\documentclass[12pt,reqno]{amsart}

\usepackage[colorlinks=true,urlcolor=blue,
bookmarks=true,bookmarksopen=true,citecolor=blue
]{hyperref}

\usepackage{color,soul}
\setstcolor{red}

\usepackage{pdfsync}
\usepackage[all, cmtip]{xy}
\usepackage{graphicx}

\usepackage{epsfig}
\usepackage{amscd}
\usepackage[mathscr]{eucal}
\usepackage{amssymb}
\usepackage{bm}
\usepackage{amsxtra}
\usepackage{amsmath}
\usepackage{latexsym} 
\usepackage[all]{xy}
\usepackage{enumerate}

\usepackage{enumitem}

\usepackage{color}
\usepackage{xcolor, soul}
\sethlcolor{green}

\usepackage{bbm}

\theoremstyle{plain}

\newtheorem{mainthm}{Theorem}
\newtheorem{thm}{Theorem}[subsection]

\newtheorem{cor}[thm]{Corollary}

\newtheorem{lem}[thm]{Lemma}
\newtheorem{prop}[thm]{Proposition}

\theoremstyle{definition}
\newtheorem{dfn}[thm]{Definition}

\newtheorem{claim}[thm]{Claim}
\newtheorem*{claim-nonum}{Claim}

\theoremstyle{remark}
\newtheorem{rem}[thm]{Remark}

\newtheorem{rems}[thm]{Remarks}

\theoremstyle{plain}

\newcommand{\hooklongrightarrow}{\lhook\joinrel\longrightarrow}

\newcommand{\cobto}{\leadsto}
\newcommand{\id}{\textnormal{id}}

\newcommand{\R}{\mathbb{R}}
\newcommand{\Z}{\mathbb{Z}}

\newcommand{\C}{\mathbb{C}}

\newcommand{\fuk}{\mathcal{F}uk}
\newcommand{\fukcob}{\mathcal{F}uk_{cob}}

\newcommand{\mor}{{\textnormal{Mor\/}}}

\newcommand{\bme}{\bm{\epsilon}}
\newcommand{\eh}{\epsilon^h}
\newcommand{\bmeh}{\bm{\epsilon}^h}
\newcommand{\ea}{\epsilon^{\mathcal{A}}}
\newcommand{\bmea}{\bm{\epsilon}^{\mathcal{A}}}
\newcommand{\emm}{\epsilon^m}
\newcommand{\bmemm}{\bm{\epsilon}^m}
\newcommand{\bmd}{\bm{\delta}}
\newcommand{\ua}{u^{\mathcal{A}}}

\newcommand{\on}{O(\mathcal{N})}
\newcommand{\onx}[1]{O({#1})}

\newcommand{\tmod}{\textnormal{mod}}
\newcommand{\rmod}{\mathrm{mod}}
\newcommand{\mdly}{\mathscr{Y}}
\newcommand{\tcn}{{\mathcal{C}one}}
\newcommand{\assmpe}{\hyperlink{h:asmp-e}{\lbe(\bmemm, \bmea)}}
\newcommand{\assmpen}{\hyperlink{h:asmp-e}{\lbe}}

\newcommand{\pbaddress}{biran@math.ethz.ch}
\newcommand{\ocaddress}{cornea@dms.umontreal.ca}
\newcommand{\esaddress}{shelukhin@dms.umontreal.ca}

\begin{document}

\title{Lagrangian shadows and triangulated categories}

\date{\today}

\thanks{The second author was supported by an individual NSERC Discovery grant. The third author was supported by an individual NSERC Discovery grant, and by the FRQNT start up grant.}

\author{Paul Biran, Octav Cornea, and Egor Shelukhin}

\address{Paul Biran, Department of Mathematics, ETH-Z\"{u}rich,
  R\"{a}mistrasse 101, 8092 Z\"{u}rich, Switzerland}\email{\pbaddress}
 
\address{Octav Cornea, Department of Mathematics and Statistics,
  University of Montreal, C.P. 6128 Succ.  Centre-Ville Montreal, QC
  H3C 3J7, Canada} \email{\ocaddress}

\address{Egor Shelukhin, Department of Mathematics and Statistics,
  University of Montreal, C.P. 6128 Succ.  Centre-Ville Montreal, QC
  H3C 3J7, Canada} \email{\esaddress}

\bibliographystyle{alphanum}

%

%

\begin{abstract} Under certain assumptions (such as weak exacteness or monotonicity) we show that splitting Lagrangians through cobordism has an energy cost and, from this cost being smaller than certain explicit bounds, we deduce some strong forms of rigidity of Lagrangian intersections.  As a consequence, we construct some new pseudo-metrics and metrics on certain classes of Lagrangian submanifolds. We also fit these constructions in a more general setting, independent of Lagrangian cobordism. As a main technical tool, we develop aspects of  the theory of (weakly) filtered $A_{\infty}$ - categories. 
 
\end{abstract}

\maketitle

%
%



\section{Introduction and main results}\label{sb:intro} 
Let $M$ be a symplectic manifold and consider  a collection $\mathcal{L}$ of Lagrangian submanifolds of $M$.  
Given a Lagrangian submanifold $L \subset M$ we are interested in the ``splitting'' (or decomposition) of $L$ into
Lagrangian submanifolds picked from the collection $\mathcal{L}$. The type of splitting that we mainly
focus on is through Lagrangian cobordisms $V$ with a single positive end equal to $L$ and multiple negative ends, $V:L\cobto (L_{1},\ldots, L_{k})$ (the definition of cobordisms is recalled just below, in \S\ref{sb:intro-lcob}). This perspective on cobordism is natural not least because, as is known from previous work \cite{Bi-Co:cob1,Bi-Co:lcob-fuk} and under appropriate constraints on $\mathcal{L}$, such  cobordisms induce genuine (iterated cone) decompositions of $L$ with factors the negative ends $L_{i}$ in the derived Fukaya category of $M$.  Cobordisms $V$  are  Lagrangian
 submanifolds of $\R^{2}\times M$ and  they have an elementary measure called shadow \cite{Co-She:metric}, $\mathcal{S}(V)$, which is the area of their projection on $\R^{2}$ together with all the bounded regions bounded by this projection.

The central point of view for this paper is to view the shadow of a cobordism $V$ as an energy cost for the  splitting corresponding to $V$. We address three natural questions from this point of view: 

\begin{itemize}
\item[I.] Assuming $L$ and $L_{1},\ldots, L_{k}$ fixed, find a lower bound for the minimal
energy cost required to split $L$ in  the factors $L_{i}$  (see Theorem \ref{t:main-A}) ? 
\item[II.] Is there some form of Lagrangian intersections rigidity that is specific to low energy splittings (see Theorem \ref{t:main-B}) ? 
\end{itemize}
A consequence of our answers to I and II is to define some new metrics and pseudo-metrics on certain classes of Lagrangian submanifolds. This leads to a third natural question.

\begin{itemize}
\item[III.] Are there some other types of energy costs that are relevant to this setting and, in particular, that also lead to associated pseudo-metrics but are independent from  cobordisms and their shadows (see \S\ref{sb:intro-wei-frag-metr})?
\end{itemize}

%
%
%

\subsection{Decomposition by Lagrangian cobordism}
\label{sb:intro-lcob}

Let $M$ be a symplectic manifold, compact or convex at infinity. A {\em Lagrangian cobordism}~\cite{Ar:cob-1}
(see~\cite{Bi-Co:cob1} for the formalism in use here) is a Lagrangian
submanifold $V \subset \mathbb{R}^2 \times M$ that, outside a compact
interval $\times \ \mathbb{R}$ looks like a finite disjoint union of
horizontal semi-infinite rays, each of which being multiplied by a
Lagrangian submanifold in $M$. More precisely, there exists a compact
interval $[a_-, a_+] \subset \mathbb{R}$ such that
$$V \setminus ([a_-,a_+] \times \mathbb{R} \times M) = 
\Bigl(\coprod_{i=1}^k \ell_- \times \{i\} \times L_i \Bigr) \coprod \,
\Bigl(\coprod_{j=1}^{k'} \ell_+ \times \{j\} \times L'_j\Bigr),$$
where $\ell_- = (-\infty, a_-]$, $\ell_+ = [a_+, \infty)$ and the
$L_i$'s and $L'_j$'s are Lagrangian submanifolds of $M$.  The $L_i$'s
are called the negative ends of $V$ and the $L'_j$'s the positive ends.
We write $V: (L'_1, \ldots, L'_{k'}) \cobto (L_1, \ldots, L_k)$.  Note
that we allow any of $k'$ or $k$ to be $0$ in which case the
positive or negative end of the cobordism is void. In what follows we
will be mainly interested in cobordisms
$V: L \cobto (L_1, \ldots, L_k)$ between a single Lagrangian $L$ and a
tuple $(L_1, \ldots, L_k)$. We will view such a cobordism as a means to
decompose (or split) $L$ into the ``factors'' (or pieces) $L_1, \ldots, L_k$.

We need to restrict the class of
Lagrangian submanifolds as follows. Denote by $\mathcal{L}ag^*(M)$ either the
class of closed Lagrangian submanifolds in $M$ that are weakly
exact or the class of closed Lagrangians in $M$ that are 
monotone (in the latter case there is an additional
constraint that will be explained later). Similarly, we will work with
cobordisms $V$ that are either weakly exact or monotone (where again,
in the monotone case there will be an additional constraint). We
denote the class of such cobordisms by
$\mathcal{L}ag^*(\mathbb{R}^2 \times M)$.

Our first theorem shows that the shadow of cobordisms
$V: L \cobto (L_1, \ldots, L_k)$ between fixed $L$ and
$(L_1, \ldots, L_k)$ cannot become arbitrarily small unless these
Lagrangians are placed in a very particular position.

\begin{mainthm} \label{t:main-A} Let $L, L_1, \ldots, L_k \subset M$
  be weakly exact Lagrangian submanifolds. Assume that
  $L \not \subset L_1 \cup \cdots \cup L_k$. Then there exists
  $\delta = \delta(L; S) > 0$ which depends only on $L$ and
  $S := L_1 \cup \cdots \cup L_k$, such that for every weakly exact
  Lagrangian cobordism $V : L \cobto (L_1, \ldots, L_k)$ we have
  \begin{equation} \label{eq:delta-L-S}
    \mathcal{S}(V) \geq \tfrac{1}{2}\delta.
  \end{equation}
\end{mainthm}
The proof is given~\S\ref{s:main-geom}. A non-technical outline of the
proof is presented in~\S\ref{sb:outline-prf-A}.

The next theorem establishes relations between $L$ and
$L_1, \ldots, L_k$ in case they are related by a Lagrangian cobordism
with small shadow.

\begin{mainthm} \label{t:main-B} Let $L, L_1, \ldots, L_k \subset M$
  be weakly exact Lagrangians and $S$ as in
  Theorem~\ref{t:main-A}. Let $N \subset M$ be another weakly exact
  Lagrangian and  assume that the family of Lagrangians $N,L, L_{1}\ldots, L_{k}$ are in general position.
  \begin{enumerate} [label=(\alph*),ref=\alph*]
  \item There
    exists $\delta' = \delta'(N, S)>0$ that depends on $N$ and
    $S$ (but not on $L$) such that for every weakly Lagrangian
    cobordism $V : L \cobto (L_1, \ldots, L_k)$ with
    $\mathcal{S}(V) < \tfrac{1}{2}\delta'$ we have:
    \begin{equation} \label{eq:N-cap-L-1} \# (N \cap L) \geq
      \sum_{i=1}^k \# (N \cap L_i).
    \end{equation}
  \item There exists $\delta'' = \delta''(N, S)>0$ that
    depends on $N$ and $S$ (but not on $L$) such that for every weakly
    exact Lagrangian cobordism $V: L \cobto (L_1, \ldots, L_k)$ with $\mathcal{S}(V) < \frac{1}{4}\ \delta''$ we
    have:
    \begin{equation} \label{eq:N-cap-L-2} \# (N \cap L) \geq
      \sum_{i=1}^k \dim_{\Lambda} HF(N, L_i).
    \end{equation}
    Here $HF(N,L_i)$ is the Floer homology of $(N,L_i)$ defined with
    coefficients in the Novikov field $\Lambda$.
  \end{enumerate}
\end{mainthm}
The numbers $\delta$, $\delta'$ and $\delta''$ are variants of the
Gromov width from~\cite{Bar-Cor:Serre}.  Namely, $\delta$ is the Gromov width of $L$ in the complement of $S$, $\delta'$ is a symplectic measure of the intersections $S\cap L$ and $\delta''$ is a symplectic measure of the double points of $S$ in the complement of $N$.
The precise definitions are given in~\S\ref{s:main-geom}, where more precise versions of
Theorems~\ref{t:main-A} and~\ref{t:main-B} are restated and proved as parts of a single statement,  Theorem \ref{thm:fission}.
A similar statement holds also in the monotone case,
see~\S\ref{s:main-geom}.

\subsection{From shadows to metrics on the space of
  Lagrangians} \label{sb:intro-shadow-metric}

Interestingly, the notion of shadow gives rise to a new class of
pseudo-metrics on the space of Lagrangians $\mathcal{L}ag^*(M)$. The
simplest of them, which we call the {\em shadow pseudo-metric}, is defined as
follows. Fix a family of Lagrangians
$\mathcal{F} \subset \mathcal{L}ag^*(M)$. The shadow pseudo-metric
$d^{\mathcal{F}}$ relative to $\mathcal{F}$ is defined
by
\begin{equation} \label{eq:dF-intro} d^{\mathcal{F}}(L,L'):= \inf_{V}
  \bigr\{ \mathcal{S}(V) \mid V : L \cobto (F_1, \ldots, F_{i-1}, L',
  F_i, \ldots, F_k), \, k \geq 0, F_i \in \mathcal{F} \bigr\},
\end{equation}
for every $L, L' \in \mathcal{L}ag^*(M)$. The infimum here is taken
only over cobordisms $V \in \mathcal{L}ag^*(\mathbb{R}^2 \times M)$,
i.e. the $V$'s are assumed to satisfy the same constraints as the
Lagrangians in $\mathcal{L}ag^*(M)$. These cobordisms are not assumed to be connected.
Here and in what follows we use
the convention that $\inf \emptyset = \infty$, so that
$d^{\mathcal{F}}(L,L') = \infty$ if there is no cobordism
$V \in \mathcal{L}ag^*(\mathbb{R}^2 \times M)$ as
in~\eqref{eq:dF-intro}.

It is easy to see that $d^{\mathcal{F}}$ is a pseudo-metric (possibly
with infinite values). However, unless $\mathcal{F} = \emptyset$,
$d^{\mathcal{F}}$ is generally degenerate (yet not identically zero).

Now fix a second family of Lagrangians
$\mathcal{F}' \subset \mathcal{L}ag^*(M)$. The preceding recipe
yields another pseudo-metric $d^{\mathcal{F}'}$. Consider now
\begin{equation} \label{eq:d-ff'} \widehat{d}^{\mathcal{F},
    \mathcal{F}'} = \max\{ d^{\mathcal{F}}, d^{\mathcal{F}'} \}.
\end{equation}

An easy consequence of Theorem~\ref{t:main-A} is that if
$(\overline{\cup_{K \in \mathcal{F}} K}) \cap (\overline{\cup_{K' \in
    \mathcal{F'}} K'})$ is totally disconnected (e.g. discrete), then
$\widehat{d}^{\mathcal{F}, \mathcal{F}'}$ is a non-degenerate metric
(possibly with infinite values) on $\mathcal{L}ag^*(M)$. We call it
the shadow metric associated to the pair of families $\mathcal{F}$,
$\mathcal{F}'$.

Shadow metrics can be viewed as generalizations of the ``simple
cobordism metric'' from~\cite{Co-She:metric} (which coincides with
$d^{\emptyset}$), which itself generalizes the Lagrangian Hofer
metric~\cite{Chek:finsler}. However, the shadow metrics are in general
different from these two metrics. In particular with appropriate
choices for the families $\mathcal{F}$, $\mathcal{F}'$, the shadow
metrics are finite for a larger collection of Lagrangians in
$\mathcal{L}ag^*(M)$. It is already known from \cite{Co-She:metric} that, 
without appropriate rigidity constraints  (such as those imposed here) on the class of 
Lagrangians under consideration, even the simple cobordism metric is degenerate.

\subsection{Weighted fragmentation pseudo-metrics on triangulated
  categories} \label{sb:intro-wei-frag-metr} The construction of the shadow pseudo-metrics
 can be generalized to a more abstract setting, as discussed in \S\ref{subsubsec:weights1}. 
 In short, given a triangulated category $\mathcal{X}$ there is an associated category, $T^{S}\mathcal{X}$ \cite{Bi-Co:lcob-fuk}, whose  objects are ordered finite  families $(K_{1},\ldots, K_{r})$, $K_{i}\in \mathcal{O}b(\mathcal{X})$, and whose morphisms parametrize the  iterated cone-decompositions of the objects in $\mathcal{X}$.  A {\em weight} on $\mathcal{X}$ associates a value $\in [0,\infty]$ to the morphisms in $T^{S}\mathcal{X}$ in a way that is sub-additive with respect to composition. The weight of a morphism $\bar{\phi}\in \mor_{T^{S}\mathcal{X}}$
 can be viewed as a measure of the energy required for the (iterated cone)-decomposition represented by $\bar{\phi}$ to take place.
   Given such a weight $w$, and  for a fixed family  $\mathcal{F}\subset \mathcal{O}b(\mathcal{X})$, we can compare two objects
 $K,K'$ in $\mathcal{X}$ by  infimizing $w(\bar{\phi})$ among all morphisms $\bar{\phi}$ in $T^{S}\mathcal{X}$,  $\bar{\phi}: K\to (F_{1},\ldots, F_{i}, K', F_{i+1},\ldots, F_{r})$, $F_{i}\in \mathcal{F}$. The resulting measurement satisfies the triangle inequality and can be symmetrized formally, thus providing a pseudo-metric
 on the objects of $\mathcal{X}$. These pseudo-metrics are called {\em weighted fragmentation pseudo-metrics}. This construction recovers the shadow pseudo-metrics by using a functor $\widetilde{\Phi}:\mathcal{C}ob^{\ast}(M)\to T^{S}D\fuk^{\ast}(M)$ from
 \cite{Bi-Co:lcob-fuk} to define a shadow-weight, $w_{\mathcal{S}}$,  on  $D\fuk^{\ast}(M)$ - see (\ref{eq:shadow-weight}). The shadow pseudo-metrics are the weighted fragmentation pseudo-metrics associated to $w_{\mathcal{S}}$.

More interestingly, this abstract construction is applied in \S\ref{subsubsec:alg-we} to a different type of weight on $D\fuk^{\ast}(M)$ whose definition is based on algebraic measurements associated to filtered chain-complexes and modules and is independent of cobordism.  By applying the method above to this weight we deduce the  existence of additional pseudo-metrics on  $\mathcal{L}ag^{\ast}(M)$,  that are all more algebraic in nature compared to their shadow counterparts. 
An additional key point, contained in Corollary \ref{cor:alg-weights}, is that, by carefully separating the geometric steps from the algebraic ones in the proof of Theorem \ref{t:main-A}, we show that an 
 analogue of this theorem remains valid for the algebraic weights.  As a consequence, the algebraic weights can also be used to define actual metrics (possibly taking infinite values) on $\mathcal{L}ag^{\ast}(M)$.

\subsection{Outline of the proof of
  Theorem~\ref{t:main-A}} \label{sb:outline-prf-A}   
  We focus here on the proof of Theorem \ref{t:main-A}. The proof of Theorem \ref{t:main-B} makes use 
of the similar ideas.  We consider a symplectic embedding
of a standard ball $e:B(r)\to M$ such that $e^{-1}(L)=B_{\R}(r)$, $e(B(r))\cap (L_{1}\cup L_{1}\cup\ldots \cup L_{k})=\emptyset$ and we put $P=e(0)$. 
The purpose of the proof is to show that for any almost complex structure $J$ on $M$ there exists a $J$-holmorphic polygon $u$ in $M$ with a boundary edge on $L$ (and possibly on the other $L_{i}$'s) going through $P$ and of symplectic area smaller or equal than $\mathcal{S}(V)$.  By a suitable choice of $J,$ an application of the Lelong inequality, and the definition of $\delta=\delta(L;S)$ as in (\ref{eq:delta1}) this implies the claim.

To be able to control energy bounds in our considerations we first set up in \S\ref{s:wf-ai-theory} the 
machinery of $A_{\infty}$-categories and modules in the (weakly) filtered setting.  Variants of this already appear in the literature,
for instance in \cite{FO3:book-vol1,FO3:book-vol2} (in somewhat different form), but we give enough details such as to be able to extend - in \S\ref{s:floer-theory} - the results from \cite{Bi-Co:lcob-fuk} to this setting. The wording {\em weakly} means that, to achieve regularity, we allow for small Hamiltonian perturbations in the definition of the various algebraic structures. As  a consequence, these structures are 
filtered only up to a system of small, controllable errors. We also prove in \S\ref{sb:wf-ic} a structural result,  Theorem \ref{t:itcones}, concerning 
iterated cones $\mathcal{K}$ of (weakly) filtered $A_{\infty}$ modules and, in particular, we show 
that each such cone admits a quasi-isomorphic model $\mathcal{K}'$ which is an iterated cone with the same factors as $\mathcal{K}$ and such that $\mathcal{K}'$ has a filtration that is well controlled with respect  to that of $\mathcal{K}$ 
and, most importantly for us,  the $\mu_{1}$ operation of  $\mathcal{K}'$ can be written explicitly in terms of higher
$\mu_{k}$'s of the underlying $A_{\infty}$-category  - see (\ref{eq:aij}). This result is based on a (weakly) filtered version of the following property of the Yoneda embedding \cite{Se:book-fukaya-categ}: for an $A_{\infty}$-module $\mathcal{M}$ and an object $Y$ there is a natural quasi-isomorphism $\mathcal{M}(Y) \cong\mor (\mathcal{Y}, \mathcal{M})$ (where $\mathcal{Y}$ is  the Yoneda module of $Y$). 
We prove  in \S\ref{sb:l-map} a weakly filtered version of this property which seems to be new 
(and somewhat delicate to prove).

With this preparation, the proof of the theorem is given in \S\ref{s:main-geom}. By neglecting a number of technicalities, the 
argument can be sketched as follows. We consider a new cobordism 
$W:\emptyset \cobto (L,L_{1},\ldots, L_{k})$ obtained from $V$ by bending the end $L$ of $V$  clokwise half a turn, as in Figure \ref{f:cob-v-w}. The main result in
\cite{Bi-Co:lcob-fuk} implies that  the Yoneda modules $\mathcal{L}$, $\mathcal{L}_{i}$ associated to the negative 
ends of the cobordism $W$ fit into an iterated cone of $A_{\infty}$-modules over the Fukaya category, $\fuk^{\ast}(M)$:
$$\mathcal{M}=\tcn(\mathcal{L}_{k} \stackrel{\varphi_{k}}{\longrightarrow} \tcn (\mathcal{L}_{k-1}\stackrel{\varphi_{k-1}}{\longrightarrow}\ldots \tcn(\mathcal{L}_{1}\stackrel{\varphi_{1}}{\longrightarrow} \mathcal{L})\ldots )~,~$$  and moreover the module $\mathcal{M}$ is acyclic. In view of our preparatory step
all the modules and structures involved here are filtered. The complex $\mathcal{M}(L)$ admits a geometric description
as follows. Let $\gamma$ be a curve in $\C$ that lies to left of the ``bulb'' of the cobordism $W$ and is crossing the negative ends of $W$ in a vertical way - as in 
Figure \ref{f:gamma-gamma'}. Then,  by neglecting in this sketch the other choices (of perturbations etc) required, we have
$\mathcal{M}(L)=CF(\gamma\times L, W)$.  There is a Hamiltonian isotopy of energy very close 
to $\mathcal{S}(W)$ that isotopes the curve $\gamma$ to a curve $\gamma'$ that passes to the right of the projection of $W$ onto $\C$ and thus $\gamma'\times L$ is disjoint from $W$ - see again Figure \ref{f:gamma-gamma'}. As a consequence, 
by standard Floer theoretic considerations, we deduce that there exists a null-homotopy $\xi$ of the identity of $\mathcal{M}(L)$ that shifts filtrations by at most $\rho\leq \mathcal{S}(V)+\epsilon$ where we can take $\epsilon$ as small as desired. 
A cycle $e_L$ in $CF(L,L)$ representing the fundamental class $[L] \in HF(L,L)$ still remains a cycle in $\mathcal{M}(L)$. We deduce that it has to be the boundary of some element in $\mathcal{M}(L)$ of filtration higher than that of $e_L$ by not more than $\rho$. In a more streamlined formalism, we say that the boundary depth (a notion introduced by Usher in \cite{Usher1} - see also \S\ref{s:filt-ch}) of 
$e_L$ is at most $\rho$.  By suitable choices of perturbation data, we arrange it so that $e_L$ is the maximum point of a Morse function on $L,$ which is achieved at $P$. For the next step it is crucial that $\mathcal{M}$ is an iterated cone of (weakly) filtered {\em modules} (in contrast to  $\mathcal{M}(L)$ only being an iterated cone of filtered chain complexes). We now use the structural Theorem \ref{t:itcones} to associate to $\mathcal{M}$ the quasi-isomorphic module $\mathcal{M}'$. Because the filtrations
on $\mathcal{M}'$ and $\mathcal{M}$ are tightly related, we deduce that the boundary depth of $e_L$ in $\mathcal{M}'(L)$
is at most $\rho+\epsilon'$ where $\epsilon'>0$ can be taken arbitrarily small. From the special form  of the differential of $\mathcal{M}'(L)$ which involves the higher order $A_{\infty}$-operations $\mu_{d}$, we conclude that there
is a pseudoholomorphic polygon $u$ in $M$ with boundary on $L$ and on some of the $L_{i}$'s that  appears in the differential of $\mathcal{M}'(L)$ and that passes through $P$. Moreover, the area of this polygon is not more than $\rho+\epsilon'$. This concludes
 the proof by making $\epsilon,\epsilon'\to 0$ (in practice, the last step of finding the $J$-holomorphic curve $u$ 
 is a bit more complicated as, for the complex $CF(\gamma\times L, W)$ to be defined, some perturbations need to be introduced
 and these have to be sent to $0$ to deduce the existence of $u$).

\begin{rem}
From the fact that  $e_L$ becomes a boundary in $\mathcal{M}(L)$ one can directly deduce that there is a holomorphic curve
passing through $P$. However, this curve is in $\R^{2}\times M$ and its boundary
is along $\gamma\times L$ and $W$ - this is not the curve $u$ we are looking for. To obtain a curve $u$ in $M$ with the 
desired properties, the algebraic Theorem \ref{t:itcones} is essential in our argument. 
We do not know a geometric argument bypassing this theorem and 
providing directly the specific form of differential of $\mathcal{M}'$ (except for the very special case of a 
cobordism $W$ with at most three ends). Notice also that the curve $u$ that is identified in the proof can be shown to be
associated to an operation $\mu_{d}$ with $d\geq 2$. \end{rem}

\subsection{Structure of the
  paper} \label{sb:intro-further-rems}
  After two sections, \S \ref{s:wf-ai-theory} algebraic 
  and \S\ref{s:floer-theory} that adjusts the Fukaya category constructions to
  the weakly filtered setting,  we assemble  in \S\ref{s:main-geom} all these tools to prove Theorem \ref{thm:fission} -  that combines the statements of both Theorems \ref{t:main-A} and \ref{t:main-B}.
The last section of the paper, \S\ref{sec:examples-metr}, contains a number of examples and calculations as well as the construction of the weighted fragmentation pseudo-metrics mentioned earlier in the introduction. 
For the most part, this section can be read directly after the introduction. Except for Corollary \ref{cor:alg-weights}, it only uses the statements of Theorems \ref{t:main-A} and \ref{t:main-B} and not their proofs
 and just a few algebraic notions  from \S\ref{s:filt-ch}.

 \subsection*{Aknowledgements.} We thank Misha Khanevsky for mentioning to us the argument in Remark \ref{rem:Misha-etc} b.  The second author thanks Luis Diogo for useful comments. The last two authors thank
 the {\em Institute for Advanced Study} and Helmut Hofer for generously hosting them there for a part of this work. The second author also thanks the {\em Forschungsinstitut f\"ur Mathematik}  for support during repeated visits to Z\"urich.

\tableofcontents 

\section{Weakly filtered $A_{\infty}$-theory}
\label{s:wf-ai-theory}

In this section we develop a general framework for weakly filtered
$A_{\infty}$-categories, with an emphasis on weakly filtered modules
over such categories. In our context ``weakly filtered'' generally
means that the morphisms in the category are filtered chain complexes
but the higher $A_{\infty}$-operations do not necessarily preserve
these filtrations. Rather they preserve them up to prescribed errors
which we call {\em discrepancies}. In the same vein one can consider
also weakly filtered $A_{\infty}$-functors and modules. Related
notions of filtered $A_{\infty}$-structures have been considered in
the literature (e.g.~\cite{FO3:book-vol1, FO3:book-vol2}), but the
existing theory seems to differ from ours in its scope
and applications.

Below we will cover only the most basic concepts of
$A_{\infty}$-theory in the weakly filtered setting. In particular we
will not go into the topics of derived categories, split closure or
generation in the weakly filtered framework. Our main goal is in fact
much more modest: to provide an effective description of iterated
cones of modules in the weakly filtered setting in terms of weakly
filtered twisted complexes.

Some readers may find the details of the weakly filtered setting
somewhat overwhelming, especially in what concerns keeping track of
the discrepancies. If one assumes all the discrepancies to vanish, the
theory becomes ``genuinely filtered'' and is easier to
follow. However, the additional difficulty due to the weakly filtered
setting is in large superficial. Indeed, significant parts of the
theory developed in this section do not become easier if one
works in the genuinely filtered setting, except in terms of notational convenience. 
We also remark that, as far
as we know, a good part of the theory developed in this section, particularly the study 
of iterated cones, is new even
in the genuinely filtered case. The reason for developing the theory
in the weakly filtered setting (rather than filtered) has to do with
the geometric applications we aim at which have to do with Fukaya
categories of symplectic manifolds. For technical reasons, the weakly
filtered framework fits better with the standard implementations of
these categories.

\subsection{Weakly filtered $A_{\infty}$-categories}
\label{sb:wf-ai-categ}

In the following we will often deal with sequences
\label{pg:conv-seq}
$\bme = (\epsilon_1, \ldots, \epsilon_d, \ldots)$ of real numbers that
we will refer to as discrepancies. We will use the following
abbreviations and conventions. For two sequences $\bme$, $\bme'$ we
write $\bme \leq \bme'$ in order to say that
$\epsilon_d \leq \epsilon'_d$ for all $d$. For $c \in \mathbb{R}$ we
write $\bme + c$ for the sequence
$(\epsilon_1+c, \ldots, \epsilon_d+c, \ldots)$.  For a finite number
of sequences $\bme^{(1)}, \ldots, \bme^{(r)}$ we define
$\max\{\bme^{(1)}, \ldots, \bme^{(r)} \}$ to be the sequence
$\bme = (\epsilon_1, \ldots, \epsilon_d, \ldots)$ with
$\bme_d := \max\{\epsilon^{(1)}_d, \ldots, \epsilon^{(r)}_d\}$.

Fix a commutative ring $R$, which for simplicity we will henceforth
assume to be of characteristic $2$ (i.e. $2r=0$ for all $r \in R$).
Unless otherwise stated all tensor products will be taken over $R$.

The $A_{\infty}$-theory developed below will be carried out in the
ungraded framework. Also, in contrast to standard texts on the subject
such as~\cite{Se:book-fukaya-categ}, we will work in a homological
(rather than cohomological) setting, following the conventions
from~\cite{Bi-Co:lcob-fuk}.

Let $\mathcal{A}$ be an $A_{\infty}$-category over $R$. To simplify
notation, in what follows we will denote the morphisms between two
objects $X,Y \in \textnormal{Ob}(\mathcal{A})$ by
$C(X,Y) := \hom_{\mathcal{A}}(X,Y)$. We denote the composition maps of
$\mathcal{A}$ by $\mu^{\mathcal{A}}_d$, $d \geq 1$.

Let
$\bm{\epsilon}^{\mathcal{A}} = (\epsilon_1^{\mathcal{A}},
\epsilon_2^{\mathcal{A}}, \ldots, \epsilon_d^{\mathcal{A}}, \ldots)$
be an infinite sequence of non-negative real numbers, with
$\epsilon_1^{\mathcal{A}}=0$. We call $\mathcal{A}$ a {\em weakly
  filtered} $A_{\infty}$-category with discrepancy
$\leq \bm{\epsilon}^{\mathcal{A}}$ if the following holds:
\begin{enumerate}
\item For every $X, Y \in \textnormal{Ob} (\mathcal{A})$, $C(X,Y)$ is
  endowed with an increasing filtration of $R$-modules indexed by the
  real numbers. We denote by $C^{\leq \alpha}(X,Y) \subset C(X,Y)$ the
  part of the filtration corresponding to $\alpha \in \mathbb{R}$. By
  {\em increasing filtration} we mean that
  $C^{\leq \alpha'}(X,Y) \subset C^{\leq \alpha''}(X,Y)$ for every
  $\alpha'\leq \alpha''$.
\item The $\mu_d$-operation preserves the filtration up to an
  ``error'' of $\epsilon_d^{\mathcal{A}}$. More precisely, for every
  $X_0, \ldots, X_d \in \textnormal{Ob}(\mathcal{A})$ and
  $\alpha_1, \ldots, \alpha_d \in \mathbb{R}$ we have:
  $$\mu_d \bigl( C^{\leq \alpha_1}(X_0, X_1) \otimes \cdots \otimes C^{\leq
    \alpha_d}(X_{d-1},X_d) \bigr) \subset C^{\leq \alpha_1 + \cdots +
    \alpha_d + \epsilon_d^{\mathcal{A}}} (X_0, X_d).$$
\end{enumerate}
Note that since $\epsilon_1^{\mathcal{A}}=0$, $\mu_1^{\mathcal{A}}$
preserves the filtration, hence each $C^{\leq \alpha}(X,Y)$,
$\alpha \in \mathbb{R}$, is a sub-complex of $C(X,Y)$. Note also that
the discrepancy is not uniquely defined -- in fact we can always
increase it if needed. Namely, if
$\bm{\epsilon}' = (\epsilon'_1=0, \epsilon'_2, \ldots, \epsilon'_d,
\ldots)$ is another sequence like $\bm{\epsilon}^{\mathcal{A}}$ but
with $\bmea \leq \bme'$ then $\mathcal{A}$ is also weakly filtered
with discrepancy $\leq \bm{\epsilon}'$.

By analogy with symplectic topology we will often refer to the index of
the filtration as an {\em action} and say that elements of
$C^{\leq \alpha}(X,Y)$ have action $\leq \alpha$.

\subsubsection{Unitality} \label{sbsb:unitality-A} Let $\mathcal{A}$
be a weakly filtered $A_{\infty}$-category and assume that
$\mathcal{A}$ is homologically unital (h-unital for
short). \label{pg:h-unital-categ} We say that $\mathcal{A}$ is
h-unital in the weakly filtered sense if the following holds. There
exists $u^{\mathcal{A}} \in \mathbb{R}_{\geq 0}$ such that for every
$X \in \textnormal{Ob}(\mathcal{A})$ we have a cycle
$e_X \in C^{\leq u^{\mathcal{A}}}(X,X)$ representing the homology unit $[e_X] \in H(C(X,X), \mu^{\mathcal{A}}_1)$. We view the choices of
$e_X$, $X \in \textnormal{Ob}(\mathcal{A})$ and $u^{\mathcal{A}}$ as
part of the data of a weakly filtered h-unital
$A_{\infty}$-category. We call $u^{\mathcal{A}}$ the discrepancy of
the units.

Occasionally we will have to impose the following additional
assumption on $\mathcal{A}$.

\newcommand{\lbue}{U^e} \hypertarget{h:asmp-ue}{\noindent
  \textbf{Assumption~$\lbue$.}}  \label{l:assm-ue} Let $\mathcal{A}$
be a weakly filtered $A_{\infty}$-category which is h-unital in the
weakly filtered sense. Let
$2u^{\mathcal{A}} + \epsilon_2^{\mathcal{A}} \leq \zeta \in
\mathbb{R}$. We say that $\mathcal{A}$ satisfies
Assumption~$\lbue(\zeta)$ if for every
$X \in \textnormal{Ob}(\mathcal{A})$ we have:
$$\mu_2^{\mathcal{A}}(e_X, e_X) = e_X + \mu_1^{\mathcal{A}}(c)$$ for some 
$c \in C^{\leq \zeta}(X,X)$. Put in different words, the assumption
$\lbue$ says that $[e_X] \cdot [e_X] = [e_X]$ in
$H_*(C^{\leq \zeta}(X,X))$, where $\cdot$ stands for the product
induced by $\mu_2^{\mathcal{A}}$ in homology. (The superscript $e$ in
$\lbue$ indicates that the assumption deals with the cycles $e_X$
representing the units.) Below we will sometimes write
$\mathcal{A} \in \lbue(\zeta)$ to say that $\mathcal{A}$ satisfies
Assumption~$\lbue(\zeta)$.

\subsection{Typical classes of examples} \label{sb:typ-exp} Before we
go on with the general algebraic theory of weakly filtered
$A_{\infty}$-structures, it might be useful to make a short digression
in order to exemplify what types of filtrations will actually occur in
our applications. We resume with the general algebraic theory
in~\S\ref{sb:functors} below.

The weakly filtered $A_{\infty}$-categories that will appear in our
applications are Fukaya categories associated to symplectic
manifolds. They will mostly be of the following types, described
in~\S\ref{sbsb:filtration-action} - \ref{sbsb:families-class} below.

\subsubsection{Filtrations induced by an ``action'' functional on the
  generators} \label{sbsb:filtration-action} In this class of weakly
filtered $A_{\infty}$-categories the morphisms $C(X,Y)$ between two
objects are assumed to be free $R$-modules with a distinguished basis
$B(X,Y)$, i.e.  $C(X,Y) = \bigoplus_{b \in B(X,Y)} Rb$. We also have a
function $\mathbf{A}: B(X,Y) \longrightarrow \mathbb{R}$, which (by
analogy to symplectic topology) we call the {\em action function},
defined for every $X, Y \in \textnormal{Ob}(\mathcal{A})$, and this
function induces the filtration, namely:
$$C^{\leq \alpha}(X,Y) = \bigoplus_{b \in B(X,Y), \mathbf{A}(b)\leq \alpha} Rb.$$
We will mostly assume that $C(X,Y)$ has finite rank and that $R$ is a
field.

\subsubsection{Filtration coming from the Novikov
  ring} \label{sbsb:filtration-nov} Here we fix a commutative ring $A$ and
consider the (full) Novikov ring over $A$:
\begin{equation} \label{eq:Nov-ring}
  \Lambda = \Bigl\{ \sum_{k=0}^{\infty} a_k T^{\lambda_k} \mid a_k \in A, 
  \lim_{k \to \infty} \lambda_k = \infty \Bigr\},
\end{equation}
as well as the positive Novikov ring:
\begin{equation} \label{eq:Nov-ring-0}
  \Lambda_0 = \Bigl\{ \sum_{k=0}^{\infty} a_k T^{\lambda_k} \mid a_k \in A, 
  \lambda_k\geq 0, \lim_{k \to \infty} \lambda_k = \infty \Bigr\}.
\end{equation}

The weakly filtered $A_{\infty}$-categories $\mathcal{A}$ of the type
discussed here are defined over $\Lambda$, but the weakly filtered
structure is only over the ring $R = \Lambda_0$. 

As in~\S\ref{sbsb:filtration-action} above, we have
$C(X,Y) = \bigoplus_{b \in B(X,Y)} \Lambda b$. The filtration on
$C(X,Y)$ is then defined by
$$C^{\leq \alpha}(X,Y) = 
\bigoplus_{b \in B(X,Y)} T^{-\alpha} \Lambda_0 b.$$ Note that
$C^{\leq a}(X,Y)$ is not a $\Lambda$-module but rather a
$\Lambda_0$-module.

We will mostly assume that $B(X,Y)$ are finite (hence $C(X,Y)$ have
finite rank) and that $A$ is a field (in which case $\Lambda$ is a
field too).

\subsubsection{Mixed filtration} \label{sbsb:filtration-mixed} In some
situations the filtrations on our $A_{\infty}$-categories occur as
combination of~\S\ref{sbsb:filtration-action}
and~\S\ref{sbsb:filtration-nov} above.  More specifically, we have
$C(X,Y) = \Lambda B(X,Y)$ as in~\S\ref{sbsb:filtration-nov} and an
action functional $\mathbf{A}: B(X,Y) \longrightarrow \mathbb{R}$ as
in~\S\ref{sbsb:filtration-action}. We then extend $\mathbf{A}$ to a
functional
$\mathbf{A}: C(X,Y) = \Lambda \cdot B(X,Y) \longrightarrow \mathbb{R}
\cup \{-\infty\}$ by first setting $\mathbf{A}(0) = -\infty$. Then for
$P(T) \in \Lambda$ and $b \in B(X,Y)$ we define:
$$\mathbf{A}(P(T)b) := -\lambda_0 + \mathbf{A}(b),$$ where
$\lambda_0 \in \mathbb{R}$ is the minimal exponent that appears in the
formal power series of $P(T) \in \Lambda$, i.e.
$P(T) = a_0 T^{\lambda_0} + \sum_{i=1}^{\infty} a_i T^{\lambda_i}$
with $a_0 \neq 0$ and $\lambda_i > \lambda_0$ for every $i\geq 1$.
Finally, for a general non-trivial element
$c = P_1(T)b_1 + \cdots + P_l(T)b_l \in C(X,Y)$, define
$$\mathbf{A}(c) = \max \{ \mathbf{A}(P_k(T) b_k) \mid 1\leq k \leq
l\}.$$ The filtration on $C(X,Y)$ is then induced by $\mathbf{A}$:
$$C^{\leq \alpha}(X,Y) = 
\{c \in C(X,Y) \mid \mathbf{A}(c) \leq \alpha \}.$$ It is easy to see
that $C^{\leq \alpha}(X,Y)$ is a $\Lambda_0$-module.

\subsubsection{Families of weakly filtered
  $A_{\infty}$-categories} \label{sbsb:families-class}

The weakly filtered $A_{\infty}$-categories in our applications will
naturally occur in families $\{ \mathcal{A}_s \}_{s \in \mathcal{P}}$
parametrized by choices of auxiliary structures $s$ needed to define
the $A_{\infty}$-structure. The parameter $s$ will typically vary over
a subset $\mathcal{P} \subset E \setminus \{0\}$ where $E$ is a
neighborhood of $0$ in a Banach (or Fr\'{e}chet) space.  The subset
$\mathcal{P}$ will usually be residual (in the sense of Baire) so that
$0$ is in the closure of $\mathcal{P}$.

Typically all the members of the family
$\{ \mathcal{A}_s\}_{s \in \mathcal{P}}$ will be mutually
quasi-equivalent (see~\cite[Section~10]{Se:book-fukaya-categ} for
several approaches to families of $A_{\infty}$-categories). Of course,
in the weakly filtered setting the quasi-equivalences between
different $\mathcal{A}_s$'s are supposed to bear some compatibility
with respect to the weakly filtered structures on the
$\mathcal{A}_s$'s.

Apart from the above, in our applications the families
$\{\mathcal{A}_s \}_{s \in \mathcal{P}}$ will enjoy the following
additional property which will be crucial. The bounds
$\bme^{\mathcal{A}_s}$ for the discrepancies of the $\mathcal{A}_s$'s
can be chosen such that:
$$\lim_{s \to 0} \epsilon_d^{\mathcal{A}_s} = 0, \; \forall d.$$
Moreover, the categories $\mathcal{A}_s$ will mostly be h-unital with
discrepancy of units $u^{\mathcal{A}_s}$ and satisfy
Assumption~$\hyperlink{h:asmp-ue}{\lbue(\zeta_s)}$. The latter two
quantities will satisfy
$$\lim_{s \to 0} u^{\mathcal{A}_s} = \lim_{s \to 0} \zeta_s = 0.$$

Below we will encounter further notions in the framework of weakly
filtered $A_{\infty}$-categories such as weakly filtered functors and
modules. Each of these comes with its own discrepancy sequence $\bme$. In our
applications everything will occur in families and we will usually
have $\lim_{s \to 0} \epsilon_d(s) = 0$ for each $d$.

While the algebraic theory below is developed without a priori
assumptions on the size of discrepancies, it might be useful to view
the discrepancies as quantities that can be made arbitrarily small.

\subsubsection{The case of Fukaya
  categories} \label{sbsb:fukaya-class} The general description
in~\S\ref{sbsb:families-class} applies to the case of Fukaya
categories which will be central in our applications. More
specifically, in order to define the $A_{\infty}$-structure of Fukaya
categories one has to make choices of perturbation data (e.g. choices
of almost complex structures as well as Hamiltonian perturbation --
see e.g.~\cite[Sections~8, 9]{Se:book-fukaya-categ}). The space
$\mathcal{P}$ will consist of those perturbation data that are regular
(or admissible). This is normally a 2'nd category subset of the space
of all perturbations $E$.  The discrepancies occur as ``error''
curvature terms (associated to the perturbations) when defining the
$\mu_d$-operations. These discrepancies can be made arbitrarily small
(for a fixed $d$) by choosing smaller and smaller perturbations. The
same holds for discrepancy of the units and the $\zeta_s$'s.

\subsection{Weakly filtered $A_{\infty}$-functors and modules} \label{sb:functors}
Let $\mathcal{A}$, $\mathcal{B}$ two weakly filtered
$A_{\infty}$-categories and
$\mathcal{F}: \mathcal{A} \longrightarrow \mathcal{B}$ an
$A_{\infty}$-functor. Let
$\bme^{\mathcal{F}} = (\epsilon_1^{\mathcal{F}},
\epsilon_2^{\mathcal{F}}, \ldots, \epsilon_d^{\mathcal{F}}, \ldots)$
be a sequence of non-negative real numbers. In contrast to $\bmea$ and
$\bme^{\mathcal{B}}$ we do allow here that
$\epsilon_1^{\mathcal{F}} \neq 0$. We say that $\mathcal{F}$ is a
weakly filtered $A_{\infty}$-functor with discrepancy
$\leq \bme^{\mathcal{F}}$ if for all
$X_0, \ldots, X_d \in \textnormal{Ob}(\mathcal{A})$ and
$\alpha_1, \ldots, \alpha_d \in \mathbb{R}$ we have:
\begin{equation} \label{eq:wf-func}
  \mathcal{F}_d \bigl( C_{\mathcal{A}}^{\leq \alpha_1}(X_0, X_1) 
  \otimes \cdots \otimes C_{\mathcal{A}}^{\leq \alpha_d}(X_{d-1},X_d)
  \bigr) \subset C_{\mathcal{B}}^{\leq \alpha_1 + \cdots + \alpha_d +
    \epsilon_d^{\mathcal{F}}} (\mathcal{F}X_0, \mathcal{F}X_d).
\end{equation}
Here we have denoted by $C_{\mathcal{A}}$ and $C_{\mathcal{B}}$ the
$\hom$'s in $\mathcal{A}$ and $\mathcal{B}$ respectively and by
$\mathcal{F}_d$ the higher order terms of the functor $\mathcal{F}$.

There is also a notion of weakly filtered natural transformations
between weakly filtered functors but we will not go into this now as
our main focus will be on a special case -- weakly filtered modules
and weakly filtered morphisms between them.

\begin{rem} \label{r:wf-func} One could attempt to generalize the
  notion of weakly filtered functors to allow also an additional shift
  $\rho \in \mathbb{R}$ in the action filtration, namely replace the
  right-hand side of~\eqref{eq:wf-func} by
  $C_{\mathcal{B}}^{\leq \alpha_1 + \cdots + \alpha_d + \rho +
    \epsilon_d^{\mathcal{F}}} (\mathcal{F}X_0, \mathcal{F}X_d)$. While
  this might make sense theoretically, it does not seem to be in line
  with our approach (viewing the discrepancies as small errors) since
  the $A_{\infty}$-identities defining $A_{\infty}$-functors do not
  have the correct ``homogeneity'' properties with respect to $\rho$.
\end{rem}

\subsubsection{Weakly filtered modules} \label{sb:mod} Let $\mathcal{A}$
be a weakly filtered $A_{\infty}$-category with discrepancy
$\bm{\epsilon}^{\mathcal{A}}$. Let $\mathcal{M}$ be an
$\mathcal{A}$-module with composition maps $\mu^{\mathcal{M}}_d$,
$d \geq 1$. Let
$\bm{\epsilon}^{\mathcal{M}} = (\epsilon^{\mathcal{M}}_1,
\epsilon^{\mathcal{M}}_2, \ldots, \epsilon^{\mathcal{M}}_d, \ldots)$
be an infinite sequence of non-negative real numbers with
$\epsilon^{\mathcal{M}}_1=0$. We say that $\mathcal{M}$ is weakly
filtered with discrepancy $\leq \bm{\epsilon}^{\mathcal{M}}$ if the
following holds:
\begin{enumerate}
\item For every $X \in \textnormal{Ob}(\mathcal{A})$, $\mathcal{M}(X)$
  is endowed with an increasing filtration
  $\mathcal{M}^{\leq \alpha}(X)$ indexed by $\alpha \in \mathbb{R}$.
\item The $\mu^{\mathcal{M}}_d$-operation respects the filtration up
  to an ``error'' of $\epsilon^{\mathcal{M}}_d$. Namely, for all
  $X_0, \ldots, X_{d-1} \in \textnormal{Ob}(\mathcal{A})$ and
  $a_1, \ldots, a_d \in \mathbb{R}$ we have:
  $$\mu^{\mathcal{M}}_d \bigl( C^{\leq \alpha_1}(X_0, X_1) \otimes 
  \cdots \otimes C^{\leq \alpha_{d-1}}(X_{d-2}, X_{d-1}) \otimes
  \mathcal{M}^{\leq \alpha_d}(X_{d-1}) \bigr) \subset
  \mathcal{M}^{\leq \alpha_1 + \cdots + \alpha_d +
    \epsilon^{\mathcal{M}}_d} (X_0).$$
\end{enumerate}
Again, since $\epsilon^{\mathcal{M}}_1=0$, every
$(\mathcal{M}^{\leq \alpha}(X), \mu^{\mathcal{M}}_1)$ is a sub-complex
of $(\mathcal{M}(X), \mu^{\mathcal{M}}_1)$.

\begin{rems} \label{r:modules} 
  \begin{enumerate}
  \item It is easy to see that weakly filtered $\mathcal{A}$-modules
    are the same as weakly filtered functors
    $\mathcal{F}: \mathcal{A} \longrightarrow
    Ch_{\textnormal{f}}^{\textnormal{opp}}$ (having some
    discrepancy). Here $Ch_{\textnormal{f}}$ is the dg-category of
    filtered chain complexes (of $R$-modules) and
    $Ch_{\textnormal{f}}^{\textnormal{opp}}$ stands for its opposite
    category. (Note that $Ch_{\textnormal{f}}$ and
    $Ch_{\textnormal{f}}^{\textnormal{opp}}$ are in fact {\em
      filtered} dg-categories, i.e. they have discrepancies $0$.)  The
    correspondence between weakly filtered functors and weakly
    filtered modules is the same as in the ``unfiltered''
    case~\cite[Section~(1j)]{Se:book-fukaya-categ}. Note that if
    $\mathcal{F}: \mathcal{A} \longrightarrow
    Ch_{\textnormal{f}}^{\textnormal{opp}}$ has discrepancy
    $\leq \bme^{\mathcal{F}}$ then the weakly filtered module
    $\mathcal{M}$ corresponding to it has discrepancy
    $\leq \bme^{\mathcal{M}}$ with
    $\epsilon^{\mathcal{M}}_d = \epsilon^{\mathcal{F}}_{d-1}$ for
    every $d \geq 2$.
  \item In what follows it would be sometimes convenient to fix one
    discrepancy which applies to several weakly filtered modules
    together. In that case we will denote it by $\bmemm$ rather than
    by $\bm{\epsilon}^{\mathcal{M}}$ (the superscript $m$ stands for
    ``modules''). Note that we can always increase the bound on
    discrepancy in the sense that if $\mathcal{M}$ has discrepancy
    $\leq \bm{\epsilon}^{\mathcal{M}}$ than it also has discrepancy
    $\leq \bm{\epsilon}$ for every sequence $\bm{\epsilon}$ with
    $\bme \geq \bme^{\mathcal{M}}$. Therefore, whenever dealing with a
    finite collection of weakly filtered modules we can bound from
    above all their discrepancies by a single sequence.
  \end{enumerate}
\end{rems}

Next we define morphisms between weakly filtered
$\mathcal{A}$-modules. Let $\mathcal{M}_0, \mathcal{M}_1$ be two
weakly filtered $\mathcal{A}$-modules, both with discrepancy
$\leq \bm{\epsilon}^m$. Let
$f: \mathcal{M}_0 \longrightarrow \mathcal{M}_1$ be a pre-module
homomorphism. We write $f=(f_1, \ldots, f_d, \ldots)$ where the
$f_d$-component is a an $R$-linear map
$$f_d: C(X_0, X_1) \otimes \cdots C(X_{d-2}, X_{d-1}) \otimes 
\mathcal{M}_0(X_{d-1}) \longrightarrow \mathcal{M}_1(X_0).$$ Let
$\alpha \in \mathbb{R}$ and
$\bm{\epsilon}^f = (\epsilon^f_1, \ldots, \epsilon^f_d, \ldots)$ be a
vector of non-negative real numbers. In contrast to
$\bm{\epsilon}^{\mathcal{A}}$ and $\bm{\epsilon}^m$ we {\em do allow}
that $\epsilon^f_1 \neq 0$. We say that $f$ shifts action by
$\leq \rho$ and has discrepancy $\leq \bm{\epsilon}^f$ if for every
$d$, the map $f_d$ shifts action by not more than $\rho+\epsilon_d^f$,
namely:
$$f_d \bigl( C^{\leq \alpha_1} (X_0, X_1) \otimes \cdots 
C^{\leq \alpha_{d-1}}(X_{d-2}, X_{d-1}) \otimes \mathcal{M}_0^{\leq
  \alpha_d}(X_{d-1}) \bigr) \subset \mathcal{M}_1^{a_1+\cdots+a_d +
  \rho + \epsilon_d^f}(X_0).$$ We will generally refer to such $f$'s
as weakly filtered pre-module homomorphisms. 

As before, if $\rho \leq \rho'$ and $\bme^f \leq \bme$ then $f$ also
shifts filtration by $\leq \rho'$ and has discrepancy
$\leq \bm{\epsilon}$.

We will now attempt to define a filtration on the totality of
pre-module homomorphisms. Denote by
$\hom(\mathcal{M}_0, \mathcal{M}_1)$ the pre-module homomorphisms
$\mathcal{M}_0 \longrightarrow \mathcal{M}_1$. The filtration will
depend on an additional ``discrepancy'' parameter $\bmeh$ which is a
sequence
$\bm{\epsilon}^h = (\epsilon^h_1, \epsilon^h_2, \ldots, \epsilon^h_d,
\ldots)$ of non-negative real numbers (the superscript $h$ stands for
``homomorphisms''). Again, we {\em do not} assume here that
$\epsilon^h_1$ is $0$. Our filtration is indexed by $\mathbb{R}$ and
is defined as follows. The part of the filtration corresponding to
$\rho \in \mathbb{R}$ is denoted by
$\hom^{\leq \rho; \bm{\epsilon}^h}(\mathcal{M}_0, \mathcal{M}_1)$ and
consists of all pre-module homomorphisms
$f: \mathcal{M}_0 \longrightarrow \mathcal{M}_1$ which shift action by
not more than $\rho$ and have discrepancy $\leq \bmeh$. Clearly this
yields an increasing filtration on
$\hom(\mathcal{M}_0, \mathcal{M}_1)$. Note however, that in general
this filtration might not be exhaustive since not every pre-module
homomorphism must be weakly filtered.

A few words are in order about the structure of the totality of weakly
filtered $\mathcal{A}$-modules. Recall that $\mathcal{A}$-modules form
a dg-category $\textnormal{mod}_{\mathcal{A}}$ with composition maps
$\mu_1^{\textnormal{mod}}$, $\mu_2^{\textnormal{mod}}$
(See~\cite[Section~(1j)]{Se:book-fukaya-categ} for the
definitions). We would like to view the weakly filtered modules (say
of given discrepancy) as a sub-category and define a weakly filtered
structure on it. However, when trying do that one encounters the
following problem. For general choices of
$\bm{\epsilon}^{\mathcal{A}}$, $\bm{\epsilon}^{m}$ and
$\bm{\epsilon}^h$ and two weakly filtered modules $\mathcal{M}_0$,
$\mathcal{M}_1$ with discrepancy $\leq \bm{\epsilon}^m$ the
differential $\mu_1^{\textnormal{mod}}$ does not preserve
$\hom^{\leq \rho; \bm{\epsilon}^h}(\mathcal{M}_0,
\mathcal{M}_1)$. Nevertheless it is possible to correct this problem
by restricting the choice of $\bm{\epsilon}^h$ as follows:

\medskip \newcommand{\lbe}{\mathcal{E}}
\hypertarget{h:asmp-e}{\noindent \textbf{Assumption~$\lbe$.}}  
\label{pp:assump-E}
A sequence
$\bm{\varepsilon} = (\varepsilon_1, \ldots, \varepsilon_d, \ldots)$ is
said to satisfy Assumption~$\lbe$ if for every $d \geq 1$ we have
$\varepsilon_d \geq \max \bigl\{ \epsilon^m_i + \varepsilon_j, \,
\epsilon_i^{\mathcal{A}} + \varepsilon_j \mid i+j=d+1 \bigr\}$.

Sometimes we will need to emphasize the dependence of
Assumption~$\lbe$ on the choices of $\bmea$ and $\bmemm$ in which case
we will refer to it as 
Assumption~\hyperlink{h:asmp-e}{$\lbe(\bmemm, \bmea)$}. Alternatively
we will sometimes write
$\bm{\varepsilon} \in~\hyperlink{h:asmp-e}{\lbe(\bmemm, \bmea)}$.

\medskip

An inspection of definition of $\mu_1^{\textnormal{mod}}$ (see
e.g.~\cite[Section~(1j)]{Se:book-fukaya-categ}) shows that if
$\bm{\epsilon}^h$ satisfies Assumption~\hyperlink{h:asmp-e}{$\lbe$}
then $\hom^{\leq \rho; \bm{\epsilon}^h}(\mathcal{M}_0, \mathcal{M}_1)$
is preserved by $\mu_1^{\textnormal{mod}}$ hence is a chain complex.

The following will be useful later on:
\begin{lem} \label{l:choice-eps-1} Given $\bm{\epsilon}^{\mathcal{A}}$
  and $\bm{\epsilon}^m$ it is always possible to find
  $\bm{\varepsilon}$ that satisfies
  Assumption~\hyperlink{h:asmp-e}{$\lbe$}. Moreover, there exists a
  sequence of real numbers $\{A_d\}_{d \in \mathbb{N}}$ which is
  universal in the sense that it does not depend on $\bmea$ or
  $\bmemm$ and has the following property: for any given sequence
  $\bm{\delta} = (\delta_1, \ldots, \delta_d, \ldots)$ of non-negative
  real numbers there exists an $\bm{\varepsilon}$ that satisfies
  {Assumption~\hyperlink{h:asmp-e}{$\lbe$}} and such that for all $d$:
  \begin{equation} \label{eq:ineq-Ad} \delta_d \leq \varepsilon_d \leq
    A_d \sum_{j=1}^d (\ea_j + \emm_j + \delta_j).
  \end{equation}
\end{lem}

\begin{rem} \label{r:universal-1} The bound from above on
  $\varepsilon_d$ in~\eqref{eq:ineq-Ad} becomes useful if for every
  $d$ all the quantities
  $\epsilon^{\mathcal{A}}_j, \epsilon^m_j, \delta_j$,
  $j=1, \ldots, d$, can simultaneously be made arbitrarily small. This
  will indeed be the case in our applications where the weakly
  filtered $A_{\infty}$-categories and modules appear in families -
  see~\S\ref{sbsb:families-class} for further details.
\end{rem}

\begin{proof}[Proof of Lemma~\ref{l:choice-eps-1}]
  One can easily construct $\varepsilon_d$ and $A_d$ inductively:
  start with $\varepsilon_1 := \delta_1$ then set
  $\varepsilon_2 := \max\{\epsilon^m_2+\varepsilon_1,
  \epsilon^{\mathcal{A}}_2 + \varepsilon_1, \delta_2\}$ and so
  on. (Note that $\epsilon^{\mathcal{A}}_1=\epsilon^m_1=0$ so that the
  inequality in Assumption~\hyperlink{h:asmp-e}{$\lbe$} is obviously
  satisfied for $i=1$, $j=d$.)
\end{proof}

\begin{rems} \label{r:eh}
  \begin{enumerate}
  \item If $\bm{\varepsilon}$ satisfies
    {Assumption~\hyperlink{h:asmp-e}{$\lbe$}} then so does
    $\widetilde{\bm{\varepsilon}}: = \bm{\varepsilon} +
    \textnormal{const}$.
  \item When dealing with $\hom^{\leq \rho; \bmeh}$ we can always
    arrange that $\epsilon^h_1=0$ by applying the following
    procedure. Suppose that $\bm{\epsilon}^h$ satisfies
    Assumption~\hyperlink{h:asmp-e}{$\lbe$}. Put
    $$\widetilde{\rho} := \rho + \epsilon^h_1, \quad 
    \widetilde{\epsilon}^h_d := \epsilon^h_d - \epsilon^h_1.$$ Note
    that $\widetilde{\epsilon}^h_1 = 0$,
    $\widetilde{\epsilon}^h_d \geq 0$ and that
    $\widetilde{\bm{\epsilon}}^h$ still satisfies
    Assumption~\hyperlink{h:asmp-e}{$\lbe$}. It is easy to see that
    $$\hom^{\leq \widetilde{\rho};
      \widetilde{\bm{\epsilon}}^h}(\mathcal{M}_0, \mathcal{M}_1) =
    \hom^{\leq \rho; \bm{\epsilon}^h}(\mathcal{M}_0,
    \mathcal{M}_1).$$ \label{i:rems-eh}
  \end{enumerate}
\end{rems}

We now turn to the $\mu_2^{\tmod}$ operation. Let $\mathcal{M}_0$,
$\mathcal{M}_1$, $\mathcal{M}_2$ be weakly filtered
$\mathcal{A}$-modules with discrepancy $\leq \bmemm$. Let
$f: \mathcal{M}_0 \longrightarrow \mathcal{M}_1$,
$g: \mathcal{M}_1 \longrightarrow \mathcal{M}_2$ be two weakly
filtered pre-module homomorphism with
$f \in \hom^{\leq \rho^f; \bme^f}(\mathcal{M}_0, \mathcal{M}_1)$,
$g \in \hom^{\leq \rho^g; \bme^g}(\mathcal{M}_1, \mathcal{M}_2)$. Set
$\varphi := \mu_2^{\tmod}(f,g): \mathcal{M}_0 \longrightarrow
\mathcal{M}_2$. A simple calculation shows that $\varphi$ is weakly
filtered and that
$\varphi \in \hom^{\leq \rho^f + \rho^g; \bme^f *
  \bme^g}(\mathcal{M}_0, \mathcal{M}_2)$, where the sequence of
discrepancies $\bme^f * \bme^g$ is defined as:
\begin{equation} \label{eq:disc-f*g}
  (\bme^f * \bme^g)_d = \max \{\epsilon_i^f + \epsilon_j^g \mid i+j=d+1\}.
\end{equation}
Moreover, a simple calculation shows that if
$\bme^f, \bme^g \in~\assmpe$ then the same holds for
$\bme^f * \bme^g$.

\subsubsection{Do weakly filtered $\mathcal{A}$-modules form a
  dg-category?} \label{sbsb:wfm-categ} Fix $\bmemm$ and
$\bmeh \in \assmpe$. As explained in the preceding section, for weakly
filtered $\mathcal{A}$-modules
$\mathcal{M}_0, \mathcal{M}_1, \mathcal{M}_2$ with discrepancy
$\leq \bmemm$, the operation $\mu_2^{\tmod}$ satisfies:
$$\mu_2^{\tmod}: \hom^{\leq \rho'; \bmeh}(\mathcal{M}_0, \mathcal{M}_1) \otimes 
\hom^{\leq \rho''; \bmeh}(\mathcal{M}_1, \mathcal{M}_2)
\longrightarrow \hom^{\leq \rho'+\rho''; \bmeh * \bmeh}(\mathcal{M}_0,
\mathcal{M}_2),$$ where $\bmeh*\bmeh$ is defined
by~\eqref{eq:disc-f*g}.  Thus in order for the $\mu_2^{\tmod}$
operation to have values in
$\hom^{\rho'+\rho''; \bmeh}(\mathcal{M}_0, \mathcal{M}_2)$ we need
$\bmeh*\bmeh \leq \bmeh$. It is easy to see that such an inequality
cannot hold unless $\eh_1=0$ (the latter condition being only a
necessary one).

One way to solve this problem is based on the fact that
$$\hom^{\leq \rho'+\rho''; \bmeh * \bmeh}(\mathcal{M}_0, \mathcal{M}_2)
= \hom^{\leq \rho' + \rho'' + \eh_1; \bmeh*\bmeh-\eh_1}(\mathcal{M}_0,
\mathcal{M}_2).$$ Therefore, we only need to arrange that
$\bmeh*\bmeh-\eh_1 \leq \bmeh$. Unwrapping this gives the following
system of inequalities. For all $d\geq 1$:
\begin{equation} \label{eq:eheh-eh1} \eh_d +\eh_1 \geq \eh_i + \eh_j,
  \; \; \forall \; i+j=d+1.
\end{equation}

It is easy to see (by induction on $d$) that sequences $\bmeh$
satisfying both Assumption~$\assmpen$ and~\eqref{eq:eheh-eh1} do
exist. In fact, Lemma~\ref{l:choice-eps-1} continues to hold (for
$\bm{\varepsilon} = \bmeh$) with the stronger assertion that $\bmeh$
satisfies both $\lbe$ as well as~\eqref{eq:eheh-eh1}.

Now let $\bmeh$ be a sequences that satisfies both $\lbe$
and~\eqref{eq:eheh-eh1}. We define a dg-category of weakly filtered
modules $\tmod^{\textnormal{wf}}_{\mathcal{A}}(\bmea, \bmemm, \bmeh)$
as follows. (The superscript wf stands for ``weakly filtered''). The
objects of this category are the weakly filtered $\mathcal{A}$-modules
with discrepancy $\leq \bmemm$. The morphisms
$\hom^{\bmeh}(\mathcal{M}_0, \mathcal{M}_1)$ between two such modules
$\mathcal{M}_0, \mathcal{M}_1$ are defined by:
$$\hom^{\bmeh}(\mathcal{M}_0, \mathcal{M}_1) := 
\bigcup_{\rho \in \mathbb{R}} \hom^{\leq \rho; \bmeh}(\mathcal{M}_0,
\mathcal{M}_1).$$ We endow
$\tmod^{\textnormal{wf}}_{\mathcal{A}}(\bmea, \bmemm, \bmeh)$ with the
operations $\mu_1^{\tmod}$ and $\mu_2^{\tmod}$ coming from the
dg-category $\rmod_{\mathcal{A}}$. Finally, we consider the obvious
filtration on the $\hom^{\bmeh}$'s given by $\hom^{\leq \rho; \bmeh}$,
$\rho \in \mathbb{R}$.

The following is easy to check.
\begin{prop} \label{p:wf-mod-categ} With the structures defined above
  $\tmod^{\textnormal{wf}}_{\mathcal{A}}(\bmea, \bmemm, \bmeh)$ is a
  weakly filtered dg-category with discrepancy
  $\leq (0, \eh_1, 0, \ldots, 0, \ldots)$. In particular, if $\eh_1=0$
  this category is a filtered dg-category.
\end{prop}

\begin{rems}
  \begin{enumerate}
  \item One can attempt to enlarge this category by taking unions over
    different choices of $\bmeh$ (say with fixed $\eh_1$) and possibly
    get a category that does not depend on $\bmeh$, but we will not do
    that here.
  \item The category
    $\tmod^{\textnormal{wf}}_{\mathcal{A}}(\bmea, \bmemm, \bmeh)$ is
    in general not triangulated. The reason is that while cones of
    morphisms in that category are weakly filtered they might have
    higher discrepancy (see~\S\ref{sb:wf-mc} below). One can attempt
    to rectify this issue by enlarging this category to contain weakly
    filtered modules of arbitrary discrepancy and labeling each
    pre-module morphisms by a pair of quantities indicating the bounds
    on their action shift and their discrepancy. We will not pursue
    this direction in this paper since it will not be needed for our
    applications.
  \item With the above definitions one can see that the Yoneda
    embedding restricts to a weakly filtered functor from
    $\mathcal{A}$ to the category of weakly filtered
    $\mathcal{A}$-modules with discrepancy $\leq \bmea$ and where
    $\bmeh$ is taken to be any sequences
    in~$\hyperlink{h:asmp-e}{\lbe(\bmea,\bmea)}$, with the additional
    condition that
    $\eh_d \geq \epsilon^{\mathcal{A}}_{d+1} \; \forall d$.
  \end{enumerate}
\end{rems}

At the same time, for the applications intended in this paper, we will
not really need the weakly filtered modules to form a
dg-category. Therefore we will not insist that $\bmeh$ satisfies
condition~\eqref{eq:eheh-eh1}. In contrast, Assumption~$\assmpen$ will
continue to play an important role since it assures that the
$\hom^{\leq \rho; \bmeh}$'s are preserved by $\mu_1^{\tmod}$. Thus we
will mostly continue to assume~$\assmpen$.

\subsubsection{Action-shifts} \label{sbsb:action-shifts} Let
$\mathcal{M}$ be a weakly filtered module over a weakly filtered
$A_{\infty}$-category $\mathcal{A}$. Let $\nu_0 \in \mathbb{R}$.
Define a new weakly filtered $\mathcal{A}$-module
$S^{\nu_0}\mathcal{M}$ to be the same module as $\mathcal{M}$ only
that its filtration is shifted by $\nu_0$, namely:
$$(S^{\nu_0}\mathcal{M})^{\leq \alpha}(N) := 
\mathcal{M}^{\leq \alpha + \nu_0}(N) \; \; \forall \; N \in
\textnormal{Ob}(\mathcal{A}), \; \alpha \in \mathbb{R}.$$ Clearly
$S^{\nu_0}\mathcal{M}$ has the same discrepancy as $\mathcal{M}$. We
call $S^{\nu_0}\mathcal{M}$ the ``action-shift of $\mathcal{M}$ by
$\nu_0$''.

In what follows we will use the same notation $S^{\nu_0}$ also for
action shifts of other filtered objects such as filtered chain
complexes or more generally filtered $R$-modules.

\subsubsection{Homologically unital
  $\mathcal{A}$-modules} \label{sbsb:hu-mod} We have already discussed
h-unital $A_{\infty}$ categories in the weakly filtered sense on
page~\pageref{pg:h-unital-categ}. In what follows we will sometimes
need an analogous, yet somewhat stronger, notion for modules.

\medskip \newcommand{\lbh}{U} \hypertarget{h:asmp-H}{\noindent
  \textbf{Assumptions~$\lbh_w$ and~$\lbh_s$.}}  Let $\mathcal{A}$ be a
weakly filtered $A_{\infty}$-category with discrepancy $\leq
\bmea$. Assume that $\mathcal{A}$ is h-unital in the weakly filtered
sense as defined in~\S\ref{sb:wf-ai-categ}, i.e. we have
$u^{\mathcal{A}} \geq 0$ and choices of cycles
$e_X \in C^{\leq u^{\mathcal{A}}}(X,X)$ for every
$X \in \textnormal{Ob}(\mathcal{A})$ representing the units in
homology.

Let $\mathcal{M}$ be a weakly filtered $\mathcal{A}$-module with
discrepancy $\leq \bmemm$, and let
$u^{\mathcal{A}} + \epsilon_2^{\mathcal{M}} \leq \kappa \in
\mathbb{R}$. We say that $\mathcal{M}$ satisfies
{Assumption~$\lbh_w(\kappa)$} if for every
$X \in \textnormal{Ob}(\mathcal{A})$ and every $\alpha \in \mathbb{R}$
the map
\begin{equation} \label{eq:mu-2-e}
  \mathcal{M}^{\leq \alpha}(X) \longrightarrow 
  \mathcal{M}^{\leq \alpha + \kappa}(X), \quad b \longmapsto
  \mu_2^{\mathcal{M}}(e_X, b)
\end{equation}
induces in homology the same map as the one induced by the inclusion
$\mathcal{M}^{\leq \alpha}(X) \longrightarrow \mathcal{M}^{\leq \alpha
  + \kappa}(X)$. Note that in particular, $\mathcal{M}$ is an h-unital
module.

Sometimes the module $\mathcal{M}$ will be a Yoneda module
$\mathscr{Y}$ associated to an object
$Y \in \textnormal{Ob}(\mathcal{A})$. In that case we will sometimes
write $Y \in \lbh_w(\kappa)$ instead of
$\mathscr{Y} \in \lbh_w(\kappa)$. Note that in this case the map
in~\eqref{eq:mu-2-e} becomes
\begin{equation*} \label{eq:mu-2-e-2} C^{\leq \alpha}(Y, X)
  \longrightarrow C^{\leq \alpha + \kappa}(Y, X), \quad b \longmapsto
  \mu_2^{\mathcal{A}}(e_X, b).
\end{equation*}

There is also a homotopical version of $\lbh_w$ which is
stronger. Namely, if we replace Assumption~$\lbh_w$ by the stronger
one requiring that~\eqref{eq:mu-2-e} is chain homotopic to the
inclusion, then we say that $\mathcal{M}$ satisfies
Assumption~$\lbh_s(\kappa)$. (The subscripts $s$ and $w$ stand for
``strong'' and ``weak'' -- $\lbh_w$ is weaker than $\lbh_s$.) For
modules $\mathcal{M}$ satisfying~$\lbh_w(\kappa)$ we will sometimes
write $\mathcal{M} \in \lbh_w(\mathcal{M})$ and similarly for $\lbh_s$.

\begin{rem} \label{r:us-uw} If the ground ring $R$ is a field then
  $U_s(\kappa) = U_w(\kappa)$. This is so because for chain complexes
  $C_{\bullet}, D_{\bullet}$ over a field, two chain maps
  $C_{\bullet} \longrightarrow D_{\bullet}$ induce the same map in
  homology iff they are chain homotopic. Or in other words, the
  canonical map
  $$H_*(\hom(C_{\bullet}, D_{\bullet})) \longrightarrow \hom(H_*(C),
  H_*(D))$$ is an isomorphism.
\end{rem}

\subsubsection{Pulling back weakly filtered
  modules} \label{sbsb:pull-back-mod} Let $\mathcal{A}, \mathcal{B}$
be two weakly filtered $A_{\infty}$-categories and
$\mathcal{F}: \mathcal{A} \longrightarrow \mathcal{B}$ a weakly
filtered $A_{\infty}$-functor with discrepancy
$\leq \bme^{\mathcal{F}}$. Let $\mathcal{M}$ be a weakly filtered
$\mathcal{B}$-module with discrepancy $\leq \bme^{\mathcal{M}}$.
Consider the $\mathcal{A}$-module $\mathcal{F}^*\mathcal{M}$ which is
obtained by pulling back $\mathcal{M}$ via $\mathcal{F}$.
We filter $\mathcal{F}^*\mathcal{M}$ by setting
$$(\mathcal{F}^*\mathcal{M})^{\leq \alpha}(N) = 
\mathcal{M}^{\leq \alpha}(\mathcal{F} N).$$ The following can be
easily proved.
\begin{lem} \label{l:pull-back-M} The module
  $\mathcal{F}^*\mathcal{M}$ is weakly filtered with discrepancy
  $\leq \bme^{\mathcal{F}^*\mathcal{M}}$, where
  $$\epsilon_d^{\mathcal{F}^*\mathcal{M}} = 
  \max \bigl\{ \epsilon_{s_1}^{\mathcal{F}} + \cdots +
  \epsilon_{s_k}^{\mathcal{F}} + \epsilon_{k+1}^{\mathcal{M}} \bigm| 1
  \leq k \leq d-1, \;\; s_1 + \cdots + s_k = d-1 \bigr\}, \; \;
  \forall \; d \geq 2.$$ In particular, if the higher order terms of
  $\mathcal{F}$ vanish, i.e. $\mathcal{F}_s = 0$ for all $s \geq 2$,
  then
  $$\epsilon_d^{\mathcal{F}^*\mathcal{M}} = (d-1)
  \epsilon_1^{\mathcal{F}} + \epsilon_d^{\mathcal{M}}, \; \; \forall
  \; d.$$
\end{lem}

Let $\mathcal{M}_0, \mathcal{M}_1$ be two weakly filtered
$\mathcal{B}$-modules and
$f: \mathcal{M}_0 \longrightarrow \mathcal{M}_1$ a weakly filtered
module homomorphism that shifts action by $\leq \rho$ and has
discrepancy $\leq \bme^f$. Pulling back $\mathcal{M}_0, \mathcal{M}_1$
and $f$ by $\mathcal{F}$ we obtain a homomorphism of
$\mathcal{A}$-modules
$\mathcal{F}^*f: \mathcal{F}^*\mathcal{M}_0 \longrightarrow
\mathcal{F}^* \mathcal{M}_1$. The following can be easily verified.
\begin{lem} \label{l:pull-back-f} The module homomorphism
  $\mathcal{F}^*f$ is weakly filtered with action shift $\leq \rho$
  and discrepancy $\leq \bme^{\mathcal{F}^*f}$, where
  $\epsilon_1^{\mathcal{F}^*f} = \epsilon_1^f$ and
  $$\epsilon_d^{\mathcal{F}^*f} = 
  \max \bigl\{ \epsilon_{s_1}^{\mathcal{F}} + \cdots +
  \epsilon_{s_k}^{\mathcal{F}} + \epsilon_{k+1}^{f} \bigm| 1 \leq k
  \leq d-1, \;\; s_1 + \cdots + s_k = d-1 \bigr\}, \; \; \forall \; d
  \geq 2.$$ In particular, if the higher order terms of $\mathcal{F}$
  vanish, i.e. $\mathcal{F}_s = 0$ for all $s \geq 2$, then
  $$\epsilon_d^{\mathcal{F}^*f} = (d-1)
  \epsilon_1^{\mathcal{F}} + \epsilon_d^{f}, \; \; \forall \; d.$$
\end{lem}
  
\subsection{Weakly filtered mapping cones} \label{sb:wf-mc} Let
$\mathcal{M}_0$, $\mathcal{M}_1$ be two weakly filtered
$\mathcal{A}$-modules with discrepancies $\leq \bme^{\mathcal{M}_0}$
and $\leq \bme^{\mathcal{M}_1}$ respectively.  Let
$f : \mathcal{M}_0 \longrightarrow \mathcal{M}_1$ be a {\em module
  homomorphism}, i.e. $f$ is a pre-module homomorphisms which is a
cycle $\mu_1^{\textnormal{mod}}(f)=0$. Assume that $f$ shifts action
by $\leq \rho$ and has discrepancy $\leq \bme^{f}$, or in other words
$f \in \hom^{\leq \rho; \bme^f}(\mathcal{M}_0, \mathcal{M}_1)$. We
{\em generally do not} assume that $\bme^f$ satisfies
Assumption~$\hyperlink{h:asmp-e}{\lbe(\bmemm, \bmea)}$ unless
explicitly specified.

Consider the mapping cone $\mathcal{C} := \mathcal{C}one(f)$ viewed as
an $\mathcal{A}$-module and endowed with its standard
$A_{\infty}$-composition maps $\mu_d^{\mathcal{C}}$. We endow
$\mathcal{C}$ with a weakly filtered structure as follows.  For
$X \in \textnormal{Ob}(\mathcal{A})$ and $a \in \mathbb{R}$, put
\begin{equation} \label{eq:fcone} \mathcal{C}^{\leq \alpha}(X) :=
  \mathcal{M}_0^{\leq \alpha - \rho - \epsilon^f_1}(X) \oplus
  \mathcal{M}_1^{\leq \alpha}(X).
\end{equation}
Define \label{pg:disc-cone}
$\bme^{\mathcal{C}} := \max \{\bme^{\mathcal{M}_0},
\bme^{\mathcal{M}_1}, \bme^f-\epsilon^f_1\}$. (See
page~\pageref{pg:conv-seq} for the precise meaning of this notation.)
Then $\mathcal{C}$ is weakly filtered with discrepancy
$\leq \bme^{\mathcal{C}}$. This follows from~\eqref{eq:fcone} and the
fact that
$$\mu^{\mathcal{C}}_d \bigl( a_1, \ldots, a_{d-1}, (b_0, b_1) \bigr) = 
\bigl( \mu_d^{\mathcal{M}_0}(a_1, \ldots, a_{d-1}, b_0), f_d(a_1,
\ldots, a_{d-1}, b_0) + \mu^{\mathcal{M}_1}_d(a_1, \ldots, a_{d-1},
b_1) \bigr).$$

\begin{rem} \label{r:epsN}
  If we assume in addition that
  $\bme^{\mathcal{M}_0}, \bme^{\mathcal{M}_1} \leq \bmemm$ and that
  $\bme^f \in \assmpe$ then we have
  $\bme^f - \epsilon^f_1 \geq \bmemm$, hence
  $\bme^{\mathcal{C}} = \bme^f - \epsilon^f_1$.
\end{rem}

It is important to note that the filtration we have defined on
$\mathcal{C}one(f)$ in~\eqref{eq:fcone} strictly depends on the
choices of $\rho$ and $\epsilon^f_1$. Therefore, whenever these
dependencies are relevant we will denote the weakly filtered cone of
$f$ by
\begin{equation} \label{eq:notation-cone} \mathcal{C}one(f; \rho,
  \bme^f) \; \; \textnormal{or by } \; \tcn ( \mathcal{M}_0
  \xrightarrow{ \; (f; \rho, \bme^f) \;} \mathcal{M}_1).
\end{equation}

\begin{rem} \label{r:cone-filtrations} 
  \begin{enumerate}
  \item The filtration~\eqref{eq:fcone} might look ad hoc since it can
    be shifted in an arbitrary way. However, we opted for the one
    in~\eqref{eq:fcone} so that the inclusion
    $\mathcal{M}_1 \longrightarrow \mathcal{C}$ would be a filtered
    map (rather than just weakly filtered).
  \item If $\rho' \geq \rho$ and $\bme' \geq \bme^f$ then clearly
    $f \in \hom^{\leq \rho'; \bme'}(\mathcal{M}_0, \mathcal{M}_1)$ and
    we can also consider
    $\tcn (\mathcal{M}_0 \xrightarrow{\; (f; \rho', \bme') \;}
    \mathcal{M}_1)$ which is the same $\mathcal{A}$-module but with a
    different weak filtration structure. It is easy to see that the
    identity homomorphism $\tcn(f) \longrightarrow \tcn(f)$ becomes a
    weakly filtered module homomorphism
    $\tcn(f; \rho, \bme^f) \longrightarrow \tcn(f; \rho', \bme')$ that
    shifts action by $\leq \rho'-\rho$ and has discrepancy
    $\leq (\epsilon'_1 - \epsilon^f_1, 0, \ldots, 0, \ldots)$. In the
    other direction the identity map becomes a weakly filtered module
    homomorphism
    $\tcn(f; \rho', \bme') \longrightarrow \tcn(f; \rho, \bme^f)$ that
    shifts action by $\leq \rho-\rho'$ and has discrepancy $0$ (so it
    is in fact a filtered homomorphism).
  \end{enumerate}
\end{rem}

We will now derive several elementary properties of weakly filtered
mapping cones that will be useful later on. We begin with the effect
of action-shifts (see~\S\ref{sbsb:action-shifts}) on mapping cones.
The following follows immediately from the definitions.
\begin{lem} \label{l:action-shift-cone} Let
  $f: \mathcal{M}_0 \longrightarrow \mathcal{M}_1$ be a weakly
  filtered module homomorphism between two weakly filtered
  $\mathcal{A}$-modules. Assume that $f$ shifts action by $\leq \rho$
  and has discrepancy $\leq \bme^f$. Let $\nu_0 \in \mathbb{R}$. Then
  we have the following equality of weakly filtered
  $\mathcal{A}$-modules:
  \begin{equation*}
    \begin{aligned}
      & S^{\nu_0} \bigl(\tcn(\mathcal{M}_0 \xrightarrow{\; (f; \rho,
        \bme^f) \;} \mathcal{M}_1) \bigr) =
      \tcn(S^{\nu_0}\mathcal{M}_0 \xrightarrow{\; (f; \rho, \bme^f)
        \;} S^{\nu_0}\mathcal{M}_1) =
      \\
      & \tcn(\mathcal{M}_0 \xrightarrow{\; (f; \rho, \bme^f-\nu_0) \;}
      S^{\nu_0}\mathcal{M}_1) = \tcn(\mathcal{M}_0 \xrightarrow{\; (f;
        \rho-\nu_0, \bme^f) \;} S^{\nu_0}\mathcal{M}_1).
    \end{aligned}
  \end{equation*}
\end{lem}

Next, we analyze (a special case of) cones over composition of module
homomorphisms, from the weakly filtered perspective. Let
$f:\mathcal{M}_0 \longrightarrow \mathcal{M}_1$ be as at the beginning
of the present section. Let $\mathcal{M}'_1$ be another weakly
filtered $\mathcal{A}$-module with discrepancy
$\leq \bme^{\mathcal{M}'}$ and let
$\xi: \mathcal{M}_1 \longrightarrow \mathcal{M}'_1$ be a weakly
filtered module homomorphism with
$\xi \in \hom^{\leq s; \bme^{\xi}}(\mathcal{M}_1,
\mathcal{M}'_1)$. Denote by
$$f' = \xi \circ f := \mu_2^{\tmod}(f,\xi): \mathcal{M}_0 \longrightarrow
\mathcal{M}'_1$$ the composition of $f$ and $\xi$. We have
$f' \in \hom^{\rho+s; \bme^{f'}}(\mathcal{M}_0, \mathcal{M}_1)$, where
$$\epsilon^{f'}_d = (\bme^f * \bme^{\xi})_d =: 
\max\{ \epsilon^f_i + \epsilon^{\xi}_j \mid i+j=d+1\}.$$ 


\begin{lem} \label{l:con-f-g} There exists a weakly filtered module
  homomorphism
  $$\psi: \tcn(f;\rho, \bme^f) \longrightarrow \tcn(f'; \rho+s, \bme^{f'})$$
  that shifts action by $\leq s$ and has discrepancy
  $\leq \bme^{\xi}$. The homomorphism $\psi$ fits into the following
  (chain level) commutative diagram of $\mathcal{A}$-modules:
  \begin{equation} \label{eq:triangles-psi} 
    \begin{gathered}
      \xymatrix{ \mathcal{M}_0
        \ar[r]^{f} \ar[d]^{\id} & \mathcal{M}_1 \ar[r]^{\,} \ar[d]^{\xi}
        & \tcn(f)
        \ar[r]^{\,} \ar[d]^{\psi} & T \mathcal{M}_0 \ar[d]^{\id} \\
        \mathcal{M}_0 \ar[r]^{f'} & \mathcal{M}'_1 \ar[r]^{\,} & \tcn(f')
        \ar[r]^{\,} & T \mathcal{M}_0 }
    \end{gathered}
  \end{equation}
  where the horizontal unlabeled maps are the standard inclusion and
  projection maps (with $0$ higher order terms). Moreover, if $\xi$ is
  a quasi-isomorphism then so is $\psi$.
\end{lem}
\begin{proof}
  Simply define $\psi_1(b_0, b_1) = \bigl( b_0, \xi_1(b_1) \bigr)$ and
  for $d\geq 2$:
  $$\psi_d\bigl (a_1, \ldots, a_{d-1}, (b_0, b_1) \bigr) := 
  \bigl(0, \xi_d(a_1, \ldots, a_{d-1}, b_1)\bigr).$$ All the
  statements asserted by the lemma can be verified by direct
  calculation.
\end{proof}

Next we discuss how the weakly filtered mapping cone changes if we
alter the cycle $f$ by a boundary. Assume now that $\mathcal{M}_0$ and
$\mathcal{M}_1$ have both discrepancies $\leq \bmemm$. Fix a sequence
$\bmeh$ that satisfies Assumption~$\assmpe$. Let
$f \in \hom^{\leq \rho; \bmeh}(\mathcal{M}_0, \mathcal{M}_1)$ be a
module homomorphism and $f' = f + \mu_1^{\tmod}(\theta)$ for some
$\theta \in \hom^{\leq \rho; \bmeh}(\mathcal{M}_0,
\mathcal{M}_1)$. Consider the two weakly filtered mapping cones
$\mathcal{C}one(f; \rho, \bmeh)$ and
$\mathcal{C}one(f'; \rho, \bmeh)$.

\begin{lem} \label{l:cone-f-f'} There exists a module homomorphism
  $\vartheta: \mathcal{C}one(f; \rho, \bmeh) \longrightarrow
  \mathcal{C}one(f'; \rho, \bmeh)$ with the following properties:
  \begin{enumerate}
  \item $\vartheta$ is a quasi-isomorphism.
  \item $\vartheta$ does not shift action (i.e. it shifts the action
    filtration by $\leq 0$) and has discrepancy
    $\leq \bme^{\vartheta} := \bmeh - \eh_1$. In particular (since
    $\epsilon^{\vartheta}_1=0$) the chain map
    $$\vartheta_1: \mathcal{C}one(f; \rho, \bmeh)(X) \longrightarrow
    \mathcal{C}one(f'; \rho, \bmeh)(X)$$ preserves the action
    filtration for every $X \in \textnormal{Ob}(\mathcal{A})$.
  \end{enumerate}
\end{lem}

\begin{proof}
  Define
  $\vartheta_1(b_0, b_1) := \bigl( (-1)^{|b_0|-1} b_0, (-1)^{|b_1|}
  b_1 + \theta_1(b_0) \bigr)$ and for $d \geq 2$ define:
  $$\vartheta_d \bigl( a_1, \ldots, a_{d-1}, (b_0, b_1) \bigr) = 
  \bigl(0, \theta_d(a_1, \ldots, a_{d-1}, b_0) \bigr).$$
  Cf.~\cite[Formula~{3.7}, page~35]{Se:book-fukaya-categ}.
  
  Note that in this paper we work with a base ring $R$ of
  characteristic $2$, hence the signs in the preceding formula for
  $\vartheta_1$ can actually be ignored. Nevertheless we included
  them, just as an indication for a possible extension to more general
  rings.
\end{proof}

The next lemma shows that weakly filtered cones are preserved under
pulling back by weakly filtered functors.
\begin{lem} \label{l:pull-back-cone} Let $\mathcal{A}, \mathcal{B}$ be
  two weakly filtered $A_{\infty}$-categories,
  $\mathcal{M}_0, \mathcal{M}_1$ weakly filtered $\mathcal{B}$-modules
  and $f: \mathcal{M}_0 \longrightarrow \mathcal{M}_1$ a weakly
  filtered module homomorphism which shifts action by $\leq \rho$ and
  has discrepancy $\leq \bme^f$. Then we have the following equality
  of weakly filtered $\mathcal{A}$-modules:
  $$\mathcal{F}^* \bigl( 
  \tcn(\mathcal{M}_0 \xrightarrow{\; (f; \rho, \bme^f) \;}
  \mathcal{M}_0) \bigr) = \tcn(\mathcal{F}^* \mathcal{M}_0
  \xrightarrow{\; (\mathcal{F}^*f; \rho, \bme^{\mathcal{F}^*f}) \;}
  \mathcal{F}^*\mathcal{M}_1),$$ where $\bme^{\mathcal{F}^*f}$ is
  given in Lemma~\ref{l:pull-back-f}.
\end{lem}
The proof is straightforward, hence omitted.

We now return briefly to unitality of modules, more specifically to
Assumptions~$\hyperlink{h:asmp-H}{\lbh_s, \lbh_w}$.  The following
lemma shows that these two assumptions are preserved under certain
quasi-isomorphisms of weakly filtered modules.
\begin{lem} \label{l:ass-usw-iso} Let $\mathcal{M}'$, $\mathcal{M}''$
  be weakly filtered $\mathcal{A}$-modules with discrepancies
  $\leq \bme^{\mathcal{M}'}$ and $\leq \bme^{\mathcal{M}''}$
  respectively. Let
  $\phi': \mathcal{M}' \longrightarrow \mathcal{M}''$ be a weakly
  filtered module homomorphism with
  $\phi' \in \hom^{\rho'; \bme^{\phi'}}(\mathcal{M}',
  \mathcal{M}'')$. Let $\phi''$ be a collection of chain maps
  $\phi''_X: \mathcal{M}''(X) \longrightarrow \mathcal{M}'(X)$,
  defined for all $X \in \textnormal{Ob}(\mathcal{A})$, and assume
  $$\phi''_X \bigl( \mathcal{M}''^{\leq \alpha}(X) \bigr) \subset 
  \mathcal{M}'^{\leq \alpha + \rho'' + \epsilon''}(X), \; \; \forall
  \, X \in \textnormal{Ob}(\mathcal{A}), \, \alpha \in \mathbb{R},$$
  for some fixed $\rho'', \epsilon'' \in \mathbb{R}$. (For example, if
  $\phi'': \mathcal{M}'' \longrightarrow \mathcal{M}'$ is a weakly
  filtered module homomorphism with
  $\phi'' \in \hom^{\rho''; \bme^{\phi''}}(\mathcal{M}'',
  \mathcal{M}')$, where $\epsilon_1^{\phi''} \leq \epsilon''$, then
  the assumptions on $\phi''$ are clearly satisfied.) Let
  $\nu, \kappa'' \in \mathbb{R}$ and assume further that:
  \begin{enumerate}
  \item For every $X \in \textnormal{Ob}(\mathcal{A})$ and every
    $\alpha \in \mathbb{R}$ the composition of chain maps
    $$\mathcal{M}'^{\leq \alpha}(X) \xrightarrow{\; \phi''_X \circ \phi'_1 \;}
    \mathcal{M}'^{\leq \alpha + \rho'+\rho''+\epsilon_1^{\phi'} +
      \epsilon''}(X) \xrightarrow{\; \text{inc} \;} \mathcal{M}'^{\leq
      \alpha + \rho'+\rho''+\epsilon_1^{\phi'} + \epsilon''+\nu}(X)$$
    induces in homology the same map as the one induced by the
    inclusion
    $$\mathcal{M}'^{\leq \alpha}(X) \longrightarrow
    \mathcal{M}'^{\leq \alpha + \rho'+\rho''+\epsilon_1^{\phi'} +
      \epsilon''+\nu}(X).$$
    \label{i:phi-phi}
  \item $\mathcal{M}'' \in \hyperlink{h:asmp-H}{\lbh_w(\kappa'')}$.
    \label{i:us-kappa''}
  \end{enumerate}
  Then $\mathcal{M}' \in \hyperlink{h:asmp-H}{\lbh_w(\kappa')}$, where
  \begin{equation} \label{eq:kappa'-1} \kappa' = \rho' + \rho'' +
    \epsilon'' + \max \{ \epsilon_1^{\phi'} +
    u^{\mathcal{A}}+\epsilon_2^{\mathcal{M}'} + \nu,
    \epsilon_1^{\phi'} + \kappa'', \epsilon_2^{\phi'} +
    u^{\mathcal{A}} \}.
  \end{equation}
  
  If we replace~(\ref{i:phi-phi}) above by the assumption that for
  every $X \in \textnormal{Ob}(\mathcal{A})$, $\alpha \in \mathbb{R}$,
  the map
  $\phi''_X \circ \phi'_1: \mathcal{M'}(X) \longrightarrow
  \mathcal{M'}(X)$ is chain homotopic to $\textnormal{id}$ via a chain
  homotopy that shifts action by $\leq \nu$ and instead of
  assuming~(\ref{i:us-kappa''}) we assume
  $\mathcal{M}'' \in \hyperlink{h:asmp-H}{\lbh_s(\kappa'')}$, then
  $\mathcal{M}' \in \hyperlink{h:asmp-H}{\lbh_s(\kappa')}$ with
  \begin{equation} \label{eq:kappa'-2}
    \kappa' = \max\{\rho'+\rho''+\epsilon_1^{\phi'} + 
    \epsilon'' + \kappa'', \rho'+\rho''+\ua +
    \epsilon_2^{\phi'}+\epsilon'', \epsilon_2^{\mathcal{M}'} + \ua +
    \nu\}.
  \end{equation}
\end{lem}

\begin{proof}
  Fix $X \in \textnormal{Ob}(\mathcal{A})$, $\alpha \in \mathbb{R}$
  and let $b \in \mathcal{M}'^{\leq \alpha}(X)$ be a {\em
    cycle}. Since $\phi'$ is a module homomorphism
  (i.e. $\mu_1^{\textnormal{mod}}(\phi')=0$) we have:
  \begin{equation*}  \phi_1'
    \mu_2^{\mathcal{M}'}(e_X, b) = \mu_2^{\mathcal{M}''}(e_X,
    \phi'_1(b)) \pm \mu_1^{\mathcal{M}''} \phi'_2(e_X, b).
  \end{equation*}
  Applying $\phi''_X$ to both sides of this identity we obtain:
  \begin{equation} \label{eq:phi''-phi'-1} \phi''_X \phi_1'
    \mu_2^{\mathcal{M}'}(e_X, b) = \phi''_X \mu_2^{\mathcal{M}''}(e_X,
    \phi'_1(b)) \pm \mu_1^{\mathcal{M}'} \phi''_X \phi'_2(e_X, b).
  \end{equation}
  Since
  $\mu_2^{\mathcal{M}'}(e_X, b) \in \mathcal{M}'^{\leq \alpha + \ua +
    \epsilon_2^{\mathcal{M}'}}(X)$ our assumption on
  $\phi''_X \circ \phi'_1$ implies that there exists
  $x \in \mathcal{M}'^{\leq \alpha + \ua +
    \epsilon_2^{\mathcal{M}'}+\rho'+\rho''+\epsilon_1^{\phi'} +
    \epsilon''+\nu}(X)$ such that
  \begin{equation} \label{eq:phi''-phi'-mu_2}
    \phi''_X \phi_1' \mu_2^{\mathcal{M}'}(e_X, b) =
    \mu_2^{\mathcal{M}'}(e_X, b) - \mu_1^{\mathcal{M}'}(x).
  \end{equation}
  Since $\mathcal{M}'' \in \hyperlink{h:asmp-H}{\lbh_w(\kappa'')}$
  there exists an element
  $y \in \mathcal{M}''^{\leq \alpha +
    \rho'+\epsilon_1^{\phi'}+\kappa''}(X)$ such that
  $$\mu_2^{\mathcal{M}''}(e_X, \phi_1'(b)) = \phi_1'(b) + 
  \mu_1^{\mathcal{M}''}(y).$$ 
  Substituting the last identity together
  with~\eqref{eq:phi''-phi'-mu_2} into~\eqref{eq:phi''-phi'-1} yields:
  \begin{equation} \label{eq:mu-2-e-b} \mu_2^{\mathcal{M}'}(e_X, b) =
    \mu_1^{\mathcal{M}'}(x) + \phi''_X \phi'_1(b) +
    \mu_1^{\mathcal{M}'}(\phi''_X(y)) + \mu_1^{\mathcal{M}'}( \phi''_X
    \phi_2' (e_X, b)).
  \end{equation}
  Using our assumption on $\phi''_X \circ \phi_1'$, we can write the
  2'nd term of~\eqref{eq:mu-2-e-b} as
  $\phi''_X \phi'_1(b) = b + \mu_1^{\mathcal{M}'}(z)$ for some
  $z \in \mathcal{M}'^{\leq \alpha + \rho'+\rho''+\epsilon_1^{\phi'} +
    \epsilon'' + \nu}(X)$. Substituting this in~\eqref{eq:mu-2-e-b} we
  obtain:
  \begin{equation} \label{eq:mu-2-e-b-II} \mu_2^{\mathcal{M}'}(e_X, b)
    = b + \mu_1^{\mathcal{M}'}(x) + \mu_1^{\mathcal{M}'}(z) +
    \mu_1^{\mathcal{M}'}(\phi''_X(y)) + \mu_1^{\mathcal{M}'}( \phi''_X
    \phi_2' (e_X, b)),
  \end{equation}
  where
  \begin{equation*}
    \begin{aligned}
      & x \in \mathcal{M}'^{\leq \alpha + \ua +
        \epsilon_2^{\mathcal{M}'}+\rho'+\rho''+\epsilon_1^{\phi'} +
        \epsilon''+\nu}(X), \quad \phi''_X (y) \in \mathcal{M}''^{\leq
        \alpha +
        \rho'+\rho''+\epsilon_1^{\phi'}+\epsilon''+\kappa''}(X), \\
      & z \in \mathcal{M}'^{\leq \alpha +
        \rho'+\rho''+\epsilon_1^{\phi'} + \epsilon'' + \nu}(X), \quad
      \phi''_X \phi_2' (e_X, b) \in \mathcal{M}'^{\leq \alpha + \ua +
        \rho' + \rho'' + \epsilon'' + \epsilon_2^{\phi'}}.
    \end{aligned}
  \end{equation*}
  The estimate~\eqref{eq:kappa'-1} for $\kappa'$ readily follows.

  We now turn to the proving the second statement, concerning
  $\mathcal{M}' \in \hyperlink{h:asmp-H}{\lbh_s(\kappa')}$.  Let
  $X \in \textnormal{Ob}(\mathcal{A})$ and $b \in \mathcal{M}'(X)$ a
  chain. Since $\phi'$ is a module homomorphism we have the following
  identity:
  \begin{equation} \label{eq:ident-phi'} \phi'_1
    \mu_2^{\mathcal{M}'}(e_X, b) = \mu_2^{\mathcal{M}''}(e_X,
    \phi'_1(b)) - \mu_1^{\mathcal{M}''} \phi'_2(e_X, b) + \phi_2'(e_X,
    \mu^{\mathcal{M}'}_1(b)).
  \end{equation}

  Since $\mathcal{M}'' \in \hyperlink{h:asmp-H}{\lbh_s(\kappa'')}$, we
  can write
  $$\mu_2^{\mathcal{M}''}(e_X, \phi'_1(b)) = b + g \mu_1^{\mathcal{M}''}
  \phi'_1(b) + \mu_1^{\mathcal{M}''} g \phi'_1(b)$$ for some map
  $g:\mathcal{M}''(X) \longrightarrow \mathcal{M}''(X)$ that shifts
  action by $\leq \kappa''$. Substituting this
  into~\eqref{eq:ident-phi'} and then applying $\phi_X''$ to both
  sides of the equation we get:
  \begin{equation} \label{eq:phi''-phi'-s1}
    \begin{aligned}
      \phi''_X \phi'_1 \mu_2^{\mathcal{M}'}(e_X, b) = & \; \phi''_X
      \phi'_1(b) + \phi''_X g \mu_1^{\mathcal{M}''}
      \phi'_1(b) + \phi''_X \mu_1^{\mathcal{M}''} g \phi'_1(b) \\
      & - \mu_1^{\mathcal{M}''} \phi''_X \phi'_2(e_X, b) + \phi''_X
      \phi'_2 (e_X, \mu_1^{\mathcal{M}'}(b)).
    \end{aligned}
  \end{equation}
  By assumption
  $\phi''_X \circ \phi'_1 = id + h \circ \mu_1^{\mathcal{M}'} +
  \mu_1^{\mathcal{M}'} \circ h$ for some map
  $h: \mathcal{M}'(X) \longrightarrow \mathcal{M}'(X)$ that shifts
  action by $\leq \nu$. Thus we can write:
  $$\phi''_X \phi'_1 \mu_2^{\mathcal{M}'}(e_X, b) = \mu_2^{\mathcal{M}'}(e_X, b) + 
  h \mu_1^{\mathcal{M}'} \mu_2^{\mathcal{M}'}(e_X, b) +
  \mu_1^{\mathcal{M}'} h \mu_2^{\mathcal{M}'}(e_X, b).$$ Substituting
  $\phi''_X \circ \phi'_1(b) = b + h \mu_1^{\mathcal{M}'}(b) +
  \mu_1^{\mathcal{M}'} h(b)$ together with the latter equality
  into~\eqref{eq:phi''-phi'-s1} we obtain:
  \begin{equation} \label{eq:mu_2-e-b-s} 
    \begin{aligned}
      \mu_2^{\mathcal{M}'}(e_X, b) = b & + h \mu_1^{\mathcal{M}'}(b) +
      \mu_1^{\mathcal{M}'} h(b) + \phi''_X g \phi'_1
      \mu_1^{\mathcal{M}'}(b) + \mu_1^{\mathcal{M}'}
      \phi''_X g \phi'_1(b) \\
      & - \mu_1^{\mathcal{M}''} \phi''_X \phi'_2(e_X, b) + \phi''_X
      \phi'_2(e_X, \mu_1^{\mathcal{M}'}(b)) \\
      & - h \mu_2^{\mathcal{M}'}(e_X, \mu_1^{\mathcal{M}'}(b)) -
      \mu_1^{\mathcal{M}'} h \mu_2^{\mathcal{M}'}(e_X, b).
    \end{aligned}
  \end{equation}
  Here we have also used the fact that $\phi''_X$ and $\phi'_1$ are
  chain maps and that
  $\mu_1^{\mathcal{M}'} \mu_2^{\mathcal{M}'}(e_X, b) =
  \mu_2^{\mathcal{M}'}(e_X, \mu_1^{\mathcal{M}'}(b))$.  The
  estimate~\eqref{eq:kappa'-2} for $\kappa'$ now easily follows.
\end{proof}

It is known that h-unitality is preserved under mapping
cones~\cite[Section~3e]{Se:book-fukaya-categ}. The following Lemma is
a weakly filtered analog concerning
Assumptions~$\hyperlink{h:asmp-H}{\lbh_s, \lbh_w}$.

\begin{lem} \label{l:assh-cones} Assume that $\mathcal{A}$ satisfies
  Assumption~\hyperlink{h:asmp-ue}{$\lbue(\zeta)$} (see
  Page~\pageref{l:assm-ue}). Let $\mathcal{M}_0$, $\mathcal{M}_1$ be
  weakly filtered $\mathcal{A}$-modules with discrepancies
  $\leq \bme^{\mathcal{M}_0}$ and $\leq \bme^{\mathcal{M}_1}$
  respectively and assume that
  $\mathcal{M}_0 \in
  \hyperlink{h:asmp-H}{\lbh_w(\kappa^{\mathcal{M}_0})}$ and
  $\mathcal{M}_1 \in
  \hyperlink{h:asmp-H}{\lbh_w(\kappa^{\mathcal{M}_1})}$. Let
  $f \in \hom^{\leq \rho; \bme^f}(\mathcal{M}_0, \mathcal{M}_1)$ be a
  module homomorphism. Then the weakly filtered module
  $\tcn(f; \rho, \bme^f)$ satisfies
  Assumption~$\hyperlink{h:asmp-H}{\lbh_w(\kappa)}$, where
  \begin{equation} \label{eq:kappa-cone} \kappa = \max
    \{2\kappa^{\mathcal{M}_0}, 2\kappa^{\mathcal{M}_1},
    2u^{\mathcal{A}} + \epsilon_3^{\mathcal{C}},
    2u^{\mathcal{A}}+2\epsilon_2^{\mathcal{C}},
    \zeta+\epsilon_2^{\mathcal{C}}\},
  \end{equation}
  and
  $\bme^{\mathcal{C}} := \max\{\bme^{\mathcal{M}_0},
  \bme^{\mathcal{M}_1}, \bme^f-\epsilon_1^f\}$. (Recall that
  $\bme^{\mathcal{C}}$ is the standard bound on the discrepancy of
  $\mathcal{C} = \tcn(f; \rho, \bme^f)$ -- see
  Page~\pageref{pg:disc-cone}.) If the ground ring $R$ is a field then
  we actually have
  $\tcn(f; \rho, \bme^f) \in~\hyperlink{h:asmp-H}{\lbh_s(\kappa)}$.
\end{lem}

\begin{proof}[Proof of Lemma~\ref{l:assh-cones}]
  Denote $\mathcal{C} = \tcn (f; \rho, \bme^f)$. Recall that this
  module has discrepancy
  $\leq \bme^{\mathcal{C}} := \max\{\bme^{\mathcal{M}_0},
  \bme^{\mathcal{M}_1}, \bme^f-\epsilon_1^f\}$. Put
  $$\delta := 
  \max \{\ua + \epsilon_2^{\mathcal{C}}, \kappa^{\mathcal{M}_0},
  \kappa^{\mathcal{M}_1} \}, \quad \kappa := \max\{2\delta,
  2\ua+\epsilon_3^{\mathcal{C}}, \zeta + \epsilon_2^{\mathcal{C}}
  \}.$$ It is easy to see that the latter expression for $\kappa$
  coincides with~\eqref{eq:kappa-cone}.

  For an $A_{\infty}$-module $\mathcal{M}$ and
  $X \in \textnormal{Ob}(\mathcal{A})$ we will typically denote by
  $$V_{\mathcal{M}}: H_*(\mathcal{M}^{\leq \alpha}(X)) 
  \longrightarrow H_*(\mathcal{M}^{\leq \alpha + r}(X))$$ the map
  induced in homology by the chain map
  $$v_{\mathcal{M}}: \mathcal{M}^{\leq \alpha}(X) \longrightarrow
  \mathcal{M}^{\leq \alpha + r}(X), \quad b \longmapsto
  \mu_2^{\mathcal{M}}(e_X, b).$$ Here $r \in \mathbb{R}$ is chosen
  such that $\ua + \epsilon_2^{\mathcal{M}} \leq r$ so that
  $v_{\mathcal{M}}$ is well defined with the above given target. We
  will need to consider such maps for different values of $r$, and
  whenever a need to distinguish between them arises we will use
  additional ``decorations'' such as
  $V'_{\mathcal{M}}, V''_{\mathcal{M}}$ etc.

  Fix $\alpha \in \mathbb{R}$, $X \in
  \textnormal{Ob}(\mathcal{A})$. Since
  $$\mathcal{C}^{\leq \alpha}(X) = Cone \bigl( 
  \mathcal{M}_0^{\leq \alpha - \rho -\epsilon_1^f}(X) \xrightarrow{\;
    f_1 \;} \mathcal{M}_1^{\leq \alpha}(X) \bigr)$$ we have a long
  exact sequence in homology:
  $$\cdots \longrightarrow H_k(\mathcal{M}_1^{\leq \alpha}(X)) 
  \xrightarrow{\;\; \iota \;\;} H_k(\mathcal{C}^{\leq \alpha}(X))
  \xrightarrow{\;\; \pi \;\;} H_k(\mathcal{M}_0^{\leq \alpha - \rho
    -\epsilon_1^f}(X)) \longrightarrow \cdots$$ where $\iota$ and
  $\pi$ are the maps in homology induced by the inclusion
  $\mathcal{M}_1(X) \longrightarrow \mathcal{C}(X)$ and the projection
  $\mathcal{C}(X) \longrightarrow \mathcal{M}_0(X)$ respectively.

  Replacing $\alpha$ by $\alpha+\delta$ and by $\alpha+\kappa$ we
  obtain two similar long exact sequences. These three sequences are
  mapped on to the other via maps induced from the inclusions of the
  coming from raising the action level from $\alpha$ to
  $\alpha+\delta$ and then to $\alpha+\kappa$. In particular, the
  degree-$k$ chunks of these exact sequences gives the following
  commutative diagram with exact rows:

  \begin{equation} \label{eq:3-LES}
    \begin{gathered}
      \xymatrix{ H_k \bigl( \mathcal{M}_1^{\leq \alpha}(X) \bigr)
        \ar[r] \ar[d]_{V'_{\mathcal{M}_1} = i'_{\mathcal{M}_1}} &
        H_k(\mathcal{C}^{\leq \alpha}(X)) \ar[r]
        \ar[d]_{V'_{\mathcal{C}}}^{i'_{\mathcal{C}}} &
        H_k(\mathcal{M}_0^{\leq \alpha-\rho-\epsilon_1^f}(X))
        \ar[d]^{V'_{\mathcal{M}_0} = i'_{\mathcal{M}_0} } \\
        H_k(\mathcal{M}_1^{\leq \alpha + \delta}(X)) \ar[r]^{\iota}
        \ar[d]_{V''_{\mathcal{M}_1} = i''_{\mathcal{M}_1}} &
        H_k(\mathcal{C}^{\leq \alpha+\delta}(X)) \ar[r]^{\pi \quad \; }
        \ar[d]_{V''_{\mathcal{C}}}^{i''_{\mathcal{C}}} &
        H_k(\mathcal{M}_0^{\leq \alpha - \rho - \epsilon_1^f +
          \delta}(X)) \ar[d]^{V''_{\mathcal{M}_0} = i''_{\mathcal{M}_0}} \\
        H_k(\mathcal{M}_1^{\leq \alpha + \kappa}(X)) \ar[r] &
        H_k(\mathcal{C}^{\leq \alpha + \kappa}(X)) \ar[r] &
        H_k(\mathcal{M}_0^{\leq \alpha-\rho-\epsilon_1^f+\kappa}(X)) }
    \end{gathered}
  \end{equation}
  The maps $i'_{\mathcal{C}}$, $i''_{\mathcal{C}}$ are induced by the
  corresponding inclusions and similarly for
  $i'_{\mathcal{M}_0}, i''_{\mathcal{M}_0}, i'_{\mathcal{M}_1},
  i''_{\mathcal{M}_1}$. By assumption (and by the choices of $\delta$
  and $\kappa$) we have
  $$V'_{\mathcal{M}_j} = i'_{\mathcal{M}_j}, 
  \quad V''_{\mathcal{M}_j} = i''_{\mathcal{M}_j}, \; \text{for} \;
  j=0,1.$$ Note also that each of the maps $V'_{\mathcal{C}}$ and
  $i'_{\mathcal{C}}$ makes the above diagram commutative, and similarly
  for $V''_{\mathcal{C}}$ and $i''_{\mathcal{C}}$.

  Denote by
  $V'''_{\mathcal{C}}: H_k(\mathcal{C}^{\leq \alpha}(X))
  \longrightarrow H_k(\mathcal{C}^{\leq \alpha + \kappa})$ the map
  induced in homology by
  $v_{\mathcal{C}}: \mathcal{C}^{\leq \alpha}(X) \longrightarrow
  \mathcal{C}^{\leq \alpha+\kappa}(X)$, and by 
  $$i'''_{\mathcal{C}}: H_k(\mathcal{C}^{\leq \alpha}(X))
  \longrightarrow H_k(\mathcal{C}^{\leq \alpha + \kappa})$$ the map
  induced by the inclusion. Clearly we have:
  $$V'''_{\mathcal{C}} = i''_{\mathcal{C}} \circ V'_{\mathcal{C}} = 
  V''_{\mathcal{C}} \circ i'_{\mathcal{C}}.$$

  We will need soon the following proposition.
  \begin{prop} \label{p:v_M} Assume that
    $\mathcal{A} \in \hyperlink{h:asmp-ue}{\lbue(\zeta)}$. Let
    $\mathcal{M}$ be a weakly filtered $\mathcal{A}$-module with
    discrepancy $\leq \bme^{\mathcal{M}}$ and let
    $X \in \textnormal{Ob}(\mathcal{A})$. Then the chain maps
    $$v: \mathcal{M}(X) \longrightarrow \mathcal{M}(X),
    \quad v(b) := \mu_2^{\mathcal{M}}(e_X, b)$$ and
    $v \circ v : \mathcal{M}(X) \longrightarrow \mathcal{M}(X)$ are
    chain homotopic via a chain homotopy that shifts action by
    $\leq \max\{2\ua + \epsilon_3^{\mathcal{M}}, \zeta +
    \epsilon_2^{\mathcal{M}} \}$.
  \end{prop}
  We postpone the proof of this Proposition to later in this section
  and proceed with the proof of the current lemma.

  To prove the lemma, we need to show that
  $V'''_{\mathcal{C}}(x) = i'''_{\mathcal{C}}(x)$ for all
  $x \in H_k(\mathcal{C}^{\leq \alpha}(X))$.  To prove the latter
  equality, we first note that since both $V'_{\mathcal{C}}$ and
  $i'_{\mathcal{C}}$ make diagram~\eqref{eq:3-LES} commutative, we
  have:
  $$V'_{\mathcal{C}}(x) - i'_{\mathcal{C}}(x) \in \ker \pi = 
  \textnormal{image\,} \iota.$$ Now write
  $V'_{\mathcal{C}}(x) - i'_{\mathcal{C}}(x) = \iota(y)$ for some
  $y \in H_k(\mathcal{M}_1^{\leq \alpha + \delta}(X))$. As both
  $V''_{\mathcal{C}}$ and $i''_{\mathcal{C}}$ make
  diagram~\eqref{eq:3-LES} commutative we also have:
  $V''_{\mathcal{C}} \circ \iota (y) = i''_{\mathcal{C}} \circ
  \iota(y)$.  It follows that
  $$V''_{\mathcal{C}} 
  \bigl( V'_{\mathcal{C}}(x) - i'_{\mathcal{C}}(x) \bigr) =
  i''_{\mathcal{C}}(V'_{\mathcal{C}}(x) - i'_{\mathcal{C}}(x)).$$
  Applying Proposition~\ref{p:v_M} with $\mathcal{M} = \mathcal{C}$ we
  obtain
  $$V'''_{\mathcal{C}}(x) - V''_{\mathcal{C}} \circ i'_{\mathcal{C}}(x)
  = i''_{\mathcal{C}} \circ V'_{\mathcal{C}}(x) -
  i'''_{\mathcal{C}}(x).$$ Since
  $V'''_{\mathcal{C}} = V''_{\mathcal{C}} \circ i'_{\mathcal{C}} =
  i''_{\mathcal{C}} \circ V'_{\mathcal{C}}$ the lemma follows.
\end{proof}

It still remains to prove the preceding Proposition.
\begin{proof}[Proof of Proposition~\ref{p:v_M}]
  The $A_{\infty}$-identities for $\mathcal{M}$ (+ the fact that $e_X$
  is a cycle) imply that for every $b \in \mathcal{M}(X)$ we have:
  \begin{equation} \label{eq:ai-M} \mu_1^{\mathcal{M}}
    \mu_3^{\mathcal{M}}(e_X, e_X, b) - \mu_2^{\mathcal{M}}(e_X,
    \mu_2^{\mathcal{M}}(e_X,b)) + \mu_3^{\mathcal{M}}(e_X, e_X,
    \mu_1^{\mathcal{M}}(b)) +
    \mu_2^{\mathcal{M}}(\mu_2^{\mathcal{A}}(e_X, e_X), b) = 0.
  \end{equation}
  Since $\mathcal{A} \in \hyperlink{h:asmp-ue}{\lbue(\zeta)}$ we have
  $\mu_2^{\mathcal{A}}(e_X, e_X) = e_X + \mu_1^{\mathcal{A}}(c)$, for
  some $c \in C^{\leq \zeta}(X,X)$. Substituting this
  in~\eqref{eq:ai-M} together with the fact that
  $\mu_2^{\mathcal{M}}(\mu_1^{\mathcal{A}}(c), b) +
  \mu_2^{\mathcal{M}}(c, \mu_1^{\mathcal{M}}(b)) +
  \mu_1^{\mathcal{M}}\mu_2^{\mathcal{M}}(c,b) = 0$ yields:
  \begin{equation*}
    \mu_2^{\mathcal{M}}(e_X, \mu_2^{\mathcal{M}}(e_X,b)) -
    \mu_2^{\mathcal{M}}(e_X, b) = \mu_1^{\mathcal{M}} h(b) + h
    \mu_1^{\mathcal{M}}(b),
  \end{equation*}
  where
  $h(b) = \mu_3^{\mathcal{M}}(e_X,e_X,b) - \mu_2^{\mathcal{M}}(c,b)$.
  Clearly the chain homotopy $h$ shifts action by not more than
  $\max \{2\ua + \epsilon_3^{\mathcal{M}},
  \zeta+\epsilon_2^{\mathcal{M}} \}$.
\end{proof}

\subsection{The $\lambda$-map} \label{sb:l-map} Let $\mathcal{A}$ be
an $A_{\infty}$-category and $\mathcal{M}$ and
$\mathcal{A}$-module. Let $Y \in \textnormal{Ob}(\mathcal{A})$ and
denote by $\mdly$ the Yoneda module corresponding to $Y$.  Consider
the following map:
\begin{equation} \label{eq:lambda} 
  \begin{aligned}
    & \lambda: \mathcal{M}(Y) \longrightarrow \hom(\mdly,
    \mathcal{M}), \quad c \longmapsto \lambda(c) = \bigl(\lambda(c)_1,
    \lambda(c)_2, \ldots, \lambda(c)_d,
    \ldots\bigr), \\
    & \textnormal{where} \quad \lambda(c)_d \bigl(a_1, \ldots,
    a_{d-1}, b \bigr) = \mu_{d+1}^{\mathcal{M}}(a_1, \ldots, a_{d-1},
    b, c).
  \end{aligned}
\end{equation}
This map was defined by
Seidel~\cite[Section~(1l)]{Se:book-fukaya-categ} in the context of the
Yoneda embedding of $A_{\infty}$-categories. A straightforward
calculation shows that it is a chain map. We will refer to it from now
on as the {\em $\lambda$-map}.

Seidel proves in~\cite[Lemma~2.12]{Se:book-fukaya-categ} that, under
the additional assumptions that $\mathcal{A}$ and $\mathcal{M}$ are
h-unital, the $\lambda$-map is a quasi-isomorphism. Our goal is to
establish a weakly filtered analog of this result.

We begin with a technical assumption on a given object
$Y \in \textnormal{Ob}(\mathcal{A})$.

\newcommand{\lbur}{U^{R,e}} \newcommand{\lbul}{U^{e,L}}
\hypertarget{h:asmp-ur}{\noindent \textbf{Assumption~$\lbur$.}}
\label{l:assm-ur} Let $\kappa \geq \epsilon_2^{\mathcal{A}} + \ua$ be
a real number. We say that $Y$ satisfies Assumption~$\lbur(\kappa)$
(or $Y \in \lbur(\kappa)$ for short) if for every
$X \in \textnormal{Ob}(\mathcal{A})$ the map
$$C(X,Y) \longrightarrow C(X,Y), \quad b \longmapsto \mu_2(b, e_Y)$$
is chain homotopic to the identity via a chain homotopy $h_X$ that
shifts action by $\leq \kappa$. The superscript $R,e$ stand for
``Right''-multiplication with $e_Y$.

\medskip

We now define the right setting for the $\lambda$-map in the weakly
filtered case. Assume that $\mathcal{A}$ and $\mathcal{M}$ are both
weakly filtered with discrepancies $\leq \bmea$ and
$\leq \bme^{\mathcal{M}}$ respectively. Clearly $\mdly$ is also a
weakly filtered module with discrepancy $\leq \bmea$.

Without loss of generality we assume from now on that
$\bme^{\mathcal{M}} \geq \bmea$ so that $\mdly$ can be regarded also
as a weakly filtered module with discrepancy
$\leq \bme^{\mathcal{M}}$. (If needed, we can always increase
$\bme^{\mathcal{M}}$ and $\mathcal{M}$ will continue being weakly
filtered with discrepancy $\leq$ than the increased
$\bme^{\mathcal{M}}$.)

Let $\bmeh$ be any sequence that satisfies
Assumption~\hyperlink{h:asmp-e}{$\lbe$} and assume in addition that:
\begin{equation} \label{eq:eh-emm} \eh_d \geq
  \epsilon^{\mathcal{M}}_{d+1} \; \forall \, d.
\end{equation}

Under these assumptions, the $\lambda$-map restricts to maps:
\begin{equation} \label{eq:lambda-alpha} \lambda: \mathcal{M}^{\leq
    \alpha}(Y) \longrightarrow \hom^{\leq \alpha; \bmeh}(\mdly,
  \mathcal{M}),
\end{equation}
defined for all $\alpha \in \mathbb{R}$.
Since $\bmeh$ satisfies Assumption~\hyperlink{h:asmp-e}{$\lbe$}, the
right-hand side of~\eqref{eq:lambda-alpha} is a chain complex with
respect to $\mu_1^{\tmod}$ and the $\lambda$-map
from~\eqref{eq:lambda-alpha} is a chain map.

Let $\mathcal{A}$, $\mathcal{M}$, $Y$ and $\mdly$ be as at the
beginning of~\S\ref{sb:l-map}. Fix also $\bme^{\mathcal{M}}$, $\bmeh$
as above.  For every $\alpha \in \mathbb{R}$ set
$\mathcal{H}^{\leq \alpha} := \hom^{\leq \alpha; \bmeh}(\mdly,
\mathcal{M})$ and for every $k \geq 1$:
\begin{equation*}
  Q^{\leq \alpha}_{(k)} := \bigl\{ t \in \mathcal{H}^{\leq \alpha} 
  \mid t_1 = \cdots = t_k = 0 \bigr \}, \quad 
  \mathcal{H}^{\leq \alpha}_{(k)} := 
  \mathcal{H}^{\leq \alpha} \big/ Q^{\leq \alpha}_{(k)}.
\end{equation*}
As explained above, the $\lambda$-map restricts to maps
$\lambda^{\alpha} : \mathcal{M}^{\leq \alpha}(Y) \longrightarrow
\mathcal{H}^{\leq \alpha}$ for every $\alpha \in \mathbb{R}$ and we
also have the following induced maps:
$$\lambda^{\alpha}_{(k)} : \mathcal{M}^{\leq \alpha}(Y) 
\longrightarrow \mathcal{H}^{\leq \alpha}_{(k)}$$ defined by composing
$\lambda^{\alpha}$ with the quotient map
$\pi_{(k)}: \mathcal{H}^{\leq \alpha} \longrightarrow
\mathcal{H}^{\leq \alpha}_{(k)}$.

\begin{prop} \label{p:lmap} Suppose that $\mathcal{A}$ is h-unital in
  the weakly filtered sense with discrepancy of units
  $\leq u^{\mathcal{A}}$. Let $\kappa \in \mathbb{R}$ such that
  $\kappa \geq u^{\mathcal{A}} + \epsilon_2^{\mathcal{M}}, \ua +
  \epsilon_2^{\mathcal{A}}$ and assume that
  $\mathcal{M} \in \hyperlink{h:asmp-H}{\lbh_w(\kappa)}$ and
  $Y \in \hyperlink{h:asmp-ur}{\lbur(\kappa)}$.  Let
  $\alpha \in \mathbb{R}$. Fix $1 \leq \ell \in \mathbb{Z}$ and put
  $\alpha' := \alpha + \ell \kappa$. Consider the following
  commutative diagram in cohomology:
  \begin{equation} \label{eq:H-M} 
    \begin{gathered}
      \xymatrix{ H^*
        \bigl(\mathcal{M}^{\leq \alpha}(Y) \bigr) \ar[r]^{\quad
          \lambda^{\alpha}_*} \ar[d]_{i^H} &
        H^* \bigl(\mathcal{H}^{\leq \alpha} \bigr) \ar[d]^{i^H} \\
        H^* \bigl(\mathcal{M}^{\leq \alpha'}(Y) \bigr) \ar[r]^{\quad
          \lambda^{\alpha'}_*} \ar[d]_{\id}&
        H^* \bigl(\mathcal{H}^{\leq \alpha'} \bigr) \ar[d]^{\pi^H_{(\ell)}}\\
        H^* \bigl(\mathcal{M}^{\leq \alpha'}(Y) \bigr) \ar[r]^{\quad
          \lambda^{\alpha'}_{(\ell)*}} & H^*
        \bigl(\mathcal{H}_{(\ell)}^{\leq \alpha'}\bigr) }
      \end{gathered}
  \end{equation}
  where the $i^H$ maps are induced by the inclusions
  $\mathcal{M}^{\leq \alpha}(Y) \longrightarrow \mathcal{M}^{\leq
    \alpha'}(Y)$ and
  $\mathcal{H}^{\leq \alpha} \longrightarrow \mathcal{H}^{\leq
    \alpha'}$ and $\pi^H_{(\ell)}$ is induced by the projection
  $\pi_{(\ell)}: \mathcal{H}^{\leq \alpha'} \longrightarrow
  \mathcal{H}^{\leq \alpha'}_{(\ell)}$. Then for every
  $b \in H^*(\mathcal{H}^{\leq \alpha})$ there exists
  $c \in H^*(\mathcal{M}^{\leq \alpha'}(Y))$ such that
  $\pi^H_{(\ell)} \circ i^H(b) = \lambda^{\alpha'}_{(\ell) *}(c)$. In
  other words, for every cycle $\beta \in \mathcal{H}^{\leq \alpha}$
  there exists a cycle $\gamma \in \mathcal{M}^{\leq \alpha'}(Y)$ such
  that
  \begin{equation} \label{eq:beta-gamma} \beta = \lambda(\gamma) +
    \mu_1^{\tmod}(\theta) + \tau,
  \end{equation}
  for some $\theta \in \mathcal{H}^{\leq \alpha'}$ and some cycle
  $\tau = (\tau_1, \tau_2, \ldots) \in \mathcal{H}^{\leq \alpha'}$
  with $\tau_1 = \cdots = \tau_{\ell} = 0$.
\end{prop}

\begin{proof}
  The proof below follows the general scheme of the proof of
  Lemma~2.12 from of~\cite{Se:book-fukaya-categ}, however the weakly
  filtered setting entails significant adjustments with respect
  to~\cite{Se:book-fukaya-categ}.

  We begin with some preparations regarding the weakly filtered
  version of the $\lambda$-map. Let $\rho \in \mathbb{R}$,
  $d \in \mathbb{N}$. Recall that we have the chain map
  $\lambda^{\rho}: \mathcal{M}^{\leq \rho}(Y) \longrightarrow
  \mathcal{H}_{(d)}^{\leq \rho}$ and consider its mapping cone:
  $$\mathcal{K}_{(d)}^{\rho} = Cone \bigl( \mathcal{M}^{\leq \rho}(Y)
  \xrightarrow{\; \lambda^{\rho} \;} \mathcal{H}_{(d)}^{\leq
    \rho}\bigr).$$ Define a decreasing filtration
  $F^r \mathcal{K}_{(d)}^{\rho}$, $r \in \mathbb{Z}_{\geq 0}$ on this
  chain complex by setting
  \begin{equation} \label{eq:FK}
    F^r \mathcal{K}_{(d)}^{\rho} =
    \begin{cases}
      \mathcal{K}_{(d)}^{\rho}, & \text{if $r = 0$;} \\
      \mathcal{H}_{(d)}^{\leq \rho}, & \text{if $r=1$;} \\
      \bigl\{f \in \mathcal{H}_{(d)}^{\leq \rho} \mid f_1 = \cdots =
      f_{r-1} = 0\bigr\}, & \text{if $2\leq r$.}
    \end{cases}
  \end{equation}
  Note that this is a bounded filtration and we actually have
  $F^r\mathcal{K}_{(d)}^{\rho} = 0$ for every $r \geq d+1$.

  Consider now the cohomological spectral sequence
  $\{E_r^{p,q}(\rho), \partial_r\}_{r \in \mathbb{Z}_{\geq 0}}$
  associated to the filtration $F^{\bullet}$. Since the filtration is
  bounded the spectral sequence converges to
  $H^*(\mathcal{K}_{(d)}^{\rho})$.

  Note also that for $\rho \leq \rho'$ we have an obvious inclusion of
  chain complexes
  $i: \mathcal{K}_{(d)}^{\rho} \longrightarrow
  \mathcal{K}_{(d)}^{\rho'}$. Moreover, this inclusion preserves the
  filtrations $F^{\bullet}$ on the corresponding chain
  complexes. Therefore, $i$ induces a map of spectral
  sequences
  $$i_r^E: E_r^{p,q}(\rho) \longrightarrow E_{r+1}^{p,q}(\rho'), \quad
  \forall \; r \geq 0, \; \forall \; p,q.$$

  We now describe more explicitly the first two pages of
  $E_r^{p,q}(\rho)$. A simple calculation gives the following
  description of the $E_0$-page of this spectral sequence. We have
  $E_0^{p,\bullet}(\rho) = 0$ for $p > d$ and for $p<0$. Next we have:
  $E_0^{0,\bullet}(\rho) = \mathcal{M}^{\leq \rho}(Y)^{\bullet}$,
  where the superscript $^{\bullet}$ stands here for the
  (cohomological) grading of the chain complex
  $\mathcal{M}^{\leq \rho}(Y)$. The differential
  $\partial_0: E_0^{0,q}(\rho) \longrightarrow E_0^{0, q+1}(\rho)$ is
  simply $\mu_1^{\mathcal{M}}$.

  The rest of the columns,
  $1 \leq p \leq d$, are:
  \begin{align} \label{eq:E_0^p}
    & E_0^{p,\bullet}(\rho) = \\
    & \prod_{X_0, \ldots, X_{p-1} \in \textnormal{Ob}(\mathcal{A})}
      \hom_{R}^{\leq \rho+\epsilon_p^h; \bullet} \bigl( C(X_0, X_1)
      \otimes \cdots \otimes C(X_{p-2}, X_{p-1}) \otimes C(X_{p-1},
      Y), \mathcal{M}(X_0) \bigr), \notag
  \end{align}
  where the superscript $^{\bullet}$ stand again for (cohomological)
  grading and $\hom_R^{\leq \rho+\epsilon_p^h}$ stands for $R$-linear
  homomorphism that shift action by not more than
  $\rho+\epsilon_p^h$. (Recall that $\bmeh$ has been fixed at the
  beginning of~\S\ref{sb:l-map} and is used in the definitions of
  $\mathcal{H}^{\leq \rho}$ and $\mathcal{H}^{\leq \rho}_{(d)}$.)  The
  differentials
  $\partial_0: E_0^{p,q}(\rho) \longrightarrow E_0^{p,q+1}(\rho)$ for
  $1\leq p\leq d$ are induced in a standard way from
  $\mu_1^{\mathcal{A}}$ and $\mu_1^{\mathcal{M}}$.

  The $E_1$-page is consequently the following:
  $E_1^{p,\bullet}(\rho) = 0$ for all $p > d$ and for $p<0$. For $p=0$
  we have $E_1^{0,q}(\rho) = H^q(\mathcal{M}^{\leq \rho}(Y))$ for all
  $q$. And for $1\leq p\leq d$ we have:
  \begin{align} \label{eq:E_1^pq}
    & E_1^{p,q}(\rho) = \\ 
    & \prod_{X_0, \ldots, X_{p-1} \in
      \textnormal{Ob}(\mathcal{A})} H^q \Bigl( \hom_{R}^{\leq
      \rho+\epsilon_p^h} \bigl( C(X_0, X_1) \otimes \cdots \otimes
    C(X_{p-2}, X_{p-1}) \otimes C(X_{p-1}, Y), \mathcal{M}(X_0) \bigr)
    \Bigr). \notag
  \end{align}

  We now describe the differentials
  $\partial_1: E_1^{p,q} \longrightarrow E_1^{p+1,q}$ on the
  $E_1$-page.  We start with $p=0$. Let
  $[c] \in E_1^{0,q} = H^q(\mathcal{M}^{\leq \rho}(Y))$, where $c$ is
  a cycle.  Then
  $$\partial_1 [c] \in E_1^{1,q} = 
  \prod_{X \in \textnormal{Ob}(\mathcal{A})} H^q \Bigl(\hom_R^{\leq
    \rho+\epsilon_1^h} \bigl(C(X,Y), \mathcal{M}(X) \bigr) \Bigr)$$ is
  the cycle represented by the following homomorphism:
  $$C(X,Y) \ni b \longmapsto \mu_2^{\mathcal{M}}(b,c) \in \mathcal{M}(X).$$
  It is easy to check that this homomorphism is a cycle and that it
  shifts action by not more than $\rho+\epsilon_1^h$. (The latter hold
  because $\epsilon_1^h \geq \epsilon_2^{\mathcal{M}}$
  by~\eqref{eq:eh-emm}.)  The formula for $\partial_1$ for
  $1\leq p \leq d-1$ is the following. Let $f$ be an element in the
  RHS of~\eqref{eq:E_0^p} which is a cycle. Then
  $\partial_1 [f] =[g] \in E_1^{p+1,q}$, where $g$ is a collection of
  $R$-linear homomorphism
  $$g: C(X_0, X_1) \otimes \cdots \otimes C(X_{p-1}, X_{p}) \otimes
  C(X_{p}, Y) \longrightarrow \mathcal{M}(X_0),$$ defined for all
  objects $X_0, \ldots, X_p \in \textnormal{Ob}(\mathcal{A})$ and is
  given by the following formula:
  \begin{equation} \label{eq:del_1-2}
    \begin{aligned}
      g(a_1, \ldots, a_p, b) \longmapsto & \pm
      \mu_2^{\mathcal{M}}(a_1, f(a_2, \ldots, a_p, b)) \pm f(a_1,
      \ldots, a_{p-1},
      \mu_2^{\mathcal{A}}(a_p, b)) \\
      & + \sum_{n=0}^{p-2} \pm f(a_1, \ldots,
      \mu_2^{\mathcal{A}}(a_{n+1}, a_{n+2}), \ldots, b).
    \end{aligned}
  \end{equation}
  This follows from a direct calculation. See the proof of Lemma~2.12
  in~\cite{Se:book-fukaya-categ} for the precise signs in
  formula~\eqref{eq:del_1-2}. Note also that $g$ shifts action by
  $\leq \max \{\rho + \epsilon^h_{p-1} + \epsilon^{\mathcal{M}}_2,
  \rho + \epsilon^{\mathcal{A}}_{p-1}\} \leq \rho + \epsilon^h_p$,
  where the latter inequality follows from
  Assumption~\hyperlink{h:asmp-e}{$\lbe(\bmemm, \bmea)$}.
  
  Consider now the inclusion
  $i: \mathcal{K}_{(d)}^{\rho} \longrightarrow \mathcal{K}_{(d)}^{\rho
    + \kappa}$. As indicated earlier this induces a map of spectral
  sequences $i^E: E(\rho) \longrightarrow E(\rho+\kappa)$, namely
  $$i^E_r: E_r^{p,q}(\rho) \longrightarrow 
  E_r^{p,q}(\rho+ \kappa), \; \forall r\geq 0.$$

  \begin{claim} \label{cl:homotopic} For every $q$ the chain map
    $i_1^E : E_1^{\bullet, q}(\rho) \longrightarrow E_1^{\bullet,
      q}(\rho+\kappa)$ is chain homotopic to $0$ in the degree range
    $0 \leq \bullet \leq d-1$. In other words, for every $q$ there
    exist homomorphisms
    $S^{p,q}: E_1^{p, q}(\rho) \longrightarrow E_1^{p-1,
      q}(\rho+\kappa)$, defined for all $p$, such that
    \begin{equation} \label{eq:i_1-S_pq}
      i_1^E|_{E_1^{p,q}(\rho)} = \partial_1 \circ S^{p,q} + S^{p+1,q}
      \circ \partial_1, \; \; \forall \, 0\leq p \leq d-1.
    \end{equation}
  \end{claim}
  We postpone the proof of this claim till later in this section and
  continue now with the proof of Proposition~\ref{p:lmap}.

  Claim~\ref{cl:homotopic} implies that
  $i^E_2: E_2^{p,q}(\rho) \longrightarrow E_2^{p,q}(\rho + \kappa)$ is
  the $0$ map for every $0\leq p \leq d-1$ and every $q$. It follows
  that the same holds for the maps
  $i^E_r: E_r^{p,q}(\rho) \longrightarrow E_r^{p,q}(\rho+\kappa)$ for
  every $r \geq 2$.

  Since both the spectral sequences converge after a finite number of
  pages (in fact they collapse at page $r=d+1$) we conclude that
  $i_{\infty}^E : E_{\infty}^{p,q}(\rho) \longrightarrow
  E_{\infty}^{p,q}(\rho+\kappa)$ is $0$ for all $0\leq p \leq d-1$ and
  all $q$. Denote by $F^{\bullet} H^*(\mathcal{K}_{(d)}^{\rho})$ the
  filtration on $H^*(\mathcal{K}_{(d)}^{\rho})$ induced by
  $F^{\bullet} \mathcal{K}_{(d)}^{\rho}$. Since
  $$E_{\infty}^{p,q}(\rho) = F^p H^{p+q}(\mathcal{K}_{(d)}^\rho) /
  F^{p+1}H^{p+q}(\mathcal{K}_{(d)}^{\rho}),$$ and similarly for
  $E_{\infty}^{p,q}(\rho+\kappa)$, we have proved the following:
  \begin{cor} \label{c:i^H-map} The map
    $i^H: H^n(\mathcal{K}_{(d)}^{\rho}) \longrightarrow
    H^n(\mathcal{K}_{(d)}^{\rho+\kappa})$ induced in homology by the
    inclusion
    $i: \mathcal{K}_{(d)}^{\rho} \longrightarrow
    \mathcal{K}_{(d)}^{\rho+\kappa}$ sends
    $F^pH^n(\mathcal{K}_{(d)}^{\rho})$ to
    $F^{p+1}H^n(\mathcal{K}_{(d)}^{\rho+\kappa})$ for every $n$ and
    $0 \leq p \leq d-1$.
  \end{cor}

  We are now in position to conclude the proof of
  Proposition~\ref{p:lmap}. Fix the parameters $\alpha$, $\ell$ and
  $\alpha'$ as in the statement of the proposition. Choose
  $d \gg \ell$ and apply what we have proved above to
  $\mathcal{K}_{(d)}^{\alpha}$ (i.e. take $\rho =
  \alpha$). Corollary~\ref{c:i^H-map}, applied with $p=0$, implies
  that $i^H$ maps $H^n(\mathcal{K}_{(d)}^{\alpha})$ to
  $F^1 H^n(\mathcal{K}_{(d)}^{\alpha+\mathcal{K}^{\mathcal{M}}})$ for
  all $n$. We now apply the Corollary~\ref{c:i^H-map} again, this time
  with $p=1$, $\rho = \alpha + \mathcal{K}^{\mathcal{M}}$ and the
  inclusion
  $\mathcal{K}_{(d)}^{\alpha+\mathcal{K}^{\mathcal{M}}}
  \longrightarrow
  \mathcal{K}_{(d)}^{\alpha+2\mathcal{K}^{\mathcal{M}}}$ and together
  with the previous conclusion infer that $i^H$ maps
  $H^n(\mathcal{K}_{(d)}^{\alpha})$ to
  $F^2 H^n(\mathcal{K}_{(d)}^{\alpha+2\mathcal{K}^{\mathcal{M}}})$ for
  all $n$. (By a slight abuse of notation we have denoted here by the
  same symbol, $i^H$, the maps induced in homology by the following
  different inclusions:
  $\mathcal{K}_{(d)}^{\rho} \longrightarrow
  \mathcal{K}_{(d)}^{\rho+\kappa}$,
  $\mathcal{K}_{(d)}^{\rho+\mathcal{K}^{\mathcal{M}}} \longrightarrow
  \mathcal{K}_{(d)}^{\rho+2\kappa}$ and
  $\mathcal{K}_{(d)}^{\rho} \longrightarrow
  \mathcal{K}_{(d)}^{\rho+2\kappa}$. Below we will continue with this
  notation.)

  Applying the same argument over and over again, $\ell$ times, we
  conclude that the map
  $i^H: H^n(\mathcal{K}_{(d)}^{\alpha}) \longrightarrow
  H^n(\mathcal{K}_{(d)}^{\alpha+ \ell \kappa})$ induced by the
  inclusion
  $\mathcal{K}_{(d)}^{\alpha} \longrightarrow
  \mathcal{K}_{(d)}^{\alpha+ \ell \kappa}$ maps
  $H^n(\mathcal{K}_{(d)}^{\alpha})$ to
  $F^{\ell}H^n(\mathcal{K}_{(d)}^{\alpha+\ell \kappa})$.

  Let now $\beta \in \mathcal{H}^{\leq \alpha}$ be a cycle and denote
  by $\overline{\beta}$ its image in
  $\mathcal{H}_{(d)}^{\leq \alpha'}$, where
  $\alpha' = \alpha + \ell \kappa$. Consider the cycle
  $(0, \overline{\beta}) \in \mathcal{K}_{(d)}^{\alpha'}$. By what we
  have proved before we have
  $[(0, \overline{\beta})] \in F^{\ell}
  H^*(\mathcal{K}_{(d)}^{\alpha'})$.

  It follows that there exists
  $\tau' \in \hom^{\leq \alpha'; \bmeh}(\mathscr{Y}, \mathcal{M})$
  with $\tau'_1 = \cdots = \tau'_{\ell} = 0$ such that
  $[(0, \overline{\beta})] = [(0, \tau')]$ in
  $H^*(\mathcal{K}_{(d)}^{\alpha'})$. Therefore, there exist
  $\gamma \in \mathcal{M}^{\leq \alpha'}(Y)$ and
  $\theta \in \hom^{\leq \alpha'; \bmeh}(\mathscr{Y}, \mathcal{M})$
  such that
  $$(0, \overline{\beta}) = (0, \tau') + 
  \bigl( \mu_1^{\mathcal{M}}(\gamma),
  \lambda_{(d)}^{\alpha'}(\gamma)+\mu_1^{\text{hom}}(\theta) \bigr)$$
  in $\mathcal{K}_{(d)}^{\alpha'}$.  In order to lift the last
  equation from $\mathcal{K}_{(d)}^{\alpha'}$ to
  $Cone \bigl( \mathcal{M}^{\leq \alpha'}(Y) \xrightarrow{\;
    \lambda^{\alpha'} \;} \mathcal{H}^{\leq \alpha'}\bigr)$ we can
  correct the terms beyond order $d$ if necessary by replacing $\tau'$
  with a suitable $\tau$ that coincides with $\tau'$ up to order $d$
  (recall that $d \gg \ell$).

  Summing up, we have proved that there exists a cycle
  $\gamma \in \mathcal{M}^{\leq \alpha'}(Y)$, and a pre-module
  homomorphism
  $\theta \in \hom^{\leq \alpha'; \bmeh}(\mathscr{Y}, \mathcal{M})$
  such that
  $$\beta = \lambda(\gamma) + \mu_1^{\text{mod}}(\theta) + \tau,$$
  where
  $\tau \in \hom^{\leq \alpha'; \bmeh}(\mathscr{Y}, \mathcal{M})$ is a
  cycle with $\tau_1 = \cdots = \tau_{\ell} = 0$, as claimed by the
  Proposition. (The statement in the Proposition concerning
  diagram~\eqref{eq:H-M} is a rephrasing of what we have just proved.)

  This concludes the proof of Proposition~\ref{p:lmap}, modulo the
  proof of Claim~\ref{cl:homotopic} which we will carry out next.
\end{proof}

\begin{proof}[Proof of Claim~\ref{cl:homotopic}]
  Fix $q$. We define the chain homotopy $S^{\bullet, q}$ as
  follows. Define $S^{0,q}=0$. (Note that $E_1^{-1,q}(\rho) = 0$.)
  Next, to define $S^{1,q}$, let
  $$f \in E_0^{1,q}(\rho) = \prod_{X \in \textnormal{Ob}(\mathcal{A})} 
  \hom_R^{\leq \rho+\epsilon_1^h} \bigl( C(X,Y), \mathcal{M}(X)
  \bigr)$$ be a $\partial_0$-cycle. We define
  $$S^{1,q}[f] := [f(e_Y)] \in E_1^{0,q}(\rho + \kappa) = 
  H^q(\mathcal{M}^{\leq \rho + \epsilon_1^h + \kappa}(X)).$$ Since
  $\kappa \geq \ua$, $f(e_Y)$ indeed belongs to
  $E_0^{0,q}(\rho + \kappa)$. Moreover, a straightforward calculation
  shows that $f(e_Y)$ is a $\partial_0$-cycle and that its homology
  class $[f(e_Y)]$ depends only on the homology class
  $[f] \in E_1^{1,q}(\rho)$.

  for the range of degrees $2 \leq p \leq d$ we define $S^{p,q}$ by a
  similar formula: let $f \in E_0^{p,q}(\rho)$, i.e. a collection of
  $R$-linear homomorphism as in~\eqref{eq:E_0^p}. Assume that $f$ is a
  $\partial_0$-cycle. Define $S^{p,q}[f]$ to be the homology class
  $[g] \in E_1^{p-1,q}(\rho + \kappa)$ of the element
  $g \in E_0^{p-1,q}(\rho+\kappa)$ given by:
  \begin{equation*}
    \begin{aligned}
      & g(a_1, \ldots, a_{p-2}, b) = f(a_1, \ldots, a_{p-2}, b, e_Y), \\
      & \forall \; a_i \in C(X_{i-1}, X_i), \; i=1, \ldots, p-2,
      \quad b \in C(X_{p-2}, Y).
    \end{aligned}
  \end{equation*}
  Since $\kappa \geq \ua$, we have $g \in E_0^{p-1,q}(\rho+\kappa)$. A
  straightforward calculation shows that $g$ is a $\partial_0$-cycle and
  moreover its homology class, $[g] \in E_1^{p-1,q}(\rho + \kappa)$
  depends only on the homology class $[f]$ of $f$. This concludes the
  definition of the maps $S^{p,q}$. (For $p$'s outside of the range
  $0,\ldots, d$ we can define $S^{p,q}$ in an arbitrary way.)

  We verify now the identity~\eqref{eq:i_1-S_pq}. We begin with
  $2 \leq p \leq d-1$. Let $f \in E_0^{p,q}(\rho)$ be a cycle. A
  straightforward calculation shows that
  $(\partial_1 S^{p,q} + S^{p+1,q} \partial_1 )[f] = [\widetilde{f}],$
  where
  $$\widetilde{f}(a_1, \ldots, a_{p-1}, b) = f(a_1, \ldots, a_{p-1},
  \mu_2^{\mathcal{A}}(b, e_Y)).$$

  We claim that $[\widetilde{f}] = [f]$ in
  $E_1^{p,q}(\rho+\kappa)$. Indeed, since
  $Y \in \hyperlink{h:asmp-ur}{\lbur(\kappa)}$ there exists a chain
  homotopy $h_{X_{p-1}}: C(X_{p-1}, Y) \longrightarrow C(X_{p-1}, Y)$
  that shifts action by $\leq \kappa$ such that
  $$\mu_2(b, e_Y) = b + h_{X_{p-1}} \mu_1^{\mathcal{A}}(b) + 
  \mu_1^{\mathcal{A}} h_{X_{p-1}}(b), \; \; \forall \; b \in
  C(X_{p-1},Y).$$ Define $\psi \in E_0^{p,q-1}(\rho+\kappa)$ by
  $$\psi(a_1, \ldots, a_{p-1}, b) := f(a_1, \ldots, a_{p-1}, 
  h_{X_{p-1}}(b)).$$ A straightforward calculation shows that
  $\widetilde{f} - f = \partial_0 \psi$, hence $[\widetilde{f}] = [f]$
  in $E_1^{p,q}(\rho+\kappa)$. This proves~\eqref{eq:i_1-S_pq} for
  $2 \leq p \leq d$. A similar argument shows that~\eqref{eq:i_1-S_pq}
  holds also for $p=1$.

  It remains to verify~\eqref{eq:i_1-S_pq} the case $p=0$.  Let
  $m \in \mathcal{M}^{\leq \rho}(Y)$ be a cycle. We have:
  $$(\partial_1 S^{0,q} + S^{1,q} \partial_1)[m] = 
  S^{1,q} \partial_1[m) = (\partial_1[m])(e_Y) =
  [\mu_2^{\mathcal{M}}(e_Y, m)].$$ By assumption
  $\mathcal{M} \in \hyperlink{h:asmp-H}{\lbh_w(\kappa)}$, hence
  $[\mu_2^{\mathcal{M}}(e_Y, m)] = [m]$ in
  $H^q(\mathcal{M}^{\leq \rho + \kappa}(Y))$. This
  proves~\eqref{eq:i_1-S_pq} for $p=0$ and concludes the proof of
  Claim~\ref{cl:homotopic}. 
\end{proof}

\subsection{Structure theorem for weakly filtered iterated cones} \label{sb:wf-ic}

Let $\mathcal{A}$ be a an h-unital weakly filtered
$A_{\infty}$-category with discrepancy $\leq \bmea$ and discrepancy of
units $\ua$. Let $L_0, \ldots, L_r \in \textnormal{Ob}(\mathcal{A})$
and denote by $\mathcal{L}_i$ the Yoneda module associated to $L_i$,
viewed as weakly filtered modules. In this section we analyze iterated
cones in the weakly filtered framework. By iterated cones we mean
modules of the type:
\begin{equation} \label{eq:itcone}
  \tcn (\mathcal{L}_r \xrightarrow{\; \phi_r \;} \tcn(\mathcal{L}_{r-1} 
  \xrightarrow{\; \phi_{r-1} \;} \tcn( \cdots \tcn(\mathcal{L}_2
  \xrightarrow{\; \phi_2 \;} \tcn(\mathcal{L}_1 \xrightarrow{\; \phi_1
    \;} \mathcal{L}_0 )) {\cdot}{\cdot}{\cdot}))).
\end{equation}
The weakly filtered structure is defined by iterating the construction
from~\S\ref{sb:wf-mc}. More precisely, we define a sequence of weakly
filtered $\mathcal{A}$-modules $\mathcal{K}_0, \ldots, \mathcal{K}_r$
as follows. We start by setting $\mathcal{K}_0 := \mathcal{L}_0$ which
is a weakly filtered module with discrepancy
$\leq \bme^{\mathcal{K}_0} := \bmea$. Note that all the modules
$\mathcal{L}_i$ have discrepancy $\leq \bmea$ too. Suppose that
$\phi_1 \in \hom^{\leq \rho_1; \bm{\delta}^{\phi_1}}(\mathcal{L}_1,
\mathcal{K}_0)$ is a module homomorphism, where
$\rho_1 \in \mathbb{R}$ and $\bm{\delta}^{\phi_1}$ is some
sequence. We {\em do not} assume that $\bm{\delta}^{\phi_1}$ satisfies
anything like Assumption~$\assmpen$. We define
$\mathcal{K}_1 = \tcn(\mathcal{L}_1 \xrightarrow{\; (\phi_1; \rho_1,
  \bmd^{\phi_1}) \;} \mathcal{K}_0)$. Since
$\bme^{\mathcal{K}_0} = \bmea$, the discrepancy of $\mathcal{K}_1$ is
$\leq \bme^{\mathcal{K}_1} := \max \{\bmea,
\bmd^{\phi_1}-\delta^{\phi_1}_1\}$.

Let $i\geq 1$ and suppose that we have already defined the weakly
filtered modules $\mathcal{K}_0, \ldots, \mathcal{K}_i$. Let
$\phi_{i+1} : \mathcal{L}_{i+1} \longrightarrow \mathcal{K}_i$ be a
module homomorphism that shifts action by $\leq \rho_{i+1}$ and has
discrepancy $\leq \bmd^{\phi_{i+1}}$. Again, we {\em do not} assume
that $\bmd^{\phi_{i+1}}$ satisfies any assumption of the
type~$\assmpen$. We define
$\mathcal{K}_{i+1} = \tcn( \mathcal{L}_{i+1} \xrightarrow{\;
  (\phi_{i+1}; \rho_{i+1}, \bmd^{\phi_{i+1}})\;} \mathcal{K}_i)$. The
$\mathcal{A}$-module $\mathcal{K}_{i+1}$ has discrepancy
$\leq \bme^{\mathcal{K}_{i+1}} := \max\{ \bme^{\mathcal{K}_i},
\bmd^{\phi_{i+1}}-\delta^{\phi_{i+1}}_1 \}$ because (by induction)
$\bme^{\mathcal{K}_i} \geq \bmea$.  The final $\mathcal{A}$-module
$\mathcal{K}_r$ is precisely the one described by~\eqref{eq:itcone}
and moreover now it also has the structure of a weakly filtered
module.

The following expressions will be used frequently in what follows:
\begin{equation} \label{eq:chi-xi}
  \begin{aligned}
    \chi_{m,d} & := \sum_{j=1}^m \sum_{i=1}^{d+m} \delta_i^{\phi_j} +
    \sum_{i=1}^{d+m} \epsilon_i^{\mathcal{A}},\\
    \xi_q & := \kappa + \sum_{i=1}^{q+3} \epsilon_i^{\mathcal{A}} +
    \sum_{j=1}^q \sum_{i=1}^{q+2} \delta_i^{\phi_j}.
  \end{aligned}
\end{equation}

\begin{thm} \label{t:itcones} Let $\mathcal{K}_i$, $0 \leq i \leq r$
  be as above. Assume that $\mathcal{A}$ is h-unital in the weakly
  filtered sense with discrepancy of units $\leq u^{\mathcal{A}}$. Let
  $\kappa \geq 2\ua + \epsilon_2^{\mathcal{A}}$ be a real number and
  assume that $\mathcal{A}$ and the objects $L_i$ satisfy the
  following two conditions:
  \begin{enumerate}  
  \item $\mathcal{A} \in \hyperlink{h:asmp-ue}{\lbue(\kappa)}$.
  \item For every $0 \leq i \leq r$,
    $L_i \in \hyperlink{h:asmp-ur}{\lbur(\kappa)}$, and
    $L_i \in \hyperlink{h:asmp-H}{\lbh_w(\kappa)}$.
  \end{enumerate}

  Then there exists a weakly filtered $\mathcal{A}$-module
  $\mathcal{M}$ with the following properties:
  \begin{enumerate}
  \item For every $X \in \textnormal{Ob}(\mathcal{A})$, we have
    $\mathcal{M}(X) = \mathcal{K}_r(X)$ as $R$-modules, namely the
    $R$-module $\mathcal{M}(X)$ is a direct sum:
    \begin{equation}
      \label{eq:M-direct-sum}
      \mathcal{M}(X) = C(X, L_0) \oplus C(X, L_1) \oplus \cdots 
      \oplus C(X, L_r).
    \end{equation}
  \item Denote by $\mu_1^{\mathcal{M}}$ the differential of the chain
    complex $\mathcal{M}(X)$. Then the matrix of $\mu_1^{\mathcal{M}}$
    with respect to the splitting~\eqref{eq:M-direct-sum} has the
    following shape:
    $\mu_1^{\mathcal{M}} = (a_{ij})_{0\leq i, j \leq r}$ with
    $a_{i,j}: C(X,L_j) \longrightarrow C(X,L_i)$, where:
    \begin{enumerate}
    \item $a_{i,j}=0$ for every $i>j$. In other words, the matrix of
      $\mu_1^{\mathcal{M}}$ is upper triangular.
    \item
      $a_{i,i} = \mu^{\mathcal{A}}_1: C(X,L_i) \longrightarrow
      C(X,L_i)$.
    \item There exist elements $c_{q,p} \in C(L_q,L_p)$ for all
      $0\leq p < q \leq r$, such that for for every $i<j$ the
      $(i,j)$'th entry of the matrix of $\mu_1^{\mathcal{M}}$ is given
      by:
      \begin{equation} \label{eq:aij} a_{i,j}(-) = \sum_{2 \leq d,\,
          \underline{k}} \mu^{\mathcal{A}}_{d}(-, c_{k_{d}, k_{d-1}},
        \ldots, c_{k_2,k_1}),
      \end{equation}
      where $\underline{k}=(k_1, \ldots, k_d)$ runs over all
      partitions $i=k_1 < k_2 < \cdots < k_{d-1} < k_d=j$. (Note that
      the sum in~\eqref{eq:aij} is finite. In fact, the index $d$ in
      this sum cannot exceed $j-i$ which in turn is $\leq r$.)
    \item $c_{q,p} \in C^{\leq \alpha_{q,p}}(L_q,L_p)$, where
      \begin{equation} \label{eq:alpha_qp}
        \alpha_{q,p} = \rho_q - \rho_p + B_q \xi_q,
      \end{equation}
      where $B_q$ is a universal constant in the sense that it depends
      only on $q$, but not on $\mathcal{A}$, the modules
      $\mathcal{K}_i$ or their discrepancy
      data. (In~\eqref{eq:alpha_qp} and in what follows we use the
      convention that $\rho_0 = 0$.)
      \label{i:univ-const-E}
    \end{enumerate}
  \item There exists a quasi-isomorphism of $\mathcal{A}$-modules
    $\sigma: \mathcal{K}_r \longrightarrow \mathcal{M}$ which shifts
    action by $\leq \rho^{\sigma}$ and has discrepancy
    $\leq \bme^{\sigma}$. The latter quantities admit the following
    estimates:
    \begin{equation} \label{eq:rho-sigma} \rho^{\sigma} \leq C_r
      \xi_r, \quad \epsilon_d^{\sigma} \leq D_{r,d} \chi_{r,d},
    \end{equation}
    where the constants $C_r$ and $\{D_{r,d}\}_{d \in \mathbb{N}}$ are
    universal in the sense mentioned at point~\eqref{i:univ-const-E}
    above.
  \item The 1'st order part
    $\sigma_1: \mathcal{K}_r(X) \longrightarrow \mathcal{M}(X)$ of the
    quasi-isomorphism $\sigma$ is an isomorphism of chain complexes
    for all $X \in \textnormal{Ob}(\mathcal{A})$, and the matrix
    corresponding to $\sigma_1$ with respect to the
    splitting~\eqref{eq:M-direct-sum} (taken both for
    $\mathcal{K}_r(X)$ and $\mathcal{M}(X)$) is upper triangular with
    $\id$-maps along its diagonal. \label{pp:sigma-1-iso}
  \item The inverse
    $\sigma_1^{-1}:\mathcal{M}(X) \longrightarrow \mathcal{K}_r(X)$ of
    $\sigma_1$ is action preserving (i.e. it is filtered and shifts
    action by $\leq 0$).
    \label{i:sigma_1-inv}
  \item For every $0 \leq j \leq r$ the diagonal element
    $$\Delta_j = \textnormal{pr}_{C(X,L_j)} 
    \circ \sigma_1|_{C(X,L_j)}: C(X,L_j) \longrightarrow C(X,L_j)$$ is
    the identity map (as follows from point~\eqref{pp:sigma-1-iso}
    above). However, when the domain inherits filtration from
    $\mathcal{K}_r(X)$ and the target from $\mathcal{M}(X)$ this map
    shifts action by $\leq \rho^{\sigma}$. (Note that for $j\geq 1$,
    $C(X,L_j)$ is in general not a subcomplex of either
    $\mathcal{K}_r(X)$ or of $\mathcal{M}(X)$). For $j=0$, $C(X,L_0)$
    is a subcomplex of both $\mathcal{K}_r(X)$ and of $\mathcal{M}(X)$
    and the two inherited filtrations on $C(X,L_0)$ coincide, hence
    $\Delta_0 = \id$ preserves filtration (i.e. shifts action by
    $\leq 0$).
    \label{pp:Delta_j-Delta_0}
  \end{enumerate}
\end{thm}

\newcommand{\ephia}{(1)}
\newcommand{\ephib}{(2)}
\newcommand{\ephic}{(3)}
\newcommand{\ephii}{(i)}
\newcommand{\ephij}{(j)}
\newcommand{\ephip}{(p)}

\begin{proof}[Proof of Theorem~\ref{t:itcones}]
  We will construct inductively a sequence of weakly filtered modules
  $\mathcal{M}_i$, $i=1, \ldots, r$ such that $\mathcal{M}_i$ is
  quasi-isomorphic to $\mathcal{K}_i$ and whose differential
  $\mu_1^{\mathcal{M}_i}$ has a matrix of the type describe
  by~\eqref{eq:aij}. The desired module $\mathcal{M}$ will then be
  $\mathcal{M}_r$. In the course of the construction we will
  successively apply Proposition~\ref{p:lmap} and
  Lemmas~\ref{l:cone-f-f'},~\ref{l:con-f-g}. 

  Fix once and for all $\ell := r+2$. We begin the construction with
  $i=1$. Put $\mathcal{M}_0 = \mathcal{K}_0 = \mathcal{L}_0$,
  $\mathcal{K}'_1 = \mathcal{K}_1$. Set also $\kappa_0 = \kappa$, so
  that
  $\mathcal{L}_1, \mathcal{K}_0 \in
  \hyperlink{h:asmp-H}{\lbh_w(\kappa_0)}$. Define an auxiliary weakly
  filtered module
  $$\mathcal{K}_1'' := \tcn (\mathcal{L}_1 
  \xrightarrow{\; (\phi_1; \rho_1 + \ell \kappa_0, \bme^{\ephia}) \;}
  \mathcal{K}_0),$$ where $\bme^{\ephia}$ is chosen such that:
  $$\bme^{\ephia} \geq \bm{\delta}^{\phi_1}, \quad 
  \bme^{\ephia} \in \hyperlink{h:asmp-e}{\lbe(\bmea,
    \bme^{\mathcal{K}_0})}, \quad \epsilon_d^{\ephia} \geq
  \epsilon_{d+1}^{\mathcal{K}_0} \;\; \forall \; d.$$ By
  Proposition~\ref{p:lmap} there exists a cycle
  $c_1 \in \mathcal{K}_0^{\leq \rho_1 + \ell \kappa_0} (L_1) = C^{\leq
    \rho_1 + \ell \kappa_0}(L_1, L_0)$ as well as
  $\theta_1, \tau_1 \in \hom^{\leq \rho_1 + \ell \kappa_0;
    \bme^{\ephia}}(\mathcal{L}_1, \mathcal{K}_0)$ with $\tau_1$ a
  cycle and $(\tau_1)_1 = \cdots = (\tau_1)_{\ell} = 0$, such that:
  $$\phi_1 = \lambda(c_1) + \mu_1^{\textnormal{mod}}(\theta) + \tau_1$$
  in
  $\hom^{\leq \rho_1 + \ell \kappa_0; \bme^{\ephia}}(\mathcal{L}_1,
  \mathcal{K}_0)$.
  Define now
  $$\mathcal{M}_1 := \tcn ( \mathcal{L}_1 \xrightarrow{\; (\phi_1 - 
    \mu_1^{\textnormal{mod}}(\theta_1); \rho_1 + \ell \kappa_0,
    \bme^{\ephia}) \;} \mathcal{K}_0 ).$$ Note that
  $\bme^{\mathcal{M}_1} = \max \{\bmea, \bme^{\mathcal{K}_0},
  \bme^{\ephia}-\epsilon_1^{\ephia} \} = \bme^{\ephia} -
  \epsilon_1^{\ephia}$ because
  $\bme^{\ephia} \in \hyperlink{h:asmp-e}{\lbe(\bmea,
    \bme^{\mathcal{K}_0})}$.
  
  For later use we will need to address Assumption~$\lbh_w$ for the
  module $\mathcal{M}_1$. Indeed, by Lemma~\ref{l:assh-cones} we have
  $\mathcal{M}_1 \in \hyperlink{h:asmp-H}{\lbh_w(\kappa_1)}$, where
  $$\kappa_1 := \max \{ 2\kappa_0, 
  2\ua+ \epsilon_3^{\ephia} - \epsilon_1^{\ephia}, 2\ua +
  2\epsilon_2^{\ephia} - 2\epsilon_1^{\ephia}, \kappa_0 +
  \epsilon_2^{\ephia}-\epsilon_1^{\ephia} \}.$$

  The modules $\mathcal{K}'_1 = \mathcal{K}_1$, $\mathcal{K}_1''$ and
  $\mathcal{M}_1$ are related by weakly filtered quasi-isomorphisms as
  follows. The identity homomorphism can be viewed as a weakly
  filtered quasi-isomorphism
  $I_1: \mathcal{K}_1 \longrightarrow \mathcal{K}_1''$ which shifts
  action by $\leq \ell \kappa_0$ and has discrepancy
  $\leq (\epsilon_1^{\ephia} - \delta_1^{\phi_1}, 0, \ldots, 0,
  \ldots)$. Lemma~\ref{l:cone-f-f'} provides a quasi-isomorphism
  $\vartheta_1: \mathcal{K}_1'' \longrightarrow \mathcal{M}_1$ which
  shifts action by $\leq 0$ and has discrepancy
  $\leq \bme^{\ephia} - \epsilon_1^{\ephia}$. Consider
  $\eta_1: \mathcal{K}_1 \to \mathcal{M}_1$ given by the composition
  $$\eta_1 := \vartheta_1 \circ I_1.$$ 
  This is a quasi-isomorphism which shifts action by
  $\leq \ell \kappa_0$ and has discrepancy
  $\leq \bme^{\ephia} - \delta_1^{\phi_1}$.

  The first order part
  $(\eta_1)_1:\mathcal{K}_1(X) \longrightarrow \mathcal{M}_1(X)$ of
  the module homomorphism $\eta_1$ is an isomorphism of chain
  complexes for all $X$ and its matrix (with respect to the splitting
  $C(X,L_0) \oplus C(X,L_1)$ of $\mathcal{K}_1(X)$ and
  $\mathcal{M}_1(X)$ as $R$-modules) is upper triangular with $\id$'s
  along the diagonal. This follows from the explicit formula of
  $(\vartheta_1)_1$ from the proof of Lemma~\ref{l:cone-f-f'}. The
  same formula also shows that $(\vartheta_1)_1^{-1}$ shift action by
  $\leq 0$ and the same holds for $(I_1)_1^{-1}$. It follows that the
  inverse $(\eta_1)_1^{-1}$ of $(\eta_1)_1$ shifts action by $\leq 0$.

  Next, consider the composition
  $\eta_1 \circ \phi_2: \mathcal{L}_2 \longrightarrow \mathcal{M}_1$.
  This is a module homomorphism that shifts action by
  $\leq \rho_2 + \ell \kappa_0$ and has discrepancy
  $\leq \bme^{\eta_1 \circ \phi_2} = \bme^{\eta_1} \ast
  \bm{\delta}^{\phi_2}$. Define now:
  \begin{equation*}
    \begin{aligned}
      \mathcal{K}_2' & = \tcn ( \mathcal{L}_2 \xrightarrow{\;
        (\eta_1 \circ \phi_2; \rho_2 + \ell \kappa_0,
        \bme^{\eta_1} \ast \bm{\delta}^{\phi_2}) \;}
      \mathcal{M}_1 ), \\
      \mathcal{K}_2'' & = \tcn ( \mathcal{L}_2 \xrightarrow{\;
        (\eta_1 \circ \phi_2; \rho_2 + \ell \kappa_1 + \ell
        \kappa_0, \bme^{\ephib}) \;} \mathcal{M}_1 ),
    \end{aligned}
  \end{equation*}
  where $\bme^{\ephib}$ is chosen such that:
  $$\bme^{\ephib} \geq \bme^{\eta_1} \ast \bm{\delta}^{\phi_2}, \quad 
  \bme^{\ephib} \in \hyperlink{h:asmp-e}{\lbe(\bmea,
    \bme^{\mathcal{M}_1})}, \quad \epsilon^{\ephib}_d \geq
  \epsilon^{\mathcal{M}_1}_{d+1} \;\; \forall \; d.$$
  
  Applying Proposition~\ref{p:lmap} we can write:
  $$\eta_1 \circ \phi_2 = \lambda(c_2) + 
  \mu_1^{\textnormal{mod}}(\theta_2) + \tau_2,$$ where
  $c_2 \in \mathcal{M}_1^{\leq \rho_2 + \ell \kappa_1 + \ell
    \kappa_0}(L_2)$, is a cycle and
  $\theta_2, \tau_2 \in \hom^{\leq \rho_2 + \ell \kappa_1 + \ell
    \kappa_0; \bme^{\ephib}}(\mathcal{L}_2, \mathcal{M}_1)$ with
  $\tau_2$ being a cycle such that
  $(\tau_2)_1 = \cdots = (\tau_2)_{\ell} = 0$.

  We define now
  $$\mathcal{M}_2 := \tcn ( \mathcal{L}_2 \xrightarrow{\; 
    (\eta_1 \circ \phi_2 - \mu_1^{\textnormal{mod}}(\theta_2);
    \rho_2 + \ell \kappa_1 + \ell \kappa_0, \bme^{\ephib})}
  \mathcal{M}_1).$$ The discrepancy of $\mathcal{M}_2$ is
  $\leq \bme^{\mathcal{M}_2} := \max \{\bmea, \bme^{\mathcal{M}_1},
  \bme^{\ephib} - \epsilon_1^{\ephib} \} = \max\{\bme^{\ephia} -
  \epsilon_1^{\ephia}, \bme^{\ephib} - \epsilon_1^{\ephib}\}$.

  By Lemma~\ref{l:assh-cones} we have
  $\mathcal{M}_2 \in \hyperlink{h:asmp-H}{\lbh_w(\kappa_2)}$, where
  \begin{equation*}
    \begin{aligned}
      \kappa_2 := \max \{ 2\kappa_1, &
      2\ua+\epsilon_3^{\ephia}-\epsilon_1^{\ephia}, 2\ua +
      \epsilon_3^{\ephib}-\epsilon_1^{\ephib}, \\
      & 2\ua + 2\epsilon_2^{\ephia}-2\epsilon_1^{\ephia},
      2\ua+2\epsilon_2^{\ephib}-2\epsilon_1^{\ephib}, \\
      & \kappa_0+\epsilon_2^{\ephia}-\epsilon_1^{\ephia}, \kappa_0 +
      \epsilon_2^{\ephib} - \epsilon_1^{\ephib} \}.
    \end{aligned}
  \end{equation*}

  The modules $\mathcal{K}_2$, $\mathcal{K}_2'$, $\mathcal{K}_2''$ and
  $\mathcal{M}_2$ are related by weakly filtered quasi-isomorphisms:
  $$\mathcal{K}_2 \xrightarrow[\simeq]{\; \psi_2 \;} \mathcal{K}_2' 
  \xrightarrow[\simeq]{\; I_2 \;} \mathcal{K}_2''
  \xrightarrow[\simeq]{\; \vartheta_2 \;} \mathcal{M}_2,$$ where the
  shifts in action and discrepancies of these maps are given by:
  \begin{equation*}
    \begin{aligned}
      & \text{shift}(\psi_2) \leq \ell \kappa_0, \quad \bme^{\psi_2}
      \leq
      \bme^{\ephia}-\delta_1^{\phi_1}, \\
      & \text{shift}(I_2) \leq \ell \kappa_1, \quad \bme^{I_2} \leq
      (\epsilon_1^{\ephib}-\delta_1^{\phi_2} -
      \epsilon_1^{\ephia}+\delta_1^{\phi_1}, 0, \ldots, 0, \ldots), \\
      & \text{shift}(\vartheta_2) \leq 0, \quad \bme^{\vartheta_2}
      \leq \bme^{\ephib}-\epsilon_1^{\ephib}.
    \end{aligned}
  \end{equation*}
  The quasi-isomorphism $\psi_2$ is obtained from
  Lemma~\ref{l:con-f-g} and $\vartheta_2$ from
  Lemma~\ref{l:cone-f-f'}. The quasi-isomorphism $I_2$ is basically
  the identity map, relating the same module with two (slightly)
  different structures of weakly filtered module.

  Consider now the composition
  $\eta_2 = \vartheta_2 \circ I_2 \circ \psi_2: \mathcal{K}_2
  \longrightarrow \mathcal{M}_2$. This quasi-isomorphism has the
  following action shift and discrepancy:
  $$\text{shift}(\eta_2) \leq \ell (\kappa_1+\kappa_0), 
  \quad \bme^{\eta_2} \leq \bme^{\ephib} \ast \bme^{\ephia} -
  (\delta_1^{\phi_2} + \epsilon_1^{\ephia}).$$

  As in the previous step, the first order part
  $(\eta_2)_1: \mathcal{K}_2(X) \longrightarrow \mathcal{M}_2(X)$ of
  $\eta_2$ is an isomorphism of chain complexes and its matrix (with
  respect to the splitting $C(X,L_0) \oplus C(X,L_1) \oplus C(X,L_2)$
  of $\mathcal{K}_2(X)$ and $\mathcal{M}_2(X)$ as $R$-modules) is
  upper triangular with $\id$'s along the diagonal. Moreover, the
  inverse $(\eta_2)_1^{-1}$ of $(\eta_2)_1$ shifts action by $\leq
  0$. These assertions easily follows from the explicit formulas of
  $(\psi_2)_1$ and $(\vartheta_2)_1$ given in the proofs of
  Lemmas~\ref{l:con-f-g} and~\ref{l:cone-f-f'} respectively and the
  fact, already shown in the previous step, that $(\eta_1)_1$ is a
  chain isomorphism represented by an upper triangular matrix with
  $\id$'s along the diagonal. Recall also from the previous step that
  $(\eta_1)_1^{-1}$ shifts action by $\leq 0$. An examination of the
  action shifts shows that each of the maps $(I_2)_1^{-1}$,
  $(\psi_2)_1^{-1}$ and $(\vartheta_2)_1^{-1}$ shifts action by
  $\leq 0$, hence the same holds for $(\eta_2)_1^{-1}$.

  To exemplify the structure of the proof, let us work out yet another
  step - i.e. the construction of $\mathcal{M}_3$.  Consider the
  homomorphism
  $\eta_2 \circ \phi_3: \mathcal{L}_3 \longrightarrow \mathcal{M}_2$. Its
  action-shift and discrepancy are given by:
  $$\text{shift}(\eta_2 \circ \phi_3) \leq 
  \rho_3 + \ell (\kappa_1+\kappa_0), \quad \bme^{\eta_2 \circ
    \phi_3} \leq \bm{\delta}^{\phi_3} \ast \bme^{\eta_2}.$$ Choose
  $\bme^{\ephic}$ such that:
   $$\bme^{\ephic} \geq \bme^{\eta_2} \ast \bm{\delta}^{\phi_3}, \quad 
   \bme^{\ephic} \in \hyperlink{h:asmp-e}{\lbe(\bmea,
     \bme^{\mathcal{M}_2})}, \quad \epsilon^{\ephic}_d \geq
   \epsilon^{\mathcal{M}_2}_{d+1} \;\; \forall \; d.$$ Define:

  \begin{equation*}
    \begin{aligned}
      \mathcal{K}_3' & = \tcn ( \mathcal{L}_3 \xrightarrow{\;
        (\eta_2 \circ \phi_3; \rho_3 + \ell (\kappa_1+\kappa_0),
        \bme^{\eta_2} \ast \bm{\delta}^{\phi_3}) \;}
      \mathcal{M}_2 ), \\
      \mathcal{K}_3'' & = \tcn ( \mathcal{L}_3 \xrightarrow{\;
        (\eta_2 \circ \phi_3; \rho_3 + \ell (\kappa_2 + \kappa_1 +
        \kappa_0), \bme^{\ephic}) \;} \mathcal{M}_2 ).
    \end{aligned}
  \end{equation*}
  By Proposition~\ref{p:lmap} we can write:
  $$\eta_2 \circ \phi_3 = \lambda(c_3) + 
  \mu_1^{\textnormal{mod}}(\theta_3) + \tau_3,$$ where
  $c_3 \in \mathcal{M}_2^{\leq \rho_3 + \ell (\kappa_2 + \kappa_1 +
    \kappa_0)}(L_3)$ is a cycle,
  $\theta_3, \tau_3 \in \hom^{\leq \rho_3 + \ell (\kappa_2 + \kappa_1
    + \kappa_0);\bme^{\ephic}}(\mathcal{L}_3, \mathcal{M}_2)$ and
  $\tau$ is a cycle such that:
  $(\tau_3)_1 = \cdots = (\tau_3)_{\ell} = 0$.

  We now define:
  $$\mathcal{M}_3 := 
  \tcn ( \mathcal{L}_3 \xrightarrow{\; (\eta_2 \circ \phi_3 -
    \mu_1^{\textnormal{mod}}(\theta_3); \rho_3 + \ell (\kappa_2 +
    \kappa_1 + \kappa_0), \bme^{\ephic})} \mathcal{M}_2).$$ The
  discrepancy of $\mathcal{M}_3$ is
  $\leq \bme^{\mathcal{M}_3} := \max \{\bmea, \bme^{\mathcal{M}_2},
  \bme^{\ephic} - \epsilon_1^{\ephic} \} = \max\{\bme^{\ephii} -
  \epsilon_1^{\ephii} \mid 1 \leq i \leq 3 \}.$

  By Lemma~\ref{l:assh-cones} we have
  $\mathcal{M}_3 \in \hyperlink{h:asmp-H}{\lbh_w(\kappa_3)}$, where
  \begin{equation*}
    \kappa_3 := \max \{ 2\kappa_2,
    2\ua+\epsilon_3^{\ephii}-\epsilon_1^{\ephii}, 
    2\ua + 2\epsilon_2^{\ephii}-2\epsilon_1^{\ephii}, 
    \kappa_0+\epsilon_2^{\ephii}-\epsilon_1^{\ephii} \mid 1 \leq i \leq 3 \}.
  \end{equation*}

  The modules $\mathcal{K}_3$, $\mathcal{K}_3'$, $\mathcal{K}_3''$ and
  $\mathcal{M}_3$ are related by weakly filtered quasi-isomorphisms:
  $$\mathcal{K}_3 \xrightarrow[\simeq]{\; \psi_3 \;} \mathcal{K}_3' 
  \xrightarrow[\simeq]{\; I_3 \;} \mathcal{K}_3''
  \xrightarrow[\simeq]{\; \vartheta_3 \;} \mathcal{M}_3,$$ where the
  shifts in action and discrepancies of these maps are given by:
  \begin{equation*}
    \begin{aligned}
      & \text{shift}(\psi_3) \leq \ell (\kappa_1 + \kappa_0), \quad
      \bme^{\psi_3} \leq
      \bme^{\ephib}-\delta_1^{\phi_2}, \\
      & \text{shift}(I_3) \leq \ell \kappa_2, \quad \bme^{I_3} \leq
      (\epsilon_1^{\ephic}-\delta_1^{\phi_3} -
      \epsilon_1^{\ephib}+\delta_1^{\phi_2}, 0, \ldots, 0, \ldots), \\
      & \text{shift}(\vartheta_3) \leq 0, \quad \bme^{\vartheta_3}
      \leq \bme^{\ephic}-\epsilon_1^{\ephic}.
    \end{aligned}
  \end{equation*}
  The quasi-isomorphism $\psi_3$ is obtained from
  Lemma~\ref{l:con-f-g} and $\vartheta_3$ from
  Lemma~\ref{l:cone-f-f'}. The quasi-isomorphism $I_3$ is basically
  the identity map, relating the same module with two (slightly)
  different structures of weakly filtered module.  Define
  $\eta_3: \mathcal{K}_3 \longrightarrow \mathcal{M}_3$ to be the
  composition $\eta_3 = \vartheta_3 \circ I_3 \circ \psi_3$. We
  have:
  $$\text{shift}(\eta_3) \leq \ell (\kappa_2+\kappa_1+\kappa_0), 
  \quad \bme^{\eta_3} \leq \bme^{\ephic} \ast \bme^{\ephib} \ast
  \bme^{\ephia} - (\delta_1^{\phi_3} + \epsilon_1^{\ephib} +
  \epsilon_1^{\ephia}).$$

  As in the previous step, the first order part $(\eta_3)_1$ of
  $\eta_3$ is a chain isomorphism represented by an upper triangular
  matrix with $\id$'s along the diagonal and its inverse
  $(\eta_3)_1^{-1}$ shifts action by $\leq 0$. This is proved in the
  same way as in the previous step for $(\eta_2)_1$, keeping in mind
  that $(\eta_2)_1^{-1}$ shifts action by $\leq 0$.

 \medskip

  Continuing as above by induction we obtain the following things for
  every $1 \leq j \leq r$:
  \begin{enumerate}
  \item A weakly filtered module $\mathcal{M}_j$.
  \item Two sequences of non-negative real numbers $\bme^{\ephij}$ and
    $\bme^{\eta_j}$ that satisfy: \label{i:seqs-sig-phi}
    \begin{enumerate}
    \item
      $\bme^{\ephij} \geq \bme^{\eta_{j-1}} \ast
      \bm{\delta}^{\phi_j}, \quad \bme^{\ephij} \in
      \hyperlink{h:asmp-e}{\lbe(\bmea, \bme^{\mathcal{M}_{j-1}})},
      \quad \epsilon^{\ephij}_d \geq
      \epsilon^{\mathcal{M}_{j-1}}_{d+1} \;\; \forall \; d.$ \label{i:eps-i}
    \item
      $\bme^{\eta_j} \leq \bme^{\ephij} \ast \cdots \ast
      \bme^{\ephia} - (\delta_1^{\phi_j} + \sum_{i=1}^{j-1}
      \epsilon_1^{\ephii}).$ \label{i:eps-sig-j}
    \end{enumerate}
    We use the convention that
    $\bme^{\eta_0} = (0, \ldots, 0, \ldots)$.
  \item A positive real number $\kappa_j$ defined (inductively) by:
    \begin{equation*}
      \kappa_j := \max \{ 2\kappa_{j-1},
      2\ua+\epsilon_3^{\ephii}-\epsilon_1^{\ephii}, 
      2\ua + 2\epsilon_2^{\ephii}-2\epsilon_1^{\ephii}, 
      \kappa_0+\epsilon_2^{\ephii}-\epsilon_1^{\ephii} \mid 1 \leq i \leq j \}.
    \end{equation*}
    (Recall that $\kappa_0 = \kappa$.)
  \item A cycle
    $c_j \in \mathcal{M}_{j-1}^{\leq \rho_j + \sum_{i=0}^{j-1}
      \kappa_i}(L_i)$.
  \item The module $\mathcal{M}_j$ is related to $\mathcal{M}_{j-1}$ by:
    \begin{equation} \label{eq:cone-M_j}
      \mathcal{M}_j = \tcn \bigl( \mathcal{L}_j \xrightarrow{\; (
        \lambda(c_j) + \tau_j; \rho_j + \ell \sum_{i=0}^{j-1} \kappa_i,
        \bme^{\ephij}) \;} \mathcal{M}_{j-1} \bigr),
    \end{equation}
    where
    $\tau_j \in \hom^{\leq \rho_j + \ell(\sum_{i=0}^{j-1} \kappa_i);
      \bme^{\ephij}}(\mathcal{L}_j, \mathcal{M}_{j-1})$ is a cycle
    with $(\tau_j)_1 = \cdots = (\tau_j)_{\ell} = 0$.
  \item The discrepancy of $\mathcal{M}_j$ is given by:
    $$\bme^{\mathcal{M}_j} \leq \max \{ \bme^{\ephii} -
    \epsilon_1^{\ephii} \mid 1 \leq i \leq j\}.$$
  \item $\mathcal{M}_j \in \hyperlink{h:asmp-H}{\lbh_w(\kappa_j)}$.
  \item A weakly filtered quasi-isomorphism
    $\eta_j: \mathcal{K}_j \longrightarrow \mathcal{M}_j$ which shifts
    action by $\leq \ell(\kappa_0 + \cdots + \kappa_{j-1})$ and with
    discrepancy $\leq \bme^{\eta_j}$, where the sequences
    $\bme^{\eta_j}$ is the one from point~(\ref{i:seqs-sig-phi})
    above. Moreover, the first order part $(\eta_j)_1$ is a chain
    isomorphism represented by an upper triangular matrix with $\id$'s
    along the diagonal (with respect to the splitting
    $CF(X,L_0) \oplus \cdots \oplus C(X,L_r)$) and its inverse
    $(\eta_j)_1^{-1}$ shifts action by $\leq 0$.
  \end{enumerate}

  The module $\mathcal{M}$ claimed in the statement of the theorem is
  the module $\mathcal{M}_r$, and the quasi-isomorphism of
  $\mathcal{A}$-modules
  $\sigma:\mathcal{K}_r \longrightarrow \mathcal{M}$ is $\eta_r$.

  \medskip

  Next, we analyze the differential $\mu_1^{\mathcal{M}_j}$ on the
  modules $\mathcal{M}_j$. We begin with the module
  \begin{equation} \label{eq:cone-M_1}
    \mathcal{M}_1 = \tcn ( \mathcal{L}_1 
    \xrightarrow{\; (\lambda(c_1) + \tau_1; \rho_1 + \ell \kappa_0,
      \bme^{\ephia}) \;} \mathcal{K}_0 ).
  \end{equation}
  Recall that
  $c_1 \in \mathcal{K}_0^{\leq \rho_1 + \ell \kappa_0} (L_1) = C^{\leq
    \rho_1 + \ell \kappa_0}(L_1, L_0)$.  For further use, we will
  write $c_{1,0} := c_1$.

  Let $X \in \textnormal{Ob}(\mathcal{A})$. Write
  $$\mathcal{M}_1(X) = C(X, L_1) \oplus C(X,L_0)$$ as $R$-modules.
  By the definition of the map $\lambda$ we have according to this
  splitting:
  $$\mu_1^{\mathcal{M}_1}(b_1, b_0) = 
  \bigl( \mu_1^{\mathcal{A}}(b_1), \mu_1^{\mathcal{A}}(b_0) +
  \mu_2^{\mathcal{A}}(b_1, c_{1,0}) \bigr), \quad \forall \; b_1 \in
  C(X, L_1), b_0 \in C(X, L_0).$$ More generally, the higher
  operations $\mu_d^{\mathcal{M}_1}$ have the following form. Let
  $1 \leq d \leq \ell-1$ and
  $X_0, \ldots, X_{d-1} \in \textnormal{Ob}(\mathcal{A})$. Then:
  \begin{equation} \label{eq:mu_d-M_1}
    \begin{aligned}
      & \mu_d^{\mathcal{M}_1}(a_1, \ldots, a_{d-1}, (b_1, b_0)) \\
      & \quad = \bigl( \mu_d^{\mathcal{A}}(a_1, \ldots, a_{d-1}, b_1),
      \mu_d^{\mathcal{A}}(a_1, \ldots, a_{d-1}, b_0) +
      \mu_{d+1}^{\mathcal{A}}(a_1, \ldots, a_{d-1}, b_1, c_{1,0}) \bigr), \\
      & \forall a_i \in C(X_{i-1}, X_i), \; i=1, \ldots, d, \; \;
      \forall \; (b_1, b_0) \in C(X_d, L_1) \oplus C(X_d, L_0).
    \end{aligned}
  \end{equation}
  Note that the term $\tau_1$ in~\eqref{eq:cone-M_1} does not play any
  role in the expression for $\mu_d^{\mathcal{M}_1}$ as long as
  $d \leq \ell-1$, since $(\tau_1)_1 = \cdots = (\tau_1)_{\ell} =
  0$. Recall also that $\ell \gg r$.

  We now analyze $\mathcal{M}_2$. Recall that:
  \begin{equation} \label{eq:cone-M_2} \mathcal{M}_2 := \tcn (
    \mathcal{L}_2 \xrightarrow{\; (\lambda(c_2) + \tau_2; \rho_2 +
      \ell \kappa_1 + \ell \kappa_0, \bme^{\ephib})} \mathcal{M}_1),
  \end{equation}
  where
  $c_2 \in \mathcal{M}_1^{\leq \rho_2 + \ell (\kappa_1 +
    \kappa_0)}(L_2)$. Recall that
  $$\mathcal{M}_1^{\leq \rho_2 + \ell (\kappa_1 + 
    \kappa_0)}(L_2) = C^{\leq \rho_2 - \rho_1 + \ell \kappa_1 -
    \epsilon_1^{\ephia}}(L_2, L_1) \oplus C^{\leq \rho_2 + \ell
    (\kappa_1 + \kappa_0)}(L_2, L_0)$$ as $R$-modules. Write
  $c_2 = (c_{2,1}, c_{2,0})$ with respect to this splitting.
  
  Let $X \in \textnormal{Ob}(\mathcal{A})$ and write
  \begin{equation} \label{eq:M_2-splitting-1} \mathcal{M}_2(X) = C(X,
    L_2) \oplus \mathcal{M}_1(X) = C(X, L_2) \oplus C(X, L_1) \oplus
    C(X, L_0)
  \end{equation}
  as $R$-modules. By the definition of $\lambda$ together
  with~\eqref{eq:mu_d-M_1} we have:
  \begin{equation} \label{eq:mu_1-M_2}
    \begin{aligned}
      & \mu_1^{\mathcal{M}_2}(b_2, b_1, b_0) = \bigl(
      \mu_1^{\mathcal{A}}(b_2), \mu_1^{\mathcal{M}_1}(b_1, b_0) +
      \mu_2^{\mathcal{M}_1}(b_2, c_2) \bigr) \\
      & = \bigl( \mu_1^{\mathcal{A}}(b_2), \mu_1^{\mathcal{A}}(b_1)
      + \mu_2^{\mathcal{A}}(b_2, c_{2,1}), \mu_1^{\mathcal{A}}(b_0)
      + \mu_2^{\mathcal{A}}(b_1, c_{1,0}) + \mu_2^{\mathcal{A}}(b_2,
      c_{2,0}) + \mu_3^{\mathcal{A}}(b_2, c_{2,1}, c_{1,0}) \bigr).
    \end{aligned}
  \end{equation}
  In other words, the matrix of $\mu_1^{\mathcal{M}_2}$ has the
  following shape:

  \begin{equation} \label{eq:mat-mu_1-M_2}
    \mu_1^{\mathcal{M}_2} = 
    \begin{pmatrix}
      \mu_1^{\mathcal{A}}(-) & \mu_2^{\mathcal{A}}(-, c_{1,0}) &
      \mu_2^{\mathcal{A}}(-, c_{2,0}) + \mu_3^{\mathcal{A}}(-,
      c_{2,1}, c_{1,0}) \\
      0 & \mu_1^{\mathcal{A}}(-) & \mu_2^{\mathcal{A}}(-, c_{2,1}) \\
      0 & 0 & \mu_1^{\mathcal{A}}(-)
    \end{pmatrix}
  \end{equation}
  Here the matrix has been calculated with respect to the splitting
  $$\mathcal{M}_2(X) = 
  C(X, L_0) \oplus C(X, L_1) \oplus C(X, L_2)$$ (in contrast
  to~\eqref{eq:M_2-splitting-1} and~\eqref{eq:mu_1-M_2}) in order to
  be compatible with~\eqref{eq:M-direct-sum}.

  A similar formula holds also for the higher operations
  $\mu_d^{\mathcal{M}_2}$. More precisely, Let $1 \leq d \leq \ell-2$
  and $X_0, \ldots, X_{d-1} \in \textnormal{Ob}(\mathcal{A})$. Then:
  \begin{equation} \label{eq:mu_d-M_2}
    \begin{aligned}
      \mu_d^{\mathcal{M}_2}(\underline{a}, b_2, b_1, b_0)
      \\
      = \bigl( \mu_d^{\mathcal{A}}(\underline{a}, b_2), \,  &
      \mu_d^{\mathcal{A}}(\underline{a}, b_1) +
      \mu_{d+1}^{\mathcal{A}}(\underline{a}, b_2, c_{2,1}),  \\
      & \mu_d^{\mathcal{A}}(\underline{a}, b_0) +
      \mu_{d+1}^{\mathcal{A}}(\underline{a}, b_1, c_{1,0}) +
      \mu_{d+1}^{\mathcal{A}}(\underline{a}, b_2, c_{2,0}) + 
      \mu_{d+2}^{\mathcal{A}}(\underline{a}, b_2, c_{2,1}, c_{1,0})
      \bigr),
    \end{aligned}
  \end{equation}
  for all
  $\underline{a} \in C(X_0, X_1) \otimes \cdots \otimes C(X_{d-2},
  X_{d-1})$.

  Continuing by induction as above, we obtain the elements
  $c_{q,p} \in C(L_q,L_p)$ for all $0\leq q < p \leq r$ and the
  operators $a_{i,j}$, $i>j$, as described in~\eqref{eq:aij}, which
  form the matrix of the differentials $\mu_1^{\mathcal{M}}$ for the
  module $\mathcal{M} = \mathcal{M}_r$.

  Note that the $\mu_k^{\mathcal{M}_j}$-operation of the intermediate
  module $\mathcal{M}_j$ involves expressions containing
  $\mu_d^{\mathcal{A}}$ for $d \leq j+k$ but no higher order $\mu$'s.
  It is also important to remark that at every step of the
  construction, the operations $\mu_d^{\mathcal{M}_j}$ for
  $d \leq r+1-j$ will depends on the cycles $c_{q,p}$ with
  $0\leq p<q\leq j$ but {\em not} on the elements $\tau_i$ that appear
  in~\eqref{eq:cone-M_j}. The reason is that
  $(\tau_i)_1 = \cdots = (\tau_i)_{\ell} = 0$ and we have chosen in
  advance $\ell = r+2$.

  Next, we estimate the action levels $\alpha_{q,p}$ of $c_{q,p}$
  from~\eqref{eq:alpha_qp} and the action shift and discrepancy of the
  quasi-isomorphism $\sigma = \eta_r$ as claimed
  in~\eqref{eq:rho-sigma}.

  An inspection of the previous steps in the proof shows that
  $$c_{q,p} \in C^{\leq \rho_q-\rho_p + 
    \ell(\kappa_p+ \cdots + \kappa_{q-1}) -
    \epsilon_1^{\ephip}}(L_q, L_p).$$ Thus we need to estimate the
  $\kappa_j$'s. This, in turn, would require to estimate the
  $\bme^{\ephii}$'s.
  
  Note that we can choose at every step of the previous inductive
  construction the sequence $\bme^{\ephij}$ at~(\ref{i:eps-i}) to
  satisfy:
  $$\epsilon_d^{\ephij} \leq 
  \epsilon_{d+1}^{\mathcal{M}_{j-1}} + \sum_{i=1}^d \bigl(
  \epsilon_i^{\eta_{j-1}} + \delta_i^{\phi_j} +
  \epsilon_i^{\mathcal{A}} + \epsilon_i^{\mathcal{M}_{j-1}}
  \bigr).$$ A simple inductive argument now implies the desired
  estimates for the $\bme_d^{\ephij}$'s the $\kappa_j$'s as well as
  for the action shift of $\eta_j$ and its discrepancy.

  \medskip Finally, the first statement at
  point~\eqref{pp:Delta_j-Delta_0} follows easily from the induction
  process defining the maps $\eta_i$, $i=1, \ldots, r$, by examining
  the filtrations induced on $CF(X,L_j)$ by each of $\mathcal{K}_i(X)$
  and $\mathcal{M}_i(X)$ for $j \leq i \leq r$.  That $C(X,L_0)$ is a
  subcomplex of both $\mathcal{K}_i(X)$ and $\mathcal{M}_i(X)$ follows
  from the fact that $\mathcal{K}_i$ and $\mathcal{M}_i$ are both
  iterated cones starting with the object $\mathcal{L}_0$.

\end{proof}

\subsection{Invariants and measurements for filtered chain
  complexes} \label{s:filt-ch}

As a supplement to the previous material we describe
here a number of numerical invariants of filtered chain complexes that
will be useful in~\S\ref{s:main-geom} when we prove our main geometric
results.

We begin with basic definitions. Fix a commutative ring $\mathcal{R}$
with unity. By a filtered chain complex we mean a chain complex
$(C, d^C)$ of $\mathcal{R}$-modules endowed with an increasing
filtration by sub-chain complexes $C^{\leq \alpha} \subset C$, indexed
by the real numbers $\alpha \in \mathbb{R}$. An $\mathcal{R}$-linear
map $f: C \longrightarrow D$ between two filtered chain complexes
$(C, d^C)$, $(D, d^D)$ is called {\em filtered} if there exists
$\rho \in \mathbb{R}$ such that
$f(C^{\leq \alpha}) \subset D^{\leq \alpha + \rho}$ for every
$\alpha$. In that case we also say that $f$ shifts action by
$\leq \rho$.  In case $f$ preserves the filtrations (i.e.~it shifts
filtration by $\leq 0$) we say that $f$ is {\em strictly filtered}.

Let $C, D$ be two chain complexes and assume $D$ is filtered. Every
chain map $\phi:C \longrightarrow D$ induces a filtration on $C$ by
$C^{\leq \alpha} := \phi^{-1}(D^{\leq \alpha})$ which we call the {\em
  pull-back filtration} by $\phi$. Endowing $C$ with this filtration
makes the map $\phi$ strictly filtered.

Let $C$ be a filtered chain complex, and $x \in C$. Define
$A(x) \in \mathbb{R} \cup \{-\infty, \infty\}$ to be the infimal
filtration level of $C$ which contains $x$, i.e.
$A(x) := \inf \{\alpha \in \mathbb{R} \mid x \in C^{\leq \alpha}\}$.
We call $A(x)$ the {\em action level} of $x$. Sometimes we will write
$A(x;C)$ instead of $A(x)$ in order to keep track of the chain complex
that $x$ belongs to. By our conventions we have $A(0) = -\infty$ and
if $\cap_{\alpha \in \mathbb{R}} C^{\leq \alpha} = \{0\}$ then
$A(x)=-\infty$ iff $x=0$. Also, if the filtration on $C$ is
exhaustive, i.e. $\cup_{\alpha \in \mathbb{R}} C^{\leq \alpha} = C$,
then $A(x) < \infty$ for every $x \in C$.

Another measurement relevant for our considerations is the following.
\label{pp:delta-f}
Let $(C,d^{C})$, $(D, d^{D})$ be filtered chain complexes and
$f: C \longrightarrow D$ a strictly filtered $\mathcal{R}$-linear
map. Define the ``{\em action drop}'' of $f$ as
$\delta_{f} = \sup\{r\in [0,\infty) \mid \forall a\in \R, \ f(C^{\leq
  a}) \subset D^{\leq a-r}\}$. An important special case is when $C=D$
and $f = d^C$ the differential of $C$:
\begin{equation} \label{eq:delta-dC} \delta_{d^C} = \sup\{r\in
  [0,\infty) \mid \forall a\in \R, \ d^C(C^{\leq a}) \subset C^{\leq
    a-r}\}~.~
\end{equation}

A central measurement in our framework is the following. Let
$\psi: C \longrightarrow D$ be a filtered chain map and assume that
$\psi$ is null-homotopic. Define its {\em homotopical boudary level}
$B_h(\psi)$ to be the infimal action shift needed for a chain homotopy
between $\psi$ and $0$. More precisely:
\begin{equation} \label{eq:Bh-1}
  \begin{aligned}
    B_h(\psi) = \inf \bigl\{ \rho \in \mathbb{R} \mid \, & \exists \;
    \text{an} \; \mathcal{R} \text{-linear map} \; h: C
    \longrightarrow D\; \text{which shifts action by} \, \leq \rho, \\
    & \text{and such that} \; \psi = hd^C + d^Dh \bigr\}.
  \end{aligned}
\end{equation}
This notion is closely related to the boundary depth measurement
introduced by Usher~\cite{Usher1, Usher2}. In order to put things in
the right context we will explain this relation and further notions in
the next section.

\subsubsection{Boundary depth and related algebraic
  notions} \label{sec:bdry-depth} 

Let $(C,d^C)$, $(D, d^D)$ be filtered chain complexes and
$\phi: C \longrightarrow D$ a {\em strictly} filtered chain
map. Denote by $Z_C^{\leq \alpha} \subset C^{\leq \alpha}$ the cycles
of $C^{\leq \alpha}$ and by $B_D \subset D$ the boundaries of $D$.
\begin{dfn}\label{def:b-depth1}
  The boundary depth of $\phi$ is defined by:
  \begin{equation}
    \begin{aligned}
      \beta (\phi)=\inf \Bigl\{ r \geq 0 \ \bigm| \ & \forall \,
      \alpha \in \R, \forall \, x \in Z_C^{\leq \alpha} \; \text{with}
      \;
      \phi(x) \in B_D, \\
      & \exists \, b \in D^{\leq \alpha+r} \; \text{such that} \;
      \phi(x) = d^D(b)\Bigr \}.
  \end{aligned}
  \end{equation}
\end{dfn}

This notion was introduced and studied extensively in symplectic
topology (in a slightly different formulation) by Usher~\cite{Usher1,
  Usher2}.

A special case of interest is the following. Let $c\in (C,d^{C})$ be a
boundary. Denote by $\langle c \rangle$ be the chain complex with a
single generator $c$ and zero differential and let
$j : \langle c \rangle \longrightarrow C$ be the inclusion. Endow
$\langle c \rangle$ with the pull-back filtration by $j$. Note that
with this filtration $j$ is a strictly filtered map.  Define now
\begin{equation} \label{eq:beta-c}
  \beta(c;C) := \beta \big(\langle c \rangle \stackrel{j}{\longrightarrow}
  C \bigr)~.~
\end{equation}
It is easy to see that
$\beta(c;C) = \inf \{ r \geq 0 \mid c \; \text{is a boundary in} \;
C^{\leq A(c)+r} \}$.  Note also that for every boundary $c\in C$ we
have $\beta(c;C)\geq \delta_{d^C}$, where $\delta_{d^C}$ is defined
in~\eqref{eq:delta-dC}.

Sometimes it would be more convenient to work with the following
quantity instead of $\beta(c; C)$. Let $c \in C$ be a boundary. Define
\begin{equation} \label{eq:B-c} B(c;C) = \inf \bigl\{ \alpha \in
  \mathbb{R} \mid \, \exists \, b \in C^{\leq \alpha} \; \text{such
    that} \; c=d^Cb \bigr\}.
\end{equation}
Clearly, $B(c;C) = A(c;C) + \beta(c;C)$. We call $B(c;C)$ the {\em
  boundary level} of the element $c$.

Returning to the boundary depth measurement, another special case is
the boundary depth of a chain complex $C$. This is the boundary depth
of the identity of $C$, i.e. $\beta(C)=\beta(id_{C})$. This
measurement is especially useful when $C$ is acyclic.

Yet another instance of the same notion appears in the following
setting. Let $(C,d^C)$ and $(D,d^D)$ be two filtered chain
complexes. Denote by $\hom_{\mathcal{R}}(C,D)$ the
$\mathcal{R}$-linear maps $C \longrightarrow D$.  Then
$\hom_{\mathcal{R}}(C,D)$ is a chain complex with differential
$d^{\hom}(f) = d^D \circ f - f \circ d^C$. (In case $C$ and $D$ are
graded we take $\hom_{\mathcal{R}}(C,D)$ to consist only of linear
maps that preserve degree up to a shift, and the differential is
modified to be $d^{\hom}(f) = d^D \circ f - (-1)^{|f|} f \circ d^C$,
where $|f|$ is the degree-shift of $f$.) The cycles in
$\hom_{\mathcal{R}}(C,D)$ are chain maps and the boundaries are the
null-homotopic chain maps. The chain complex $\hom_{\mathcal{R}}(C,D)$
is filtered where $\hom_{\mathcal{R}}^{\leq \gamma}(C,D)$ is the
subcomplex consisting of all $\mathcal{R}$-linear maps
$C \longrightarrow D$ that shift action by $\leq \gamma$.

Now let $\psi: C \longrightarrow D$ be a filtered chain map and assume
that $\psi$ is null-homotopic (i.e. $\psi$ is a boundary in
$\hom_{\mathcal{R}}(C,D)$). We define the {\em homotopical boundary
  depth} $\beta_h(\psi)$ of $\psi$ by
\begin{equation} \label{eq:betta-h}
  \beta_{h}(\psi) := \beta(\psi; 
  \hom_{\mathcal{R}}(C,D))~.~
\end{equation}
More explicitly, $\beta_{h}(\psi)$ is the infimal $r \geq 0$ for which
$\psi$ is null-homotopic via a chain homotopy that shifts filtration
by $\leq A(\psi) + r$. As before, we define also the {\em homotopical
  boundary level} $$B_h(\psi) = B(\psi; \hom_{\mathcal{R}}(C,D)).$$

Note that if $\psi: C \longrightarrow D$ is a strictly filtered chain
map which is null homotopic then:
\begin{equation} \label{eq:beta-beta_h} \beta(\psi) \leq B_h(\psi) =
  A(\psi) + \beta_h(\psi) \leq \beta_h(\psi),
\end{equation}
where the last inequality holds because $\psi$ is strictly filtered,
hence $A(\psi)\leq 0$.

%

\begin{rem} \label{r:beta-filtered} We have defined the boundary depth
  $\beta(\psi)$ only for {\em strictly} filtered chain maps. However,
  the homotopical boundary depth $\beta_h(\psi)$ is defined for all
  filtered (null-homotopic) chain maps $\psi$, not only for the
  strictly filtered ones. This is because, in contrast to
  $\beta(\psi)$, the homotopical boundary depth $\beta_h(\psi)$ is
  defined as the boundary depth of the element $\psi$ inside the
  filtered chain complex $\hom_{\mathcal{R}}(C,D)$.
\end{rem}

Finally, here is another variant of the same measurement.  Let
$\mathcal{A}$ be a weakly filtered $A_{\infty}$-category with
discrepancy $\leq \bmea$ (see~\S\ref{s:wf-ai-theory}). Let
$\bmemm = (\epsilon_1^m=0, \epsilon_2^m, \ldots, \epsilon_d^m,
\ldots)$ be a sequence of non-negative real numbers, and let
$\mathcal{M}_0, \mathcal{M}_1$ be two weakly filtered
$\mathcal{A}$-modules with discrepancy $\leq \bmemm$. Let $\bmeh$ be
another sequence of non-negative real numbers, and assume that
$\bmeh \in~\hyperlink{h:asmp-e}{\lbe(\bmemm,
  \bmea)}$. (See~\S\ref{sb:mod}, page~\pageref{pp:assump-E}.)

Denote by $\hom^{\bmeh}(\mathcal{M}_0, \mathcal{M}_1)$ the weakly
filtered pre-module homomorphisms
$\mathcal{M}_0 \longrightarrow \mathcal{M}_1$ with discrepancy
$\leq \bmeh$ (and arbitrary action shift). As explained
in~\S\ref{sb:mod}, $\hom^{\bmeh}(\mathcal{M}_0, \mathcal{M}_1)$ is a
chain complex when endowed with the differential $\mu_1^{\text{mod}}$
of the dg-category of $\mathcal{A}$-modules. Moreover, this chain
complex is filtered by
$\hom^{\leq \rho; \bmeh}(\mathcal{M}_0, \mathcal{M}_1)$,
$\rho \in \mathbb{R}$.

Now let $\psi: \mathcal{M}_0 \longrightarrow \mathcal{M}_1$ be a
weakly filtered module homomorphism with discrepancy $\leq \bmeh$, and
assume that $\psi$ is a boundary in
$\hom^{\bmeh}(\mathcal{M}_0, \mathcal{M}_1)$ (i.e. $\psi$ is chain
homotopic to $0$ via a chain homotopy of pre-module maps with
discrepancy $\leq \bmeh$. Then we can define
$\beta_h(\psi; \bmeh) := \beta \bigl(\psi; \hom^{\bmeh}(\mathcal{M}_0,
\mathcal{M}_1)\bigr)$. Now, let $X \in \textnormal{Ob}(\mathcal{A})$
and denote by
$\psi_1^X: \mathcal{M}_0(X) \longrightarrow \mathcal{M}_1(X)$ the
chain map which is the 1'st order component of $\psi$ (corresponding
to the object $X$). By the preceding discussion we can associate to
$\psi_1^X$ its homotopical boundary depth
$\beta_h(\psi_1^X) = \beta \bigl(\psi_1^X;
\hom_{\mathcal{R}}(\mathcal{M}_0(X), \mathcal{M}_1(X)) \bigr)$ by the
recipe~\eqref{eq:betta-h}.  The two versions of the boundary depth
$\beta_h(\psi; \bmeh)$ and $\beta_h(\psi_1^X)$ are related as follows:
$$\beta_h(\psi_1^X) + 
A(\psi_1^X) \leq \beta_h(\psi; \bmeh) + A(\psi) + \epsilon_1^h.$$ Here
$A(\psi_1^X)$ is the action-level of $\psi_1^X$ viewed as an element
of the filtered chain complex
$\hom_{\mathcal{R}}(\mathcal{M}_0(X), \mathcal{M}_1(X))$ while
$A(\psi)$ is the action level of $\psi$, viewed as an element of the
filtered chain complex $\hom^{\bmeh}(\mathcal{M}_0, \mathcal{M}_1)$.

In case $\psi_1^X$ is strictly filtered, then we also have the
following two inequalities:
$$\beta(\psi_1^X) \leq \beta_h(\psi; \bmeh) + A(\psi) + \epsilon_1^h, 
\quad \beta(\psi_1^X) \leq \beta_h(\psi_1^X),$$ where the 2'nd
inequality comes from~\eqref{eq:beta-beta_h}.

%
%

\subsubsection{Algebraic implications} \label{sb:alg-impl}

We begin with a simple algebraic approximation lemma which says that a
altering a chain isomorphism by a chain homotopy yields an injective
map provided that the chain homotopy shifts action by a small enough
amount.

\begin{lem} \label{l:rig-cplx-n} Let $(C, d^C)$ and $(D, d^D)$ be
  filtered chain complexes, and assume that the filtration on $C$ is
  exhaustive (i.e.
  $\cup_{\alpha \in \mathbb{R}} C^{\leq \alpha} = C$) and separated
  (i.e.  $\cap_{\alpha \in \mathbb{R}} C^{\leq \alpha} = 0$). Let
  $f, g: C \longrightarrow D$ be chain maps with the following
  properties:
  \begin{enumerate}
  \item $g$ is an isomorphism.
  \item Both $g$ and $g^{-1}$ are strictly filtered.
  \item $f-g$ is null-homotopic and
    $B_h(f-g) < \min\{\delta_{d^C}, \delta_{d^D}\}$.
  \end{enumerate}
  Then $f$ is strictly filtered and moreover $f$ is injective.
\end{lem}
In our geometric applications $D=C$, $g$ will be the identity, and $f$
will be the composition of two chain morphisms
$C\stackrel{f_{1}}{\longrightarrow}
C'\stackrel{f_{2}}{\longrightarrow} C$ that are constructed
geometrically. The Lemma shows in this case that the middle complex
$C'$ contains $C$ as a retract. Results of this sort are familiar in symplectic topology  since 
\cite{Co-Ra:Morse-Novikov}. 

\begin{proof}[Proof of Lemma~\ref{l:rig-cplx-n}]
  Since the filtartion on $C$ is both exhaustive and separated, we
  have $-\infty < A(x) < \infty$ for every $x \neq 0$, and
  $A(0) = -\infty$.

  Set $\rho := B_h(f-g) + \epsilon$, where $\epsilon>0$ is small
  enough such that $\rho < \min \{\delta_{d^C}, \delta_{d^D}\}$.
  Write $f = g + \eta d^C + d^D \eta$, where
  $\eta : C \longrightarrow D$ is an $\mathcal{R}$-linear map that
  shifts action by $\leq \rho$.  Since
  $\rho < \min\{\delta_{d^C}, \delta_{d^D}\}$ and $g$ is strictly
  filtered we have
  $$A(f(x)) = A\bigl(g(x) + \eta d^C(x) + d^D \eta(x)\bigr) \leq A(x), 
  \; \forall x \in C,$$ hence $f$ is strictly filtered.

  For the injectivity of $f$, assume that $f(x)=0$ for some
  $x \neq 0$. Then $$g(x) = - (\eta d^C(x) + d^D \eta(x)),$$ and using
  again that $\rho < \min\{\delta_{d^C}, \delta_{d^D}\}$ we obtain
  that
  $$A(g(x)) = A(\eta d^C(x) + d^D \eta(x)) < A(x).$$ 
  The last inequality together with the assumption that $g^{-1}$ is
  strictly filtered imply that
  $$A(x) = A(g^{-1}g(x)) \leq A(g(x)) < A(x).$$ A contradiction.
\end{proof}

Under additional assumptions we can obtain a somewhat stronger
result. Before we state it, here are a couple of relevant notions.
The filtration $C^{\leq \alpha} \subset C$, $\alpha \in \mathbb{R}$
induces a topology on $C$ which is generated by the cosets of
$C^{\leq \alpha}$, $\alpha \in \mathbb{R}$, as basic open subsets. The
assumption that the filtration is separated (i.e.
$\cap_{\alpha \in \mathbb{R}} C^{\leq \alpha} = 0$) implies that $C$
is Hausdorff in this topology.

The filtration on $C$ is called complete if the obvious map
$C \longrightarrow \displaystyle{\varprojlim_{\alpha}} (C / C^{\leq
  \alpha})$ is surjective. This assumption implies that the previously
mentioned topology on $C$ turns $C$ into a complete topological space
(in the sense that every Cauchy sequence converges).

\begin{lem} \label{lem:rig-cplx-c} Let $(C,d^C)$, $(D,d^D)$, $f$, $g$
  be as in Lemma~\ref{l:rig-cplx-n} and assume in addition that the
  filtration on $C$ is complete. Then $f$ is a strictly filtered
  isomorphism and moreover $f^{-1}$ is also strictly filtered.
\end{lem}

\begin{proof}
  In view of Lemma~\ref{l:rig-cplx-n} we only need to show that $f$ is
  an isomorphism and that $f^{-1}$ is strictly filtered. 

  We will use a well-known inversion trick, that has already been used
  in a similar setting in~\cite{Usher1, Usher2}. Fix
  $0< \epsilon <(\min\{\delta_{d^C}, \delta_{d^D}\}-B_{h}(f-g))/2$. By
  the definition of $B_{h}$ there is an $\mathcal{R}$-linear map
  $\eta: C \longrightarrow D$ that shifts actions by
  $\leq B_h(f-g)+\epsilon$ such that $f-g=d^{C}\eta + \eta d^{C}$.
  Note that $f-g$ {\em decreases} action by at least $\epsilon$.

  Now write $f = g + (f - g) = g (\id+g^{-1}(f-g)) = g(\id - k)$,
  where $k: C \longrightarrow C$ is defined by $k=-g^{-1}(f-g)$. Since
  $g^{-1}$ is strictly filtered and $f-g$ decreases filtration by at
  least $\epsilon$, the same is true for $k$. As the filtration on $C$
  is complete, the series $a = \id + \sum_{n \geq 1} k^n$ converges,
  and satisfies $(\id-k)a = a(\id-k) = \id$. Therefore $f$ is
  invertible with inverse $a g^{-1}$, a strictly filtered
  chain-morphism of $\mathcal{R}$-modules.
\end{proof}

For the next result we will assume that $\mathcal{R} = \Lambda_0$ (the
positive Novikov ring over any field $R$). Recall that the Novikov
ring $\Lambda$ is the field of fractions of $\Lambda_0$.  Denote by
$\nu: \Lambda \longrightarrow \mathbb{R} \cup
\{\infty\}$ the standard valuation defined by
\begin{equation} \label{eq:val-Nov} \nu \Bigl(a_0 T^{\lambda_0} +
  \sum_{i=1}^{\infty} a_i T^{\lambda_i} \Bigr) = \lambda_0,
\end{equation}
where $a_0 \neq 0$ and $\lambda_i > \lambda_0$ for every $i\geq 1$. As
usual we set $\nu(0) = \infty$.

Let $C$ be a finite dimensional chain complex over $\Lambda$. Fix a
basis $\mathcal{G}$ of $C$ over $\Lambda$ and let
$A: \mathcal{G} \longrightarrow \mathbb{R}$ be a function. Similarly
to~\S\ref{sbsb:filtration-mixed} we will use $A$ to define a
filtration on $C$ by $\Lambda_0$-modules. Extend $A$ to a function
$A:C \to \R \cup \{-\infty\},$ by
$$A\Bigl(\sum \lambda_j e_j \Bigr) = \max\{-\nu(\lambda_j) + A(e_j)\},$$ 
where $e_j$ are the elements of the basis $\mathcal{G}$,
$0 \neq \lambda_j \in \Lambda$, $A(e_j)$ is the
pre-determined value of $A$ on the generator $e_j,$ and $\nu$ is the
preceding valuation. Define now 
$$C^{\leq \alpha} := \{ x\in C \mid A(x) \leq \alpha\}.$$
It is easy to see that $C^{\leq \alpha} \subset C$,
$\alpha \in \mathbb{R}$, is an increasing filtration of $C$ by
$\Lambda_0$-modules (though not by vector space over $\Lambda$). Since
$A(x) = -\infty$ iff $x=0$, this filtration is separated. Moreover, it
is exhaustive and complete.

From now on we will make the following {\em standing assumption:}
$A(d^C x) \leq A(x)$, $\forall \, x \in C$. In other words, we assume
that each $C^{\leq \alpha} \subset C$, $\alpha \in \mathbb{R}$, is a
subcomplex of $C$ (over $\Lambda_0$).

It is important to note that the function $A$, as defined above,
coincides with the action level of the preceding filtration on $C$, as
defined at the beginning of~\S\ref{s:filt-ch}. Thus no confusion
should arise by denoting them both by $A$.

We will make use of the following defintion from~\cite{Usher-Zhang}.
\begin{dfn} \label{d:delta-robust}
  A subspace $V \subset Ker(d^C) \subset C$ is called {\em
    $\delta$-robust} if for all $v \in V$ and $w \in C$ such that
  $v = d^C(w),$ we have $A(w) \geq A(v) + \delta.$
\end{dfn}

\begin{rem} \label{rem:delta-robust-subspace-in-kernel} According to
  the above definition, a complement $W$ in $Ker(d^C)$ to $Im(d^C)$ is
  a $\delta$-robust subspace for all $\delta > 0.$ Hence if
  $V \subset Im(d^C)$ is $\delta$-robust then $V \oplus W$ is also
  $\delta$-robust. We will call a $\delta$-robust subspace
  $V \subset Im(d^C)$ a {\em proper $\delta$-robust} subspace.
\end{rem}

\begin{prop}\label{prop:rig-cplx2} Let $(C,d^{C})$ be a chain complex
  as above, and let $f:C \longrightarrow C$ be a chain map. Assume
  that $d^{C}$ splits as a sum $d^{C}=d_{0}+d_{1}$ such that $d_0$ is
  an $\Lambda$-linear differential which (like $d^C$)
  also preserves the given filtration on $C$. Furthermore, assume that
  $\dim_{\Lambda}(H_{\ast}(C,d_{0})) \geq
  \dim_{\Lambda}(H_{\ast}(C,d^C))$. If
  $B_{h}(f-id_{C})< \delta_{d_{1}}$, then
  $$\dim_{\Lambda}(Im(f))\geq 
  \dim_{\Lambda}(H_{\ast}(C,d_{0}))~.~$$
\end{prop}

The proposition follows directly from the following two lemmas.

\begin{lem}\label{lem:delta-robust-existence} Let $(C, d^C)$ 
  be a chain complex as above, and assume that its differential splits
  as $d^C = d_0+d_1$ with $d_0$ satisfying the same assumptions as in
  Proposition~\ref{prop:rig-cplx2}. Then
  $\dim_{\Lambda}(H_{\ast}(C,d_{0})) -
  \dim_{\Lambda}(H_{\ast}(C,d^C))$ is
  even. Furthermore, denote the latter number by $2k$ and assume that
  $k \geq 0$. Then $(C,d_C)$ admits a proper $\delta_{d_1}$-robust
  subspace of dimension at least $k$.
\end{lem}

\begin{lem} \label{lem:delta-robust-injective} Let $(C, d^C)$ be a
  chain complex as in Lemma~\ref{lem:delta-robust-existence} and
  $f: C \longrightarrow C$ be a chain map. Let $0< \epsilon < \delta$
  and suppose that $B_h(f-id_{C}) = \delta - \epsilon$. Then $f$ is
  injective on each (resp. proper) $\delta$-robust subspace, and maps
  it to a (resp. proper) $\epsilon$-robust subspace.
\end{lem}

\begin{proof}[Proof of Proposition~\ref{prop:rig-cplx2}]
  By Lemma~\ref{lem:delta-robust-existence}, there exists a proper
  $\delta_{d_{1}}$-robust subspace $V$ in $(C,d^C)$ of dimension $k$
  (where $k$ is given by that lemma). By
  Lemma~\ref{lem:delta-robust-injective}, $f(V)$ will be a proper
  $\epsilon$-robust subspace of dimension $k$. Consider a subspace
  $V' \subset C$ of dimension $k$ such that $d^C(V') = V$, and a
  complement $W$ in $Ker(d^C)$ to $Im(d^C).$ Then $d^C(f(V')) = f(V),$
  showing that $\dim d^C(f(V')) = k,$ and $f(W)$ will again be a
  complement in $Ker(d^C)$ to $Im(d^C)$. (Note that
  $f(W) \cap d^C(C) = 0$ because, by assumption, $f-\id_C$ is
  null-homotopic, so $f$ induces an isomorphism in homology.) Now, by
  Lemma~\ref{lem:delta-robust-injective} again, $f(W)$ will have the
  correct dimension. Finally the three subspaces $f(V),f(V'),f(W)$ are
  direct summands of $C$ whence
  $\dim_{\Lambda}(Im(f)) \geq
  \dim_{\Lambda}(H_{\ast}(C,d^C)) + 2k$, finishing the
  proof.
\end{proof}

\begin{proof}[Proof of Lemma \ref{lem:delta-robust-existence}]
  The identities
  \begin{equation*}
    \begin{aligned}
      & \dim_{\Lambda}(C) =
      \dim_{\Lambda} (H_{\ast}(C,d^C)) + 2
      \dim_{\Lambda} (Im(d_C)), \\
      & \dim_{\Lambda}(C) =
      \dim_{\Lambda} (H_{\ast}(C,d_{0})) + 2
      \dim_{\Lambda} (Im(d_{0})),
    \end{aligned}
  \end{equation*}
  show that
  $\dim_{\Lambda}(H_{\ast}(C,d_{0})) -
  \dim_{\Lambda}(H_{\ast}(C,d^C))$ is even. Moreover we
  obtain:
  \begin{equation} \label{eq:dimension-relation}
    \dim_{\Lambda} (Im(d^C)) =
    \dim_{\Lambda} (Im(d_{0})) + k.
  \end{equation}

  From~\cite[Proposition 7.4]{Usher-Zhang}, it is immediate to
  construct a projection $\pi : C \longrightarrow Im(d_{0})$, that
  restricts to the identity on $Im(d_0)$ and satisfies
  $A(\pi(x)) \leq A(x)$ for all $x \in C$.

  From~\eqref{eq:dimension-relation} we now have that
  $\dim(Ker(\pi|_{Im(d^C)})) \geq k.$ We claim that
  $V = Ker(\pi|_{Im(d^C)})$ is $\delta_{d_1}$-robust. Indeed, if
  $v \in V, w \in C,$ and $v = d^C w,$ then writing
  $d^C w = d_0 w + d_1 w,$ and using $\pi(v) = 0$ we obtain
  $d_0w = \pi(d_0 w) = - \pi(d_1 w),$ whence $v = (id - \pi) (d_1 w).$
  Therefore
  $A(v) = A((id - \pi) (d_1 w)) \leq A(d_1 w) \leq A(w) -
  \delta_{d_1}$.  This implies $A(w) \geq A(v) + \delta_{d_1}$,
  concluding the proof.
\end{proof}

\begin{proof}[Proof of Lemma~\ref{lem:delta-robust-injective}]
  Let $V \subset C$ be a $\delta$-robust subspace. We write
  $f = id_C + d^C h - h d^C,$ where
  $A(h(x)) \leq A(x) + (\delta - \epsilon)$, for all $x \in C$. If
  $v \in V$ is such that $f(v) = 0,$ we would have
  $v + d^C(h(v)) = 0,$ which would yield $w = - h(v),$ with $v = dw$
  and $A(w) \leq A(v) + \delta - \epsilon$. On the other hand
  $\delta$-robustness implies $A(w) \geq A(v) + \delta$. A
  contradiction.

  If $f(v) = d^C z,$ we would have $v + d^C(h(v)) = d^C z,$ which
  would yield $w = z - h(v),$ with $v = dw.$ Therefore by
  $\delta$-robustness we obtain
  $A(v) + \delta \leq A(z- h(v)) \leq \max \{A(h(v)),A(z)\}.$ Since
  $A(h(v)) \leq A(v) + \delta - \epsilon,$ we get
  $A(v) + \delta \leq A(z),$ and
  $A(f(v)) \leq \max \{ A(v), A(h(v)) \} \leq A(v) + \delta - \epsilon
  \leq A(z) - \epsilon.$ We conclude that
  $A(z) \geq A(f(v)) + \epsilon,$ which finishes the proof.
\end{proof} 



\section{Floer theory and Fukaya categories} \label{s:floer-theory}

The goal of this section is to set up the variant of Floer theory that
will be used in this paper. In particular we will show how choose the
auxiliary parameters of this theory so that the Fukaya category
becomes a {\em weakly filtered} $A_{\infty}$-category.

Let $(M, \omega)$ be a symplectic manifold, either closed or convex at
infinity. We always assume $M$ to be connected. Denote by
$\mathcal{L}ag^{we}(M)$ the collection of all closed connected
Lagrangian submanifolds $L \subset M$ that are {\em weakly
  exact}. Recall that $L \subset M$ is weakly exact if for every
$A \in H_2^D(M,L)$ we have $\int_{A} \omega = 0$. (The group of
$H_2^D(M,L) \subset H_2(M,L)$ is by definition the image of the
Hurewicz homomorphism $\pi_2(M,L) \longrightarrow H_2(M,L)$).

Let $\mathcal{C} \subset \mathcal{L}ag^{we}$ be a collection of weakly
exact Lagrangians. Unless explicitly stated otherwise, we henceforth
make the following mild assumption on $\mathcal{C}$, whenever $M$ is
not compact. There exists an open domain $U_0 \subset M$ with compact
closure, such that all Lagrangians $L \subset \mathcal{C}$ lie inside
$U_0$. For further use, also fix another open domain with compact
closure $U_1 \supset \overline{U}_0$ as well as an $\omega$-compatible
almost complex structure $J_{\textnormal{conv}}$ which is compatible
with the convexity of $M$ outside of $\overline{U}_1$.

Fix a base ring $R$ of characteristic $2$ (e.g. $R = \mathbb{Z}_2$)
and let $\Lambda$ be the Novikov ring over $R$ as defined
in~\eqref{eq:Nov-ring}. Denote by $\fuk(\mathcal{C})$ the Fukaya
category, with coefficients in $\Lambda$, whose objects are
$L \in \mathcal{C}$. We mostly follow here the implementation of the
Fukaya category due to Seidel~\cite{Se:book-fukaya-categ} with several
modifications that will be explained shortly.

As in~\cite{Se:book-fukaya-categ}, for every pair of Lagrangians
$L_0, L_1 \in \mathcal{C}$ we choose a Floer datum
$\mathscr{D}_{L_0, L_1} = (H^{L_0, L_1}, J^{L_0, L_1})$ consisting of
a Hamiltonian function
$H^{L_0, L_1}: [0,1] \times M \longrightarrow \mathbb{R}$ and a
time-dependent $\omega$-compatible almost complex structure
$J^{L_0,L_1} = \{J^{L_0,L_1}_t\}_{t \in [0,1]}$. In case $M$ is not
compact we require that outside of $U_1$ we have $H \equiv 0$ and
$J^{L_0,L_1}_t \equiv J_{\textnormal{conv}}$.

Denote by $\mathcal{O}(H^{L_0,L_1})$ the set of orbits
$\gamma:[0,1] \longrightarrow M$ of the Hamiltonian flow
$\phi^{H^{L_0,L_1}}_t$ generated by $H^{L_0,L_1}$ such that
$\gamma(0) \in L_0$ and $\gamma(1) \in L_1$. The Floer complex
$CF(L_0, L_1; \mathscr{D}_{L_0, L_1})$ is a free $\Lambda$-module
generated by the set $\mathcal{O}(H^{L_0,L_1})$:
\begin{equation} \label{eq:CF-1} CF(L_0, L_1; \mathscr{D}_{L_0, L_1})
  = \bigoplus_{\gamma \in \mathcal{O}(H^{L_0,L_1})} \Lambda
  \gamma.
\end{equation}
By adding additional structure it is possible to grade the Floer
complexes but we will not do that in this paper, hence work in an
ungraded setting.  The differential $\mu_1$ on the Floer complex is
defined by counting solutions $u$ of the Floer equation:
\begin{equation} \label{eq:floer-eq-1}
  \begin{aligned}
    & u: \mathbb{R} \times [0,1] \longrightarrow M, \quad u(\mathbb{R}
    \times 0) \subset L_0, \; u(\mathbb{R} \times 1) \subset L_1, \\
    & \partial_s u + J^{H_0,H_1}_t (u) \partial_t u =
    -\nabla H^{L_0,L_1}_t (u), \\
    & E(u) := \int_{-\infty}^{\infty} \int_{0}^{1} \lvert \partial_su
    \rvert^2 dt ds < \infty.
    \end{aligned}
\end{equation}
where $(s,t) \in \mathbb{R} \times [0,1]$. Here,
$H_t^{L_0,L_1}(x) := H^{L_0, L_1}(t,x)$ and $\nabla H_t^{L_0,L_1}$ is
the gradient of the function
$H_t^{L_0,L_1} : M \longrightarrow \mathbb{R}$ with respect to the
Riemannian metric
$g_t(\cdot, \cdot) = \omega(\cdot, J^{L_0, L_1}_t \cdot)$ associated
to $\omega$ and $J^{L_0, L_1}_t$. The quantity $E(u)$ in the last line
of~\eqref{eq:floer-eq-1} is the energy of a solution $u$ and we
consider only finite energy solutions. (Note also that the norm
$\lvert \partial_s u \rvert$ in the definition of $E(u)$ is calculated
with respect to the metric $g_t$.) Solutions $u$
of~\eqref{eq:floer-eq-1} are also called Floer trajectories.

For $\gamma_-, \gamma_+ \in \mathcal{O}(H^{L_0,L_1})$ consider the
space of {\em parametrized} Floer trajectories $u$ connecting
$\gamma_{-}$ to $\gamma_+$:
\begin{equation} \label{eq:floer-traj-space-1} \mathcal{M}(\gamma_-,
  \gamma_+; \mathscr{D}_{L_0,L_1}) = \Big\{ u \mid u \text{
    solves~\eqref{eq:floer-eq-1} and} \lim_{s \to \pm \infty} u(s,t) =
  \gamma_{\pm}(t) \Big\}.
\end{equation}
Note that $\mathbb{R}$ acts on this space by translations along the
$s$-coordinate. This action is generally free, with the only exception
being $\gamma_- = \gamma_+$ and the stationary solution
$u(s,t) = \gamma_-(t)$ at $\gamma_-$.

Whenever, $\gamma_- \neq \gamma_+$ we denote by
\begin{equation}
  \mathcal{M}^*(\gamma_-,
  \gamma_+; \mathscr{D}_{L_0,L_1}) := \mathcal{M}(\gamma_-,
  \gamma_+; \mathscr{D}_{L_0,L_1}) \big/ \mathbb{R}
\end{equation}
the quotient space (i.e. the space of non-parametrized solutions).  In
the case $\gamma_- = \gamma_+$ we define
$\mathcal{M}^*(\gamma_-, \gamma_-; \mathscr{D}_{L_0,L_1})$ in the same
way only that we omit the stationary solution at $\gamma_-$. (Note
that if $\omega$ is exact then
$\mathcal{M}(\gamma_-, \gamma_-; \mathscr{D}_{L_0,L_1})$ contains only
the stationary solution at $\gamma_-$, hence
$\mathcal{M}^* = \emptyset$. But this might not be the case if
$\omega$ is not exact.)

For a generic choice of Floer datum $\mathscr{D}$ the space
$\mathcal{M}^*(\gamma_-, \gamma_+; \mathscr{D}_{L_0,L_1})$ is a smooth
manifold (possibly with several components having different
dimensions). Moreover, its $0$-dimensional component
$\mathcal{M}^*_0(\gamma_{-}, \gamma_{+}; \mathscr{D}_{L_0,L_1})$ is
compact hence a finite set.

Define now
$\mu_1: CF(L_0, L_1; \mathscr{D}) \longrightarrow CF(L_0, L_1;
\mathscr{D})$ by
\begin{equation} \label{eq:mu-1} \mu_1(\gamma_{-}) :=
  \sum_{\gamma_{+}} \sum_{u} T^{\omega(u)} \gamma_+, \quad \forall \;
  \gamma_{-} \in \mathcal{O}(H^{L_0,L_1}),
\end{equation}
and extending linearly over $\Lambda$. Here, the outer sum runs over
all $\gamma_{+} \in \mathcal{O}(H^{L_0,L_1})$ and the inner sum over
all solutions
$u \in \mathcal{M}^*_0(\gamma_{-}, \gamma_{+};
\mathscr{D}_{L_0,L_1})$.  The term $\omega(u)$ is a shorthand notation
for the symplectic area of a Floer trajectory $u$, namely
$\omega(u) := \int_{\mathbb{R} \times [0,1]} u^*\omega$.

It is well known that $\mu_1$ is a differential and we denote the
homology of $CF(L_0,L_1; \mathscr{D}_{L_0,L_1})$ by
$HF(L_0,L_1; \mathscr{D}_{L_0,L_1})$. This homology is independent of
the choice of the Floer datum in the sense that for every two regular
choices of Floer data $\mathscr{D}_{L_0,L_1}$,
$\mathscr{D}'_{L_0,L_1}$ there is a canonical isomorphism
$\psi_{\mathscr{D}, \mathscr{D}'}: HF(L_0,L_1; \mathscr{D}_{L_0,L_1})
\longrightarrow HF(L_0,L_1; \mathscr{D}'_{L_0,L_1})$ which form a
directed system. Therefore we can regard this collection of
$\Lambda$-modules as one and denote it by $HF(L_0,L_1)$. Note however,
that the canonical isomorphisms $\psi_{\mathscr{D}, \mathscr{D}'}$ do
not preserve action-filtrations in general, hence there is no meaning
to $H(CF^{\leq \alpha}(L_0,L_1))$ without specifying the Floer datum.

\medskip 

The higher operations $\mu_d$, $d \geq 2$, follow the same scheme as
in~\cite{Se:book-fukaya-categ}, with the main difference being that we
work over the Novikov ring $\Lambda$.

More precisely, we first make a choice of strip-like ends along the
compactification of the moduli-spaces $\mathcal{R}^{d+1}$, $d \geq 2$,
of disks with $(d+1)$-boundary punctures. For every
$r \in \mathcal{R}^{d+1}$ denote by $S_r$ the punctured disk
corresponding to $r$. Denote the punctures by $\zeta_i$,
$i=0, \ldots, d$, going in {\em clockwise} direction. The puncture
$\zeta_0$ will be called the exit and $\zeta_1, \ldots, \zeta_d$ the
entry punctures. We denote the arc along $\partial S_r$ connecting
$\zeta_i$ to $\zeta_{i+1}$ by $C_i$, with the convention that
$\zeta_{d+1} := \zeta_0$. (Of course, $\zeta_i$ and $C_i$ all depend
on $r$ but we suppress this from the notation.)

Next we make a choice of perturbation data
$\mathscr{D}_{L_0, \ldots, L_d} = (K^{L_0, \ldots, L_d}, J^{L_0,
  \ldots, L_d})$ for every tuple of $d+1$ Lagrangians
$L_0, \ldots, L_d \in \mathcal{C}$. The first item in the perturbation
data is a family of $1$-forms
$K^{L_0, \ldots, L_d} = \{ K^{L_0, \ldots, L_r}_r\}_{r \in
  \mathcal{R}^{d+1}}$, parametrized by $r \in \mathcal{R}^{d+1}$, with
values in the space of Hamiltonian functions
$M \longrightarrow \mathbb{R}$. The second one is a family of
$\omega$-compatible domain-dependent almost complex structures on $M$,
$J^{L_0, \ldots, L_d} = \{ J_r^{L_0, \ldots, L_d} \}_{r \in
  \mathcal{R}^{d+1}}$, parametrized by $r \in \mathcal{R}^{d+1}$. In
other words for every $r \in \mathcal{R}^{d+1}$,
$J^{L_0, \ldots, L_d}_r$ is itself a family
$\{ J^{L_0, \ldots, L_d}_{r,z} \}_{z \in S_r}$ of $\omega$-compatible
almost complex structure on $M$, parametrized by $z \in S_r$.

The perturbation data are required to satisfy several additional
conditions. The first one is that along each of the strip-like ends
the perturbation data coincides with the Floer data associated to the
pair of Lagrangians corresponding to that end. More precisely, along
the strip like end corresponding to the puncture $\zeta_i$ of $S_r$ we
have
\begin{equation} \label{eq:pert-floer-da} K^{L_0, \ldots, L_r}_r =
  H_t^{L_{i-1}, L_i} dt, \quad J^{L_0, \ldots, L_d} = J_t^{L_{i-1},
    L_i}, \quad \forall \; 1\leq i \leq d+1,
\end{equation}
where we have used here the convention that $L_{d+1}=L_0$. Here
$(s,t)$ are the conformal coordinates corresponding to the strip-like
ends.

The second condition is that along the arc $C_i$ we have
\begin{equation} \label{eq:K=0-on-del-C_i} K^{L_0, \ldots,
    L_d}(\xi)|_{L_i} = 0, \quad \forall \; \xi \in T(C_i), \; i=0,
  \ldots, d.
\end{equation}

Moreover, the choices of strip-like ends and perturbation data along
$\mathcal{R}^{d+1}$ are required to be compatible with gluing and
splitting, or in the language of~\cite{Se:book-fukaya-categ}
``consistent''.  This means essentially that these choices extend smoothly over
the compactification $\overline{\mathcal{R}}^{d+1}$ of the space of
boundary-punctured disks. In turn, this requires that for every $d$,
the choices of strip-like ends and perturbation data done over
$\mathcal{R}^{d+1}$ are compatible with those that appear on all the
strata of the boundary $\partial \overline{\mathcal{R}}^{d+1}$ of the
compactification $\overline{\mathcal{R}}^{d+1}$ of
$\mathcal{R}^{d+1}$. We refer the reader
to~\cite[Chapter~9]{Se:book-fukaya-categ} for the precise definitions
and implementation.

In case $M$ is not compact we add the following conditions on the
perturbation data. For every $r \in \mathcal{R}^{d+1}$ and
$\xi \in T(S_r)$ the Hamiltonian function
$K_r^{L_0, \ldots, L_d}(\xi)$ is required to vanish outside of $U_1$
and $J_r^{L_0, \ldots, L_d} \equiv J_{\textnormal{conv}}$ outside of
$U_1$.

Once we have made consistent choices of strip-like ends and
perturbation data we define the higher operations $\mu_d$ for
$L_0, \ldots, L_d \in \mathcal{C}$ as follows. For
$r \in \mathcal{R}^{d+1}$, $z \in S_r$ and $\xi \in T_z(S_r)$ define
$Y_{r,z}(\xi)$ to be the Hamiltonian vector field of the function
$K^{L_0, \ldots, L_d}_{r,z}(\xi) : M \longrightarrow
\mathbb{R}$. Consider now the following Floer equation:
\begin{equation} \label{eq:floer-eq-2}
  \begin{aligned}
    & u: S_r \longrightarrow M, \quad u(C_i) \subset L_i \;\; \forall
    \;
    0 \leq i \leq d, \\
    & Du_z + J^{L_0, \ldots, L_d}_{r,z} (u) \circ Du_z \circ j_r =
    Y_{r,z}(u) + J^{L_0, \ldots, L_d}_{r,z} \circ Y_{r,z}(u) \circ
    j_r, \\
    & E(u) := \int_{S_r} \lvert Du - Y_r \rvert^2_J \sigma < \infty.
  \end{aligned}
\end{equation}
Here $j_r$ stands for the complex structure on $S_r$. The last
quantity in~\eqref{eq:floer-eq-2} is the energy of a solution $u$ and
we consider only solutions of finite energy. The definition of $E(u)$
involves an area form $\sigma$ on $S_r$ and the norm
$\lvert \cdot \rvert_J$ on the space of linear maps
$T_z(S_r) \longrightarrow T_{u(z)}(M)$ which is induced by $j_r$,
$J := J_r^{L_0, \ldots, L_d}$ and $\sigma$. See~\cite[Section~2.2,
Page~20]{McD-Sa:jhol} for the definition. Note that $E(u)$ does not
depend on $\sigma$.

Given orbits $\gamma_{-}^1, \ldots, \gamma_{-}^d, \gamma_+$ with
$\gamma_{-}^i \in \mathcal{O}(H^{L_{i-1}, L_i})$ and
$\gamma_+ \in \mathcal{O}(H^{L_0, L_d})$ define the space of so called
{\em Floer polygons} connecting $\gamma_{-}^1, \ldots, \gamma_{-}^d$
to $\gamma_+$ to be the space of all pairs $(r, u)$ with
$r \in \mathcal{R}^{d+1}$ and $u:S_r \longrightarrow M$ such that:
\begin{enumerate}
\item $u$ is a solution of~\eqref{eq:floer-eq-2}.
\item On the strip-like end corresponding to puncture $\zeta_i$ we
  have $\lim_{s \to \infty} u(s,t) = \gamma^i_{-}(t)$ for
  $1 \leq i \leq d$, where $(s,t) \in (-\infty, 0] \times [0,1]$ are
  the conformal coordinates on the strip-like end of $\zeta_i$.
\item On the strip like end corresponding to puncture $\zeta_0$ we
  have $\lim_{s \to \infty} u(s,t) = \gamma_+(t)$ for
  $1 \leq i \leq d$, where $(s,t) \in [0,\infty) \times [0,1]$ are the
  conformal coordinates on the strip-like end of $\zeta_0$.
\end{enumerate}
We denote this space by
$\mathcal{M}(\gamma^1_-, \ldots, \gamma^d_-, \gamma_+;
\mathscr{D}_{L_0, \ldots, L_d})$. For a generic choice of Floer data
and perturbation data this space is a smooth manifold and its
$0$-dimensional component
$\mathcal{M}_0(\gamma^1_-, \ldots, \gamma^d_-, \gamma_+;
\mathscr{D}_{L_0, \ldots, L_d})$ is compact hence a finite
set. (Recall that $d\geq 2$ hence we do not divide here by any
reparametrization group.)  Define now
\begin{equation} \label{eq:mu_d-def}
  \mu_d(\gamma^1_{-}, \ldots, \gamma^d_{-}) = 
  \sum_{\gamma_+} \sum_{(r,u)} T^{\omega(u)} \gamma_+ \in CF(L_0, L_d;
  \mathscr{D}_{L_0, L_d}),
\end{equation}
where the first sum goes over all
$\gamma_+ \in \mathcal{O}(H^{L-0, L_d})$ and the second sum over all
$(r,u) \in \mathcal{M}_0(\gamma^1_-, \ldots, \gamma^d_-, \gamma_+;
\mathscr{D}_{L_0, \ldots, L_d})$. The term $\omega(u)$ stands for the
symplectic area of $u$, namely $\omega(u) = \int_{S_r} u^*\omega$.

Extending $\mu_d$ multi-linearly over $\Lambda$ we obtain an
operation:
$$\mu_d: CF(L_0, L_1; \mathscr{D}_{L_0, L_1}) \otimes \cdots \otimes
CF(L_{d-1}, L_d; \mathscr{D}_{L_{d-1}, L_d}) \longrightarrow CF(L_0,
L_d; \mathscr{D}_{L_0, L_d}).$$

With all the operations above $\fuk(\mathcal{C})$ becomes an
$A_{\infty}$-category. The proof of this is essentially the same as
the one in~\cite{Se:book-fukaya-categ}, the only difference is that
one needs to keep track of the areas appearing as exponents in the
variable $T$ of the Novikov ring.

\subsection{Units} \label{sb:units} We now explain briefly how to
construct homology units in $\fuk(\mathcal{C})$. More details can be
found in~\cite[Chapter~8]{Se:book-fukaya-categ}. Denote by
$S = D \setminus {\zeta_0}$ the unit disk punctured at one boundary
point $\zeta_0 \in \partial D$. Fix a strip-like end around $\zeta_0$
making $\zeta_0$ an exit puncture and let $(s,t)$ be the conformal
coordinates associated to this strip-like end. Let $L \in \mathcal{C}$
and $\mathscr{D}^{L,L}$ be a regular Floer datum for the pair $(L,L)$.
Pick a regular perturbation datum $\mathscr{D}_S = (K, J)$, as
described earlier with the only difference that $K$ and $J$ are are
defined only on $S$ (i.e. there is no dependence on any space like
$\mathcal{R}^{d+1}$). As before, we require that $D_S$ coincides with
the Floer datum $\mathscr{D}_{L,L}$ along the strip-like ends in the
sense of~\eqref{eq:pert-floer-da}. For $z \in S$, $\xi \in T_z(S)$
define $Y_z(\xi)$ as before. Given $\gamma \in \mathcal{O}(H^{L,L})$
consider the space $\mathcal{M}(\gamma; \mathscr{D}_S)$ of solutions
$u:(S, \partial S) \longrightarrow (M, L)$ of the last two lines of
equation~\eqref{eq:floer-eq-2}, with $S_r$, $Y_{r,z}$,
$J^{L_0, \ldots, L_d}_{r,z}$, $j_r$ replaced by $S$, $Y_z$, $J_z$ and
$i$ respectively, and such that along the strip-like end at $\xi_0$ we
have $\lim_{s \to \infty} u(s,t) = \gamma(t)$.  Define now an element
$e_L \in CF(L,L; \mathscr{D}_{L,L})$ by
\begin{equation} \label{eq:e_L} e_L := \sum_{\gamma \in
    \mathcal{O}(H^{L,L})} \sum_{u} T^{\omega(u)} \gamma,
\end{equation}
where the second sum runs over all solutions $u$ in the
$0$-dimensional component $\mathcal{M}_0(\gamma; \mathscr{D}_S)$ of
$\mathcal{M}(\gamma; \mathscr{D}_S)$. By standard theory $e_L$ is a
cycle and its homology class in $HF(L,L)$ is independent of the choice
of the Floer and perturbation data. Moreover, $[e_L] \in HF(L,L)$ is a
unit for the product induced by $\mu_2$.

\subsection{Families of Fukaya categories} \label{sb:fukaya-families}
The Fukaya category $\fuk(\mathcal{C})$ depends on all the choices
made - strip-like ends, Floer and perturbation data. We fix once and
for all a consistent choice of strip-like ends and denote by $E$ the
space of all consistent choices of perturbation data (compatible with
the fixed choice of strip-like ends). The space $E$ can be endowed
with a natural topology (and a structure of a Fr\'{e}chet manifold),
induced from a larger space in which one allows perturbation data in
appropriate Sobolev spaces
(see~\cite[Chapter~9]{Se:book-fukaya-categ}. The subspace
$E_{\textnormal{reg}} \subset E$ of regular perturbation data is
residual hence a dense subset.

The space $E$ contains a distinguished subspace \label{pp:E-reg}
$\mathcal{N} \subset E$ consisting of all consistent choices of
perturbation data $\mathscr{D} = (K, J)$ with $K \equiv 0$. Fix a
subset $E'_{\textnormal{reg}} \subset E_{\textnormal{reg}}$ whose
closure $\overline{E'}_{\textnormal{reg}}$ contains
$\mathcal{N}$. \label{pp:E'}

For $p \in E'_{\textnormal{reg}}$ we denote by $\fuk(\mathcal{C}; p)$
the associated Fukaya category with choice of perturbation data
$p$. We thus obtain a family of $A_{\infty}$-categories
$\{ \fuk(\mathcal{C}; p)\}_{p \in E'_{\textnormal{reg}}}$,
parametrized by $p \in E'_{\textnormal{ref}}$. It is well known that
this is a {\em coherent} system of $A_{\infty}$-categories
(see~\cite[Chapter~10]{Se:book-fukaya-categ}), in particular they are
all mutually quasi-equivalent.

In what follows we will sometimes use the following notation. Given
$L_0, L_1 \in \mathcal{C}$ and $p \in E'_{\textnormal{reg}}$ we write
$CF(L_0,L_1;p)$ for $CF(L_0,L_1; \mathscr{D}_{L_0,L_1})$, where
$\mathscr{D}_{L_0,L_1}$ is the Floer datum prescribed by the choice
$p \in E'_{\textnormal{reg}}$.

\subsection{Weakly filtered structure on Fukaya categories}
\label{sb:wf-fukaya}

We start by defining filtrations on the Floer complexes of pairs of
Lagrangians in $\mathcal{C}$. We follow here the general recipe
from~\S\ref{sbsb:filtration-mixed}.

Let $L_0, L_1 \in \mathcal{C}$ be two Lagrangians and
$\mathscr{D}_{L_0,L_1} = (H^{L_0,L_1}, J^{L_0,L_1})$ a Floer datum.
We define an ``action functional''
$$\mathbf{A}: CF(L_0, L_1; \mathscr{D}_{L_0,L_1}) \longrightarrow
\mathbb{R} \cup \{-\infty\}$$ as follows. Let
$P(T) = \sum_{i=0}^{\infty} a_i T^{\lambda_i} \in \Lambda$ with
$\lambda_0 < \lambda_i$ for all $i\geq 1$, and $a_0 \neq 0$. Let
$\gamma \in \mathcal{O}(H^{L_0,L_1})$ be a Hamiltonian orbit.  We
first define:
$$\mathbf{A}\bigl( P(T)\gamma \bigr) := -a_0 + \int_{0}^{1}
H^{L_0,L_1}_t(\gamma(t)) dt.$$ Now let
$\sum_{k=1}^l P_k(T) \gamma_k \in CF(L_0,L_1;\mathscr{D}_{L_0,L_1})$
be a general non-trivial element, where the $\gamma_k$'s are mutually
distinct.  We extend the definition of $\mathbf{A}$ to such an element
by:
$$\mathbf{A} \bigl( P_1(T)\gamma_1 + \cdots + P_l(T) \gamma_l \bigr) 
:= \max \bigl\{ \mathbf{A} \bigl(P_k(T)\gamma_k\bigr) \mid k=1,
\ldots, l \bigr\}.$$ Finally, we put $\mathbf{A}(0) = -\infty$.

We now define a filtration on $CF(L_0,L_1;\mathscr{D}_{L_0,L_1})$
by:
\begin{equation} \label{eq:CF-filtration}
  CF^{\leq \alpha} (L_0,L_1; \mathscr{D}_{L_0,L_1}) := 
  \{x \in CF(L_0,L_1; \mathscr{D}_{L_0,L_1}) \mid \mathbf{A}(x) \leq
  \alpha\}.
\end{equation}
Before we go on, a quick remark regarding the Hamiltonian functions
$H^{L_0,L_1}$ in the Floer data is in order. We do {\em not} assume
that these functions are normalized (e.g. by requiring them to have
zero mean when $M$ is closed, or to be compactly supported when $M$ is
open). This means that if we replace $H^{L_0, L_1}$ by
$H^{L_0, L_1} + c(t)$ for some family of constants $c(t)$, we get the
same chain complex as $CF(L_0,L_1; \mathscr{D}_{L_0,L_1})$ but with a
shifted action-filtration. See Remark~\ref{r:ham-normalization} for
more on this issue and the reason for not insisting on normalizing
these Hamiltonian functions.

Returning to~\eqref{eq:CF-filtration}, it is easy to see that
$CF^{\leq \alpha} (L_0,L_1; \mathscr{D}_{L_0,L_1})$ is a
$\Lambda_0$-module (though {\em NOT} a $\Lambda$-module!).  The fact
that this filtration is preserved by $\mu_1$ and moreover, that it
provides $\fuk(\mathcal{C})$ with a structure of a weakly filtered
$A_{\infty}$-category are the subject of the following proposition.

\begin{prop} \label{p:wf-fukaya} There exists a choice
  $E'_{\textnormal{reg}} \subset E_{\textnormal{reg}} \setminus
  \mathcal{N}$ with
  $\overline{E'}_{\textnormal{reg}} \supset \mathcal{N}$ and such that
  the following holds. Let $p \in E'_{\textnormal{reg}}$ and
  $\fuk(\mathcal{C}; p)$ be the corresponding Fukaya category. Then
  there exist a sequence of non-negative real numbers
  $\bme(p) = (\epsilon_1(p)=0, \epsilon_1(p), \ldots, \epsilon_d(p),
  \ldots)$ and $u(p), \zeta(p), \kappa(p) \in \mathbb{R}_+$, depending
  on $p$, such that:
  \begin{enumerate}
  \item With the filtrations described above on the Floer complexes,
    $\fuk(\mathcal{C};p)$ becomes a weakly filtered
    $A_{\infty}$-category with discrepancy $\leq \bme(p)$.
  \item $\fuk(\mathcal{C}; p)$ is h-unital in the weakly filtered
    sense and there is a choice of homology units with discrepancy
    $\leq u(p)$.
  \item
    $\fuk(\mathcal{C}; p) \in \hyperlink{h:asmp-ue}{\lbue(\zeta(p))}$.
    \label{i:fuk-p-ue}
  \item Let $L \in \mathcal{C}$ and denote by $\mathcal{L}$ its Yoneda
    module. Then
    $\mathcal{L} \in \hyperlink{h:asmp-H}{\lbh_s(\kappa(p))}$.
    \label{i:fuk-p-us}
  \item For every $p_0 \in \mathcal{N} \subset E$ (see
    Page~\pageref{pp:E-reg}) we have:
    $$\lim_{p \to p_0} \epsilon_d(p) = 0, \; \; \forall \; d \geq 2, \quad 
    \lim_{p \to p_0} u(p) = \lim_{p \to p_0} \zeta(p) = \lim_{p \to
      p_0} \kappa(p) = 0.$$
    \label{i:u-z-k-to-0}
  \end{enumerate}
\end{prop}

\begin{proof}
  We will only give a sketch of the proof, as most of the ingredients
  are standard in the theory (see e.g.~\cite{Se:book-fukaya-categ}).

  The precise definition of the set of choices of perturbation data
  $E'_{\textnormal{reg}}$ will be given in the course of the proof.

  We begin by showing that the filtration~\eqref{eq:CF-filtration} is
  preserved by $\mu_1$. Let $L_0, L_1 \in \mathcal{C}$ and
  $\mathscr{D}_{L_0,L_1} = (H^{L_0,L_1}, J^{L_0,L_1})$ be a Floer
  datum. Let $\gamma_{-},\gamma_{+} \in \mathcal{O}(H^{L_0,L_1})$ be
  two generators of $CF(L_0,L_1; \mathscr{D}_{L_0,L_1})$ and let
  $u \in \mathcal{M}_0(\gamma_{-}, \gamma_{+}; \mathscr{D}_{L_0,L_1})$
  be an element of the $0$-dimensional component of Floer trajectories
  connecting $\gamma_{-}$ to $\gamma_{+}$. By~\eqref{eq:mu-1}, the
  contribution of $u$ to $\mu_1(\gamma_{-})$ is
  $T^{\omega(u)} \gamma_{+}$.  We now have the following standard
  energy-area identity for solutions
  $u \in \mathcal{M}(\gamma_{-}, \gamma_{+}; \mathscr{D}_{L_0,L_1})$
  of the Floer equation:
  \begin{equation} \label{eq:E-omega-strips} E(u) = \omega(u) +
    \int_{0}^{1} H^{L_0,L_1}_t(\gamma_{-}(t)) dt - \int_{0}^{1}
    H^{L_0,L_1}_t(\gamma_{+}(t)) dt.
  \end{equation}
  It immediately follows that 
  $$\mathbf{A}(T^{\omega(u)} \gamma_{+}) = -\omega(u) + \int_{0}^{1}
  H^{L_0,L_1}_t(\gamma_{+}(t)) dt \leq \int_{0}^{1}
  H^{L_0,L_1}_t(\gamma_{-}(t)) dt = \mathbf{A}(\gamma_{-}).$$ This
  shows that $\mu_1$ preserves the filtration~\eqref{eq:CF-filtration}
  on $CF(L_0,L_1; \mathscr{D}_{L_0,L_1})$.

  We move on to analyze the behavior of the higher operations $\mu_d$,
  $d \geq 2$, with respect to our filtration. Let
  $L_0, \ldots, L_d \in \mathcal{C}$ and
  $\mathscr{D}_{L_0, \ldots, L_d}$ be the corresponding perturbation
  data. Let $\gamma_{-}^i \in \mathcal{O}(H^{L_{i-1}, L_i})$,
  $\gamma_+ \in \mathcal{O}(H^{L_0, L_d})$, and
  $(r,u) \in \mathcal{M}_0(\gamma_{-}^1, \ldots, \gamma_{-}^d,
  \gamma_+; \mathscr{D}_{L_0, \ldots, L_d})$. The contribution of $u$
  to $\mu_d(\gamma_{-}^1, \ldots, \gamma_{-}^d)$ is
  $T^{\omega(u)} \gamma_{+}$.

  Similarly to~\eqref{eq:E-omega-strips} we have the following
  energy-area identity for solutions $u$ of~\eqref{eq:floer-eq-2}:
  \begin{equation} \label{eq:E-omega-polygons} E(u) = \omega(u) -
    \int_0^1 H^{L_0,L_d}_t(\gamma_+(t))dt + \sum_{j=1}^d \int_0^1
    H_t^{L_{j-1},L_j}(\gamma_{-}^{j}(t))dt - \int_{S_r} R^{K^{L_0,
        \ldots, L_d}}(u),
  \end{equation}
  where $R^{K^{L_0, \ldots, L_d}}$ is the curvature $2$-form on $S_r$
  associated to the perturbation form $K^{L_0,\ldots, L_d}$. In local
  conformal coordinates $(s,t) \in S_r$ it can be written as follows.
  Write $K^{L_0, \ldots, L_d} = F_{s,t}ds + G_{s,t} dt$ for some
  functions $F_{s,t}, G_{s,t}: M \longrightarrow \mathbb{R}$. Then
  \begin{equation} \label{eq:curvature-term} R_{s,t}^{K^{L_0, \ldots,
        L_d}} = \bigl( -\tfrac{\partial F_{s,t}}{\partial t} +
    \tfrac{\partial G_{s,t}}{\partial s} - \{F_{s,t}, G_{s,t}\} \bigr)
    ds \wedge dt,
  \end{equation}
  where $\{F_{s,t}, G_{s,t}\} := -\omega(X^{F_{s,t}}, X^{G_{s,t}})$ is
  the Poisson bracket of the functions $F_{s,t}$, $G_{s,t}$.

  We now need to bound the term
  $\int_{S_r} R^{K^{L_0, \ldots, L_d}}(u)$
  from~\eqref{eq:E-omega-polygons} independently of $(r,u)$. To this
  end, first note that for any given $r \in \mathcal{R}^{d+1},$ the curvature
  $R^{K^{L_0, \ldots, L_d}}$ vanishes identically along the strip-like
  ends of $S_r$ by assumption on the perturbation $1$-form. Next, let $\mathcal{S}^{d+1}$ be the universal
  family of disks with $d+1$ boundary punctures
  (see~\cite[Chapter~9]{Se:book-fukaya-categ},see
  also~\cite[Section~3.1]{Bi-Co:lcob-fuk}). This is a fiber bundle
  over $\mathcal{R}^{d+1}$ whose fiber over $r \in \mathcal{R}^{d+1}$
  is the surface $S_r$. The space $\mathcal{S}^{d+1}$ admits a partial
  compactification $\overline{\mathcal{S}}^{d+1}$ over
  $\overline{\mathcal{R}}^{d+1}$ and can be endowed with a smooth
  structure. Since the perturbation data
  $\mathscr{D}_{L_0, \ldots, L_d}$ was chosen consistently, the forms
  $K^{L_0, \ldots, L_d}$ extend to the partial compactification
  $\overline{S}^{d+1}$ over $\overline{\mathcal{R}}^{d+1}$. Now let
  $\mathcal{W} \subset \overline{S}^{d+1}$ be the union of all the
  strip-like ends corresponding to all the surfaces parametrized by
  $r \in \overline{\mathcal{R}}^{d+1}$. Then
  $\overline{\mathcal{S}}^{d+1} \setminus
  \textnormal{Int\,}\mathcal{W}$ is compact.  It follows that for all
  $(r,u) \in \mathcal{M}_0(\gamma_{-}^1, \ldots, \gamma_{-}^d,
  \gamma_+; \mathscr{D}_{L_0, \ldots, L_d})$ we have:
  \begin{equation} \label{eq:bound-R} \Bigl \lvert \int_{S_r} R^{K^{L_0,
        \ldots, L_d}}(u) \Bigr \rvert \leq \epsilon_d(K^{L_0, \ldots,
      L_d}),
  \end{equation}
  where $\epsilon_d(K^{L_0, \ldots, L_d})$ depends only on the
  $C^1$-norm of $K^{L_0, \ldots, L_d}$ (defined in the
  $\mathcal{S}^{d+1}$ as well as $M$ directions). Moreover, we have
  $\epsilon_d(K^{L_0, \ldots, L_d}) \longrightarrow 0$ as
  $K^{L_0, \ldots, L_d} \longrightarrow 0$ in the $C^1$-topology
  (along
  $\overline{\mathcal{S}}^{d+1} \setminus
  \textnormal{Int\,}\mathcal{W}$ and $M$).

  A few words are in order for the case when $M$ is not compact. In
  that case the arguments above continue to work due to our choice of
  perturbation data. More precisely, recall that we had two open
  domains $U_0, U_1 \subset M$ with compact closure, with
  $\overline{U}_0 \subset U_1$, and with the following properties: all
  Lagrangians $L \in \mathcal{L}$ lie in $U_0$ and outside of $U_1$ we
  have $\mathscr{D}^{L_0, \ldots, L_d} = (0, J_{\textnormal{conv}})$
  for all $r \in \mathcal{R}^{d+1}$. This implies that the Floer
  equations~\eqref{eq:floer-eq-1} and~\eqref{eq:floer-eq-2} become
  homogeneous at the points where $u(z) \in M \setminus U_1$. Since
  $(M, \omega, J_{\textnormal{conv}})$ is convex at infinity, the
  maximum principle implies that all solutions $u$ lie within one
  compact domain of $M$.  Thus the estimate~\eqref{eq:bound-R} follows
  by bounding the $C^1$-norm of $K^{L_0, \ldots, L_d}$ only along that
  compact domain.

  Coming back to the estimate~\eqref{eq:bound-R}, it follows
  from~\eqref{eq:E-omega-polygons} that:
  \begin{equation} \label{eq:eps-d-K} \mathbf{A}(T^{\omega(u)} \gamma_+)
    \leq \mathbf{A}(\gamma_{-}^1) + \cdots + \mathbf{A}(\gamma_{-}^d) +
    \epsilon_d(K^{L_0, \ldots, L_d}).
  \end{equation}
  In order to obtain a weakly filtered structure on
  $\fuk(\mathcal{C};p)$ we need to bound from above
  $\epsilon_d(K^{L_0, \ldots, L_d})$ uniformly in
  $L_0, \ldots, L_d \in \mathcal{C}$, so that the ultimate discrepancy
  $\epsilon_d(p)$ depends only on the choice of
  $p \in E'_{\textnormal{reg}}$. This is easily done by restricting
  the set $E'_{\textnormal{reg}}$ to choices of perturbation data
  $p = \{ \mathscr{D}_{L_0, \ldots, L_d} \}_{L_0, \ldots, L_d \in
    \mathcal{C}}$ for which the $C^1$-norms of the forms
  $K^{L_0, \ldots, L_d}$ are uniformly bounded (in
  $L_0, \ldots, L_d$). Since $E_{\textnormal{reg}} \subset E$ is
  residual it follows that the restricted set of choices
  $E'_{\textnormal{reg}}$ still has $\mathcal{N}$ in its closure.

  This concludes the proof that
  $\fuk(\mathcal{C}; p)_{p \in E'_{\textnormal{reg}}}$ is a family of
  weakly filtered $A_{\infty}$-categories, and that the bounds on
  their discrepancies $\bm{\epsilon}(p)$ have the property that for
  all $p_0 \in \mathcal{N}$ we have
  $\lim_{p \to p_0} \epsilon_d(p) = 0$ for every $d \geq 2$.

  We now turn to the statements about the unitality of the categories
  $\fuk(\mathcal{C}; p)$ and their Yoneda modules. Let
  $p \in E'_{\textnormal{reg}}$. Fix $L \in \mathcal{C}$ and let
  $\mathscr{D}_{L,L} = (H^{L,L}, J^{L,L})$ be the Floer datum of
  $(L,L)$ prescribed by $p$. Recall that a homology unit
  $e_L \in CF(L,L;\mathscr{D}_{L,L})$ can be defined
  by~\eqref{eq:e_L}. Let $S = D \setminus \{\zeta_0\}$ and
  $\mathscr{D}_{S} = (K,J)$ as in~\S\ref{sb:units}.  Let
  $\gamma \in \mathcal{O}(H^{L,L})$ and
  $u \in \mathcal{M}_0(\gamma; \mathscr{D}_S)$. According
  to~\eqref{eq:e_L} the contribution of $\gamma$ and $u$ to $e_L$ is
  $T^{\omega(u)}\gamma$. The energy-area identity for $u$ gives:
  $$E(u) = \omega(u) - \int_0^1 H_t^{L,L}(\gamma(t))dt  - 
  \int_S R^{K}(u),$$ where $R^{K}(u)$ is the curvature associated to
  the $1$-form $K$ from the perturbation datum $\mathscr{D}_S$ and is
  defined in a similar way as in~\eqref{eq:curvature-term}. Note that
  we can choose the perturbation datum $\mathscr{D}_S = (K, J)$ such
  that the $C^1$-norm of the $1$-form $K$ is of the same order size as
  the $C^1$-norm of $H^{L,L}$ (i.e.
  $\lVert K \rVert_{C^1} \leq C \lVert H^{L,L} \rVert_{C^1}$ for some
  constant $C$). By doing that we obtain
  $\lvert \int_{S}R^{K}(u) \rvert \leq C' \lVert H^{L,L} \rVert_{C^1}$
  for some constant $C'$. It follows that
  $\mathbf{A}(e_L) \leq C' \lVert H^{L,L} \rVert_{C^1}$. By
  restricting all the Hamiltonians $H^{L,L}$, for all
  $L \in \mathcal{C}$, to have a uniformly bounded $C^1$-norm we
  obtain one constant $u(p)$ (that depends on the choice $p$) such
  that for all every $L \in \mathcal{C}$ we have
  $\mathbf{A}(e_L) \leq u(p)$. Moreover, $u(p) \longrightarrow 0$ as
  $p \longrightarrow p_0 \in \mathcal{N}$ in the $C^1$-topology.  This
  proves the statement about the discrepancy of the units in
  $\fuk(\mathcal{C}; p)$.
 
  We now turn to proving statements~(\ref{i:fuk-p-ue})
  and~(\ref{i:fuk-p-us}) of the proposition and the corresponding
  claims on $\zeta(p)$ and $\kappa(p)$ from
  statement~(\ref{i:u-z-k-to-0}).

  Let $L, L' \in \mathcal{C}$. Choose $S$, $\mathscr{D}_S$ and define
  $e_L$ as explained above. Denote by $\mathscr{D}_{L,L,L'}$ be the
  perturbation datum of the triple $(L,L,L')$ as prescribed by
  $p$. Consider also a disk
  $S' = D \setminus \{\zeta_0', \zeta_1', \zeta_2'\}$ with three
  boundary punctures, ordered clockwise along $\partial D$. We fix
  strip-like ends near these three punctures such that
  $\zeta_0', \zeta_1'$ are entries and $\zeta_2'$ is an exit.
  Consider a $1$-parametric family
  $(\{S''_{\tau}\}_{\tau \in (0,1]}, j_{\tau})$ of surfaces (endowed
  with complex structures) obtained by performing gluing $S$ and $S'$
  at the points $\zeta_0$, $\zeta_0'$ respectively. We construct this
  family so that $S''_{\tau} \longrightarrow S \coprod S'$ as
  $\tau \longrightarrow 0$ and $S''_1 = \mathbb{R} \times [0,1]$ is
  the standard strip. Next, we choose a generic family
  $\{\mathscr{D}_{\tau}\}_{\tau \in (0,1]}$ of perturbation data over
  the family $\{S''_{\tau}\}_{\tau \in (0,1]}$ such that:
  \begin{enumerate}
  \item for $\tau \longrightarrow 0$, $\mathscr{D}_{\tau}$ converges
    to $\mathscr{D}_{S}$ on the $S$ component and
    $\mathscr{D}_{L,L,L'}$ on the $S''$ component.
  \item $\mathscr{D}_1 = \mathscr{D}_{L,L}$.
  \end{enumerate}
  As the family $\{\mathscr{D}_{\tau}\}_{\tau \in (0,1]}$ is generic,
  none of the elements in the $\mathscr{D}_{\tau}$, $\tau <1$ is
  invariant under reparametrization by any non-trivial automorphism
  $\sigma \in Aut(S_{\tau})$.

  Let $\gamma, \lambda \in \mathcal{O}(H^{L,L'})$ and consider the
  space $\mathcal{M}(\gamma, \lambda ; \{ \mathscr{D}_{\tau} \})$ of
  all pairs $(\tau, u)$, with $\tau \in (0,1]$ and
  $u:S_{\tau} \longrightarrow M$ a solution of the Floer
  equation~\eqref{eq:floer-eq-2} with the obvious modifications:
  namely, the lower part of $\partial S_{\tau}$ is mapped by $u$ to
  $L$ and the upper one to $L$', $u$ converges to $\gamma$ at the
  entry $\zeta_1'$ and to $\lambda$ at the exit $\zeta_2'$,
  $(S_r, j_r)$ is replaced by $(S_{\tau}, j_{\tau})$, and $J_{r,z}$
  and $Y_{r,z}$ are replaced by the corresponding structures from
  $\mathscr{D}_{\tau}$.

  Assume that $\gamma \neq \lambda$, and consider the $0$-dimensional
  component
  $\mathcal{M}_0(\gamma, \lambda ; \{ \mathscr{D}_{\tau} \})$. This is
  compact $0$-dimensional manifold hence a finite set. It gives rise
  to a map
  \begin{equation} \label{eq:ch-homotopy-phi-1}
    \begin{aligned}
      & \Phi: CF(L,L'; \mathscr{D}_{L,L'}) \longrightarrow CF(L,L';
      \mathscr{D}_{L,L'}), \\
      & \Phi(\gamma) := \sum_{\lambda} \sum_{(\tau,u)} T^{\omega(u)}
      \lambda, \;\; \forall \gamma \in \mathcal{O}(H^{L,L'}),
    \end{aligned}
  \end{equation}
  where the outer sum is over all $\lambda \in \mathcal{O}(H^{L,L'})$
  with $\lambda \neq \gamma$ and the second sum is over all
  $(\tau, u) \in \mathcal{M}_0(\gamma, \lambda ; \{ \mathscr{D}_{\tau}
  \})$. We extend the formula in the 2'nd line
  of~\eqref{eq:ch-homotopy-phi-1} linearly over $\Lambda$.

  We claim that the following formula holds:
  \begin{equation} \label{eq:ch-homotopy-phi-2} \mu_2(e_L, x) = x +
    \mu_1 \circ \Phi(x) + \Phi \circ \mu_1(x), \; \; \forall \; x \in
    CF(L,L'; \mathscr{D}_{L,L'}),
  \end{equation}
  i.e. $\Phi$ is a chain homotopy between the map $\mu_2(e_L, \cdot)$
  and the identity.

  To prove this, let $\gamma, \gamma_+ \in \mathcal{O}(H^{L,L'})$ and
  consider the $1$-dimensional component
  $\mathcal{M}_1(\gamma, \gamma_+ ; \{ \mathscr{D}_{\tau} \})$ of the
  space $\mathcal{M}(\gamma, \gamma_+ ; \{ \mathscr{D}_{\tau}
  \})$. This space admits a compactification
  $\overline{\mathcal{M}}_1(\gamma, \gamma_+ ; \{ \mathscr{D}_{\tau}
  \})$ which is a $1$-dimensional manifold with boundary. The elements
  in the boundary of this space fall into four types:
  \begin{enumerate}
  \item Elements corresponding to the splitting of $S''$ into $S$ and
    $S'$ at $\tau=0$. These can be written as pairs $(u_S, u_{S'})$
    with $u_S \in \mathcal{M}_0(\gamma'; \mathscr{D}_S)$ for some
    $\gamma' \in \mathcal{O}(H^{L,L})$ and
    $u_{S'} \in \mathcal{M}_0(\gamma', \gamma, \gamma_+;
    \mathscr{D}_{L,L,L'})$. \label{i:split-1}
  \item Elements corresponding to splitting of $S_{\tau}$ at some
    $0 < \tau_0<1$ into a Floer strip $u_0$ followed by a solution
    $u_1:S_{\tau_0} \longrightarrow M$ of the Floer equation for the
    perturbation datum $\mathscr{D}_{\tau_0}$. More precisely, these
    can be written as $(\tau_0,u_0, u_1)$ with $0<\tau_0<1$,
    $u_0 \in \mathcal{M}^*(\gamma, \gamma'; \mathscr{D}_{L,L'})$ and
    $u_1 \in \mathcal{M}(\gamma', \gamma_+; \mathscr{D}_{\tau_0})$ for
    some $\gamma' \in \mathcal{O}(H^{L,L'})$. \label{i:split-2}
  \item The same as~(\ref{i:split-2}) only that the splitting occurs
    in reverse order, namely first an element of
    $\mathcal{M}(\gamma, \gamma'; \mathscr{D}_{\tau_0})$ followed by
    an element of
    $\mathcal{M}^*(\gamma', \gamma_+;
    \mathscr{D}_{L,L'})$. \label{i:split-3}
  \item Elements corresponding to $\tau=1$. These are
    $u: \mathbb{R} \times [0,1] \longrightarrow M$ that belong to the
    space $\mathcal{M}_0(\gamma, \gamma_+; \mathscr{D}_{L,L'})$ or in
    other words elements of the $0$-dimensional component of the space
    $\mathcal{M}(\gamma, \gamma_+; \mathscr{D}_{L,L'})$ of {\em
      parametrized} Floer trajectories connecting $\gamma$ to
    $\gamma_+$. The latter space has a $0$-dimensional component if
    and only if $\gamma=\gamma_+$ in which case that component
    contains only the stationary trajectory at $\gamma$. Summing up,
    this type of boundary point occurs if and only if
    $\gamma = \gamma_+$ and $u$ is the stationary solution at
    $\gamma$.
    \label{i:split-4}
  \end{enumerate}
  Summing up over all these four possibilities (for every given area
  of solutions $u$) yields formula~\eqref{eq:ch-homotopy-phi-2}. Note
  that the first term (i.e. the summand $x$) on the right-hand side
  of~\eqref{eq:ch-homotopy-phi-2} comes exactly from the boundary
  points of type~(\ref{i:split-4}).

  To conclude the proof we only need to estimate the shift in action
  (or filtration) of the chain homotopy $\Phi$. This is done in a
  similar way to the argument used above to estimate $\bme(p)$, namely
  by using an energy-area identity as
  in~\eqref{eq:E-omega-polygons}. Indeed we can choose the
  perturbation data $\mathscr{D}_S$ and $\mathscr{D}_{\tau}$,
  $0< \tau < 1$, to be of the same size order (in the $C^1$-norm) as
  the Hamiltonian $H^{L,L'}$, hence the curvature term in the
  energy-area identity can be bounded by a constant $C(H^{L,L'})$ that
  depends on $H^{L,L'}$ and such that $C(H^{L,L'}) \longrightarrow 0$
  as $H^{L,L'} \longrightarrow 0$ in the $C^1$-topology. By taking all
  the Hamiltonians $H^{L,L'}$ for all $L,L' \in \mathcal{C}$ to be
  uniformly bounded in the $C^1$-topology we obtain a uniform bound
  $\kappa(p)$ on the action shift of the chain homotopy $\Phi$ that
  holds for all pairs $L,L' \in \mathcal{C}$ and such that
  $\kappa(p) \longrightarrow 0$ as
  $p \longrightarrow p_0 \in \mathcal{N}$.  This shows that the Yoneda
  module $\mathcal{L}$ satisfies
  Assumption~$\hyperlink{h:asmp-H}{\lbh_s(\kappa(p))}$. By taking
  $L'=L$ it also follows immediately that
  $\mu_2(e_L,e_L) = e_L + \mu_1(c)$ for some chain $c$ with
  $\mathbf{A}(c) \leq \kappa(p)$, hence
  $\fuk(\mathcal{C}; p) \in \hyperlink{h:asmp-ue}{\lbue(\kappa(p))}$
  (so we can actually take $\zeta(p) = \kappa(p)$).
\end{proof}

\begin{rem} \label{r:ham-normalization} In some variants of Floer
  theory it is common to normalize the Hamiltonian functions involved
  in the definition of the Floer complexes. For example, when the
  ambient manifold is closed one often normalizes the Hamiltonian
  functions to have zero mean, and for open manifolds one requires the
  Hamiltonian functions to have compact support. This solves the
  ambiguity of adding constants to the Hamiltonian functions and
  consequently provides a ``canonical'' way to define action
  filtrations. This especially makes sense when one aims to construct
  invariants of Hamiltonian diffeomorphisms (or flows) by means of
  filtered Floer homology. See e.g. the theory of spectral
  numbers~\cite{Sc:action-spectrum, Oh:spec-inv-1, Oh:spec-inv-2,
    Oh:spec-inv-3, En-Po:calqm}, see
  also~\cite{Vi:generating-functions-1} for an earlier approach via
  generating functions.

  While we could have normalized the Hamiltonian functions in the
  Floer and perturbation data, we have opted not to do so. At first
  glance, this might seem to have odd implications. For example,
  suppose that $p_1, p_2 \in E'_{\text{reg}}$ are two choices of
  perturbation data such that $p_2$ is obtained by adding a
  (different) constant to each of the Hamiltonian functions (or forms)
  in the perturbation data from $p_1$. Clearly, the Fukaya categories
  $\fuk(\mathcal{C}; p_1)$ and $\fuk(\mathcal{C}; p_2)$ are precisely
  the same, but they have completely different (and generally
  unrelated) weakly filtered structures.

  Our justification for not imposing any normalization on the
  Hamiltonian functions is that their role is purely auxiliary, and
  moreover, ideally we would like to make them arbitrarily small. More
  specifically, the focus of our study is the collections of
  Lagrangians $\mathcal{C}$ and its Fukaya category, whereas the
  Hamiltonian functions in the Floer data serve only as perturbations
  whose sole purpose is technical, namely to set up the Floer theory
  so that it fits into a (infinite dimensional) Morse theoretic
  framework. In reality, we view the Hamiltonian perturbations as
  quantities that can be taken arbitrarily small and consider families
  of Fukaya categories parametrized by choices of perturbations that
  tend to $0$. (See e.g. Proposition~\ref{p:wf-fukaya}.)

  In fact, our theory would become simpler and cleaner if we
  could set up the Fukaya category without appealing to any
  perturbations at all. If this were possible (which means that all
  the Floer trajectories and polygons are unperturbed
  pseudo-holomorphic curves) our Fukaya categories would be
  genuinely filtered rather than only ``weakly
  filtered''. (See~\cite{FO3:book-vol1, FO3:book-vol2} for a
  ``perturbation-less'' construction of an $A_{\infty}$-algebra
  associated to a single Lagrangian.)

  Another point related to the matter of normalization is that when
  extending our theory to Lagrangian cobordisms
  (see~\S\ref{sb:ext-lcob}) we are forced to work with non-compactly
  supported Hamiltonian perturbations. While one could have attempted
  a different sort of normalization in that case (suited for the class
  of non-compactly supported perturbations used for cobordisms), we
  will not do that for the very same reasons as those for not doing it for
  $\fuk(\mathcal{C};p)$.
\end{rem}
\subsection{Extending the theory to Lagrangian cobordisms} 
\label{sb:ext-lcob}
Most of the theory developed in the previous subsections
of~\S\ref{s:floer-theory} extends to Lagrangian cobordisms. We will
briefly go over the main points here and refer the reader
to~\cite{Bi-Co:lcob-fuk, Bi-Co:cob1} for a detailed account on the
theory of Lagrangian cobordisms, especially in what concerns Floer
theory and Fukaya categories.

Let $(M, \omega)$ be a symplectic manifold as at the beginning
of~\S\ref{s:floer-theory}. We fix a collection $\mathcal{C}$ of
Lagrangians in $M$ as in~\S\ref{s:floer-theory}. Consider
$\widetilde{M} := \mathbb{R}^2 \times M$ endowed with the split
symplectic structure
$\widetilde{\omega} := \omega_{\mathbb{R}^2} \oplus \omega$, where
$\omega_{\mathbb{R}^2}$ is the standard symplectic structure of
$\mathbb{R}^2$.

Fix a strip $B = [a,b] \times \mathbb{R} \subset \mathbb{R}^2$ in the
plane. Consider the collection $\widetilde{\mathcal{C}}$ of all
Lagrangian cobordisms
$V : (L'_1, \ldots, L'_s) \leadsto (L_1, \ldots, L_r)$ in
$\widetilde{M}$ that have the following additional properties. We
assume that $V$ is cylindrical (with horizontal ends) outside of
$B \times M$ and that all of its ends $L'_i, L_j$ are Lagrangian
submanifolds from the collection $\mathcal{C}$. Moreover, we assume
that the ends of $V$ are all located along horizontal ends whose
$y$-coordinates are in $\mathbb{Z}$. Finally, we further assume that
$V$ is weakly exact as a Lagrangian submanifold of $\widetilde{M}$.

The Lagrangians $(L'_1, \ldots, L'_s)$ are referred to as the positive
ends and $(L_1, \ldots, L_r)$ are the negative ends. Note that the
values of $r$ and $s$ are allowed to vary arbitrarily. We also allow
$s$ or $r$ to be $0$ in which case $V$ is a null cobordism, i.e. a
cobordism with only negative ends (if $s=0$) or only positive ends (if
$r=0$). The case $r=s=0$ means that $V$ is a closed Lagrangian
submanifold of $\widetilde{M}$.

One can associate to the collection $\widetilde{\mathcal{C}}$ a Fukaya category
which we denote $\fukcob(\widetilde{\mathcal{C}})$. This is an
$A_{\infty}$-category (or rather a family of such categories,
depending on auxiliary choices) whose objects are the elements of
$\widetilde{\mathcal{C}}$. The precise construction is detailed
in~\cite{Bi-Co:lcob-fuk}. The main ingredients in the construction are
completely analogous to the case $\fuk(\mathcal{C})$, the main
differences being the following. The Floer datum
$\mathscr{D}_{V,V'} = (\widetilde{H}^{V,V'}, \widetilde{J}^{V,V'})$ of
a pair of cobordisms $V, V'$ has a special form at infinity. Namely,
there is a compact subset $C_{V,V'} \subset B \times M$ such that
outside of $C_{V,V'}$ we have:
$$\widetilde{H}^{V,V'}(t,z,p) = h(z) + H^{V,V'}(t,p),$$ where
$z \in \mathbb{R}^2$, $p \in M$,
$H^{V,V'}:[0,1] \times M \longrightarrow \mathbb{R}$ is a Hamiltonian
function on $M$ and $h:\mathbb{R}^2 \longrightarrow \mathbb{R}$ is the
so called {\em profile function} whose purpose is to generate a
Hamiltonian perturbation at infinity which disjoins $V'$ from $V$ at
infinity while keeping both of them cylindrical and horizontal at
infinity. Note that the profile function $h$ is not (and in fact
cannot be) compactly supported. We use the same function $h$ for the
perturbation data of all pairs of Lagrangians
$V, V' \in \widetilde{\mathcal{C}}$. We remark also that $h$ can be
taken to be arbitrarily small in the $C^1$-topology. Precise details
on the construction of $h$ can be found
in~\cite[Section~3]{Bi-Co:lcob-fuk}. The almost complex structures
$\widetilde{J}^{V,V'}$ appearing in the Floer data have also a special
form whose purpose is to retain compactness of the space of Floer
trajectories. We will not repeat its definition here, since its
particular form does not have any significance to the weakly filtered
structure on $\fukcob(\widetilde{\mathcal{C}})$ that we want to
achieve. The only relevant thing is that with this choice of Floer
data, there is a compact subset $C'_{V,V'} \subset B \times M$ such
that all orbits $\mathcal{O}(\widetilde{H}^{V,V'})$ lie inside
$C_{V,V'}$ and moreover all Floer trajectories for the pair $(V,V')$
lie inside $C_{V,V'}$.

The perturbation data used for the definition of
$\fukcob(\widetilde{\mathcal{C}})$ are analogous to those used for
$\fuk(\mathcal{C})$ with the following differences. For a given tuple
$\mathcal{V} = (V_0, \ldots, V_d)$ with
$V_j \in \widetilde{\mathcal{C}}$ the perturbation data
$\mathscr{D}_{\mathcal{V}} = (\widetilde{K}^{\mathcal{V}},
\widetilde{J}^{\mathcal{V}})$ is chosen so that
$$\widetilde{K}^{\mathcal{V}}|_{S_r} = 
h \cdot da_r + \widetilde{K}^{\mathcal{V}}_0,$$ where
$a_r: S_r \longrightarrow [0,1]$ are the so called {\em transition
  functions} which depend smoothly on $r \in
\mathcal{R}^{d+1}$. See~\cite[Section~3.1]{Bi-Co:lcob-fuk} for their
precise definition. The $1$-form
$\widetilde{K}^{\mathcal{V}}_0 \in
\Omega^1(S_r,C^{\infty}(\widetilde{M}))$ is chosen so that it has the
following two additional properties. The first one is that
$\widetilde{K}^{\mathcal{V}}_0$ satisfies
condition~\eqref{eq:K=0-on-del-C_i}. The second one is that there is a
compact subset $C_{\mathcal{V}} \subset B \times M$ that contains all
the subsets $C'_{V_i, V_j}$ (mentioned earlier) such that outside of
$C_{\mathcal{V}}$ the Hamiltonian vector fields
$X^{\widetilde{K}^{\mathcal{V}}_0}(\xi)$, generated by the function
$\widetilde{K}^{\mathcal{V}}_0(\xi):\widetilde{M} \longrightarrow
\mathbb{R}$ are vertical for all $r \in \mathcal{R}^{d+1}$ and
$\xi \in T(S_r)$. By ``vertical'' we mean that
$D\pi(X^{\widetilde{K}^{\mathcal{V}}_0(\xi)}) = 0$, where
$\pi: \widetilde{M} \longrightarrow \mathbb{R}^2$ is the projection.
Note that the due to the $h\cdot da_r$ term in the perturbation form
$\widetilde{K}^{\mathcal{V}}$ this form does {\em not} satisfy
condition~\eqref{eq:K=0-on-del-C_i}. However, this will not play any
role for the purposes of establishing a weakly filtered
$A_{\infty}$-category.

The almost complex structures $\widetilde{J}^{\mathcal{V}}$ from
$\mathscr{D}_{\mathcal{V}}$ are also chosen to have restricted form,
similarly to the ones appearing in the Floer data. We refer the reader
to~\cite[Secton~3.2]{Bi-Co:lcob-fuk} for the details. With these
choices made it can be proved that there exists a compact subset
$C_{\mathcal{V}} \subset \widetilde{M}$ such that for all
$(r,u) \in \mathcal{M}(\widetilde{\gamma}^1, \ldots,
\widetilde{\gamma}_d; \mathscr{D}_{\mathcal{V}})$ we have
$\textnormal{image\,}(u) \subset
C'_{\mathcal{V}}$. See~\cite[Section~3.3]{Bi-Co:lcob-fuk} and in
particular Lemma~3.3.2 there.

Of course, apart from the above the perturbation data
$\mathscr{D}_{\mathcal{V}}$ are assumed to be consistent and also
compatible with the Floer data along the strip-like ends of the
$S_r$'s.

With these choices made we can define the $A_{\infty}$-category
$\fukcob(\widetilde{\mathcal{C}})$ by the same recipe as in the
previous subsections of~\S\ref{s:floer-theory} an in particular by
formula ~\eqref{eq:mu_d-def}. This $A_{\infty}$-category is
$h$-unital, and a choice of homology units
$e_V \in CF(V,V; \mathscr{D}_{V,V})$ can be constructed by the same
recipe as in~\S\ref{sb:units} (see
also~\cite[Remark~3.5.1]{Bi-Co:lcob-fuk} for an alternative approach).

Similarly to $\fuk(\mathcal{C})$ our category
$\fukcob(\widetilde{\mathcal{C}})$ depends on the various choices
made, namely a choice of strip-like ends and perturbation data. Note
that part of the choice made for the perturbation data is the choice
of a profile function and the choice of transition functions.

We now fix the same choice of strip-like ends as for
$\fuk(\mathcal{C})$ and denote the space of choices of perturbation
data by $\widetilde{E}$. We denote the subspace of regular choices of
perturbation data by $\widetilde{E}_{\textnormal{reg}}$. For
$\widetilde{p} \in \widetilde{E}_{\textnormal{reg}}$ we denote by
$\fukcob(\widetilde{\mathcal{C}}; \widetilde{p})$ the category
corresponding to $\widetilde{p}$.

Next we endow $\fukcob(\widetilde{\mathcal{C}}; \widetilde{p})$ with a
weakly filtered structure. This is done in precisely the same way as
for $\fuk(\mathcal{C}; p)$. More precisely we define the action
filtration on the Floer complexes
$CF(V_0, V_1; \mathscr{D}_{V_0, V_1})$ by the same recipe as
in~\S\ref{sb:wf-fukaya}.

With these filtrations fixed, we now have the following:
\begin{prop} \label{p:wf-fukaya-cob} The statement of
  Proposition~\ref{p:wf-fukaya} holds for the $A_{\infty}$-categories
  $\fukcob(\widetilde{\mathcal{C}}; \widetilde{p})$,
  $\widetilde{p} \in \widetilde{E}'_{\textnormal{reg}}$, where
  $\widetilde{E}'_{\textnormal{reg}} \subset
  \widetilde{E}_{\textnormal{reg}}$ is defined in an analogous way as
  $E'_{\textnormal{reg}} \subset E_{\textnormal{reg}}$ (see
  page~\pageref{pp:E'}).
\end{prop}
The proof of this Proposition is essentially the same as that of
Proposition~\ref{p:wf-fukaya-cob} with straightforward modifications
related to the special form of the perturbation data $\widetilde{p}$.

For technical reasons we will need in the following also enlargements
of the categories $\fukcob(\widetilde{\mathcal{C}}; \widetilde{p})$
which will be denoted
$\fuk_{cob}(\widetilde{\mathcal{C}}_{1/2}; \widetilde{p})$. These are
defined in the same was as
$\fukcob(\widetilde{\mathcal{C}}; \widetilde{p})$ only that the
collection of objects $\widetilde{\mathcal{C}}$ is extended to allow
Lagrangian cobordisms $V$ with ends from $\mathcal{C}$ but these ends
are now allowed to lie over horizontal rays with $y$-coordinate in
$\tfrac{1}{2}\mathbb{Z}$ (rather than only $\mathbb{Z}$). This larger
collection of objects is denoted by
$\widetilde{\mathcal{C}}_{1/2}$. The perturbation data, the
$A_{\infty}$-operations as well as the weakly filtered structures are
defined in an analogous way as for
$\fukcob(\widetilde{\mathcal{C}}; \widetilde{p})$. We denote the space
of choices of perturbation data for these categories by
$\widetilde{E}_{1/2}$ and the space of regular such choices by
$\widetilde{E}_{\text{reg}, 1/2}$. Similarly to $E'_{\text{reg}}$ and
$\widetilde{E}'_{\text{reg}}$ we also have the space
$\widetilde{E}'_{\text{reg},1/2} \subset \widetilde{E}_{\text{reg},
  1/2}$. An obvious analogue of Proposition~\ref{p:wf-fukaya-cob}
continues to hold for the family of categories
$\fukcob(\widetilde{\mathcal{C}}_{1/2}; \widetilde{p})$,
$\widetilde{p} \in \widetilde{E}'_{\text{reg}, 1/2}$.

The relation between $\fukcob(\widetilde{\mathcal{C}})$ and
$\fukcob(\widetilde{\mathcal{C}}_{1/2})$ is simple. Any regular choice
of perturbation data for $\fukcob(\widetilde{\mathcal{C}}_{1/2})$ can
be used, by restriction to smaller class of objects, for
$\fukcob(\widetilde{\mathcal{C}})$. Thus, with the right choices of
perturbation data we obtain a full and faithful embedding
$\fukcob(\widetilde{\mathcal{C}}) \longrightarrow
\fukcob(\widetilde{\mathcal{C}}_{1/2})$. We shall give now a more
precise description of this.

There is an obvious restriction map
$r: \widetilde{E}_{1/2} \longrightarrow \widetilde{E}$ with
$r(\widetilde{E}_{\text{reg},1/2}) \subset \widetilde{E}_{\text{reg}}$
and
$r(\widetilde{E}'_{\text{reg},1/2}) \subset
\widetilde{E}'_{\text{reg}}$ and such that the closure of
$r(\widetilde{E}'_{\text{reg},1/2})$ contains
$\widetilde{\mathcal{N}}$ (the space of perturbations with
perturbation form $0$, similarly to $\mathcal{N}$
on~\pageref{pp:E'}). We will replace from now on
$\widetilde{E}'_{\text{reg}}$ with
$r(\widetilde{E}'_{\text{reg},1/2})$ and continue to denote the latter
by $\widetilde{E}'_{\text{reg}}$.

There is also a (non-unique) right inverse to $r$ which is an
extension map
$j: r(\widetilde{E}_{1/2}) \longrightarrow \widetilde{E}_{1/2}$ with
$j(\widetilde{N}) \subset \widetilde{N}_{1/2}$ and such that
$j(\widetilde{E}'_{\text{reg}}) \subset
\widetilde{E}'_{\text{reg},1/2}$. The map $j$ induces an obvious
family of extension functors
\begin{equation} \label{eq:functor-map-j} \mathscr{J}:
  \fukcob(\mathcal{C};\widetilde{p}) \longrightarrow
  \fukcob(\widetilde{\mathcal{C}}_{1/2}; j(\widetilde{p})), \;\;\;
  \widetilde{p} \in \widetilde{E}'_{\text{reg}}.
\end{equation}
These are $A_{\infty}$-functors which are full and faithful (on the
chain level). Note also that these functors $\mathscr{J}$ are
filtered, i.e. they have discrepancy $\leq \mathbf{0}$.

From now on we replace $\widetilde{E}'_{\text{reg}, 1/2}$ with
$j(\widetilde{E}'_{\text{reg}})$ and continue to denote the latter by
$\widetilde{E}'_{\text{reg}, 1/2}$. With these conventions made, the
maps
$r|_{\widetilde{E}'_{\text{reg}, 1/2}} : \widetilde{E}'_{\text{reg},
  1/2} \longrightarrow \widetilde{E}'_{\text{reg}}$ and
$j|_{\widetilde{E}'_{\text{reg}}} : \widetilde{E}'_{\text{reg}}
\longrightarrow \widetilde{E}'_{\text{reg}}$ become bijections,
inverse one to the other. Therefore, whenever no confusion arises we
omit $j$ and $r$ from the notation and denote $j(\widetilde{p})$ by
$\widetilde{p}$ keeping in mind that $\widetilde{p}$ is a regular
choice of perturbation data for $\fukcob(\widetilde{\mathcal{C}})$
which admits an extension, still denoted by $\widetilde{p}$, to a
regular choice of perturbation data for
$\fukcob(\widetilde{\mathcal{C}}_{1/2})$.

An important property of the extension map $j$ is the following.
For every $\widetilde{p}_0 \in \widetilde{\mathcal{N}}$ we have:
\begin{equation} \label{eq:discrep-under-j} \lim_{\widetilde{p} \to
    \widetilde{p}_0} \bme_d^{\fukcob(\widetilde{\mathcal{C}}_{1/2};
    j(\widetilde{p}))} = 0, \;\; \forall \; d.
\end{equation}
This follows easily from Proposition~\ref{p:wf-fukaya-cob} together
with the fact that $j$ is continuous, that
$j(\widetilde{\mathcal{N}}) \subset \widetilde{\mathcal{N}}_{1/2}$ and
that the closure of $j(\widetilde{E}'_{\text{reg}})$ contains
$\widetilde{\mathcal{N}}_{1/2}$.

\subsection{The monotone case} \label{sb:monotone}

The theory developed earlier in the paper continues to work in the
more general setting of monotone Lagrangian submanifolds.  We will
assume henceforth all symplectic manifolds as well as Lagrangian
submanifolds to be connected.

Let $(M, \omega)$ be a symplectic manifold and $L \subset M$ a
Lagrangian submanifold. Recall that $L$ is called monotone if the
following two conditions are satisfied:
\begin{enumerate}
\item There exists a constant $\rho>$ such that
  $$\omega(A) = \rho \mu(A) \; \forall A \in H_2^D(M,L).$$ 
  Here $H_2^D(M,L) \subset H_2(M,L)$ is the image of the Hurewicz
  homomorphism $\pi_2(M,L) \longrightarrow H_2(M,L)$ and $\mu$ is the
  Maslov index of $L$.
\item The minimal Maslov number $N_L$ of $L$, defined by
  $$N_L := \min\{\mu(A) \mid A \in H_2^D(M,L), \mu(A)>0\}$$ 
  satisfies $N_L \geq 2$. (We use the convention that
  $\min \emptyset = \infty$.)
\end{enumerate}

A basic invariant of monotone Lagrangians $L$ is the {\em Maslov-$2$
  disk count}, $\mathbf{d}_L \in \Lambda_0$. This element is defined
as $\mathbf{d}_L := d T^{a}$, where $d \in \mathbb{Z}_2$ is the number
of $J$-holomorphic disks (for generic $J$) of Maslov index $2$ whose
boundaries go through a given point in $L$, and $a = 2\rho> 0$ is the
area of each of these disks. Note that if there are no $J$-holomorphic
disks of Maslov $2$ at all then $\mathbf{d}_L=0$ by definition.

It is well known that $\mathbf{d}_L$ is independent of the choices
made in the definition (the almost complex structure $J$ and the point
on $L$ through which we count the disks - recall that $L$ is assumed
to be connected). We refer the reader
to~\cite[Section~2.5.1]{Bi-Co:lagtop} for the precise definition of
the coefficient $d$ in $\mathbf{d}_L$ and its properties. (Note that
the definition in that paper is done over $\mathbb{Z}$ so the $d$
above is obtained by reducing mod $2$.) In different forms this
invariant has appeared in~\cite{Oh:HF1, Oh:HF1-add, Chek:cob,
  FO3:book-vol1, Aur:t-duality}. Under additional assumptions on $L$,
one can define a version of this invariant also over other base rings
(such as $\mathbb{Z}$ and $\mathbb{C}$) sometimes taking additional
structures (like local systems) into account (see
e.g.~\cite{Aur:t-duality, Bi-Co:lagtop}), but we will not need that in
the sequel.

Fix an element $\mathbf{d} \in \Lambda_0$ of the form
$\mathbf{d} = d T^{a}$, $d \in \mathbb{Z}_2$, $a>0$. Denote by
$\mathcal{L}ag^{mon, \mathbf{d}}(M)$ the class of closed monotone
Lagrangians $L \subset M$ with $\mathbf{d}_{L} = \mathbf{d}$.  Let
$\mathcal{C} \subset \mathcal{L}ag^{mon, \mathbf{d}}(M)$ be a
collection of Lagrangians. Then one can define the Fukaya categories
$\fuk(\mathcal{C}; p)$ in the same way as described earlier and the
theory developed above in~\S\ref{s:floer-theory} carries over without
any modifications. (The main difference in the monotone case is that
$HF(L,L)$ might not be isomorphic to $H_*(L)$, and in fact may even
vanish. This however will not affect any of our considerations. Apart
from that, the monotone case poses some grading issues for Floer
complexes, but in this paper we work in an ungraded framework.)

Before we go on, we mention another basic measurement for monotone
Lagrangians that will be relevant in the sequel. Given a monotone
Lagrangian $L \subset M$ define its {\em minimal disk area} $A_L$ by
\begin{equation} \label{eq:min-disk-area}
  A_L = \min \{ \omega(A) \mid A \in H_2^D(M,L), \omega(A)>0 \}.
\end{equation}

Turning to cobordisms, the theory continues to work if we restrict to
monotone Lagrangian cobordisms $V \subset \mathbb{R}^2 \times M$. Note
that if $V : (L'_1, \ldots, L'_s) \leadsto (L_1, \ldots, L_r)$ is
monotone then its ends $L'_i$ and $L'_j$ are automatically monotone
too. Moreover, as observed by Chekanov~\cite{Chek:cob}, if $V$ is a
monotone Lagrangian cobordism then one can define the Maslov-$2$ disk
count $\mathbf{d}_V$ in the same way as above (i.e. for closed
Lagrangian submanifolds) and $\mathbf{d}_V$ continues to be invariant
of the choices made in its definition. Furthermore, if $V$ is
connected then
$$\mathbf{d}_V = \mathbf{d}_{L'_i} = \mathbf{d}_{L_j}, \; \forall \,
i,j.$$

Given $\mathbf{d} = dT^a\in \Lambda_0$ and a collection
$\mathcal{C} \subset \mathcal{L}ag^{mon, \mathbf{d}}(M)$, denote by
$\widetilde{\mathcal{C}}$ the collection of connected monotone
Lagrangian cobordisms $V$ all of whose ends are in $\mathcal{C}$. Note
that by the preceding discussion each $V \in \widetilde{\mathcal{C}}$
must have $\mathbf{d}_V = \mathbf{d}$. Therefore we omit $\mathbf{d}$
from the notation of $\mathcal{C}$ and $\widetilde{\mathcal{C}}$. This
also keeps the notation consistent with the weakly exact case.

From now, unless explicitly indicated, we treat uniformly both the
weakly exact case as well as the monotone one. In particular the class
of admissible Lagrangians will be denoted by $\mathcal{L}ag^{*}(M)$,
where $* = we$ in the weakly exact case, and $* = (mon, \mathbf{d})$
in the monotone case. We will use similar notation
$\mathcal{L}ag^{*}(\mathbb{R}^2 \times M)$ for the admissible classes
of cobordisms.

\subsection{Inclusion functors} \label{sb:inc-functors} Let
$\gamma \subset \mathbb{R}^2$ be an embedded curve with horizontal
ends, i.e. $\gamma$ is the image of a proper embedding
$\mathbb{R} \hooklongrightarrow \mathbb{R}^2$ whose image outside of
a compact set coincides with two horizontal rays having
$y$-coordinates in $\tfrac{1}{2}\mathbb{Z}$.

In~\cite[Section~4.2]{Bi-Co:lcob-fuk} we associated to $\gamma$ a
family of mutually quasi-isomorphic $A_{\infty}$-functors
$$\mathcal{I}_{\gamma} : \fuk(\mathcal{C}) \longrightarrow
\fukcob(\widetilde{\mathcal{C}}_{1/2})$$ which we called {\em
  inclusion functors}. They all have the same action on objects which
is given by $\mathcal{I}_{\gamma}(L) = \gamma \times L$ for every
$L \in \mathcal{C}$.

Here is a more precise description of this family of functors. Denote
by $\mathcal{H}_{\text{prof}}$ the space of profile functions
(see~\S\ref{sb:ext-lcob}, see also~\cite[Section~3]{Bi-Co:lcob-fuk}
for the precise definition). The construction of the inclusion
functors from~\cite{Bi-Co:lcob-fuk} involves the following
ingredients. First, we restrict to a special subset
$\mathcal{H}'_{\text{prof}}(\gamma) \subset \mathcal{H}_{\text{prof}}$
which contains arbitrarily $C^1$-small profile functions. Apart from
being profile functions, these functions
$h:\mathbb{R}^2 \longrightarrow \mathbb{R}$ have the following
additional properties:
\begin{enumerate}
\item $h|_{\gamma}$ is a Morse function with an odd number of critical
  points $O_1, \ldots, O_l \in \gamma$, where $5 \leq
  l=\text{odd}$. Moreover, in a small Darboux-Weinstein neighborhood
  of $\gamma$, $h$ is constant along each cotangent fiber. Thus for
  every $t$ we have:
  $\phi_t^h(\gamma) \cap \gamma = \{O_1, \ldots, O_l\}$.
\item The image, $(\phi_1^h)^{-1}(\gamma)$, of $\gamma$ under the
  inverse of the time-$1$ map of the Hamiltonian diffeomorphism
  generated by $h$ is as depicted in figure~\ref{f:gamma-c}.
\end{enumerate}
We refer the reader to~\cite[Section~4]{Bi-Co:lcob-fuk} for more
details. In that paper such functions were called {\em extended
  profile functions}. The word ``extended'' indicates that these
functions are adapted to cobordisms with ends along rays having
$y$-coordinates in $\tfrac{1}{2}\mathbb{Z}$ rather than just
$\mathbb{Z}$.
\begin{figure}[htbp]
   \begin{center}
     \includegraphics[scale=0.70]{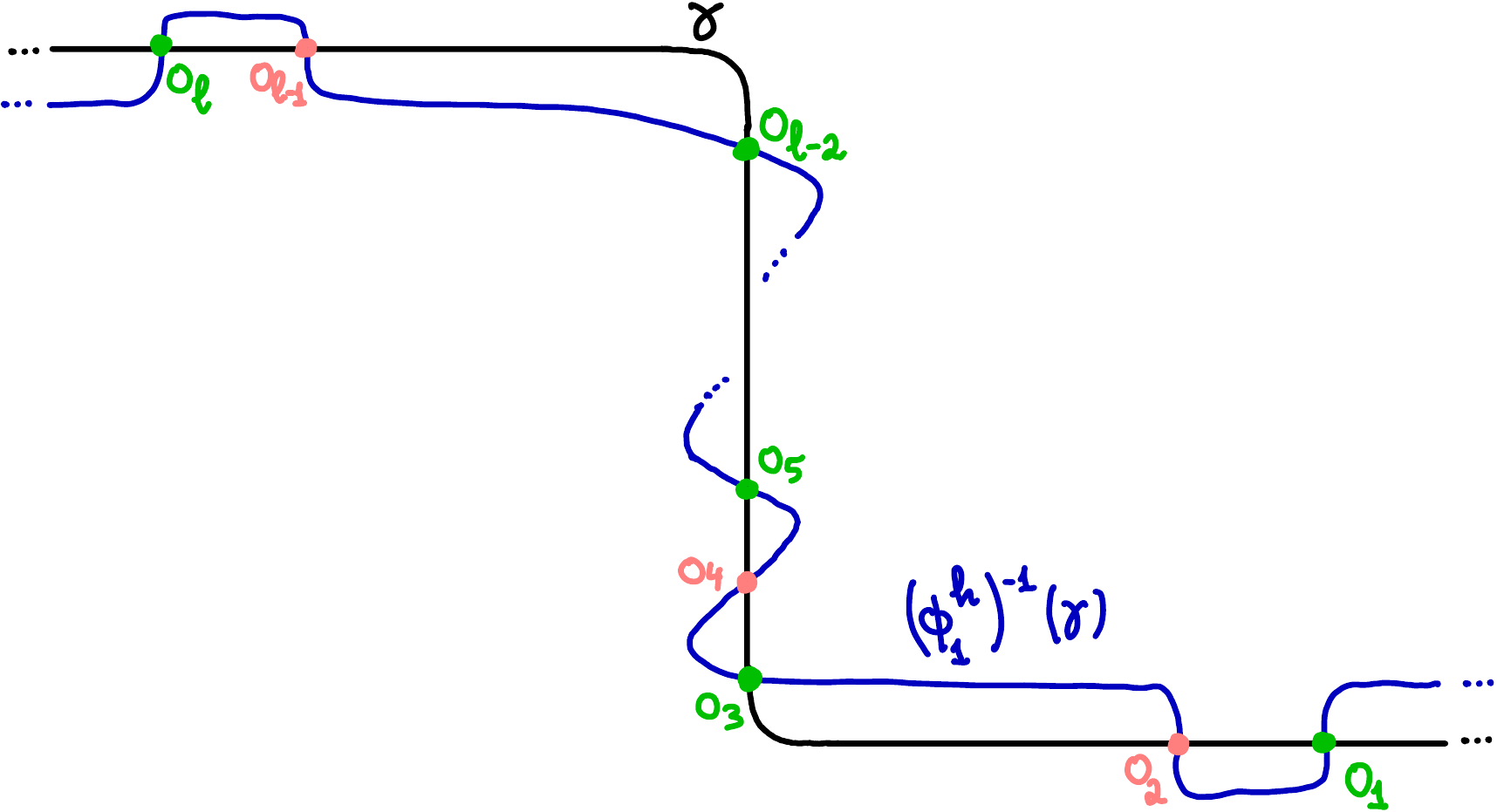}
   \end{center}
   \caption{The curves $\gamma$ and $(\phi_1^h)^{-1}(\gamma)$.}
   \label{f:gamma-c}
\end{figure}

Next, there is a map
\begin{equation} \label{eq:iota-gamma} \iota_{\gamma}:
  E'_{\textnormal{reg}} \times \mathcal{H}'_{\text{prof}}(\gamma)
  \longrightarrow \widetilde{E}'_{\textnormal{reg},1/2},
\end{equation}
and an $A_{\infty}$-functor
\begin{equation} \label{eq:I-gamma}
  \mathcal{I}_{\gamma; p,h}: \fuk(\mathcal{C}; p) 
  \longrightarrow \fukcob(\widetilde{\mathcal{C}}_{1/2};
  \iota_{\gamma}(p,h)),
\end{equation}
defined for every
$(p,h) \in E'_{\text{reg}} \times \mathcal{H}'_{\text{prof}}(\gamma)$
such that for all $(p,h)$ the following holds:
\begin{enumerate}
\item For $L_0, L_1 \in \mathcal{C}$, let
  $\mathscr{D}_{L_0,L_1} = (H^{L_0,L_1}, J^{L_0,L_1})$ be the Floer
  datum of $(L_0,L_1)$ prescribed by $p$ and
  $\mathscr{D}_{\gamma \times L_0, \gamma \times L_1} = (H^{\gamma
    \times L_0, \gamma \times L_1}, J^{\gamma\times L_0, \gamma \times
    L_1})$ the one prescribed by $\iota_{\gamma}(p,h)$.  Let
  $1 \leq j = \text{odd} \leq l$. Then for a small neighborhood
  $\mathcal{U}_j$ of $O_j$ we have
  $H^{\gamma \times L_0, \gamma \times L_1}(z,m) = h(z) +
  H^{L_0,L_1}(m)$ for all $(z,m) \in \mathcal{U}_j \times
  M$. Moreover, for every orbit $x \in \mathcal{O}(H^{L_0,L_1})$ and
  $1\leq j=\text{odd} \leq l$, we have
  $O_{j} \times x \in \mathcal{O}(H^{\gamma \times L_0, \gamma \times
    L_1})$. In the following we will denote
  $x^{(j)} := O_{j} \times x$. Furthermore, we may assume that these
  are all the orbits in
  $\mathcal{O}(H^{\gamma \times L_0, \gamma \times L_1})$, i.e.
  $\mathcal{O}(H^{\gamma \times L_0, \gamma \times L_1}) = \cup_{j}
  (O_j \times \mathcal{O}(H^{\gamma \times L_0, \gamma \times L_1}))$,
  where the union runs over all $1\leq j=\text{odd} \leq l$.
  \label{pp:iota-gamma-h}
\item $\mathcal{I}_{\gamma; p,h}(L) = \gamma \times L$ for every
  $L \in \mathcal{C}$. 
\item The 1'st order term, $(\mathcal{I}_{\gamma;p,h})_1$ is the chain
  map
  \begin{equation} \label{eq:I_1}
    \begin{aligned}
      & (\mathcal{I}_{\gamma;p,h})_1: CF(L_0,L_1;
      \mathscr{D}_{L_0,L_1}) \longrightarrow CF(\gamma \times L_0,
      \gamma \times L_1;
      \mathscr{D}_{\gamma \times L_0, \gamma \times L_1}), \\
      & (\mathcal{I}_{\gamma;p,h})_1(x) = x^{(1)} + x^{(3)} + \cdots +
      x^{(l)}, \;\; \forall \; x \in \mathcal{O}(H_{L_0,L_1}).
  \end{aligned}
\end{equation}
\item The higher terms of $\mathcal{I}_{\gamma; p,h}$ vanish:
  $(\mathcal{I}_{\gamma; p,h})_d=0$ for every $d \geq 2$.
\item The homological functor associated to
  $\mathcal{I}_{\gamma; p,h}$ is full and faithful.
\item For every $p_0 \in \mathcal{N}$ we have
  $\lim \iota_{\gamma}(p,h) \in \widetilde{\mathcal{N}}_{1/2}$ as
  $h \longrightarrow 0$, $p \longrightarrow p_0$. (The limits here are
  in the $C^1$-topology.) \label{pp:iota_gamma-lim}
\item Let $p_0 \in \mathcal{N}$. The weakly filtered
  $A_{\infty}$-categories
  $\fukcob(\widetilde{\mathcal{C}}_{1/2}; \iota(p,h))$ have
  discrepancy
  $\leq \bme^{\fukcob(\widetilde{\mathcal{C}}_{1/2}; \iota(p,h))}$,
  where for every $d$,
  $\lim \epsilon_d^{\fukcob(\widetilde{\mathcal{C}}_{1/2};
    \iota(p,h))} = 0$ as $p \longrightarrow p_0$ and
  $h \longrightarrow 0$. (The limits here are in the $C^1$-topology.)
  \label{pp:discr-fukcob-12}
\item In case the ends of $\gamma$ are along rays with $y$-coordinates
  in $\mathbb{Z}$ the map $\iota_{\gamma}$ and functors
  $\mathcal{I}_{\gamma; p,h}$ can be assumed to have values in
  $\widetilde{E}'_{\text{reg}}$ and
  $\fukcob(\widetilde{\mathcal{C}}; \widetilde{p})$ respectively. More
  precisely, the map $\iota_{\gamma}$ factors as a composition
  $$E'_{\textnormal{reg}} \times \mathcal{H}'_{\text{prof}}(\gamma)
  \xrightarrow{\; \iota'_{\gamma} \;}
  \widetilde{E}'_{\textnormal{reg}} \xrightarrow{\; j \; }
  \widetilde{E}'_{\textnormal{reg},1/2}$$ and the functors
  $\mathcal{I}_{\gamma; p,h}$ factor as the composition of the
  following two $A_{\infty}$-functors:
  \begin{equation} \label{eq:I-gamma'}
    \fuk(\mathcal{C}; p) \xrightarrow{\; \mathcal{I}'_{\gamma; p,h} \;} 
    \fukcob(\widetilde{\mathcal{C}}; \iota'_{\gamma}(p,h))
    \xrightarrow{\; \mathscr{J}\;}
    \fukcob(\widetilde{\mathcal{C}}_{1/2}; \iota_{\gamma}(p,h)).
  \end{equation}
  The map $\iota'_{\gamma}$ and the $A_{\infty}$-functor
  $\mathcal{I}'_{\gamma; p,h}$ have the same properties as described
  above for $\iota_{\gamma}$ and $\mathcal{I}_{\gamma; p,h}$
  respectively, with obvious modifications). The map $j$ and functor
  $\mathscr{J}$ are the ones introduced in~\eqref{eq:functor-map-j}.
\end{enumerate}

Naively speaking, the functor $\mathcal{I}_{\gamma; p,h}$ would have
been defined by extending perturbation data
$\mathscr{D}_{L_0, \ldots, L_d}$ for Lagrangians in $M$ to
perturbation data
$\mathscr{D}_{\gamma\times L_0, \ldots, \gamma \times L_d}$ in the
obvious way (using $h$) and then relate the $\mu_d$ operations on both
sides. However, due to technical issues related to compactness, the
realization of these functors turns out to be somewhat less
straightforward. We refer the reader
to~\cite[Section~4.2]{Bi-Co:lcob-fuk} for a detailed construction of
these functors.

\subsubsection{Additional properties relative to a given cobordisms}
Suppose we fix in advance a Lagrangian cobordism
$W \in \widetilde{\mathcal{C}}$ with the following properties.  Denote
by $K_1, \ldots, K_r \in \mathcal{C}$ the negative ends of $W$. Let
$\pi: \mathbb{R}^2 \times M \longrightarrow \mathbb{R}^2$ be the
projection. Assume that $\gamma$ intersects $\pi(W)$ only along the
projection of the horizontal cylindrical negative part of $W$
(corresponding to its negative ends) with one intersection point
corresponding to each end. Assume further that the intersection of
$\gamma$ and $\pi(W)$ is transverse and denote the intersection points
by $Q_1, \ldots, Q_r \in \mathbb{R}^2$, where $Q_j$ corresponds to the
$j'th$ negative end of $W$. Then we can restrict to profile functions
$h$ that have $O_{2j+1} = Q_j$ for every $1 \leq j \leq k$ and
redefine the spaces $\mathcal{H}_{\textnormal{prof}}$ and
$\mathcal{H}'_{\textnormal{prof}}$ by adding this restriction to the
their definitions. For simplicity, we will continue to denote these
spaces by $\mathcal{H}'_{\textnormal{prof}}$ and
$\mathcal{H}_{\textnormal{prof}}$.

Now, {\em in addition} to the previous list of properties, the map
$\iota_{\gamma}$ can be assumed to have also the following property:
let $L \in \mathcal{C}$ be a Lagrangian, and denote by
$\mathscr{D}_{L,K_j} = (H^{L,K_j}, J^{L,K_j})$ the Floer datum of
$(L,K_j)$ prescribed by $p$ and by
$\mathscr{D}_{\gamma \times L, V}=(H^{\gamma \times L, V}, J^{\gamma
  \times L, V})$ the Floer datum of $(\gamma \times L, V)$ prescribed
by $\iota_{\gamma}(p,h)$. Then we may assume that for small
neighborhoods $\mathcal{U}_j$ of $Q_j$ we have
$H^{\gamma \times L, V}(z,m) = h(z) + H^{L, K_j}(m)$ for every
$(z,m) \in \mathcal{U}_j \times M$. Moreover, we may assume that
$\mathcal{O}(H^{\gamma \times L, V}) = \cup_{j=1}^k \bigl(Q_j \times
\mathcal{O}(H^{L, K_j}) \bigr)$. \label{pp:iota-gamma-h-2}

\subsubsection{The weakly filtered structure of the inclusion
  functors} \label{sbsb:wf-inc-functors} The next proposition shows
that the inclusion functors are weakly filtered and gives more
information on their discrepancies.
\begin{prop} \label{p:wf-inc-functors} The family of
  $A_{\infty}$-functors $\mathcal{I}_{\gamma; p,h}$,
  $(p,h) \in E'_{\text{reg}} \times
  \mathcal{H}'_{\text{prof}}(\gamma)$, has the following properties:
  \begin{enumerate}
  \item $\mathcal{I}_{\gamma; p,h}$ is weakly filtered
    (see~\S\ref{sb:functors} for the definition).
  \item $\mathcal{I}_{\gamma; p,h}$ has discrepancy
    $\leq \bme^{\mathcal{I}_{\gamma;p,h}}$, where
    $\epsilon_d^{\mathcal{I}_{\gamma;p,h}} = 0 $ for every $d \geq 2$
    and
    $$\epsilon_1^{\mathcal{I}_{\gamma;p,h}} \leq \max \{ h(O_k) \mid
    1 \leq k = \text{odd} \leq l\}.$$ Note that
    $\epsilon_1^{\mathcal{I}_{\gamma; p,h}} \longrightarrow 0$ as
    $h \longrightarrow 0$ in the $C^0$-topology.
  \item $\mathcal{I}_{\gamma; p,h}$ is homologically unital. 
  \item For every $L \in \mathcal{C}$ denote by
    $e'_{\gamma \times L} = (\mathcal{I}_{\gamma; p,h})_1 (e_L) \in
    CF(\gamma \times L, \gamma \times L; \mathscr{D}_{\gamma \times L,
      \gamma \times L})$ the image of the homology unit
    $e_L \in CF(L,L; \mathscr{D}_{L,L})$ under the functor
    $\mathcal{I}_{\gamma; p,h}$. The collection of elements
    $\{ e'_{\gamma \times L} \}_{L \in \mathcal{C}}$ can be extended
    to a collection of homology units
    $\widetilde{\mathcal{E}} = \{e'_{V}\}_{V \in
      \widetilde{\mathcal{C}}}$ for
    $\fukcob(\widetilde{\mathcal{C}}_{1/2}, \iota_{\gamma}(p,h))$ with
    discrepancy $\leq \widetilde{u}'(p,h)$, where
    $\widetilde{u}'(p,h) \longrightarrow 0$ as
    $p \longrightarrow p_0 \in \mathcal{N}$ and $h \longrightarrow 0$
    in the $C^1$-topologies.
  \item With respect to the collection of homology units
    $\widetilde{\mathcal{E}}$ above we have
    $\fukcob(\widetilde{\mathcal{C}}_{1/2}; \iota_{\gamma}(p,h)) \in
    \hyperlink{h:asmp-ue}{\lbue(\widetilde{\zeta}(p,h))}$, where
    $\widetilde{\zeta}(p,h) \longrightarrow 0$ as
    $p \longrightarrow p_0 \in \mathcal{N}$ and $h \longrightarrow 0$
    in the $C^1$-topologies.
  \item Let $\mathcal{V}$ be the Yoneda module of
    $V \in \widetilde{\mathcal{C}}_{1/2}$. Then, with respect to the
    collection of homology units $\widetilde{\mathcal{E}}$ above we
    have
    $\mathcal{V} \in
    \hyperlink{h:asmp-H}{\lbh_s(\widetilde{\kappa}(p,h))}$, where
    $\widetilde{\kappa}(p,h) \longrightarrow 0$ as
    $p \longrightarrow p_0 \in \mathcal{N}$ and $h \longrightarrow 0$
    in the $C^1$-topologies.
  \end{enumerate}
  In case the ends of $\gamma$ have $y$-coordinates in $\mathbb{Z}$ an
  obvious analog holds for the family of functors
  $\mathcal{I}'_{\gamma; p, h}$ from~\eqref{eq:I-gamma'}.
\end{prop}
The proof of this proposition is straightforward, given the precise
definition of the functors $\mathcal{I}_{\gamma; p,h}$ which is
described in detail in~\cite[Section~4.2]{Bi-Co:lcob-fuk}.

\subsection{Weakly filtered iterated cones coming from cobordisms}
\label{sb:wf-icones-cobs}

Let $V \in \widetilde{\mathcal{C}}$ be a Lagrangian cobordism and
denote by $L_0, \ldots, L_r \in \mathcal{C}$ its negative ends. (In
contrast to~\S\ref{sb:ext-lcob} as well as~\cite{Bi-Co:lcob-fuk}, in
this section we index the negative ends from $0$ to $r$ rather than
from $1$ to $r$.)

Let $\gamma \subset \mathbb{R}^2$ be the curve depicted in
Figure~\ref{f:gamma-cob-c}. Let $p \in E'_{\text{reg}}$ and
$h \in \mathcal{H}'_{\text{prof}}(\gamma)$ be such that
$l := \# (\phi^h_1)^{-1}(\gamma) \cap \gamma = 2r+5$. 
\begin{figure}[htbp]
   \begin{center}
     \includegraphics[scale=0.70]{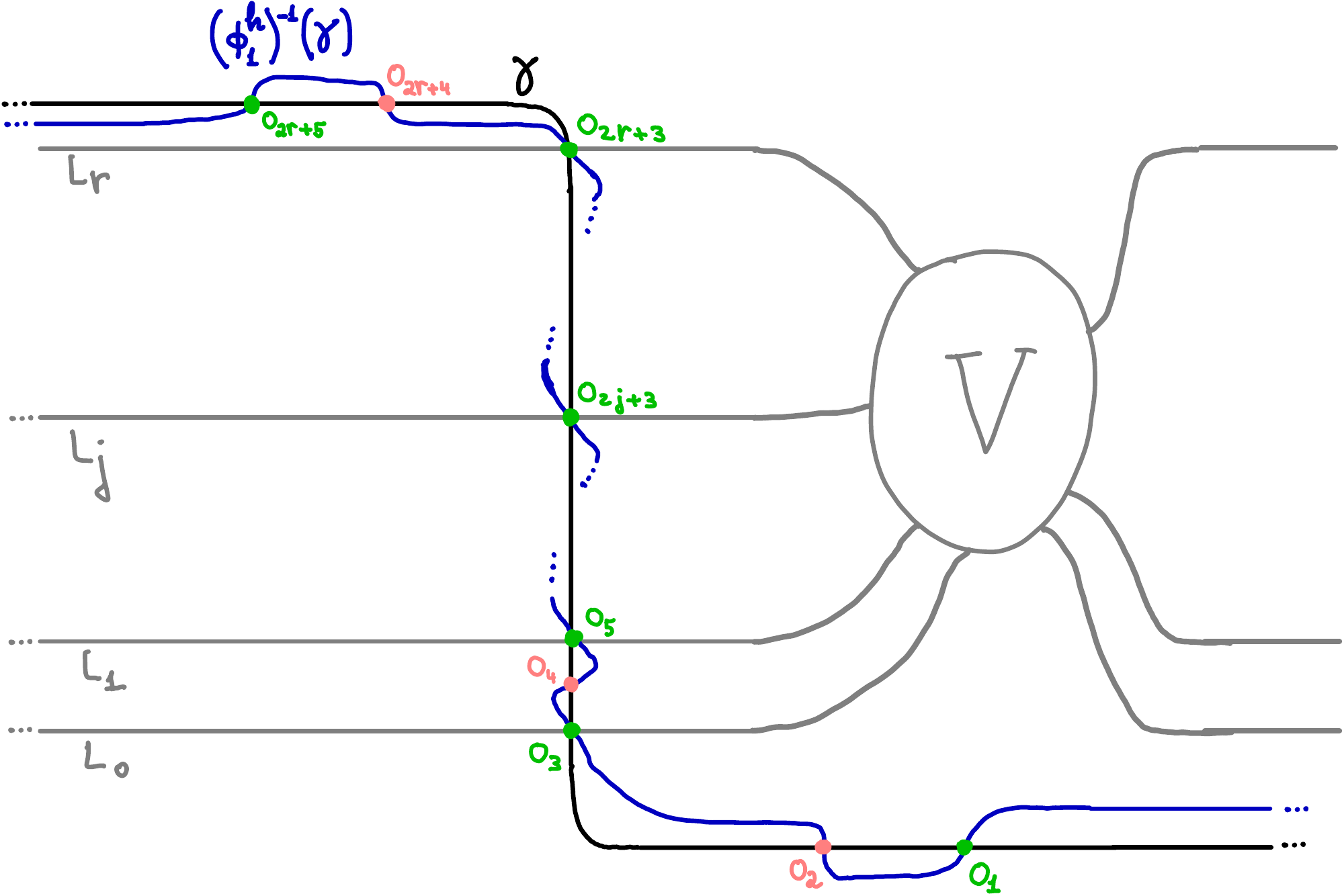}
   \end{center}
   \caption{The curves $\gamma$, $(\phi_1^h)^{-1}(\gamma)$ and the
     cobordism $V$.}
   \label{f:gamma-cob-c}
\end{figure}
Denote by $\mathcal{V}$ the Yoneda module of $V$, which we view here
as an $A_{\infty}$-module over the category
$\fukcob(\widetilde{\mathcal{C}}_{1/2};\iota_{\gamma}(p,h))$. Consider
now the pullback module
\begin{equation} \label{eq:pull-back-V} \mathcal{M}_{V; \gamma, p,h}
  := \mathcal{I}_{\gamma; p,h}^* \mathcal{V},
\end{equation}
which is a $\fuk(\mathcal{C}; p)$-module. Since
$\mathcal{I}_{\gamma; p, h}$ is a weakly filtered functor the module
$\mathcal{M}_{V; \gamma, p, h}$ is weakly filtered.
\begin{prop} \label{p:wf-M_V} The weakly filtered module
  $\mathcal{M}_{V; \gamma, p, h}$ has the following properties.
  \begin{enumerate}
  \item For every $N \in \mathcal{C}$ and $\alpha \in \mathbb{R}$ we
    have
    \begin{equation*}
      \begin{aligned}
        & \mathcal{M}_{V; \gamma, p, h}^{\leq \alpha}(N) = \\
        & CF^{\leq \alpha-h(O_3)}(N,L_0; p) \oplus CF^{\leq \alpha -
          h(O_5)}(N,L_1;p) \oplus \cdots \oplus CF^{\leq \alpha -
          h(O_{2r+3})}(N,L_r;p),
      \end{aligned}
    \end{equation*}
    where the last equality is of $\Lambda_0$-modules (but not
    necessarily of chain complexes). Here $CF(N, L_i; p)$ stands for
    $CF(N,L_i; \mathscr{D}_{N, L_i}) $, where $\mathscr{D}_{N, L_i}$
    is the Floer datum prescribed by $p \in E'_{\textnormal{reg}}$.
  \item $\mathcal{M}_{V; \gamma, p, h}$ has discrepancy
    $\leq \bme^{\mathcal{M}_{V; \gamma, p, h}}$, where
    \begin{equation} \label{eq:discrep-M_V}
      \epsilon_d^{\mathcal{M}_{V; \gamma, p, h}} \leq (d-1) \max\{
      h(O_k) \mid 1 \leq k = \text{odd} \leq 2r+5\} +
      \epsilon_d^{\fuk_{\text{cob}}(\widetilde{\mathcal{C}}_{1/2};
        \iota_{\gamma}(p,h))}.
    \end{equation}
  \end{enumerate}
\end{prop}
\begin{proof}
  The second statement follows immediately from
  Proposition~\ref{p:wf-inc-functors} and Lemma~\ref{l:pull-back-M}
  together with the fact that the higher terms of
  $\mathcal{I}_{\gamma; p,h}$ vanish. The first statement can be
  verified by a straightforward calculation.
\end{proof}

\begin{rem} \label{r:discrep-M_V-2} An inspection of the arguments
  from~\cite[Section~4.4]{Bi-Co:lcob-fuk} shows that the estimate for
  the discrepancy $\epsilon_d^{\mathcal{M}_{V; \gamma, p, h}}$
  in~\eqref{eq:discrep-M_V} can be slightly improved by replacing the
  ``$\max$'' term from~\eqref{eq:discrep-M_V} with
  $\max\{ h(O_k) \mid 3 \leq k = \text{odd} \leq 2r+3\}$. We will not
  go into details on that since this improvement will not play any
  role in our applications.
\end{rem}

Recall from~\cite[Section~4.4]{Bi-Co:lcob-fuk} that the module
$\mathcal{M}_{V; \gamma, p, h}$ is naturally isomorphic to an iterated
cone with attaching objects corresponding to the ends
$L_0, \ldots, L_r$ of $V$. More precisely, denote by $\mathcal{L}_j$
the Yoneda module corresponding to $L_j$. Then
\begin{equation*} \mathcal{M}_{V; \gamma, p, h} \cong \tcn
  (\mathcal{L}_r \xrightarrow{\; \phi_r \;} \tcn(\mathcal{L}_{r-1}
  \xrightarrow{\; \phi_{r-1} \;} \tcn( \cdots \tcn(\mathcal{L}_2
  \xrightarrow{\; \phi_2 \;} \tcn(\mathcal{L}_1 \xrightarrow{\; \phi_1
    \;} \mathcal{L}_0 )) {\cdot}{\cdot}{\cdot}))),
\end{equation*}
where $\phi_j$ is a module homomorphism between $\mathcal{L}_j$ and
the intermediate iterated cone involving the attachment of only the
first $j+1$ objects $\mathcal{L}_0, \ldots, \mathcal{L}_j$.

As we will see shortly, the module homomorphisms $\phi_j$ are weakly
filtered (and obviously the $\mathcal{L}_i$'s too) and consequently
the iterated cone $\mathcal{M}_{V; \gamma, p, h}$ can be endowed with
a weakly filtered structure by the algebraic recipe
of~\S\ref{sb:wf-mc} and~\S\ref{sb:wf-ic}. At the same time, we have
just seen that $\mathcal{M}_{V; \gamma, p, h}$ has another weakly
filtered structure being the pull back module by an inclusion functor,
as described in Proposition~\ref{p:wf-M_V}. Our goal now is to compare
these two weakly filtered structures and show that they are
essentially the same.

Consider the following collection of curves
$\gamma_1, \ldots, \gamma_r \subset \mathbb{R}^2$ with horizontal
ends, as depicted in Figure~\ref{f:gamma-j-c}. We assume that
$\gamma_r = \gamma$, the curve involved in the definition of
$\mathcal{M}_{V; \gamma, p, h}$.

We also choose profile functions
$h_1, \ldots, h_r: \mathbb{R}^2 \longrightarrow \mathbb{R}$ with
$h_j \in \mathcal{H}'_{\text{prof}}(\gamma_j)$ and such that the
following holds (see Figure~\ref{f:gamma-j-c}):
\begin{enumerate}
\item $h_r = h$. \label{pp:h_j-functions}
\item
  $(\phi_1^{h})^{-1}(\gamma) \cap \gamma = \{O_1, \ldots, O_{2r+5}\}$
\item
  $(\phi_1^{h_j})^{-1}(\gamma_j) \cap \gamma_j = \{O^j_1, \ldots,
  O^j_{2j+5}\}$, where $O^j_k = O_k$ for all $1 \leq k \leq
  2j+3$. Thus only the last two intersection points
  $O^j_{2j+4}, O^j_{2j+5}$ do not belong to the $\gamma_l$'s for
  $l>j$.
\item $h_j$ coincides with $h$ over the half-plane
  $\{y \leq y_{2j+3} + \tfrac{1}{100} \}$, where $y_{2j+3}$ is the
  $y$-coordinate of $O_{2j+3}$.
\end{enumerate}

\begin{figure}[htbp]
   \begin{center}
     \includegraphics[scale=0.80]{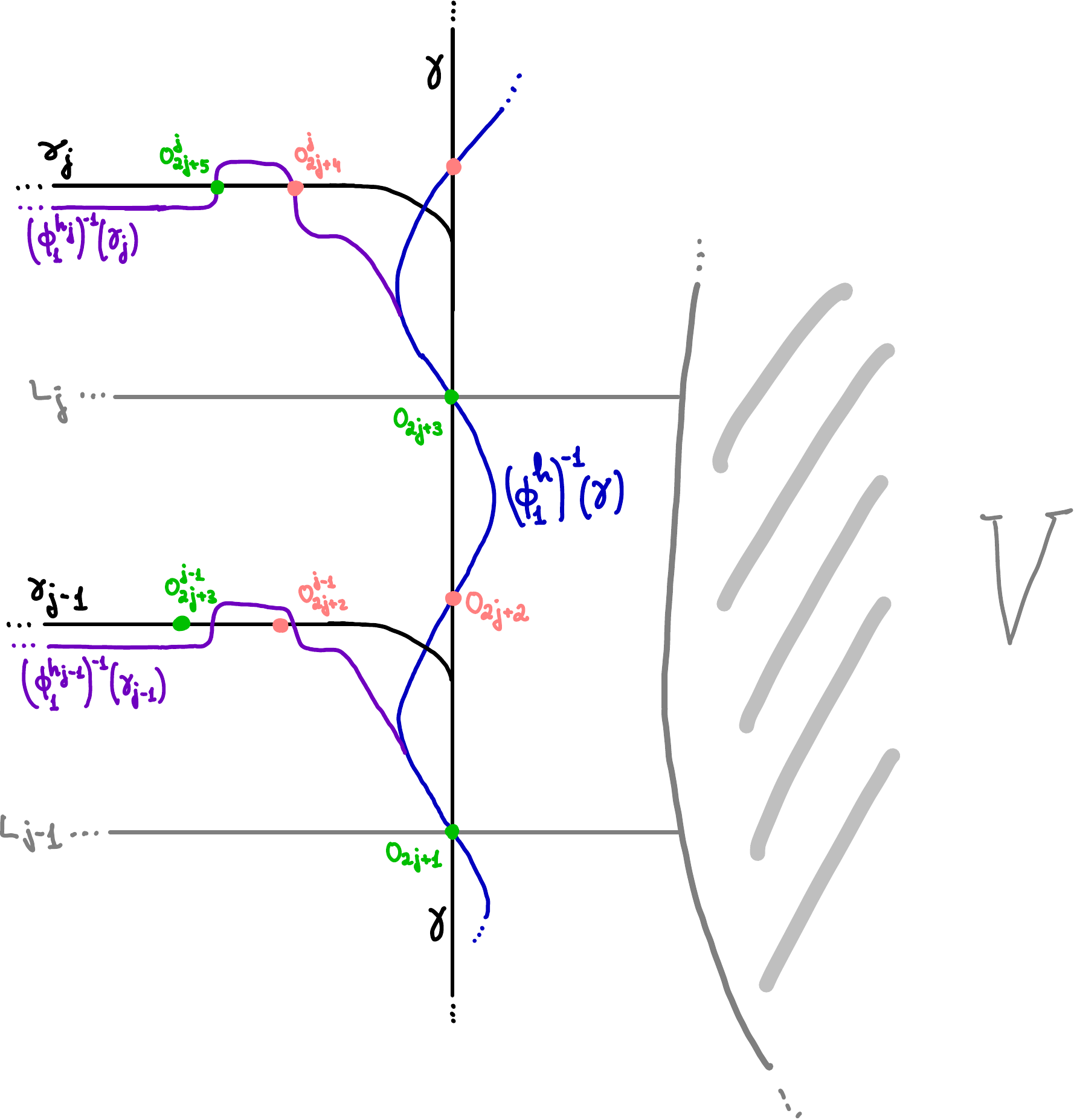}
   \end{center}
   \caption{A closer look at the curves $\gamma$, $\gamma_{j-1}$, and
     $\gamma_j$ near the $(j-1)$'th and $j$'th ends of $V$.}
   \label{f:gamma-j-c}
\end{figure}

We denote the space of all tuples of profile functions
$(h_1, \ldots, h_r)$ satisfying these conditions by
$\mathcal{H}'_{\text{prof}}(\gamma_1, \ldots, \gamma_r)$ and denote
elements of this space by $\upsilon = (h_1, \ldots, h_r)$. With this
notation, it is possible to choose maps
$\iota_{\gamma_j}: E'_{\textnormal{reg}} \times
\mathcal{H}'_{\text{prof}}(\gamma_j) \longrightarrow
\widetilde{E}'_{\textnormal{reg},1/2}$ for $j=1, \ldots, r$, as
in~\eqref{eq:iota-gamma} satisfying the following. For every
$(p, \upsilon) \in E'_{\text{reg}} \times
\mathcal{H}'_{\text{prof}}(\gamma_1, \ldots, \gamma_r)$ the choice of
data $\iota_{\gamma_j}(p,h_j) \in \widetilde{E}'_{\text{reg},1/2}$ has
the properties listed for $\iota_{\gamma}(p,h)$ on
page~\pageref{pp:iota-gamma-h} but with $\gamma$ replaced by
$\gamma_j$ and $h$ by $h_j$. (Consequently, for every
$\widetilde{p}_0 \in \mathcal{N}$ we have
$\lim \iota_{\gamma_j}(p, h_j) \in \widetilde{\mathcal{N}}_{1/2}$ as
$p \longrightarrow p_0$ and
$\upsilon = (h_1, \ldots, h_r) \longrightarrow (0, \ldots, 0)$.)
Moreover, we require that for every $j$ and
$p \in E'_{\textnormal{reg}}$,
$\upsilon = (h_1, \ldots, h_r) \in
\mathcal{H}'_{\text{prof}}(\gamma_1, \ldots, \gamma_r)$ the data
prescribed by $\iota_{\gamma_j}(p, h_j)$ is compatible with that
prescribed by $\iota_{\gamma_{j-1}}(p,h_{j-1})$ (in the obvious sense,
similar to $h_j$ being compatible with
$h_{j-1}$). By~\S\ref{sb:inc-functors}, the curves $\gamma_j$ and the
maps $\iota_{\gamma_j}$ induce a family of inclusion functors
$$\mathcal{I}_{\gamma_j; p, h_j}: \fuk(\mathcal{C}; p)
\longrightarrow \fuk(\widetilde{\mathcal{C}}_{1/2};
\iota_{\gamma_j}(p, h_j)),$$ parametrized by
$p \in E'_{\textnormal{reg}}$,
$\upsilon = (h_1, \ldots, h_r) \in
\mathcal{H}'_{\text{prof}}(\gamma_1, \ldots, \gamma_r)$.  We will use
below the notation
$\mathcal{I}_{\gamma_j; p, \upsilon} := \mathcal{I}_{\gamma_j; p,
  h_j}$, where $h_j$ is the $j$'th entry in the tuple $\upsilon$ since
it reflects better the parameters $(p, \upsilon)$ parametrizing this
family of functors. We will also write sometimes
$\iota_{\gamma_j}(p, \upsilon)$ for $\iota_{\gamma_j}(p, h_j)$.

Consider now the pullback $\fuk(\mathcal{C};p)$-modules
$$\mathcal{M}_{V; \gamma_j, p, \upsilon} =
\mathcal{I}^*_{\gamma_j; p, \upsilon} \mathcal{V}, \quad j=1, \ldots,
r.$$ We endow each of these modules with its weakly filtered structure
as defined at the beginning of~\S\ref{sb:wf-icones-cobs} and further
described by Proposition~\ref{p:wf-M_V} (where $l = 2j+5$, and
$\gamma$ should be replaced by $\gamma_j$ and $h$ by $h_j$). Next, for
every $0 \leq j \leq r$ denote by $\mathcal{L}_j$ the Yoneda module
associated to $L_j$, endowed with its weakly filtered structure
induced from $\fuk(\mathcal{C}; p)$. Finally, recall that for a weakly
filtered module $\mathcal{M}$ and $\nu \in \mathbb{R}$,
$S^{\nu}\mathcal{M}$ stands for the weakly filtered module obtained
from $\mathcal{M}$ by an action-shift of $\nu$
(see~\S\ref{sbsb:action-shifts}).

\begin{prop} \label{p:icones-M_j} For every
  $(p, \upsilon) \in E'_{\textnormal{reg}} \times
  \mathcal{H}'_{\text{prof}}(\gamma_1, \ldots, \gamma_r)$ there exist
  weakly filtered module homomorphisms
  $\phi_1: \mathcal{L}_1 \longrightarrow \mathcal{L}_0$ and
  $\phi_j: \mathcal{L}_j \longrightarrow S^{h(O_3)}\mathcal{M}_{V;
    \gamma_{j-1}, p, \upsilon}$ for $j=2, \ldots, r$ such that the
  following holds for every $1 \leq j \leq r$:
  \begin{enumerate}
  \item $\phi_j$ shifts action by $\leq 0$.
  \item The discrepancy of $\phi_j$ is $\leq \bm{\delta}^{\phi_j}$,
    where
    \begin{equation} \label{eq:delta-phi-j} \delta_d^{\phi_j} :=
      (d-1)\max_{\substack{1 \leq k \leq 2j+3 \\ k \text{ odd}}}
      h(O_k) + \epsilon_d^{\fukcob(\widetilde{\mathcal{C}}_{1/2};
        \iota_{\gamma_j}(p,\upsilon))} + h(O_{2j+3})-h(O_3).
    \end{equation}
  \item For every $1\leq j \leq r$,
    $S^{h(O_3)}\mathcal{M}_{V; \gamma_j, p, \upsilon} = \tcn(\phi_j;
    0, \bm{\delta}^{\phi_j})$ as a weakly filtered
    modules. (See~\S\ref{sb:wf-mc} for our conventions for weakly
    filtered cones.) In other words, the weakly filtered module
    $S^{h(O_3)}\mathcal{M}_{V; \gamma_j, p, \upsilon}$ coincides with
    the weakly filtered mapping cone over $\phi_j$.
  \end{enumerate}
\end{prop}

Recalling that
$\mathcal{M}_{V; \gamma, p,h} = \mathcal{M}_{V; \gamma_r, p,
  \upsilon}$, the above Proposition implies that
\begin{equation}\label{eq:filtered-cob-it-cone} 
  \begin{aligned}
    S^{h(O_3)} & \mathcal{M}_{V; \gamma, p, h} \\ = \; & \tcn
    (\mathcal{L}_r \xrightarrow{\; \overline{\phi}_r \;}
    \tcn(\mathcal{L}_{r-1} \xrightarrow{\; \overline{\phi}_{r-1} \;}
    \tcn( \cdots \tcn(\mathcal{L}_2 \xrightarrow{\; \overline{\phi}_2
      \;} \tcn(\mathcal{L}_1 \xrightarrow{\; \overline{\phi}_1 \;}
    \mathcal{L}_0 )) {\cdot}{\cdot}{\cdot}))),
  \end{aligned}
\end{equation}
where $\overline{\phi}_j := (\phi_j; 0, \bm{\delta}^{\phi_j})$ and the
cones in~\eqref{eq:filtered-cob-it-cone} are endowed with the
filtrations as defined in~\S\ref{sb:wf-mc}. In other words, up to a
small action-shift, $\mathcal{M}_{V; \gamma, p, h}$ can be viewed as a
weakly filtered iterated cone by the very same recipe described at the
beginning of~\S\ref{sb:wf-ic} (with $\rho_j=0$ and
$\mathcal{K}_j = S^{h(O_3)} \mathcal{M}_{V; \gamma_j, p,
  \upsilon}$). Consequently, we can apply Theorem~\ref{t:itcones} with
$\mathcal{K}_r = S^{h(O_3))}\mathcal{M}_{V; \gamma, p, h}$.

\begin{proof}[Proof of Proposition~\ref{p:icones-M_j}]
  The proof is based on two main ingredients. The first one is the
  theory developed in~\cite[Sections~4.2,~4.4]{Bi-Co:lcob-fuk} from
  which it follows that, ignoring action-filtrations, we have
  $\mathcal{M}_{V; \gamma_j, p, \upsilon} = \tcn(\mathcal{L}_j
  \longrightarrow \mathcal{M}_{V; \gamma_{j-1}, p, \upsilon})$. The
  second one comprises direct action-filtration calculations for the
  modules $\mathcal{M}_{V; \gamma_j, p, \upsilon}$ and the
  homomorphisms $\phi_j$.

  Before we go on, we should remark a notational difference
  between~\cite{Bi-Co:lcob-fuk} and the present
  paper. In~\cite{Bi-Co:lcob-fuk} the negative ends of the cobordism
  $V$ are indexed from $1$ to $r$, whereas in the present text the
  indexing runs between $0$ and $r$. This results in several other
  indexing differences between the two texts. For example, the curves
  $\gamma_j$ in the present text are the same as $\gamma_{j+1}$
  in~\cite{Bi-Co:lcob-fuk}. In the present text, the number of
  intersection points between $\phi_1^{h_j}(\gamma_j)$ and $\gamma_j$
  is $2j+5$, whereas in~\cite{Bi-Co:lcob-fuk} this number is $2j+3$,
  etc.

  We start by adding to the collection of curves
  $\gamma_1, \ldots, \gamma_r$ another curve $\gamma_0$, defined in
  the same way as the $\gamma_j$'s only that it is adapted to the
  $L_0$-end of $V$ in the sense that the negative end of $\gamma_0$
  goes above the $L_0$-end and below the $L_1$ end. We also choose
  $h_0 \in \mathcal{H}'_{\text{prof}}(\gamma_0)$ satisfying the same
  conditions as the $h_j$'s (see page~\pageref{pp:h_j-functions}) only
  for $j=0$. We write
  $(\phi_1^*{h_0})^{-1}(\gamma_0) \cap \gamma_0 = \{O_1, O_2, O_3,
  O_4^0, O_5^0\}$. To simplify the notation we also extend the tuple
  $\upsilon = (h_1, \ldots, h_r)$ to contain also $h_0$ and write
  $\upsilon = (h_0, \ldots, h_r)$. As before we have an inclusion
  functor associated to $\gamma_0, p, h_0$ and we consider the
  pullback module
  $\mathcal{M}_{V; \gamma_0, p, h_0} := \mathcal{I}_{\gamma_0; p,
    h_0}^* \mathcal{V}$. To be consistent with the previous notation,
  we will denote this module also by
  $\mathcal{M}_{V; \gamma_0, p, \upsilon}$.

  We first claim that there exist module homomorphisms
  $\phi_j: \mathcal{L}_j \longrightarrow \mathcal{M}_{V; \gamma_{j-1},
    p, h_{j-1}}$ for all $1 \leq j \leq r$, such that
  \begin{equation} \label{eq:M_V-cone} \mathcal{M}_{V; \gamma_j, p,
      h_j} = \tcn(\mathcal{L}_j \xrightarrow{\; \phi_j \;}
    \mathcal{M}_{V; \gamma_{j-1}, p, h_{j-1}}),
  \end{equation}
  where at the moment we ignore the action filtrations. This statement
  is not explicitly stated in~\cite[Section~4.4.2]{Bi-Co:lcob-fuk},
  but it follows easily from the arguments in that paper. More
  specifically, what is stated explicitly
  in~\cite[Section~4.4.2]{Bi-Co:lcob-fuk} is that there exists an
  exact triangle (in the derived category $D\fuk(\mathcal{C}; p)$) of
  the form
  $\mathcal{L}_j \longrightarrow \mathcal{M}_{V; \gamma_{j-1}, p,
    h_{j-1}} \longrightarrow \mathcal{M}_{V; \gamma_j, p, h_j}$. Here
  however, we claim a stronger statement, namely
  that~\eqref{eq:M_V-cone} holds at the chain level. We will now
  explain how to deduce~\eqref{eq:M_V-cone} from the theory developed
  in~\cite{Bi-Co:lcob-fuk}. In doing that we will mostly follow the
  notation from that paper.

  By~\cite[Proposition~4.4.1]{Bi-Co:lcob-fuk} for every
  $0 \leq j \leq r$ we have the following:
  \begin{enumerate}
  \item $A_{\infty}$-categories $\mathcal{B}_j$ and $\mathcal{B}'_j$
    (depending on $\gamma_j, p$ and $h_j$).
  \item Quasi-isomorphisms of $A_{\infty}$-categories:
    $e_j: \fuk(\mathcal{C}; p) \longrightarrow \mathcal{B}_j$,
    $p_j: \mathcal{B}_j \longrightarrow \mathcal{B}'_j$,
    $\sigma_j: \mathcal{B}'_j \longrightarrow \fuk(\mathcal{C}; p)$
    and $q_j: \mathcal{B}'_j \longrightarrow \mathcal{B}'_{j-1}$, for
    $j \geq 1$, all with vanishing higher order terms. Moreover, they
    satisfy:
    \begin{equation} \label{eq:sigma_j-p_j-e_j} \sigma_j \circ p_j
      \circ e_j = \id, \; \; \forall \; j\geq 0, \quad \text{and}
      \quad q_j \circ p_j \circ e_j = p_{j-1} \circ e_{j-1} \; \;
      \forall \; j\geq 1.
    \end{equation}
  \item A $\mathcal{B}_j$-module $\overline{\mathcal{M}}_j$ and a
    $\mathcal{B}'_j$-module $\mathcal{M}'_j$ such that:
    \begin{equation} \label{eq:M_j-M'_j}
      \begin{aligned}
        & \mathcal{M}_{V; \gamma_j, p, h_j} =
        e_j^*\overline{\mathcal{M}}_j, \quad p_j^* \mathcal{M}'_j =
        \overline{\mathcal{M}}_j, \;\; \forall \; j\geq 0 \\
        & \mathcal{M}'_j = \tcn (\sigma_j^* \mathcal{L}_j
        \xrightarrow{\; \varphi_j \;} q_j^*\mathcal{M}'_{j-1}), \; \;
        \forall \; j\geq 1,
      \end{aligned}
    \end{equation}
    for some module homomorphism $\varphi_j$. (This homomorphism was
    denoted by $\phi_j$
    in~\cite[Proposition~4.4.1]{Bi-Co:lcob-fuk}. We have denoted it
    here by $\varphi_j$ since $\phi_j$ is already used for a slightly
    different homomorphism.)
  \item For $j=0$ we have: $\mathcal{M}'_0 =
    \sigma_0^*\mathcal{L}_0$. \label{i:M'_0}
  \end{enumerate}
  We now pull back the second line of~\eqref{eq:M_j-M'_j} by the
  functor $p_j \circ e_j$. The desired equality~\eqref{eq:M_V-cone}
  now follows by using~\eqref{eq:sigma_j-p_j-e_j} together with the
  fact that $A_{\infty}$-functors pull back mapping cones to mapping
  cones (at the chain level). Note that for $j=0$, pulling back the
  equality from point~\eqref{i:M'_0} above yields:
  $\mathcal{M}_{V;\gamma_0, p, h_0} = \mathcal{L}_0$.

  We now turn to the weakly filtered setting. Throughout the rest of
  the proof it is useful to keep in mind that $h_j(O_k) = h(O_k)$ for
  every $0 \leq j \leq r$ and $1 \leq k \leq 2j+3$.
  
  We claim that in the weakly filtered setting the correct version
  of~\eqref{eq:M_V-cone} has the form:
  \begin{equation} \label{eq:SMV-1}
    \begin{aligned}
      & S^{h(O_3)}\mathcal{M}_{V;\gamma_j, p, \upsilon} =
      \tcn(\mathcal{L}_j \xrightarrow{\; (\phi_j; 0,
        \bm{\delta}^{\phi_j}) \;} S^{h(O_3)}
      \mathcal{M}_{V;\gamma_{j-1}, p, \upsilon}), \; \; \forall \; 1\leq j\leq r, \\
      & S^{h(O_3)}\mathcal{M}_{V; \gamma_0, p, \upsilon} = \mathcal{L}_0.
    \end{aligned}
  \end{equation}
      
  Of course, by Lemma~\ref{l:action-shift-cone}, the first line 
  of~\eqref{eq:SMV-1} is equivalent to:
  \begin{equation} \label{eq:SMV-2} \mathcal{M}_{V;\gamma_j, p,
      \upsilon} = \tcn(\mathcal{L}_j \xrightarrow{\; (\phi_j; 0,
      \bm{\delta}^{\phi_j}+h(O_3)) \;} \mathcal{M}_{V;\gamma_{j-1}, p,
      \upsilon}), \; \; \forall \; 1\leq j\leq r.
  \end{equation}

  To prove~\eqref{eq:SMV-2} one needs to go over the arguments in the
  proof of~\cite[Proposition~4.4.1]{Bi-Co:lcob-fuk} and take
  action-filtrations into consideration. An inspection of these
  arguments shows that the categories $\mathcal{B}_j$,
  $\mathcal{B}'_j$ and functors $e_j$, $p_j$, $\sigma_j$, $q_j$ are
  all weakly filtered, and so are the modules $\mathcal{M}'_j$ and
  $\overline{\mathcal{M}}_j$. Moreover, we have:
  \begin{enumerate}
  \item The discrepancies of both $\mathcal{B}_j$ and $\mathcal{B}'_j$
    are
    $\leq \bme^{\fukcob(\widetilde{\mathcal{C}}_{1/2};
      \iota_{\gamma_j}(p, h_j))}$.
  \item Both functors $p_j$ and $q_j$ are filtered, i.e. have
    discrepancies $\leq \bm{0}$.
  \item $e_j$ has discrepancy $\leq \bme^{e_j}$, where
    $\epsilon_1^{e_j} = \max \{ h_j(O^j_k) \mid 1 \leq k = \text{odd}
    \leq 2j+5\}$ and $\epsilon_d^{e_j} = 0$ for all $d \geq 2$.
  \item $p_j \circ e_j$ has discrepancy $\leq \bme^{p_j \circ e_j}$,
    where
    $\epsilon_1^{p_j \circ e_j} = \max \{ h(O_k) \mid 1 \leq k =
    \text{odd} \leq 2j+3\}$ and $\epsilon_d^{p_j \circ e_j} = 0$ for
    all $d \geq 2$.
  \item $\sigma_j$ has discrepancy $\leq \bme^{\sigma_j}$, where
    $\epsilon_1^{\sigma_j} = -h(O_{2j+3})$ and
    $\epsilon_d^{\sigma_j} = 0$ for all $d \geq 2$.
  \item The module homomorphism
    $\varphi_j: \sigma_j^*\mathcal{L}_j \longrightarrow q_j^*
    \mathcal{M}'_{j-1}$ shifts action by $\leq 0$ and has discrepancy
    $\leq \bme^{\varphi_j}$, where
    $$\epsilon_d^{\varphi_j} = 
    \epsilon_d^{\fukcob(\widetilde{\mathcal{C}}_{1/2};
      \iota_{\gamma_j}(p, \upsilon))} + h(O_{2j+3}).$$
  \item The modules $\mathcal{M}'_j$ and $\overline{\mathcal{M}}'_j$
    have discrepancies
    $\leq \bme^{\fukcob(\widetilde{\mathcal{C}}_{1/2};
      \iota_{\gamma_j}(p, h_j))}$.
  \item The equalities (or identifications) from~\eqref{eq:M_j-M'_j}
    hold also in the weakly filtered sense, where the cone over
    $\varphi_j$ on the 2'nd line of~\eqref{eq:M_j-M'_j} is now taken
    over $(\varphi_j; 0, \bme^{\varphi_j})$.
  \item $\mathcal{M}'_0 = S^{-h(O_3)}\sigma_0^*\mathcal{L}_0$ as
    weakly filtered modules. \label{i:M'_0-wf}
  \end{enumerate}
  To conclude the proof of~\eqref{eq:SMV-2} we pull back the weakly
  filtered version of the 2'nd line of~\eqref{eq:M_j-M'_j} by
  $p_j \circ e_j$ and use
  Lemmas~\ref{l:pull-back-cone},~\ref{l:pull-back-M}
  and~\ref{l:pull-back-f} (recall that $p_j$, $e_j$ do not have higher
  order terms). The assertion that
  $S^{h(O_3)}\mathcal{M}_{V; \gamma_0, p, \upsilon} = \mathcal{L}_0$
  follows in a similar way.
\end{proof}


\section{Proof of the main geometric statements} \label{s:main-geom}

In this section we prove the main geometric results of the
paper.

We will make use of the following variants of the notion of Gromov
width. Let $(M^{2n},\omega)$ be a symplectic manifold, $L\subset M$ a
Lagrangian submanifold and $Q\subset M$ a
subset. Following~\cite{Bar-Cor:Serre,Bar-Cor:NATO} we define the
Gromov width $\delta(L;Q)$ of $L$ relative to $Q$ as follows. Assume
first that $L \not \subset Q$. Define:
\begin{equation} \label{eq:delta1}
  \begin{aligned}
    \delta(L;Q)=\sup \bigl\{\pi r^{2} \in (0,\infty] \ \mid \ &
    \exists \; \text{a symplectic embedding} \; e:
    B(r) \longrightarrow M \\
    & \text{such that} \; e^{-1}(L) = B_{\mathbb{R}}(r) \; \text{and}
    \; e(B(r)) \cap Q = \emptyset \; \bigr\} ~.~
  \end{aligned}
\end{equation}
Here $B(r) \subset \mathbb{R}^{2n}$ is the standard $2n$-dimensional
closed ball of radius $r$, endowed with the standard symplectic
structure from $\mathbb{R}^{2n}$, and
$B_{\mathbb{R}}(r) := B(r) \cap (\mathbb{R}^n \times \{0\})$ is the
real part of $B(r)$. In case $L \subset Q$ we set $\delta(L;Q) := 0$.

Another variant of the Gromov width is associated to an immersed
Lagrangian. Let $\widehat{\mathbb{L}}$ be a smooth closed manifold
(possibly disconnected) and let
$\iota: \widehat{\mathbb{L}} \longrightarrow M$ be a Lagrangian
immersion with image $\mathbb{L} := \iota(\widehat{\mathbb{L}})$. We
would like to measure the "size" of a subset of the double points of
$\mathbb{L}$ relative to a given subset $Q\subset M$.  More precisely,
denote by $\Sigma(\iota) \subset \mathbb{L}$ the set of points that
have more than one preimage under the immersion $\iota$. Let
$\Sigma' \subset \Sigma(\iota)$ be a non-empty subset such that each
point in $\Sigma'$ is a transverse intersection of two branches of the
immersion. As before, let $Q \subset M$ be a subset. Assume first that
$\Sigma' \not \subset Q$. We define the Gromov width
$\delta^{\Sigma'}(\mathbb{L}; Q)$ of the self-intersection set
$\Sigma'$ relative to $Q$ by:
\begin{equation*} \label{eq:delta-sigma}
  \begin{aligned}
    \delta^{\Sigma'}(\mathbb{L}; Q)=\sup \bigl\{ \pi r^{2}\in
    (0,\infty] \; \mid \; & \forall \; x \in \Sigma', \; \exists \;
    \text{a symplectic embedding} \;
    e_{x}:B(r) \longrightarrow M \; \text{with},\\
    & e_{x}(0)=x, \; e_{x}^{-1}(\mathbb{L})= B_{\mathbb{R}}(r)\cup i
    B_{\mathbb{R}}(r), \;
    e_{x}(B(r)) \cap Q = \emptyset, \\
    & \mathrm{and} \; e_{x'}(B(r))\cap e_{x''}(B(r)) = \emptyset \;
    \text{whenever} \; x' \neq x'' \bigr\}.
  \end{aligned}
\end{equation*}
Here $iB_{\mathbb{R}}(r)$ stands for the imaginary part of the ball,
$iB_{\mathbb{R}}(r) := B(r) \cap (\{0\} \times \mathbb{R}^n)$.

In case $\emptyset \neq \Sigma' \subset Q$ we set
$\delta^{\Sigma'}(\mathbb{L}; Q) = 0$. In case $\Sigma'=\emptyset$ we
set $\delta^{\emptyset}(\mathbb{L};Q)=\infty$. In what follows, if
$Q=\emptyset$, then we omit the set $Q$ from the notation in both
$\delta(L;Q)$ and $\delta^{\Sigma'}(\mathbb{L};Q)$.

The next important geometric measurement is the {\em shadow} of a
cobordism, as defined in~\cite{Co-She:metric} and already mentioned in the introduction. Let
$V \subset \mathbb{R}^2 \times M$ be a Lagrangian cobordism. Denote by
$\pi: \mathbb{R}^2 \times M \longrightarrow \mathbb{R}^2$ the
projection. The shadow $\mathcal{S}(V)$ of $V$ is defined as:
\begin{equation} \label{eq:shadow} \mathcal{S}(V) = Area \bigl(
  \mathbb{R}^2 \setminus \mathcal{U} \bigr),
\end{equation}
where $\mathcal{U} \subset \mathbb{R}^2 \setminus \pi(V)$ is the union
of all the {\em unbounded} connected components of
$\mathbb{R}^2 \setminus \pi(V)$. Equivalently one can define
$\mathcal{S}(V)$ as follows. Fix a subset
$V_0 := V \cap \pi^{-1}([-R,R] \times \mathbb{R})$ outside of which
$V$ is cylindrical. Then $\mathcal{S}(V)$ is the infimum of $Area(D)$,
where $D$ runs over all subsets $D \subset \mathbb{R}^2$ which are
symplectomorphic to a $2$-dimensional disk and such that
$D \supset \pi(V_0)$.

We now turn to the main result which we restate here for the convenience of the reader. Recall that
$\mathcal{L}ag^{\ast}(M)$ denotes the collection of closed Lagrangian
submanifolds of $M$ of class $*$, where $*$ stands either for the
weakly exact Lagrangians ($* = \text{we}$ in short), or for the
monotone Lagrangians with given Maslov-$2$ disk count
$\mathbf{d} \in \Lambda_0$ ($*=(\text{mon}, \mathbf{d})$ in short) as
introduced in~\S\ref{sb:monotone}. Similarly, we have the collection
$\mathcal{L}ag^{\ast}(\mathbb{R}^2 \times M)$ of Lagrangian cobordisms
$V \subset \mathbb{R}^2 \times M$ of class $*$, where $*$ is as above.

\begin{thm} \label{thm:fission} Let
  $L,L_{1},\ldots, L_{k}\in \mathcal{L}ag^{\text{we}}(M)$ and
  $V:L \leadsto (L_{1},\ldots, L_{k})$ a weakly exact Lagrangian
  cobordism.  Denote $S := \cup_{i=1}^k L_i$ the union of the
  Lagrangians corresponding to the negative ends of $V$. Then
  \begin{equation} \label{eq:SV-delta} \mathcal{S}(V)\geq
    \tfrac{1}{2}\delta(L;S).
  \end{equation}

  For the next two points of the theorem we will use the following
  notation. Let $N \in \mathcal{L}ag^{\text{we}}(M)$ be another weakly
  exact Lagrangian submanifold and consider
  $S = \cup_{i=1}^k L_i \subset M$ and $N \cup S \subset M$ as
  immersed Lagrangians (parametrized by $\coprod_{i=1}^k L_i$ and
  $N \coprod (\coprod_{i=1}^k L_i)$ respectively).
  \begin{enumerate} [label=(\alph*),ref=\alph*]
  \item \label{i:ch-inters} Assume that $N$ intersects each of the
    Lagrangians $L_1, \ldots, L_k$ transversely and that
    $N \cap L_i \cap L_j = \emptyset$ for all $i \neq j$. Denote
    $\Sigma' := N \cap S$. If
    $\mathcal{S}(V) <\frac{1}{2}\ \delta^{\Sigma'}(N\cup S)$ then
    \begin{equation} \label{eq:N-cap-L}
      \#(N\cap L) \geq \sum_{i=1}^{k}\# (N\cap L_{i})~.~
    \end{equation}
  \item \label{i:hf-inters} Assume that the Lagrangians
    $L_{1},\ldots, L_{k}$ intersect pairwise transversely and that no
    three of them have a common intersection point (i.e.
    $L_i \cap L_j \cap L_r = \emptyset$ for all distinct indices
    $i,j,r$). Let $\Sigma''$ be the set of all double points of $S$,
    i.e.  $\Sigma'' := \cup_{1 \leq i < j \leq k} L_{i} \cap
    L_{j}$. If $\mathcal{S}(V) < \frac{1}{4}\ \delta^{\Sigma''}(S; N)$
    then
    \begin{equation} \label{eq:N-cap-L-HF}
      \# (N\cap L)\geq \displaystyle\sum_{i=1}^{k} 
      \dim_{\Lambda} (HF(N,L_{i}))~.~
    \end{equation}
  \end{enumerate}
\end{thm}

The proof is given in~\S\ref{sb:prf-thm-fission} below.

Theorem~\ref{thm:fission} has an analog in the monotone case too.
Recall from~\S\ref{sb:monotone} the Maslov-$2$ disk count
$\mathbf{d} \in \Lambda_0$ associated to a monotone Lagrangian $L$ and
also its minimal disk area $A_L$ defined by~\eqref{eq:min-disk-area}
in~\S\ref{sb:monotone}.

\begin{thm} \label{thm:fission-mon} Let
  $L,L_{1},\ldots, L_{k} \subset M$ be monotone Lagrangians and
  $V:L \leadsto (L_{1},\ldots, L_{k})$ a connected monotone
  cobordism. Let $S$ be the same as in
  Theorem~\ref{thm:fission}. Denote by $\mathbf{d} \in \Lambda_0$ the
  Maslov-$2$ disk count of $L$ (hence by~\S\ref{sb:monotone} also of
  the $L_i$'s) and let $N \subset M$ be another monotone Lagrangian
  with $\mathbf{d}_N = \mathbf{d}$.  Then:
  \begin{equation} \label{eq:SV-delta-mon} \mathcal{S}(V)\geq \min \{
    \tfrac{1}{2}\delta(L;S), A_L\}.
  \end{equation}
  Moreover, under the above assumptions statements~\eqref{i:ch-inters}
  and~\eqref{i:hf-inters} continue to hold as stated in
  Theorem~\ref{thm:fission}.
\end{thm}
The proof is given in~\S\ref{sb:prf-thm-fission-mon}.
\subsection{Proof of
  Theorem~\ref{thm:fission}} \label{sb:prf-thm-fission}

We begin with the proof of~\eqref{eq:SV-delta}. We first assume that
the Lagrangians $L, L_0, \ldots, L_k$ intersect pairwise transversely,
and treat the general case afterwards.

\begin{figure}[htbp]
   \begin{center}
     \includegraphics[scale=0.5]{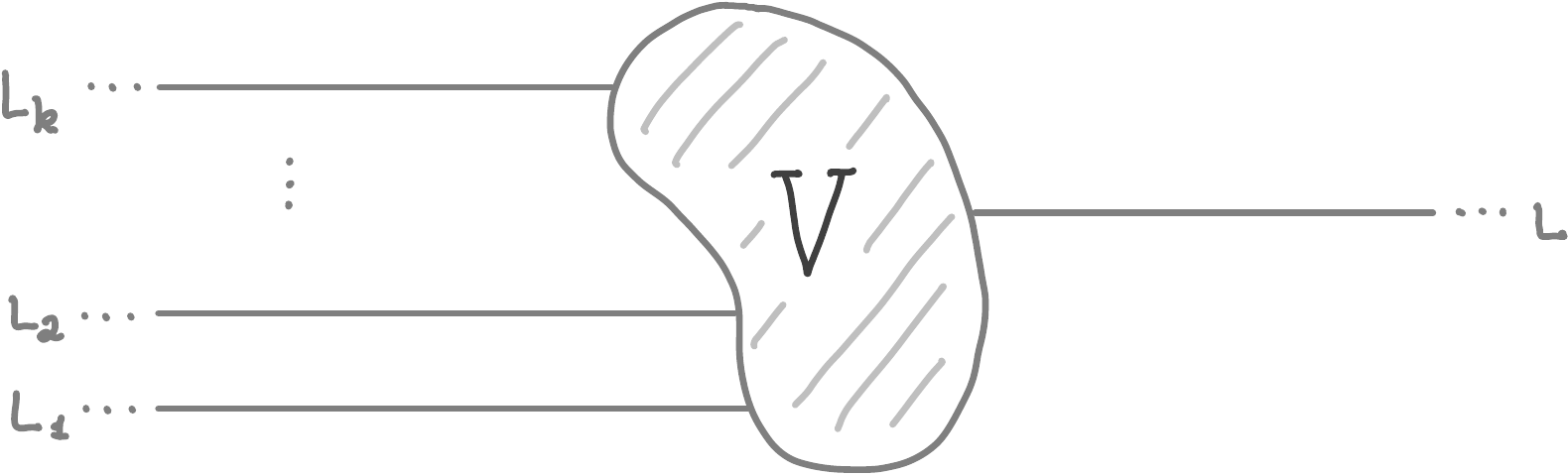} \; 
     \includegraphics[scale=0.5]{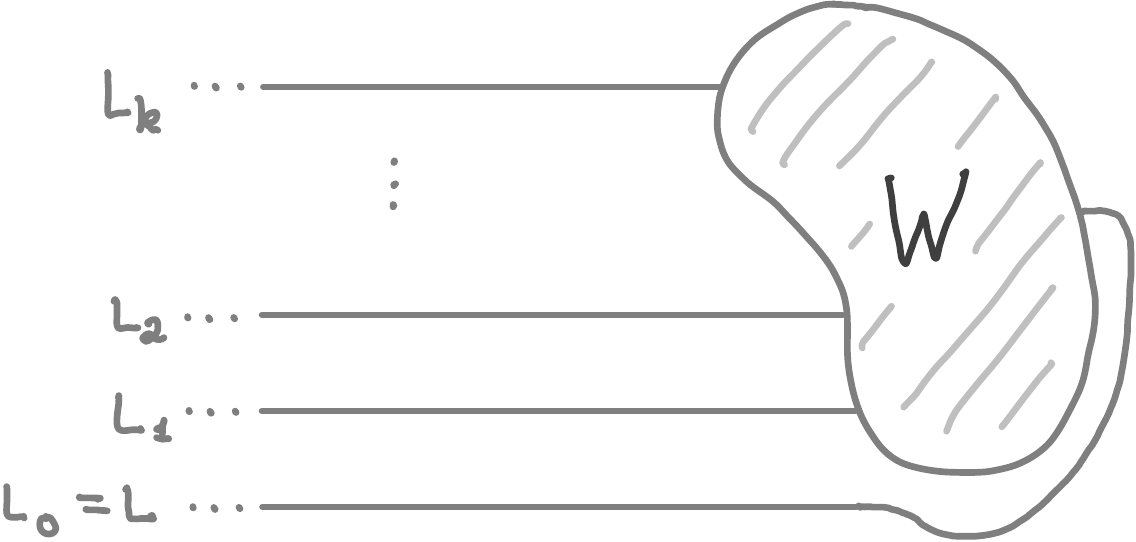}
   \end{center}
   \caption{The cobordisms $W$ obtained from $V$ by bending the
     positive end.}
   \label{f:cob-v-w}
\end{figure}

We start by bending the positive end of $V$ by $180^{\circ}$ clockwise
in such a way as to get a cobordism $W$ without positive ends, and
whose negative ends are $(L_0,L_{1},\ldots, L_{k})$, where $L_0 :=
L$. See Figure~\ref{f:cob-v-w}. Clearly
$\mathcal{S}(W)=\mathcal{S}(V)$.

Fix $\epsilon>0$. Let $\gamma$, $\gamma'$ be two curves, as depicted
in Figure~\ref{f:gamma-gamma'}, and such that there exists a (not
compactly supported) Hamiltonian isotopy, horizontal at infinity,
$\varphi_t: \mathbb{R}^2 \longrightarrow \mathbb{R}^2$, $t \in [0,1]$,
with $\varphi_0 = \id$, $\varphi_1(\gamma) = \gamma'$ and with
\begin{equation} \label{eq:l-varphi-t} \text{length} \{\varphi_t\}
  \leq \mathcal{S}(W) + \epsilon/2,
\end{equation}
where $\text{length} \{\varphi_t\}$ stands for the Hofer length of the
isotopy $\{\varphi_t\}$.

\begin{figure}[htbp]
   \begin{center}
     \includegraphics[scale=0.5]{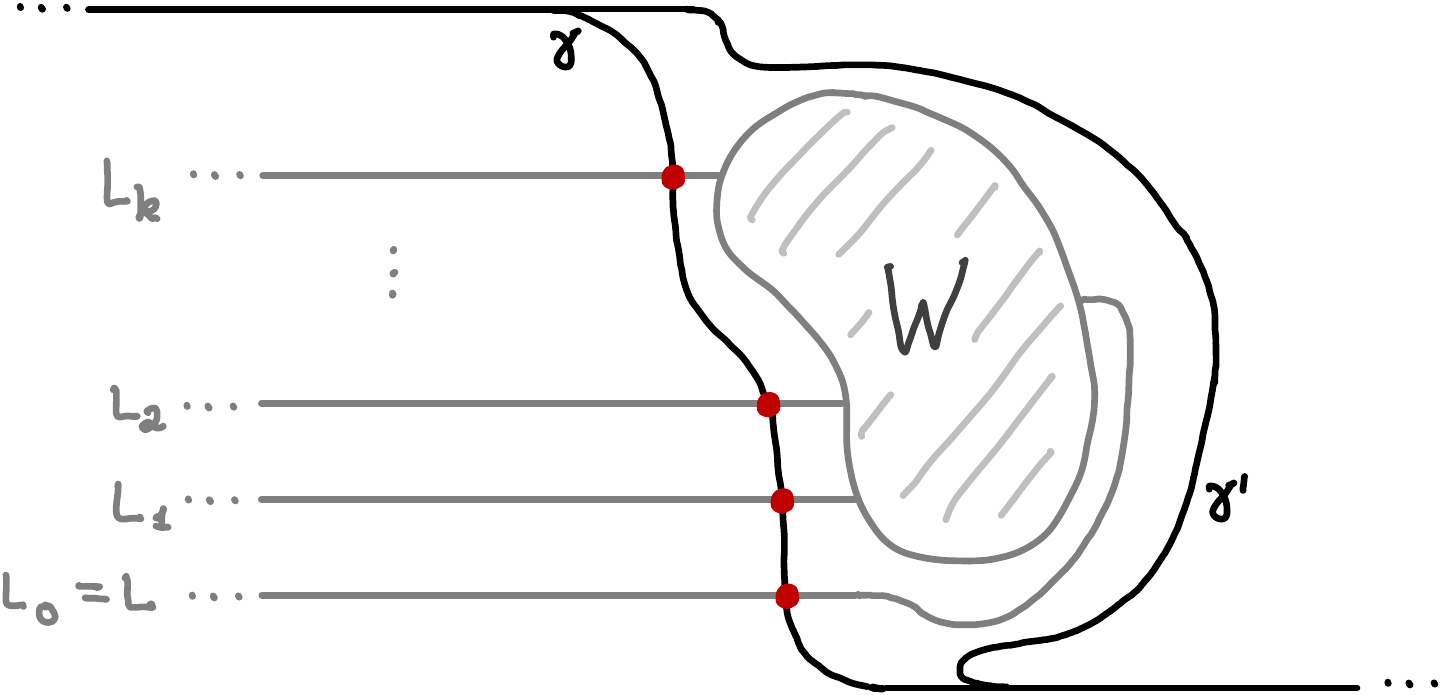}
   \end{center}
   \caption{The curves $\gamma$ and $\gamma'$ and the cobordism $W$.}
   \label{f:gamma-gamma'}
\end{figure}

Put $S = \cup_{i=1}^k L_i$ and let
$e: B(r) \longrightarrow M \setminus S$ be a symplectic embedding as
in the definition of $\delta(L_0; S)$ in~\eqref{eq:delta1}, with
$$\delta(L_0, S) - \epsilon \leq \pi r^2 \leq \delta(L_0,S).$$ 
Next, let
\begin{equation} \label{eq:JB}
  B := \textnormal{image\,}(e), \quad q := e(0) \in L_0, \quad
  J^B := e_*(J_{\textnormal{std}}),
\end{equation}
where the latter is the complex structure on $B$ corresponding to the
standard complex structure $J_{\textnormal{std}}$ of $B^{2n}(r)$ via
the embedding $e$.

Next, we fix a symplectic identification between a small
\label{pp:function-f}open
neighborhood $U$ of $L_0$ in $M$ and a neighborhood $U'$ of the
zero-section in $T^*(L)$. Let $f: L \longrightarrow \mathbb{R}$ be a
$C^1$-small Morse function with exactly one local maximum at the point
$q \in L_0$.  We extend $f$ to a function
$\widetilde{f}: U' \longrightarrow \mathbb{R}$ by setting it to be
constant along the fibers of the cotangent bundle. Finally, let
$H_f^{L_0,L_0}: M \longrightarrow \mathbb{R}$ be a smooth function
such that $H_f^{L_0,L_0}|_U$ coincides with $\widetilde{f}$ via the
identification between $U$ and $U'$ that we have just fixed.

We now turn to the Fukaya categories relevant for this proof. Let
$\mathcal{C}$ be the collection of Lagrangians $L_0, \ldots, L_k$. We
will use the Fukaya categories $\fuk(\mathcal{C})$ and
$\fukcob(\widetilde{\mathcal{C}})$ associated to $\mathcal{C}$. More
specifically, we consider regular perturbation data
$p \in E'_{\textnormal{reg}}$ and $C^1$-small profile functions
$h \in \mathcal{H}'_{\textnormal{prof}}(\gamma)$ as
in~\S\ref{sb:inc-functors}. We impose two additional restrictions on
the admissible choices of perturbation data $p$ as follows. The first
one is that the datum $\mathscr{D}_{L_0,L_0}$ of the pair $(L_0,L_0)$
should have the function $H_f^{L_0,L_0}$ as its Hamiltonian function,
defined using any choice of a $C^1$-small Morse functions $f$ as
described above. Furthermore, we allow only for functions $f$ that are
sufficiently $C^1$-small such that
$\mathcal{O}(H_f^{L_0,L_0}) = \textnormal{Crit}(f)$. Note that for
every $y \in \mathcal{O}(H_f^{L_0,L_0})$ we have $A(y) = f(y)$.

The second restriction is that the Hamiltonian functions $H^{L_i,L_j}$
in the Floer data $\mathscr{D}_{L_i, L_j}$, $i \neq j$, are all
$0$. It is possible to impose these additional restriction and still
maintain regularity since we have assumed that the Lagrangians
$L_0, L_1, \ldots, L_k$ intersect pairwise transversely. With these
choices we have for every $i \neq j$:
$$\mathcal{O}(H^{L_i,L_j}) = L_i \cap L_j, \quad A(z) = 0, 
\; \forall \, z \in \mathcal{O}(H^{L_i,L_j}).$$

We denote the space of all such regular choices of perturbation data
by $E''_{\textnormal{reg}} \subset E'_{\textnormal{reg}}$. We remark
that the Morse function $f$ is not fixed over $E''_{\textnormal{reg}}$
and each choice $p \in E''_{\textnormal{reg}}$ comes with its own
function $f$. Finally, note that $\mathcal{N}$ is still in the closure
of $E''_{\textnormal{reg}}$.

We now appeal to the theory developed in~\S\ref{s:floer-theory}.
Consider the Fukaya category $\fuk(\mathcal{C}; p)$
(see~\S\ref{sb:fukaya-families}) as well as the Fukaya category of
cobordisms
$\fukcob(\widetilde{\mathcal{C}}_{1/2}; \iota_{\gamma}(p,h))$
(see~\S\ref{sb:ext-lcob} and~\eqref{eq:I-gamma}
in~\S\ref{sb:inc-functors}). Recall that we have an ``inclusion''
functor
$\mathcal{I}_{\gamma; p,h}: \fuk(\mathcal{C}; p) \longrightarrow
\fukcob(\widetilde{\mathcal{C}}_{1/2}; \iota_{\gamma}(p,h))$.

Denote by $\mathcal{W}$ the Yoneda module corresponding to the object
$W \in \textnormal{Ob}(\fukcob(\widetilde{\mathcal{C}}_{1/2};
\iota_{\gamma}(p,h)))$ and consider its pull-back by the functor
$\mathcal{I}_{\gamma; p,h}$:
$$\mathcal{M}_{W; \gamma, p, h} := \mathcal{I}_{\gamma; p,h}^*\mathcal{W}.$$

Recall from~\S\ref{sb:wf-fukaya},~\S\ref{sb:ext-lcob}
and~\S\ref{sb:inc-functors} that the $A_{\infty}$-categories
$\fuk(\mathcal{C}; p)$,
$\fukcob(\widetilde{\mathcal{C}}_{1/2}; \iota_{\gamma}(p,h))$ as well
as the $A_{\infty}$-functor $\mathcal{I}_{\gamma; p,h}$ are all weakly
filtered. Moreover, by~\S\ref{sb:wf-icones-cobs} the module
$\mathcal{M}_{W; \gamma, p, h}$ if weakly filtered too. By
Propositions~\ref{p:wf-M_V}, and
points~\eqref{pp:iota_gamma-lim},~\eqref{pp:discr-fukcob-12} on
page~\pageref{pp:iota_gamma-lim} the discrepancy of this module is
bounded from above by
$\bme(p,h) = (\epsilon_1(p,h), \epsilon_2(p,h), \ldots,
\epsilon_d(p,h), \ldots)$ which satisfies
$\lim \epsilon_d(p,h) \longrightarrow 0$ for every $d$, as
$p \longrightarrow p_0 \in \mathcal{N}$ and $h \longrightarrow 0$ (the
latter in the $C^1$-topology).

Throughout the proof we will repeatedly deal with quantities having
the same asymptotics as $\epsilon_d(p,h)$. In order to simplify the
text we introduce the following notation. Let
$\mathcal{N}_0 \subset \mathcal{N}$ and let $(p,h) \longmapsto C(p,h)$
be a real valued function defined for $p$ in a subset of
$E'_{\textnormal{reg}}$ whose closure contains $\mathcal{N}_0$, and
$h \in \mathcal{H}'_{\textnormal{prof}}$. We will write
$C(p,h) \in \onx{\mathcal{N}_0}$ to indicate that for every
$p_0 \in \mathcal{N}_0$ we have $\lim C(p,h) = 0$ as
$p \longrightarrow p_0$ and $h \longrightarrow 0$ (the latter in the
$C^1$-topology).

By Proposition~\ref{p:icones-M_j}
(and~\eqref{eq:filtered-cob-it-cone}) we have:
\begin{equation} \label{eq:filtered-cob-it-cone-W}
  \begin{aligned}
    S^{s_h} & \mathcal{M}_{W; \gamma, p, h} \\ = \; & \tcn
    (\mathcal{L}_k \xrightarrow{\; \overline{\phi}_k \;}
    \tcn(\mathcal{L}_{k-1} \xrightarrow{\; \overline{\phi}_{k-1} \;}
    \tcn( \cdots \tcn(\mathcal{L}_2 \xrightarrow{\; \overline{\phi}_2
      \;} \tcn(\mathcal{L}_1 \xrightarrow{\; \overline{\phi}_1 \;}
    \mathcal{L}_0 )) {\cdot}{\cdot}{\cdot}))),
  \end{aligned}
\end{equation}
where $s_h \longrightarrow 0$ as $h \to 0$. (Recall
from~\S\ref{sbsb:action-shifts} that
$S^{s_h}\mathcal{M}_{V; \gamma, p, h}$ stands for the module
$\mathcal{M}_{V; \gamma, p, h}$ with action-shift by $s_h$.) The
modules $\mathcal{L}_i$ in~\eqref{eq:filtered-cob-it-cone-W} are the
Yoneda modules of the $L_i$'s. The notation $\overline{\phi_i}$ stands
for $\overline{\phi}_i = (\phi_i, 0, \bm{\delta^{(i)}})$, with
$\phi_i$ being a homomorphism of modules that shifts action by
$\leq 0$ and has discrepancy $\leq \bm{\delta^{(i)}}(p,h)$ where for
every $d$ we have $\delta^{(j)}_d (p,h) \in \on$. For simplicity of
notation set
$\bm{\delta}(p,h) := \max \{\bm{\delta^{(1)}}(p,h), \ldots,
\bm{\delta^{(k)}}(p,h)\}$, so that the discrepancy of all the
$\phi_i$'s is $\leq \bm{\delta}(p,h)$ and we still have
$\delta_d(p,h) \in \on$ for all $d$.

Consider the filtered chain complex
$\mathscr{C}_{p,h} := S^{s_{h}}\mathcal{M}_{W; \gamma, p,
  h}(L_0)$ \label{pp:C-ph} endowed with the differential coming from
the $\mu_1$-operation of $\mathcal{M}_{W; \gamma, p, h}$. By
definition
$\mathscr{C}_{p,h} = S^{s_h}CF(\gamma \times L_0, W;
\mathscr{D}_{\gamma \times L_0, W})$, where
$\mathscr{D}_{\gamma \times L_0, W}$ is the Floer datum prescribed by
$\iota_{\gamma}(p,h)$. By~\eqref{eq:filtered-cob-it-cone-W} the Floer
complex of $(L_0,L_0)$ is a subcomplex of $\mathscr{C}_{p,h}$, or more
precisely, we have an {\em action preserving} inclusion of chain
complexes:
\begin{equation} \label{eq:CF-L_0L_0} CF(L_0,L_0; p) \subset
  \mathscr{C}_{p,h},
\end{equation}
where $\mathscr{D}_{L_0,L_0}$ is specified by $p$ and is subject to
the additional restrictions imposed earlier in the proof. To simplify
the notation, we will denote from now on for a pair of Lagrangians
$(L', L'')$ by $CF(L',L''; p)$ the Floer complex
$CF(L',L''; \mathscr{D}_{L',L''})$, where $\mathscr{D}_{L',L''}$ is
the Floer datum specified by $p$.

Recall that we also have the curve $\gamma' \subset \mathbb{R}^2$ with
$\gamma' \cap \pi(W) = \emptyset$. Choose a Floer datum $\mathscr{D}'$
for $(\gamma' \times L_0, W)$ with a sufficiently $C^2$-small
Hamiltonian function so that
$CF(\gamma' \times L_0, W; \mathscr{D}')=0$. Now $\gamma \times L_0$
can be Hamiltonianly isotoped to $\gamma' \times L_0$ via an isotopy
horizontal at infinity with Hofer length
$\leq \mathcal{S}(W) + \epsilon/2$. By standard Floer theory (see
e.g.~\cite[Section~5.3.2]{FO3:book-vol1}) the identity map on
$\mathscr{C}_{p,h}$ is null homotopic via a chain homotopy which shifts
action by $\leq \mathcal{S}(W) + \epsilon/2$. Translated to the
formalism of~\eqref{eq:Bh-1} in~\S\ref{s:filt-ch} this means that
$B_h(\id_{\mathscr{C}_{p,h}}) \leq \mathcal{S}(W) + \epsilon/2$, hence
by~\eqref{eq:beta-beta_h} we have:
\begin{equation} \label{eq:Bh-C} \beta(\mathscr{C}_{p,h}) \leq
  \mathcal{S}(W) + \epsilon/2,
\end{equation}
where $\beta(\mathscr{C}_{p,h})$ is the boundary depth of the
(acyclic) chain complex $\mathscr{C}_{p,h}$ as defined
in~\eqref{sec:bdry-depth}.

%
%
We now appeal to Theorem~\ref{t:itcones} applied to the weakly
filtered iterated cone~\eqref{eq:filtered-cob-it-cone-W}. We apply
this theorem with $X=L_0$ and $\rho_i=0$. We obtain a new weakly
filtered module $\mathcal{M}$ such that $\mathcal{M}(L_0)$ has a
differential $\mu_1^{\mathcal{M}}$ as described in that theorem
together with a filtered {\em chain isomorphism}
$\sigma_1: \mathscr{C}_{p,h} \longrightarrow \mathcal{M}(L_0)$. An
inspection of the sizes of shifts and discrepancies of the various
maps involved in Theorem~\ref{t:itcones} show that there exists a
constant $s^{\sigma}(p,h) \in \on$ such that $\sigma_1$ shifts
filtration by $\leq s^{\sigma}(p,h)$.  Additionally,
Theorem~\ref{t:itcones} implies that $CF(L_0,L_0; p)$ is also a
filtered subcomplex of $\mathcal{M}(L_0)$ and that
$\textnormal{pr}_0 \circ \sigma_1$ maps
$CF(L_0,L_0; p) \subset \mathcal{C}_{P,h}$ to
$CF(L_0,L_0; p) \subset \mathcal{M}(L_0)$ via the identity map:
$(\textnormal{pr}_0 \circ \sigma_1)|_{CF(L_0,L_0; p)} = \id$. Here
$\textnormal{pr}_0: \mathcal{M}(L_0) \longrightarrow CF(L_0,L_0;p)$ is
the projection onto the $0$'th factor of $\mathcal{M}(L_0)$.

Consider now the homology unit $e_{L_0} \in CF(L_0,L_0; p)$ as
constructed in~\eqref{eq:e_L}. By standard Floer theory $e_{L_0} = q$
(recall that $q$ is the unique maximum of
$f:L_0 \longrightarrow \mathbb{R}$).

Let $c \in CF(L_0,L_0; p)$ and $\gamma \in \mathcal{O}(H^{L_0,L_0})$ a
generator, where $H^{L_0.L_0}$ is the Hamiltonian function of the
Floer datum specified by $p$ for $(L_0,L_0)$. We denote by
$\langle c, \gamma \rangle \in \Lambda$ the coefficient of $\gamma$
when writing $c$ as a linear combination of elements of
$\mathcal{O}(H^{L_0,L_0})$ with coefficient in $\Lambda$.
  
We will need the following Lemma.
\begin{lem} \label{l:unit-we} For every chain $c \in CF(L_0,L_0; p)$
  we have $\langle \mu_1(c), q \rangle = 0$.
\end{lem}
We postpone the proof of the lemma and continue with the proof of
Theorem~\ref{thm:fission}.

Put $C_f := \max_{x \in L_0} |f(x)|$,
$C^{(1)}(p,h) := C_f + s^{\sigma}(p,h)$.  By~\eqref{eq:Bh-C} there
exists 
\begin{equation}\label{eq:boundary-fund}
b' \in \mathscr{C}_{p,h}\ \textnormal{with}\ 
A(b'; \mathscr{C}_{p,h}) \leq A(e_{L_0};\mathscr{C}_{p,h}) +
\mathcal{S}(W) + \frac{\epsilon}{2} \leq C_f + \mathcal{S}(W) +
\frac{\epsilon}{2}\ ,
\end{equation} such that
$e_{L_0} = \mu^{\mathscr{C}_{p,h}}_1(b')$.

Recall from point~\eqref{pp:Delta_j-Delta_0} of
Theorem~\ref{t:itcones} that
$\textnormal{pr}_0 \circ \sigma_1|_{CF(L_0,L_0;p)} = \id$. Set
$b:=\sigma_1(b')$ and apply $\textnormal{pr}_0 \circ \sigma_1$ to the
equality $e_{L_0} = \mu^{\mathscr{C}_{p,h}}_1(b')$. We obtain:
\begin{equation} \label{eq:mu-1-b-A} e_{L_0} = \textnormal{pr}_0 \circ
  \mu_1^{\mathcal{M}}(b), \quad A(b; \mathcal{M}(L_0)) \leq
  C^{(1)}(p,h) + \mathcal{S}(W) + \epsilon/2,
\end{equation}
where $C^{(1)}(p,h) := C_f + s^{\sigma}(p,h)$. Obviously
$C^{(1)}(p,h) \in \on$. (Note that $f \longrightarrow 0$ as
$p \longrightarrow p_0 \in \mathcal{N}$.)

Using the splitting~\eqref{eq:M-direct-sum} write
$b = b_0 + \cdots + b_k$, with $b_i \in CF(L_0, L_i; p)$ and
$$A(b_i;CF(L_0, L_i; p)) \leq C^{(2)}(p,h) + \mathcal{S}(W) + \epsilon/2,$$ 
where $C^{(2)}(p,h)$ is a new constant such that
$\lim C^{(2)}(p,h) \in \on$.

Continuing to apply Theorem~\ref{t:itcones} we have:
\begin{equation} \label{eq:q-equals-sum} q = \textnormal{pr}_0 \circ
  \mu_1^{\mathcal{M}}(b) = \sum_{j=0}^k a_{0,j}(b_j) =
  \mu_1^{CF(L_0,L_0;p)}(b_0) + \sum_{j=1}^k a_{0,j}(b_j),
\end{equation}
where the operators $a_{i,j}$ are the entries of the matrix
representation of $\mu_1^{\mathcal{M}}$ with respect to the
splitting~\eqref{eq:M-direct-sum}. By Lemma~\ref{l:unit-we},
$\langle \mu_1(b_0), q \rangle = 0$, hence there exists
$1 \leq j_0 \leq k$ such that
\begin{equation} \label{eq:nu-b-j0} \langle a_{0,j_0}(b_{j_0}),
  q\rangle \neq 0, \quad \nu \bigl(\langle a_{0,j_0}(b_{j_0}),
  q\rangle \bigr) \leq \nu(1) = 0.
\end{equation}
Here $\nu$ is the standard valuation of $\Lambda$
(see~\eqref{eq:val-Nov}).

By Theorem~\ref{t:itcones} there exists chains
$c_{i',i''} \in CF(L_{i'}, L_{i''};p)$, for all $i' < i''$, with
$A(c_{i',i''}) \leq C^{(3)}(p,h)$, where $C^{(3)}(p,h) \in \on$ and
such that
$$a_{0,j_0}(b_{j_0}) = \sum_{2 \leq d, \, \underline{i}}
\mu_d^{\fuk(\mathcal{C};p)}(b_{j_0}, c_{i_d, i_{d-1}}, \ldots, c_{i_2,
  i_1}),$$ where $\underline{i} = (i_1, \ldots, i_d)$ runs over all
partitions $0=i_1 < i_2 \cdots < i_{d-1} < i_d = j_0$.

It follows that there exists a partition
$\underline{i}^0 = (i^0_1, \ldots, i^0_d)$ with $d \geq 2$, for which
$$\big \langle \mu_d^{\fuk(\mathcal{C};p)}(b_{j_0}, c_{i^0_d, i^0_{d-1}}, 
\ldots, c_{i^0_2, i^0_1}), q \big \rangle \neq 0, \quad \nu \Big( \big
\langle \mu_d^{\fuk(\mathcal{C};p)}(b_{j_0}, c_{i^0_d, i^0_{d-1}},
\ldots, c_{i^0_2, i^0_1}), q \big \rangle \Big) \leq 0.$$ Writing
$b_{j_0}$ as a linear combination (over $\Lambda$) of elements from
$L_0 \cap L_{j_0}$ and similarly for the $c_{i^0_r, i^0_{r-1}}$'s, we
deduce that there exist $x \in L_0 \cap L_{j_0}$, $P(T) \in \Lambda$,
and $z_{r} \in L_{i^0_r} \cap L_{i^0_{r-1}}$, $Q_r \in \Lambda$ for
$r=2, \ldots, d$, such that:
\begin{equation*} \label{eq:P-Q_r}
  \begin{aligned}
    & A(P(T)x) \leq C^{(2)}(p,h) + \mathcal{S}(W) + \epsilon/2, \\
    & A(Q_r(T)z_r) \leq C^{(3)}(p,h), \; \forall \, 2 \leq r \leq d, \\
    & \big \langle \mu_d^{\fuk(\mathcal{C};p)}(P(T)x, Q_d(T)z_d,
    \ldots, Q_2(T)z_2), q \big \rangle \neq 0, \\
    \nu \Big( & \big \langle \mu_d^{\fuk(\mathcal{C};p)}(P(T)x,
    Q_d(T)z_d, \ldots, Q_2(T)z_2), q \big \rangle \Big) \leq 0.
  \end{aligned}
\end{equation*}
Note that $d$, as well as the points $x$, $z_d, \ldots, z_2$, $q$, all
depend on $(p,h)$, but for the moment we suppress this from the
notation.

Since $A(P(T)x) = -\nu(P(T))$ and $A(Q_r(T)z_r)=-\nu(Q_r(T))$ we
obtain:
\begin{equation} \label{eq:nu-mu-d} \nu \Big( \big \langle
  \mu_d^{\fuk(\mathcal{C};p)}(x, z_d, \ldots, z_2), q \big \rangle
  \Big) \leq \mathcal{S}(W) + \epsilon/2 + C^{(4)}(p,h),
\end{equation}
where $C^{(4)}(p,h) \in \on$.
  
Denote by $\mathscr{D}(p) = (K(p), J(p))$ the perturbation datum
prescribed by $p \in E''_{\textnormal{reg}}$ for the tuple of
Lagrangians $(L_0, L_{j_0}, L_{i_{d-1}}, \ldots, L_{i_2}, L_0)$. It
follows from~\eqref{eq:nu-mu-d} that there exists a non-constant Floer
polygon $u \in \mathcal{M}(x, z_d, \ldots, z_2, q; \mathscr{D}(p))$
with
\begin{equation*} \label{eq:om-u-S} \omega(u) \leq \mathcal{S}(W) +
  \epsilon/2 + C^{(4)}(p,h).
\end{equation*}
  
Let $p_0 \in \mathcal{N}$ be any choice of perturbation data which
assigns to the tuple of Lagrangians
$(L_0, L_{j_0}, L_{i_{d-1}}, \ldots, L_{i_2}, L_0)$ the perturbation
data $\mathscr{D}(p_0) = (K=0, J(p_0))$, where $J(p_0)$ is a family of
almost complex structures that coincide with $J^B$ on $B$
(see~\eqref{eq:JB}).

Fix a generic $C^1$-small Morse function $f$ as on
page~\pageref{pp:function-f}. We now choose a sequence
$\{(p_n, h_n)\}$ in $E''_{\textnormal{reg}}$ with
$(p_n, h_n) \longrightarrow (p_0,0)$ as $n \longrightarrow \infty$,
and with the following additional property. The Hamiltonian function
$H^{L_0,L_0}(n)$ prescribed by $p_n$ for the Floer datum
$\mathscr{D}_{L_0,L_0}(p_n)$ of $(L_0,L_0)$ is
$H^{L_0,L_0}_{\frac{1}{n}f}$, i.e.~constructed as on
page~\pageref{pp:function-f} but with the function $\frac{1}{n}f$
instead of $f$. Consequently, the point $q$ (the maximum of
$\frac{1}{n}f$) does not depend on $n$.

Passing to a subsequence of $\{(p_n, h_n)\}$ if necessary we may
assume that both $d$ as well as the points $x, z_d, \ldots, z_2$ above
do not depend on $n$ either. (Note that by Theorem~\ref{t:itcones},
$d \leq k$, so there are only finitely many possible values for $d$.)

In summary, we obtain a sequence,
$u_n \in \mathcal{M}(x, z_d, \ldots, z_2, q; \mathscr{D}(p_n))$ with
$\omega(u_n) \leq \mathcal{S}(W) + \epsilon/2 + C^{(4)}(p_n,h_n)$. By
a compactness result~\cite{Oh-Zhu:thick-thin, Oh-Zhu:floer-traj} (see
also~\cite{Fuk-Oh:zero-loop, Oh:relative, Oh:spectral}) there exists a
subsequence of $\{u_n\}$ which converges to a union of Floer polygons
$v_0, v_1, \ldots, v_l$, $l \geq 0$, together with a (possibly broken)
negative gradient trajectory $\eta$ of $f$. (``Broken'' means that the
trajectory might pass through several critical points of $f$.)

The Floer polygons $v_i$ map the boundary components of their domains
of definition to some of the Lagrangians in the collection
$L_0, L_{j_0}, L_{i_{d-1}}, \ldots, L_{i_2}, L_0$. Moreover, $v_0$
maps one of its boundary components to $L_0$. The maps $v_i$ satisfy
the Floer equation corresponding to the perturbation data prescribed
by $p_0$. Consequently they are all genuine pseudo-holomorphic
(i.e.~without Hamiltonian perturbations) with respect to the
(domain-dependent) almost complex structures prescribed by $p_0$. In
particular, $\omega(v_i) \geq 0$ for every $i$.

Since
$\omega(u_n) \leq \mathcal{S}(W) + \epsilon/2 + C^{(4)}(p_n,h_n)$ for
every $n$, it follows that
$\sum_{i=0}^l \omega(v_i) \leq \mathcal{S}(W) + \epsilon/2$, hence
\begin{equation} \label{eq:om-v_0} \omega(v_0) \leq \mathcal{S}(W) +
  \epsilon/2.
\end{equation}

The other part of the limit of $\{u_n\}$, namely the negative gradient
trajectory $\eta$ of $f$, emanates from an $L_0$-boundary point of one
of the polygons, say $v_0$, and ends at the point $q$.

Consider now $v_0$ and $\eta$. Note that $\eta$ must be the constant
trajectory at the point $q$ since it goes into $q$ which is a maximum
of $f$. It follows that the polygon $v_0$ passes (along its boundary)
through the point $q$.

We now appeal to the special form of $J(p_0)$ over the ball
$B$. Recall that $v_0$ is $J(p_0)$-holomorphic. Thus restricting $v_0$
to the subdomain (of its definition) which is mapped to
$\textnormal{Int\,}(B)$ we obtain a {\em proper} $J^B$-holomorphic
curve $v_0'$ parametrized by a non-compact Riemann surface with one
boundary component. Moreover that boundary component is mapped to
$B \cap L_0$, and $q \in B$ which is the center of the ball is in the
image of that boundary component. Passing to the standard ball $B(r)$
via the symplectic embedding $e$ mentioned in~\eqref{eq:JB} we obtain
from $v'_0$ a proper $J_{\textnormal{std}}$-holomorphic curve $v''_0$
inside $B(r)$ which passes through $0$ and its boundary is mapped to
$B_{\mathbb{R}}(r) \subset B^{2n}(r)$. Applying a reflection along
$\mathbb{R}^n \times 0$ to $v''_0$, and gluing the result to $v''_0$
we obtain a proper $J_{\textnormal{std}}$-holomorphic curve (without
boundary) $\widetilde{v}''_0$ in $\textnormal{Int\,} B^{2n}(r)$ which
passes through $0$. By the Lelong inequality we have
$\pi r^2 \leq \omega_{\textnormal{std}}(\widetilde{v}''_0)$. Putting
everything together we obtain:
$$\delta(L_0,S) - \epsilon \leq 
\pi r^2 \leq \omega(\widetilde{v}''_0)= 2\omega(v''_0) \leq
2\omega(v_0) \leq 2\mathcal{S}(W) + \epsilon.$$ Since this
inequality holds for all $\epsilon>0$ the desired
inequality~\eqref{eq:SV-delta} follows.

\medskip

We now turn to the proof of Lemma stated earlier in the proof.
\begin{proof}[Proof of Lemma~\ref{l:unit-we}]
  Recall that the Hamiltonian function in the Floer data
  $\mathscr{D}_{L_0,L_0}$ of $(L_0,L_0)$ is $H_f^{L_0,L_0}$ and we
  have $\mathcal{O}(H_f^{L_0,L_0}) = \textnormal{Crit}(f)$.

  Let $u: \mathbb{R} \times [0,1] \longrightarrow M$ be a Floer strip
  connecting $x_{-}$ to $x_{+}$, where
  $x_{\pm} \in \textnormal{Crit}(f)$. Identifying
  $(D \setminus \{-1,+1\}, \partial D \setminus \{-1,+1\})$ with
  $(\mathbb{R} \times [0,1], \mathbb{R} \times \{0\} \cup \mathbb{R}
  \times \{1\})$ we obtain from $u$ a map
  $u': (D \setminus \{-1,+1\}, \partial D \setminus \{-1,+1\})
  \longrightarrow (M,L_0)$ that extends continuously to a map
  $\overline{u}': (D, \partial D) \longrightarrow (M,L_0)$.  Since
  $L_0$ is weakly exact we have $\omega(\overline{u}') = 0$, hence
  $\omega(u) = 0$.

  By~\eqref{eq:E-omega-strips} it follows that
  $f(x_{-}) = f(x_{+}) + E(u)$, where $E(u)$ is the energy of $u$
  (see~\eqref{eq:floer-eq-1}). As $E(u)\geq 0$ we have
  $f(x_{-}) \geq f(x_{+})$ with equality iff $E(u)=0$.

  Suppose by contradiction that $\langle \mu_1(x), q \rangle \neq 0$
  for some $x \in \textnormal{Crit}(f)$. Let $u$ be a Floer strip that
  contributes to $\mu_1(x)$ and connects $x$ to $q$. By the above, we
  have $f(x) \geq f(q)$. Since $q$ is the unique maximum of $f$ it
  follows that $x = q$. Moreover, $E(u) = 0$. The latter implies that
  $\partial_su \equiv 0$. But this can happen only if $u$ is the
  constant strip at $q$ which contradicts the fact that $u$
  contributes to $\mu_1(x)$.
\end{proof}

To complete the proof of~\eqref{eq:SV-delta} it remains only to treat
the case when the Lagrangians $L_0, L_1, \ldots, L_k$ do not intersect
pairwise transversely.

Fix $\epsilon>0$. We apply $k$ Hamiltonian isotopies, one to each
Lagrangian $L_i$, $1\leq i \leq k$, such that the following holds:
\begin{enumerate}
\item The images $L'_1, \ldots, L'_k$ of $L_1, \ldots, L_k$ after
  these isotopies are such that $L_0, L'_1, \ldots, L'_k$ intersect
  pairwise transversely.
\item The Hofer length of each of these isotopies is
  $\leq \epsilon/k$.
\item $\delta(L_0; S) - \epsilon \leq \delta(L_0; S')$, where
  $S' = L'_1 \cup \ldots \cup L'_k$.
\end{enumerate}
Let $V: L_0 \leadsto (L_1, \ldots, L_k)$ be a weakly exact
cobordism. We now glue to each of the negative ends $L_i$ of $V$ the
Lagrangian suspension associated to the preceding Hamiltonian isotopy
used to move $L_i$ to $L'_i$. The result is a new cobordism
$V':L_0 \leadsto (L'_1, \ldots, L'_k)$ whose shadow satisfies:
$\mathcal{S}(V') \leq \mathcal{S}(V) + \epsilon$.

Since the ends of $V'$ intersect pairwise transversely it follows from
what we have already proved that
$\mathcal{S}(V') \geq \frac{1}{2}\delta(L_0; S')$. Therefore:
$$\tfrac{1}{2}\delta(L_0; S) - \tfrac{\epsilon}{2} \leq 
\tfrac{1}{2}\delta(L_0; S') \leq \mathcal{S}(V') \leq \mathcal{S}(V) +
\epsilon.$$ As this holds for all $\epsilon>0$ the result readily
follows.

This completes the proof of~\eqref{eq:SV-delta}.

\medskip

We now turn to the proofs of the other two statements of the theorem.
\paragraph{{\sl Proof of statement~\eqref{i:ch-inters}}}
As in the previous part of the proof, we first assume that the
Lagrangians $L_1, \ldots, L_k$ intersect pairwise transversely.

Fix $\epsilon>0$ small enough such that
\begin{equation} \label{eq:SV-eps} \mathcal{S}(V) + \epsilon <
  \tfrac{1}{2}\delta^{\Sigma'}(N \cup S) - \tfrac{1}{2}\epsilon.
\end{equation}
Fix also $r>0$ with
\begin{equation} \label{eq:delta-sigma-p-N} \delta^{\Sigma'}(N\cup S)
  - \epsilon \leq \pi r^2 < \delta^{\Sigma'}(N \cup S).
\end{equation}
Write $\Sigma' = N \cap S = \{x_1, \ldots, x_m \}$ for the double
points of $N \cup S$, and let $e_{x_i}: B(r) \longrightarrow M$,
$i=1, \ldots, m$, be a collection of symplectic embeddings with the
properties as in the definition of $\delta^{\Sigma'}$ on
page~\pageref{eq:delta-sigma} (we take $\mathbb{L} = N \cup S$ and
$Q = \emptyset$ in that definition). Denote
$B := \cup_{i=1}^m \textnormal{image\,}(e_{x_i})$ and let $J^B$ be the
complex structure on $B$ whose value on
$\textnormal{image\,}(e_{x_i})$ is the push forward
$(e_{x_i})_* (J_{\textnormal{std}})$ of the standard complex structure
$J_{\textnormal{std}}$ of $B(r)$ via the map $e_{x_i}$.

We consider now two curves $\gamma, \gamma'$ of the the same shape as
in the earlier part of the proof (see Figure~\ref{f:gamma-gamma'-2})
and such that (similarly to~\eqref{eq:l-varphi-t}) there exists a
Hamiltonian isotopy, horizontal at infinity,
$\varphi_t: \mathbb{R}^2 \longrightarrow \mathbb{R}^2$, $t \in [0,1]$,
with $\varphi_0 = \id$, $\varphi_1(\gamma) = \gamma'$ and with
\begin{equation} \label{eq:l-varphi-t-2} \text{length} \{\varphi_t\}
  \leq \mathcal{S}(V) + \epsilon/2.
\end{equation}

\begin{figure}[htbp]
   \begin{center}
     \includegraphics[scale=0.5]{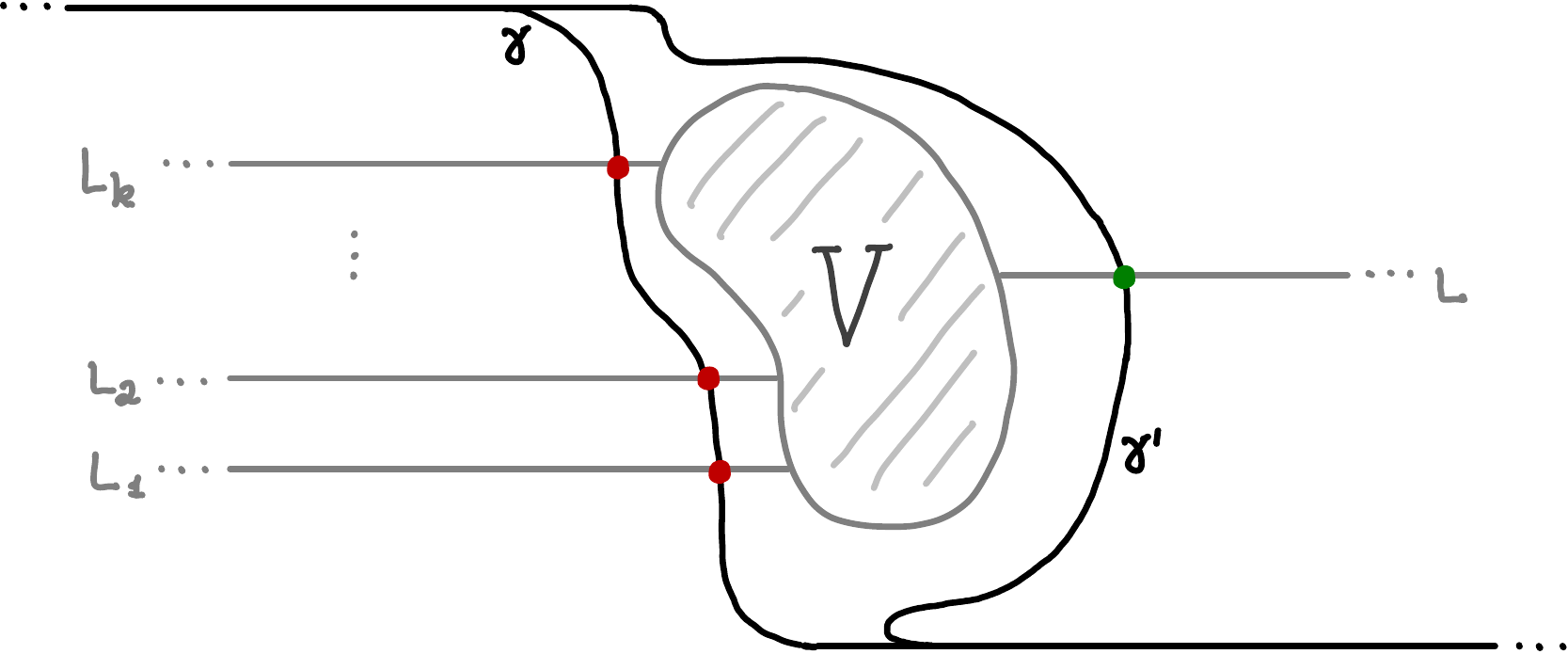}
   \end{center}
   \caption{The curves $\gamma$ and $\gamma'$ and the cobordism $V$.}
   \label{f:gamma-gamma'-2}
\end{figure}

Next we set up the Fukaya categories involved in the proof. Let
$\mathcal{C}$ be the collection of Lagrangians $L_1, \ldots, L_k,
L$. We will work with the Fukaya $\fuk(\mathcal{C};p)$ defined with
choices of perturbation data $p$ with the following restrictions. The
Floer data of $(N, L_i)$, prescribed by $p$, are of the type
$\mathscr{D}_{N, L_i} = (H^{N,L_i}=0, J(p))$, where $J(p)=\{J_t(p)\}$
is a family of almost complex structures such that $J_t(p)|_{B} = J^B$
for all $t$. The Floer data $\mathscr{D}_{L_i,L_j}$ $i \neq j$ have
the $0$ Hamiltonian function. Finally, the perturbation data
$\mathscr{D}_{N, L_{i_d}, \ldots, L_{i_1}}$, $d \geq 2$, all have
vanishing Hamiltonian form, i.e.~they are of the type $(K=0, J)$.  Due
to the assumption that $N, L_1, \ldots, L_k$ intersect pairwise
transversely, regular choices of perturbation data with the above
properties do exist. We denote the space of such regular choices by
$E''_{\textnormal{reg}}$. (It is important to note that the
restriction that $J_t(p)|_{B} = J^B$ for every $t$ does not pose any
regularity problem since every Floer strip or polygon relevant for the
definition of $\fuk(\mathcal{C}; p)$ cannot have its image lying
entirely inside $B$, and outside of $B$ we have not posed any
restrictions on the choice of almost complex structures.)

We set up the Fukaya categories
$\fukcob(\widetilde{\mathcal{C}}_{1/2}, \iota_{\gamma}(p,h))$ and
$\fukcob(\widetilde{\mathcal{C}}_{1/2}, \iota_{\gamma'}(p,h))$ and the
associated inclusion functors in the same way as in the previous part
of the proof.

Let $\mathcal{V}$, $\mathcal{V}'$ be the Yoneda modules corresponding
to $V$, one time viewed as an object
$V \in \textnormal{Ob}(\fukcob(\widetilde{\mathcal{C}}_{1/2};
\iota_{\gamma}(p,h)))$ and one time as
$V \in \textnormal{Ob}(\fukcob(\widetilde{\mathcal{C}}_{1/2};
\iota_{\gamma'}(p,h)))$. Consider the pull-back modules
$$\mathcal{M}_{V; \gamma, p, h} := \mathcal{I}^*_{\gamma; p, h}
\mathcal{V}, \quad \mathcal{M}_{V ; \gamma', p, h} :=
\mathcal{I}^*_{\gamma; p, h} \mathcal{V}'.$$

By Proposition~\ref{p:icones-M_j}
(and~\eqref{eq:filtered-cob-it-cone}) we have:
\begin{equation} \label{eq:filtered-cob-it-cone-V} S^{s_h}
  \mathcal{M}_{V; \gamma, p, h} \\ = \; \tcn (\mathcal{L}_k
  \xrightarrow{\; \overline{\phi}_k \;} \tcn(\mathcal{L}_{k-1}
  \xrightarrow{\; \overline{\phi}_{k-1} \;} \tcn( \cdots
  \tcn(\mathcal{L}_2 \xrightarrow{\; \overline{\phi}_2 \;}
  \mathcal{L}_1 )) {\cdot}{\cdot}{\cdot})),
\end{equation}
where $s_h \longrightarrow 0$ as $h \to 0$, and similarly to what we
have had on page~\pageref{eq:filtered-cob-it-cone-W},
$\overline{\phi}_i = (\phi_i, 0, \bm{\delta^{(i)}})$, with $\phi_i$ a
homomorphism of modules that shifts action by $\leq 0$ and has
discrepancy $\leq \bm{\delta}(p,h)$, where for every $d$ we have
$\delta_d (p,h) \in \on$. By similar arguments,
$S^{s'_h} \mathcal{M}_{V ; \gamma', p, h} = \mathcal{L}$, where
$\mathcal{L}$ is the Yoneda module of $L$, and
$s'_h \longrightarrow 0$ as $h \longrightarrow 0$.

Consider now the chain complexes
$$\mathcal{C}_{p,h} := \mathcal{M}_{V; \gamma, p, h}(N), \quad
\mathcal{C}'_{p,h} := \mathcal{M}_{V; \gamma', p, h}(N)$$ endowed with
the differential coming from the $A_{\infty}$-modules
$\mathcal{M}_{V; \gamma, p, h}$, $\mathcal{M}_{V; \gamma', p, h}$.
(Note that $\mathcal{C}_{p,h}$ defined here is different than the
$\mathcal{C}_{p,h}$ from page~\pageref{pp:C-ph}.)

By definition
$\mathcal{C}_{p,h} = CF(\gamma \times N, V; \mathscr{D}_{\gamma \times
  N, V})$, where $\mathscr{D}_{\gamma \times N, V}$ is the Floer datum
prescribed by $\iota_{\gamma}(p,h)$. Similarly
$\mathcal{C}'_{p,h} = CF(\gamma' \times N, V; \mathscr{D}_{\gamma'
  \times N, V})$, where $\mathscr{D}_{\gamma' \times N, V}$ is the
Floer datum prescribed by $\iota_{\gamma'}(p,h)$. Consider now the
Hamiltonian isotopy
$\widetilde{\varphi}_t := \varphi_t \times \id: \mathbb{R}^2 \times M
\longrightarrow \mathbb{R}^2 \times M$, $t \in [0,1]$, where
$\varphi_t$ is the Hamiltonian isotopy from
page~\pageref{eq:l-varphi-t-2}. Note that $\{\widetilde{\varphi}_t\}$
is horizontal at infinity and by~\eqref{eq:l-varphi-t-2} has Hofer
length $\leq \mathcal{S}(V) + \tfrac{1}{2}\epsilon$. Since
$\widetilde{\varphi}_1(\gamma \times N) = \gamma' \times N$, by
standard Floer theory (see e.g.~\cite[Chapter~5]{FO3:book-vol1}) this
isotopy induces two chain maps:
$$\phi: \mathcal{C}_{p,h} \longrightarrow \mathcal{C}'_{p,h}, 
\quad \psi: \mathcal{C}'_{p,h} \longrightarrow \mathcal{C}_{p,h},$$
which are both filtered and such that $\psi \circ \phi$ is chain
homotopic to $\id$ by a chain homotopy that shifts action by
$\leq \mathcal{S}(V) + \epsilon$. More specifically
$$\psi \circ \phi = id + K d^{\mathcal{C}_{p,h}} + 
d^{\mathcal{C}_{p,h}} K,$$ where $d^{\mathcal{C}_{p,h}}$ is the
differential of $\mathcal{C}_{p,h}$ and $K$ is a $\Lambda$-linear map
that shifts action by $\leq \mathcal{S}(V) + \epsilon$. Using the
formalism of~\eqref{eq:Bh-1} this means that
\begin{equation} \label{eq:Bh-psi-phi-id}
  B_h(\psi \circ \phi - \id) \leq \mathcal{S}(V) + \epsilon.
\end{equation}

We now appeal to Theorem~\ref{t:itcones}, by which we obtain the
following:
\begin{enumerate}
\item A chain complex $\mathcal{M}(N)$ whose underlying
  $\Lambda$-module coincides with $\mathcal{C}_{p,h}$ and whose
  differential $\mu_1^{\mathcal{M}(N)}$ is described by~\eqref{eq:aij}
  from Theorem~\ref{t:itcones}.
\item An isomorphism of chain complexes
  $\sigma_1: \mathcal{M}(N) \longrightarrow \mathcal{C}_{p,h}$ such
  that both $\sigma_1$ and its inverse
  $\sigma_1^{-1}:\mathcal{C}_{p,h} \longrightarrow \mathcal{M}(N)$
  shift action by $\leq C^{(1)}(p,h)$, where $C^{(1)}(p,h) \in \on$.
\end{enumerate}

We now estimate the action drop $\delta_{\mu_1^{\mathcal{M}(N)}}$ (as
defined in~\eqref{eq:delta-dC}) of the differential
$\mu_1^{\mathcal{M}(N)}$ of the chain complex $\mathcal{M}(N)$.
\label{pp:estim-delta_mu_1}
By Theorem~\ref{t:itcones} the differential $\mu_1^{\mathcal{M}(N)}$
comprises various $\mu_d$-operations, $1\leq d\leq k$, associated to
tuples of Lagrangians of the type
$(N, L_{i_d}, L_{i_{d-1}}, \ldots, L_{i_2}, L_{i_1})$, where
$i=i_1 < \cdots < i_d\leq j$, $1 \leq i\leq j \leq k$. Recall also
that the perturbation data $p \in E''_{\textnormal{reg}}$ were chosen
with vanishing Hamiltonian perturbation for tuple of Lagrangians as
above. Therefore, the above mentioned $\mu_d$-operations are defined
by counting (unperturbed) pseudo-holomorphic polygons $u$ with corners
mapped to intersection points between consecutive pairs of Lagrangians
in tuples as above. Each polygon $u$ contributing to these
$\mu_d$-operation has an intersection point in $N \cap L_j$ as one of
its inputs and an intersection point in $N \cap L_i$ as its
output. Moreover, these polygons are $J^B$-holomorphic over $B$.  We
thus obtain:
$$\omega(u) \geq \omega(\textnormal{image\,}(u) \cap B) \geq
\tfrac{1}{4} \pi r^2 + \tfrac{1}{4}\pi r^2 = \tfrac{1}{2}\pi r^2,$$
where the first inequality hold because $u$ is
unperturbed-pseudo-holomorphic over its entire domain, while the
second inequality follows from a Lelong-inequality type of argument
(see e.g.~\cite{Bar-Cor:Serre, Bar-Cor:NATO}). Combining the preceding
inequalities with~\eqref{eq:delta-sigma-p-N} we deduce that every
Floer polygon $u$ that participate in the calculation of the
differential $\mu_1^{\mathcal{M}(N)}$ must satisfy
$\omega(u) \geq \tfrac{1}{2}\delta^{\Sigma'}(N\cup S) -
\tfrac{\epsilon}{2}$. It follows that
\begin{equation} \label{eq:delta-mu-1-M}
  \delta_{\mu_1^{\mathcal{M}(N)}} \geq
  \tfrac{1}{2}\delta^{\Sigma'}(N\cup S) - \tfrac{\epsilon}{2}.
\end{equation}

In view of the map $\sigma_1$ and its inverse $\sigma_1^{-1}$,
mentioned earlier in the proof, we deduce the following estimate for
the action drop of the differential of $\mathcal{C}_{p,h}$:
$$\delta_{d^{\mathcal{C}_{p,h}}} \geq
\tfrac{1}{2}\delta^{\Sigma'}(N\cup S) - \tfrac{\epsilon}{2} -
2C^{(1)}(p,h).$$ As $C^{(1)}(p,h) \in \on$, by choosing
$p \in E''_{\textnormal{reg}}$ close enough to $\mathcal{N}$ and the
profile function $h$ small enough, we may assume in view
of~\eqref{eq:SV-eps} that
$\tfrac{1}{2}\delta^{\Sigma'}(N\cup S) - \tfrac{\epsilon}{2} -
2C^{(1)}(p,h) > \mathcal{S}(V) + \epsilon$. Combining the above
together with~\eqref{eq:Bh-psi-phi-id} we obtain:
$$\delta_{d^{\mathcal{C}_{p,h}}} > 
\mathcal{S}(V)+ \epsilon \geq B_h(\psi \circ \phi - \id).$$ By
Lemma~\ref{l:rig-cplx-n} (applied with $C = \mathcal{C}_{p,h}$,
$f = \psi \circ \phi$, $g = \id$) we deduce that $\psi \circ \phi$ is
injective. It follows that
$\phi: \mathcal{C}_{p,h} \longrightarrow \mathcal{C}'_{p,h}$ is
injective too, hence
$\dim_{\Lambda} \mathcal{C}_{p,h} \leq \dim_{\Lambda}
\mathcal{C}'_{p,h}$. But
$$\mathcal{C}_{p,h} = \bigoplus_{i=1}^k \bigoplus_{x \in N \cap L_i}
\Lambda \cdot x, \quad \mathcal{C}'_{p,h} = \bigoplus_{x \in N \cap
  L} \Lambda \cdot x,$$ which implies the desired
inequality~\eqref{eq:N-cap-L}. This completes the proof of
statement~\eqref{i:ch-inters} under the additional assumption that
$L_1, \ldots, L_k$ intersect pairwise transversely.

It remains to treat the case when the Lagrangians $L_1, \ldots, L_k$
do not necessarily intersect pairwise transversely.

Let $V$ be a cobordism as in the statement of the theorem. Fix $r>0$
and $\epsilon>0$ with
\begin{equation} \label{eq:perturb-a-eps}
  \mathcal{S}(V) + \epsilon < \pi r^2 < \delta^{\Sigma'}(N \cup S).
\end{equation}
Let $e_x: B(r) \longrightarrow M$, $x \in \Sigma' = N \cap S$, be a
collection of symplectic embeddings as in the definition of
$\delta^{\Sigma'}(N \cup S)$ on page~\pageref{eq:delta-sigma}.  Since
$\Sigma' = \cup_{i=1}^k (N \cap L_i)$ and the latter union is disjoint
every $x \in \Sigma'$ belongs to {\em precisely} one of the
Lagrangians $L_1, \ldots, L_k$. Now let $y \in \Sigma'$ and assume
that $y \in N \cap L_i$. Let $j \neq i$. It is easy to see from the
assumptions imposed on the embeddings $e_x$ in the definition of
$\delta^{\Sigma'}(N \cup S)$ that $L_j \cap e_y(B(r)) = \emptyset$. In
particular $L_j \cap L_i$ lies outside of $e_y(B(r))$. It follows that
$\cup_{i'<i''} (L_{i'} \cap L_{i''})$ lies outside of
$B := \cup_{x \in \Sigma'} \textnormal{image\,} e_x(B(r))$.

Next, apply a small Hamiltonian perturbation to each of the
Lagrangians $L_1, \ldots, L_k$, keeping them fixed inside $B$, so as
to obtain new Lagrangians $L'_1, \ldots, L'_k$ that intersect pairwise
transversely. By taking these perturbations small enough we may also
assume that no new intersection points between $S$ and $N$ have been
created, i.e.~$L'_i \cap N = L_i \cap N$ for every $i$.  Moreover, we
take these Hamiltonian perturbations to be small in the Hofer metric
so that the Hofer length of each of the isotopies generating the above
perturbations is $\leq \tfrac{\epsilon}{2k}$.

We now glue to each of the negative ends $L_i$ of $V$ the Lagrangian
suspension associated to the preceding Hamiltonian isotopy used to
move $L_i$ to $L'_i$. The result is a new cobordism
$V':L \leadsto (L'_1, \ldots, L'_k)$ with
$\mathcal{S}(V') \leq \mathcal{S}(V) + \frac{1}{2}\epsilon$. Combining
with~\eqref{eq:perturb-a-eps} we get:
$$\mathcal{S}(V') \leq \delta^{\Sigma'}(N \cup S').$$ 
As the ends $L'_i$ of $V'$ intersect pairwise transversely, by what we
have proved earlier we have:
$$\#(N\cap L) \geq \sum_{i=1}^{k}\# (N\cap L'_{i}).$$ Since
$N \cap L'_i = N \cap L_i$ for all $i$, the results follows and
completes the proof of statement~\eqref{i:ch-inters}.

\medskip
\paragraph{{\sl Proof of statement~\eqref{i:hf-inters}}}
The proof is similar to the proof of statement~\eqref{i:ch-inters}
above, only that now we use Proposition~\ref{prop:rig-cplx2} instead
of Lemma~\ref{l:rig-cplx-n} to estimate $\# (N\cap L)$
in~\eqref{eq:N-cap-L-HF}. Below we will mainly go over the points in
the proof that differ from the proof of statement~\eqref{i:ch-inters}.

Fix $\epsilon>0$ small enough and $r>0$ so that:
\begin{equation} \label{eq:SV-pr2-delta} \mathcal{S}(V) + \epsilon <
  \tfrac{1}{4}\delta^{\Sigma''}(S;N) - \tfrac{1}{4}\epsilon, \quad
  \delta^{\Sigma''}(S; N)-\epsilon \leq \pi r^2 <
  \delta^{\Sigma''}(S;N).
\end{equation}
Next, fix symplectic embeddings $e_x: B(r) \longrightarrow M$,
$x \in \Sigma''$, as in the definition of $\delta^{\Sigma''}(S;N)$ on
page~\pageref{eq:delta-sigma}. Fix also curves $\gamma$, $\gamma'$ as
in the proof of statement~\eqref{i:ch-inters}. We set up the Fukaya
categories $\fuk(\mathcal{C};p)$,
$\fukcob(\widetilde{\mathcal{C}}_{1/2}, \iota_{\gamma}(p,h))$,
$\fukcob(\widetilde{\mathcal{C}}_{1/2}, \iota_{\gamma'}(p,h))$ and the
inclusion functors $\mathcal{I}_{\gamma; p, h}$,
$\mathcal{I}_{\gamma'; p, h}$, in the same way as in the proof of
statement~\eqref{i:ch-inters}. We then define the chain complexes
$\mathcal{C}_{p,h}$, $\mathcal{C}'_{p,h}$, and the two chain maps
$\phi: \mathcal{C}_{p,h} \longrightarrow \mathcal{C}'_{p,h}$,
$\psi: \mathcal{C}'_{p,h} \longrightarrow \mathcal{C}_{p,h}$, with
\begin{equation} \label{eq:psi-circ-phi-2}
  \psi \circ \phi = \id + K \circ d^{\mathcal{C}_{p,h}} +
  d^{\mathcal{C}_{p,h}} \circ K,
\end{equation}
where $K$ shifts action by $\leq \mathcal{S}(V) + \epsilon$.

As before, we now use Theorem~\ref{t:itcones} and obtain a chain
complex $\mathcal{M}(N)$ whose underlying $\Lambda$-module coincides
with $\mathcal{C}_{p,h}$ and equals
\begin{equation} \label{eq:MN-split}
  \mathcal{M}(N) = \oplus_{i=1}^k CF(N,L_i; \mathscr{D}_{N,L_i}).
\end{equation}
By Theorem~\ref{t:itcones} the differential $\mu_1^{\mathcal{M}(n)}$
can be written with respect to the splitting~\eqref{eq:MN-split} as an
upper triangular matrix of operators $(a_{i,j})$ with diagonal
elements $a_{i,i} = \mu_1^{CF(N,L_i; \mathscr{D}_{N,L_i})}$. Write
$\mu_1^{\mathcal{M}(N)} = d_0 + d_1$, where
$d_0 = \oplus_{i=1}^k \mu_1^{CF(N,L_i; \mathscr{D}_{N,L_i})}$ with
respect to the splitting~\eqref{eq:MN-split} and
$d_1: \mathcal{M}(N) \longrightarrow \mathcal{M}(N)$ is the operator
represented by the part of the matrix $(a_{i,j})$ that lies strictly
above the diagonal.

The operator $d_1$ consists of sums of $\mu_d$-operations, $d \geq 2$,
where among the inputs of each such operation there is at least one
point from $L_i \cap L_j$, $i<j$. A similar argument to the one used
on page~\pageref{pp:estim-delta_mu_1} in estimating
$\delta_{\mu_1^{\mathcal{M}(N)}}$ in the proof of
statement~\eqref{i:ch-inters} shows that
$\delta_{d_1} \geq \tfrac{1}{4}\pi r^2$. Here, $\delta_{d_1}$ is the
action drop of $d_1$ (see~\S\ref{s:filt-ch},
page~\pageref{pp:delta-f}). Combining with~\eqref{eq:SV-pr2-delta} we
get:
\begin{equation} \label{eq:delta-d_1-ineq}
  \delta_{d_1} \geq \tfrac{1}{4}\delta^{\Sigma''}(S;N) -
  \tfrac{1}{4}\epsilon.
\end{equation}

Put
$f':= \psi \circ \phi : \mathcal{C}_{p,h} \longrightarrow
\mathcal{C}_{p,h}$. By~\eqref{eq:psi-circ-phi-2} we have
$B_h(f'-\id) \leq \mathcal{S}(V) + \epsilon$. Recall from
Theorem~\ref{t:itcones} the isomorphism of chain complexes
$\sigma_1: \mathcal{C}_{p,h} \longrightarrow \mathcal{M}(N)$ such that
both $\sigma_1$ and its inverse $\sigma_1^{-1}$ shift action by
$\leq C^{(1)}(p,h)$, where $C^{(1)}(p,h) \in \on$. Consider
$$f:= \sigma_1 \circ f' \circ \sigma_1^{-1} : \mathcal{M}(N)
\longrightarrow \mathcal{M}(N).$$ We would like to apply
Proposition~\ref{prop:rig-cplx2} to $C = \mathcal{M}(N)$, $d_0$, $d_1$
as defined above and the map $f$. We have:
$$f - \id = (\sigma_1 \circ K \circ \sigma_1^{-1}) \circ
d^{\mathcal{C}_{p,h}} + d^{\mathcal{C}_{p,h}} \circ (\sigma_1 \circ K
\circ \sigma_1^{-1}),$$ hence
$B_h(f - \id) \leq \mathcal{S}(V) + \epsilon + 2 C^{(1)}(p,h)$. As
$C^{(1)}(p,h) \in \on$, by taking $p$ close enough to $\mathcal{N}$
and the profile function $h$ small enough, we may assume in view
of~\eqref{eq:SV-pr2-delta} that
$$\mathcal{S}(V) + \epsilon + 2 C^{(1)}(p,h) <
\tfrac{1}{4}\delta^{\Sigma''}(S;N) - \tfrac{1}{4}\epsilon.$$

Together with~\eqref{eq:delta-d_1-ineq} we now obtain:
\begin{equation} \label{eq:Bh-f-id} B_h(f - \id) < \delta_{d_1}.
\end{equation}

In order to apply Proposition~\ref{prop:rig-cplx2} it remains to check
that
\begin{equation} \label{eq:dim-H-ineq}
  \dim_{\Lambda} H_*(\mathcal{M}(N), d_0) \geq \dim_{\Lambda}
  H_*(\mathcal{M}(N), \mu_1^{\mathcal{M}(N)}).
\end{equation}
This follows from standard results in homological algebra since
$$H_*(\mathcal{M}(N), d_0) = \oplus_{i=1}^k HF(N, L_i), \quad 
H_*(\mathcal{M}(N), \mu_1^{\mathcal{M}(N)}) \cong
H_*(\mathcal{C}_{p,h}, d^{\mathcal{C}_{P,h}})$$ and
$\mathcal{C}_{p,h}$ is an iterated cone of the type
\begin{equation*}
  \begin{aligned}
    \mathcal{C}_{p,h} = \tcn (CF(N,L_k)
    \to \tcn( & CF(N, L_{k-1}) \\
    & \to \tcn( \cdots \tcn(CF(N, L_2) \to CF(N,L_1)))
    {\cdot}{\cdot}{\cdot})).
  \end{aligned}
\end{equation*}
  
We are now in position to apply Proposition~\ref{prop:rig-cplx2}, by
which we obtain:
$$\dim_{\Lambda} (\textnormal{image\,}(f)) \geq \dim_{\Lambda}
H_*(\mathcal{M}(N), d_0) = \sum_{i=1}^k \dim_{\Lambda} HF(N,L_i).$$ On
the other hand:
$$\dim_{\Lambda} (\textnormal{image\,}(f)) \leq \dim_{\Lambda} \mathcal{M}(N)
= \sum_{i=1}^k \#(N \cap L_i).$$ Putting the last two inequalities
together yields~\eqref{eq:N-cap-L-HF} and concludes the proof of
statement~\eqref{i:hf-inters}.  \qed

\begin{rem}\label{rem:Misha-etc}
a. The intersection
 result at the point  (a) in Theorem \ref{thm:fission} implies variants of both inequalities (\ref{eq:SV-delta})  and
 (\ref{eq:N-cap-L-HF}) but with slightly different assumptions and for different constants $\delta$. 
 This is obvious concerning (\ref{eq:N-cap-L-HF}) and for inequality (\ref{eq:SV-delta}) it is seen by applying (\ref{eq:N-cap-L}) to the  cobordism $W:\emptyset\cobto (L,L_{1},\ldots, L_{k})$ obtained by bending the positive end of $V$ half way clockwise - as in Figure \ref{f:cob-v-w} - and taking $N$ a Hamiltonian
 perturbation of $L$.
 
 b. The following argument, due to Misha Khanevsky, leads to a more direct proof of the first part of Theorem \ref{thm:fission} but gives a weaker inequality.  We reproduce the argument here with Misha's permission.
 A result of Usher (Theorem 4.9 in \cite{Usher3}) claims that, given two Lagrangians $V$ and $V'$ that intersect transversely and non-trivially, there exists $\delta>0$ depending on
 $V$ and $V'$, such that the energy required to disjoin $V$ from $V'$ is greater than $\delta$.
 This result was proven for compact or (tame at infinity) symplectic manifolds but can be adjusted without any difficulty to the case of  Lagrangians with cylindrical ends in $\C\times M$. Assume that 
 $L\not\subset \cup_{i}L_{i}$ and let $T$ be a small Lagrangian torus, disjoint from all $L_{i}$'s, and 
 such that $T$ intersects $L$ transversally and non-trivially. Let $\gamma'$ be a curve as in Figure \ref{f:gamma-gamma'-2} and let $V'=\gamma'\times T$.  Thus $V'$ and $V$ intersect non-trivially
 and transversely (see also Figure \ref{f:cob-v-w}). The isotopy taking the curve $\gamma$ to the curve $\gamma'$  in Figure \ref{f:gamma-gamma'-2} disjoins $V'$ from $V$ and thus its energy, that can be assumed to be as close as needed to $\mathcal{S}(V)$, has to exceed $\delta$.
By inspecting Usher's proof \cite{Usher3}, the constant $\delta$ is smaller than the smallest area of a non-trivial  $J$-disk, sphere or strip with boundaries on $V$ and on $V'$. By standard arguments, the latter areas can be shown to only depend on $L$ and $T$ and this finishes Khanevsky's argument. However, notice that the
dependence on $T$ of the constant $\delta$ here  means that it is generally smaller than $\delta(L;S)$ from the statement of Theorem \ref{thm:fission}.  Moreover, this argument does not imply the points (a) and (b) of the statement and it also can not be adjusted to estimate the algebraic measurements that we will see later in the paper in Corollary \ref{cor:alg-weights}.
   \end{rem}

\subsection{Proof of
  Theorem~\ref{thm:fission-mon}} \label{sb:prf-thm-fission-mon} The
proofs of statements~\eqref{i:ch-inters} and~\eqref{i:hf-inters} given
in~\S\ref{sb:prf-thm-fission} carry over to the monotone case without
any modifications.

We now explain how to adjust the proof of~\eqref{eq:SV-delta} given
in~\S\ref{sb:prf-thm-fission} in order to
prove~\eqref{eq:SV-delta-mon}.

We may assume throughout the proof that $\mathcal{S}(V) < A_L$, for
otherwise the inequality~\eqref{eq:SV-delta-mon} is trivially
satisfied. We need to prove that
$\mathcal{S}(V) \geq \tfrac{1}{2}\delta(L;S)$.

We fix $\epsilon>0$ as in the proof of~\eqref{eq:SV-delta} but we
require additionally that 
\begin{equation} \label{eq:SV-AL-eps} \mathcal{S}(W) + \epsilon < A_L.
\end{equation}
The proof now goes along the same lines as the proof
of~\eqref{eq:SV-delta}, detailed in~\S\ref{sb:prf-thm-fission}, up to
the point where we had to use Lemma~\ref{l:unit-we} (see
page~\pageref{l:unit-we}). That lemma does not hold in the monotone
case, and we will now use the following lemma instead:
\begin{lem} \label{l:unit-mo} Let $c \in \textnormal{Crit}(f)$, viewed
  as an element of $\mathcal{O}(H^{L_0,L_0}_f)$. If
  $\langle \mu^{CF(L_0,L_0;p)}_1(c), q \rangle \neq 0$ then
  $\nu(\langle \mu_1^{CF(L_0,L_0;p)}(c), q\rangle) \geq A_{L_0}$.
\end{lem}
We postpone the proof for a while and continue with the proof of
Theorem~\ref{thm:fission-mon}. As in the proof of
Theorem~\ref{thm:fission-mon} we decompose the element $b$
from~\eqref{eq:mu-1-b-A} as $b=b_0 + \cdots + b_k$ with
$b_i \in CF(L_0,L_i; p)$. We cannot deduce, as earlier, that
$\langle \mu_1(b_0), q \rangle = 0$, however by Lemma~\ref{l:unit-mo}
and~\eqref{eq:SV-AL-eps} we still obtain that
$$\nu( \langle \mu_1(b_0), q\rangle) \geq A_{L_0} - \mathcal{S}(W) -
C^{(2)}(p,h) - \tfrac{1}{2}\epsilon > \tfrac{1}{2}\epsilon -
C^{(2)}(p,h),$$ where $C^{(2)}(p,h) \in \on$. By taking $p$ close
enough to $p_0 \in \mathcal{N}$ and $h$ small enough we may assume
that $\nu( \langle \mu_1(b_0), q\rangle) > 0$. In view
of~\eqref{eq:q-equals-sum} we can now deduce, as before, that there
exists $1 \leq j_0 \leq k$ such that~\eqref{eq:nu-b-j0} holds. From
this point on, the proof continues exactly as carried out in the
weakly exact case in~\S\ref{sb:prf-thm-fission}.

It remains to proof the preceding lemma.
\begin{proof}[Proof of Lemma~\ref{l:unit-mo}]
  Let $u \in \mathcal{M}(c,q; \mathscr{D}_{L_0,L_0})$ be a Floer strip
  that goes from $c$ to $q$ and contributes to
  $\mu^{CF(L_0,L_0;p)}_1(x)$. We need to show that
  $\omega(u) \geq A_{L_0}$.

  Indeed, as in the proof of Lemma~\ref{l:unit-we} on
  page~\pageref{l:unit-we}, after identifying
  $(\mathbb{R} \times [0,1], \mathbb{R} \times \{0\} \cup \mathbb{R}
  \times \{1\})$ with
  $(D\setminus \{-1,+1\}, \partial D \setminus \{-1,+1\})$ the map $u$
  extends continuously to a map
  $\overline{u}: (D, \partial D) \longrightarrow (M,L_0)$.  The
  dimension of the component of $u$ in the space
  $\mathcal{M}^*(c,q; \mathscr{D}_{L_0,L_0})$ of {\em
    non-parametrized} Floer trajectories connecting $c$ to $q$ is
  given by
  $$\dim \mathcal{M}^*_{u}(c,q; \mathscr{D}_{L_0,L_0}) =
  |c|-|q|-1+\mu([\overline{u}]) = |c|-n-1+\mu([\overline{u}]),$$ where
  $\mu$ is the Maslov index and $[\overline{u}] \in H_2^D(M,L_0)$ is
  the homology class induced by $\overline{u}$.  Since
  $\dim \mathcal{M}^*_{u}(c,q; \mathscr{D}_{L_0,L_0}) \geq 0$ we must
  have $\mu([\overline{u}]) \geq n+1-|c| > 0$. By monotonicity of
  $L_0$ we have $\omega([\overline{u}]) \geq A_{L_0}$, hence
  $\omega(u) \geq A_{L_0}$. This concludes the proof of the lemma.
\end{proof}

The proof of Theorem~\ref{thm:fission-mon} is now complete. 
\qed


\section{Metrics on Lagrangian spaces and examples}\label{sec:examples-metr}

This section gives some context to the phenomena reflected in Theorem
\ref{thm:fission} and discusses a number of applications and ramifications.

\subsection{Shadow metrics on $\mathcal{L}ag^{\ast}(M)$}\label{subsec:shad-metric}
Let $(M,\omega)$ be a symplectic manifold. Fix a class
$\mathcal{L}ag^{*}(M)$ of Lagrangian submanifolds of $M$, where $*$
that can be either $*=we$ or $(mon, \mathbf{d})$,
see~\ref{sb:monotone}. Fix a family
$\mathcal{F} \subset \mathcal{L}ag^{*}(M)$ of Lagrangian submanifolds
of $M$. Let $L$ and $L'$ be two other Lagrangians in
$\mathcal{L}ag^{*}(M)$. Theorem~\ref{thm:fission} suggests the
definition of the following two sequences of numbers. The definition of these
numbers has a geometric underpinning in that it is based on the existence of 
certain cobordisms.

First, for each $a > 0$,  define the {\em geometric $\epsilon$-cone-length} of $L'$ relative to $L$ (with
respect to $\mathcal{F}$) as
$$l^{\mathcal{F}}_{a}(L',L ) := \min \{ k \in \mathbb{N} \mid 
\exists \ V:L' \cobto (L_1,...,L_{s-1}, L, L_{s},...,L_k), L_i \in
\mathcal{F}, \; \mathcal{S}(V)\leq a\}~.~$$ Here, the minimum
is taken only over cobordisms
$V \in \mathcal{L}ag^{*}(\mathbb{R}^2 \times M)$, i.e. in the given
class $*$. It is important to note that we allow $V$ to be disconnected. 
We use the convention that the number $l^{\mathcal{F}}_{a}(L',L )$ equals $0$ 
if $L$ and $L'$ are related by a simple cobordism $V:L'\cobto L$ of shadow $\leq a$
(a cobordism with just two possibly non-void ends, one positive and one negative, is called simple).
We set $l^{\mathcal{F}}_{a}(L',L)=\infty$ if no
cobordism $V$ as above exists. We will omit $\mathcal{F}$ from the
notation when there is no risk of confusion.

It is clear that $l_{a}^{\mathcal{F}}(L', L)$ is non-increasing
in $a$ and symmetric with respect to $L,L'$. Next, define
$l^{\mathcal{F}}(L', L) := \lim_{a\to\infty}
l^{\mathcal{F}}_{a}(L',L)$ to be the {\em absolute} cone length
of $L'$ relative to $L$ and
$l^{\mathcal{F}}_{0}(L',L) :=\lim_{a\to 0}l_{a}(L',L)$.

In view of Theorem \ref{thm:fission} it is natural to also estimate
the minimal shadow required for splittings as in the definition of
$l_{a}^{\mathcal{F}}$ and thus define a second family of
natural measurements as follows. For every $k \in \mathbb{N}$ define:
\begin{equation}\label{eq:pseudo-metric}
  d^{\mathcal{F}}_{k}(L',L) := \inf\{\mathcal{S}(V) \mid 
  \exists\ V:L' \cobto (L_1,...,L_{s-1}, L, L_{s},...,L_r), \; L_i \in
  \mathcal{F}, \, r \leq k\}.
\end{equation}
Again, the infimum is taken only over cobordisms $V$ of class $*$ 
and we allow $V$ to be disconnected.  This is significant as, for instance,
if $\mathcal{F}$ contains a representative in each Hamiltonian isotopy class of the Lagrangians in $\mathcal{L}ag^{\ast}(M)$, then $d^{\mathcal{F}}_{2}(L',L)$ is finite for all $L,L'\in \mathcal{L}ag^{\ast}(M)$ (one can take $V$ as an appropriate union $V_{0}\cup V_{1}$ of two disjoint Lagrangian suspensions $V_{0}:L\to L_{1}$, $V_{1}:\emptyset \to (L',L'_{1} )$ with $L_{1}$ and $L'_{1}$ respectively
Hamiltonian isotopic to $L$ and to $L'$). We take $d^{\mathcal{F}}_{k}(L',L)=\infty$
if no cobordisms $V$ as in (\ref{eq:pseudo-metric}) exist.

We call $d^{\mathcal{F}}_{k}$ the {\em geometric $k$-splitting energy} of $L'$
relative to $L$ (with respect to $\mathcal{F}$).  Again,
$d^{\mathcal{F}}_{k}$ is symmetric in $(L, L')$, and
$d^{\mathcal{F}}_{k}(L',L)$ is non-increasing in $k$.  Note that
$d^{\mathcal{F}}_{0}$ is the ``shadow'' metric on elementary cobordism
equivalence classes as defined in~\cite{Co-She:metric}. Thus for a
Hamiltonian diffeomorphism $\phi$, we have
$$d^{\mathcal{F}}_{0}(\phi(L),L)\leq ||\phi||_{H}$$ 
(the latter stands
for the Hofer norm of $\phi$).

The following inequality is immediate:
$$d^{\mathcal{F}}_{k+k'}(L,L'')\leq d^{\mathcal{F}}_{k}(L,L')+
d^{\mathcal{F}}_{k'}(L',L'').$$ Obviously, we also have:
$d^{\mathcal{F}}_{l^{\mathcal{F}}_{a}(L',L)}(L',L) \leq
a$ and
$l^{\mathcal{F}}_{d^{\mathcal{F}}_{k}(L',L)}(L',L) \leq k$.

Finally, we define also the following measurement:
\begin{equation}\label{eq:pseudo-metric-gen}
  d^{\mathcal{F}}(L,L')=\lim_{k\to \infty} d_{k}^{\mathcal{F}}(L,L')=\inf_{k\geq 0}d_{k}^{\mathcal{F}}(L,L')~.~
\end{equation}
From the above it follows that $d^{\mathcal{F}}(-,-)$ is a
pseudo-metric.  By definition, $d^{\mathcal{F}}(L,L')$
 is infinite only if there are no cobordisms relating $L$ to $L'$ and with all the other ends in $\mathcal{F}$.
 
 The first point of Theorem \ref{thm:fission} (Theorem \ref{t:main-A}) implies:
\begin{cor}\label{cor:vanishing}
  If $d^{\mathcal{F}}(L',L)=0$, then
  $L\subset L'\cup \overline{\cup_{K\in\mathcal{F}}K}$ and
    $L' \subset L \cup \overline{\cup_{K\in\mathcal{F}}K}$.
\end{cor}
\begin{proof}
  Indeed, if the first inclusion in this statement does not hold, then
  $\delta(L; L'\cup \overline{\cup_{K\in\mathcal{F}}K})>0$ and as a
  consequence of Theorem~\ref{thm:fission} it follows that any
  cobordism having $L$ as its positive end, with one negative end
  equal to $L'$ and with the other negative ends from $\mathcal{F}$,
  is of shadow at least
  $\delta(L; L'\cup \overline{\cup_{K\in\mathcal{F}}K})/2$ which means
  that $d^{\mathcal{F}}(L',L)$ can not vanish.

  In case the other inclusion from the statement does not hold, the
  proof is similar.
\end{proof}

It is easy to see (Remark \ref{rem:pseudo} below) that the pseudo-metric
$d^{\mathcal{F}}$ given by (\ref{eq:pseudo-metric-gen}) is in general
degenerate. However, there is a simple
way to adjust it so as to obtain a genuine
metric (possibly infinite). 
  
\begin{cor} \label{cor:deblurring} Let $\mathcal{F}$ and
  $\mathcal{F}'$ be two families of Lagrangians in $\mathcal{L}ag^{\ast}(M)$ so
  that the intersection
$(\overline{\cup_{K\in\mathcal{F}}K}) \cap
    (\overline{\cup_{K'\in \mathcal{F}'}K'})$ is totally disconnected.
     The pseudo-metric on
  $\mathcal{L}ag^{\ast}(M)$ defined by:
  $$\widehat{d}^{\mathcal{F},\mathcal{F}'}=\max\{d^{\mathcal{F}}, d^{\mathcal{F}'}\}$$
  is non-degenerate.
\end{cor}  
\begin{proof} If $\widehat{d}^{\mathcal{F},\mathcal{F}'}(L, L')=0$ we
  deduce from Corollary \ref{cor:vanishing} that 
    $L\subset L'\cup \overline{\cup_{K\in \mathcal{F}}K}$ and
    $L\subset L'\cup \overline{\cup_{K'\in \mathcal{F}'}K'}$. Assume
  that there is a point $x\in L$ so that $x\not \in L'$.  Then there
  is an open disk $D\subset L$ so that $D\cap L'=\emptyset$. It
  follows that $D\subset \overline{\cup_{K\in \mathcal{F}}K}$
    as well as $D\subset \overline{\cup_{K'\in \mathcal{F}'}K'}$
  which is not possible because the set
  $(\cup_{K\in\mathcal{F}}K) \cap (\cup_{K'\in \mathcal{F}'}K')$ is totally disconnected). 
  We
  conclude that $L\subset L'$. The roles of $L$ and $L'$ being
  symmetric, we deduce that $L=L'$.
\end{proof}
 
 Notice that if $L'=\phi(L)$ with $\phi$ a Hamiltonian isotopy, then 
 $\widehat{d}^{\mathcal{F},\mathcal{F}'}(L,L')\leq ||\phi||_{H}$. 
  
Given a family $\mathcal{F}$ that is finite (but this can also work in
more general instances) it is easy to produce an additional family
$\mathcal{F}'$ that satisfies the assumption of Corollary
\ref{cor:deblurring}.  This can be achieved, for instance, by
transporting each element of $\mathcal{F}$ by an appropriate
Hamiltonian isotopy.  

We will not analyze here in detail the properties of the metrics $\widehat{d}^{\mathcal{F},\mathcal{F}'}$ but there are two simple observations that we include.

\begin{cor}\label{cor: metrics-var} Let $\phi$ be a Hamiltonian diffeomorphism and let $|| - ||_{H}$ be the Hofer norm. We have the following two inequalities:
\begin{equation}\label{eq:quasi-isom}
|\widehat{d}^{\mathcal{F},\mathcal{F}'}(L,L')-\widehat{d}^{\mathcal{F},\mathcal{F}'}(\phi(L),\phi(L'))|\leq  2 ||\phi||_{H}
\end{equation} 
and
\begin{equation}\label{eq:quasi-isom2}
|\widehat{d}^{\phi(\mathcal{F}),\phi(\mathcal{F}')}(L,L')-\widehat{d}^{\mathcal{F},\mathcal{F}'}(L,L')|\leq  2 ||\phi||_{H}~.~
\end{equation}
Therefore, the group of Hamiltonian diffeomorphisms acts by quasi-isometries on the metric space
$(\mathcal{L}ag^{\ast}(M), \widehat{d}^{\mathcal{F},\mathcal{F}'})$. Moreover, the identity
is a quasi-isometry between the two metric spaces $$(\mathcal{L}ag^{\ast}(M), \widehat{d}^{\phi(\mathcal{F}),\phi(\mathcal{F}')}) \ ,  (\mathcal{L}ag^{\ast}(M), \widehat{d}^{\mathcal{F},\mathcal{F}'})~.~$$
\end{cor}
\begin{proof}
 A cobordism $V:L\cobto (F_{1},\ldots, L',\ldots, F_{k})$ can be extended, by gluing appropriate Lagrangian suspensions to the ends $L$ and $L'$, to a cobordism 
$V':\phi(L)\cobto (F_{1},\ldots, \phi(L'),\ldots, F_{k})$ of shadow $\mathcal{S}(V')\leq \mathcal{S}(V)+2 ||\phi||_{H}$.  The first inequality in the statement then follows rapidly, by applying the same argument to $\phi^{-1}$. Similarly, to deduce the second inequality, consider $V:L\cobto (F_{1},\ldots, L',\ldots, F_{k})$. By applying $\phi$ to $V$ we get $\phi(V):\phi(L)\cobto (\phi(F_{1}),\ldots, \phi(L'),\ldots,\phi(F_{k}))$.  Extend both ends $\phi(L)$ and $\phi(L')$ by Lagrangian suspensions thus getting $V''\cobto  (\phi(F_{1}),\ldots, L',\ldots,\phi(F_{k}))$ of shadow bounded by $\mathcal{S}(V)+2||\phi||_{H}$ and the desired inequality follows easily.
\end{proof}

The construction of the metrics $\widehat{d}^{\mathcal{F},\mathcal{F}'}$ is very flexible and can be 
modified in a variety of ways.  For instance, one can use more than two families $\mathcal{F}$.
Here is such an example. Let $\mathcal{U}=\{U_{i}\}_{i\in I}$ be a family of open sets $U_{i}\subset M$
and let $\mathcal{F}_{i}=\{ L\in \mathcal{L}ag^{\ast}(M) \mid L\cap U_{i}=\emptyset\}$. 
For each
index $i\in I$ we then have a shadow pseudo-metric $d^{\mathcal{F}_{i}}$ and we define a new pseudo-metric: $$D^{\mathcal{U}}=\sup \{ d^{\mathcal{F}_{i}} \mid i\in I\}~.~$$
In the statement below we will make use of the following quantity. For $L\in \mathcal{L}ag^{\ast}(M)$
let $\Delta(L;\mathcal{U})=\inf\{ s \mid  \ \forall\  i\in I\ \  \exists 
\ \phi \ \textnormal{Hamiltonian diffeomorphism with}\ \phi(L)\cap U_{i}=\emptyset, ||\phi||_{H}\leq s\}$.

\begin{cor}\label{cor:fract-sets} With the notation above we have:
\begin{itemize}
\item[i.] If $\mathcal{U}$ is a covering of $M$ in the sense that $\cup_{i} U_{i}=M$, then $D^{\mathcal{U}}$ is non-degenerate.
\item[ii.] For all $L,L'\in\mathcal{L}ag^{\ast}(M)$ such that $\Delta(L;\mathcal{U})$ and $
\Delta (L';\mathcal{U})$ are finite, 
we have:  $$D^{\mathcal{U}}(L,L')\leq \Delta(L;\mathcal{U})+\Delta (L';\mathcal{U})~.~$$
\end{itemize}
\end{cor}

\begin{proof} The first point follows immediately from Corollary \ref{cor:vanishing}.
For the second point fix some $s > \Delta(L;\mathcal{U})$, $s'> \Delta(L';\mathcal{U})$ and pick one
family $\mathcal{F}_{i}$. There is a cobordism $V:L\cobto (L'_{1},L', L_{1})$ such that 
$V$ is a disjoint union of two Lagrangian suspensions $V_{0}: L\cobto L_{1}$ and $V_{1}: \emptyset \cobto (L'_{1},L')$ such that $L_{1}, L'_{1}\in \mathcal{F}_{i}$ and   $\mathcal{S}(V_{0})\leq s$,  $\mathcal{S}(V_{1})\leq s'$. This means that $d^{\mathcal{F}_{i}}(L,L')\leq s+s'$ which implies the claim. 
\end{proof}

There are other variants of the definition of the metric $\widehat{d}^{\mathcal{F},\mathcal{F}'}$ that have 
interesting features. For instance,  by considering in (\ref{eq:pseudo-metric}) only cobordisms
$V: L'\cobto (L_{1},\ldots, L_{k}, L)$, in other words cobordisms for which $L'$ is the positive end and $L$ is the {\em top} negative end, one gets a measurement $t^{\mathcal{F}}_{k}(L',L)$. It has similar
properties to $d^{\mathcal{F}}_{k}$, except that it is not symmetric. We define $t^{\mathcal{F}}(L',L)$ as  in (\ref{eq:pseudo-metric-gen}) and we symmetrize by putting $r^{\mathcal{F}}(L',L)=(t^{\mathcal{F}}(L',L)+t^{\mathcal{F}}(L,L'))/2$ thus obtaining a new pseudo-metric. This pseudo-metric satifies the conclusion
of Corollary \ref{cor:vanishing} and can be used in the 
rest of the constructions before leading to metrics $\widehat{r}^{\mathcal{F},\mathcal{F}'}$
that satisfy the conclusions of  Corollaries \ref{cor:deblurring}, \ref{cor: metrics-var} and \ref{cor:fract-sets},
where in \ref{cor:fract-sets} the pseudo-metric $D^{\mathcal{U}}$ is replaced with $R^{\mathcal{U}}=\sup \{ \ r^{\mathcal{F}_{i}}\  \mid  \  i \in I\}$. An additional interesting feature of the pseudo-metrics $R^{\mathcal{U}}$ is next:

\begin{cor}\label{cor:Haus} With the notation above fix $L\in \mathcal{L}ag^{\ast}(M)$ and 
assume that $\mathcal{U}$ is a covering of $M$.
There exists a constant $\delta>0$ depending on $L$ and $\mathcal{U}$ such that, if $L'\in \mathcal{L}ag^{\ast}(M)$ is disjoint from $L$, then $R^{\mathcal{U}}(L,L')\geq \delta$.
\end{cor}
\begin{proof} The crucial remark is that, by inspecting the proof of the first part of Theorem \ref{thm:fission}, we see that given $V:L\cobto (L_{1},\ldots, L_{k}, L')$ in $\mathcal{L}ag^{\ast}(\R^{2}\times M)$ and such that $L\cap L'=\emptyset$, then $\mathcal{S}(V)\geq \delta(L;S)/2$ where $S=\cup_{i}L_{i}$
but $S$ - and thus $\delta(L,S)$  - does not depend on $L'$. As $\mathcal{U}$ is a covering of $M$
there exists some index $i\in I$ and an open set $U\subset U_{i}$ such that $U$ is the image of an embedding
$e: B(r)\to M$ with $ e^{-1}(L)=B_{\R}(r)$. Obviously, $U$ is disjoint from all the elements of $\mathcal{F}_{i}$ and thus $R^{\mathcal{U}}(L,L')\geq \pi r^{2}/4$.

\end{proof}
It is easy to find examples of sequences of Lagrangians $L_{k}$, each disjoint from a fixed $L$ (for instance
longitudes on a torus), and such that the sequence $L_{k}$ converges to $L$ in Hausdorff distance. By Corollary \ref{cor:Haus} all these Lagrangians remain away from $L$ in $R^{\mathcal{U}}$ distance.

\begin{rem}\label{rem:various-metric}
a.  It is well-known that there are other natural metrics defined on $\mathcal{L}ag^{\ast}(M)$. The most famous is Hofer's Lagrangian metric, used since the work of Chekanov \cite{Chek:finsler}, which infimizes the Hofer energy needed to carry one Lagrangian to the other. Another interesting more algebraic  metric, smaller 
than the Hofer metric, is the spectral metric due to Viterbo.  Both these metrics are infinite as soon as the two Lagrangians compared are not Hamiltonian isotopic. A metric smaller than the Hofer norm, and based on simple
Lagrangian cobordism has been introduced in \cite{Co-She:metric}: it measures the distance between $L$ and $L'$ by infimizing the shadow of cobordisms having only $L$ and $L'$ as ends. This metric is finite on each 
simple cobordism class and, with the notation above, it coincides with $d^{\emptyset}=\widehat{d}^{\emptyset,\emptyset}$. This metric is again often infinite. For instance, in the exact case, as soon as $L$ and $L'$ have non-isomorphic homologies, the simple shadow distance between $L$ and $L'$ is infinite. Indeed, if $L$ and $L'$ are related by an exact simple cobordism, then $L$ and $L'$ have isomorphic singular homologies
\cite{Bi-Co:cob1}  (more rigidity is actually true,  see \cite{Sua:s-cob}).  For other results on the simple shadow metric see \cite{Bis2,Bis1,Bis3}.

b. A notion of cone-length is familiar in homotopy theory as a measure of complexity
for topological spaces \cite{Cor:cone-LS}.

\end{rem}

\subsection{Some examples and calculations}
\subsubsection{Curves on tori and related
  examples} \label{subsec:curves_tori}

If $\phi$ is a Hamiltonian diffeomorphism, then clearly
$l_{a}(\phi(L),L)=0$ as soon as $\epsilon \geq ||\phi ||_{H}$
and so $l(\phi(L),L)=0$.  However,
we will see below classes of examples with
$0< l_{a}(\phi(L),L) < \infty$.  Intuitively, an inequality of
the type $1 \leq l_{a}(\phi(L), L) < \infty$ seems to indicate
that $\phi$ distorts $L$ (at least for our choices of classes
$\mathcal{F}$).

The examples below that satisfy
$1 \leq l_{a}(\phi(L),L)<\infty$ also satisfy
$$d_{l_{a}(\phi(L),L)}(\phi(L),L) < d_{0}(\phi(L),L)\leq
||\phi||_{H}.$$ In other words, in these examples the ``optimal'' (in
the sense of minimizing the shadow) approximation of $\phi(L)$ through
elements of the set $\{L\}\cup \mathcal{F}$ requires more elements
than just $L$.  Moreover, the relevant $d_{k}$'s are small enough so
that statement~\eqref{i:hf-inters} of Theorem~\ref{thm:fission}
applies and indeed, as predicted by the theorem, in these examples the
number of intersection points $\phi(L)\cap N$, where $N$ is an
appropriate other Lagrangian $N \in \mathcal{L}ag^*(M)$, is much
higher than the usual lower bound, given by the rank of the Floer
homology group $HF(N,L)$.

Consider the $2$-dimensional torus $M=T^2$ endowed with an area form.
We identify $T^2$ with the square $[-1,1]\times [-1,1]$ with the usual
identifications of the edges.  We consider five Lagrangians on $T^2$,
described on the square $[-1,1]\times [-1,1]$ by:
$L=[-1,1]\times \{0\}$,
$S_{1}=\{-\frac{1}{2} -\epsilon\}\times [-1,1]$,
$S_{2}=\{-\frac{1}{2} +\epsilon\}\times [-1,1]$,
$S_{3}=\{\frac{1}{2} -\epsilon\}\times [-1,1]$,
$S_{4}=\{\frac{1}{2} +\epsilon\}\times [-1,1]$. Here
$0<\epsilon \leq\frac{1}{8}$.
 
We will construct a new Lagrangian obtained through surgery between
$L$ and the $S_{i}$'s.  We use the surgery conventions from
\cite{Bi-Co:cob1} and define - see Figure \ref{fig:torus}:

\begin{equation}\label{eq:surgeries}
  L'=S_{3}\#[(S_{2}\#(L\#S_{1}))\#S_{4}]~.~
\end{equation}

\begin{figure}[htbp]
  \begin{center}
    \includegraphics[width=0.55\textwidth,
    height=0.35\textheight]{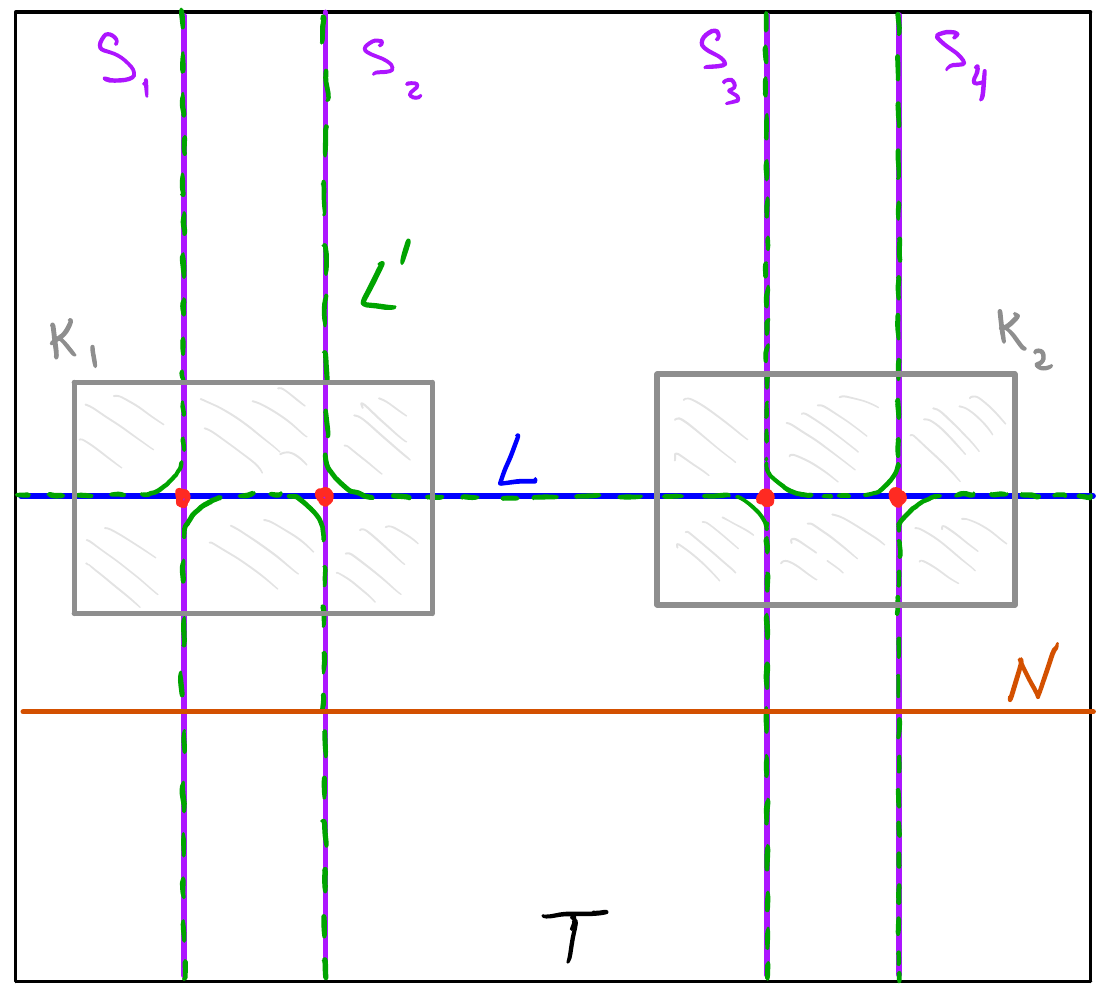}
  \end{center}
  \caption{The torus $T^2$ and the Lagrangians $L$,
    $L'=S_{3}\#[(S_{2}\#(L\#S_{1}))\#S_{4}]$, and $N$.
  \label{fig:torus}}
\end{figure}

In the surgeries above we use handles of equal size in the sense that
the area enclosed by each handle is equal to a fixed $\delta >0$ with
$\delta$ very small.  We will also make use of the two rectangles
$K_{1}=[-\frac{1}{2}-2\epsilon,-\frac{1}{2}+2\epsilon]\times
[-\epsilon,\epsilon]$ and
$K_{2}=[\frac{1}{2}-2\epsilon,\frac{1}{2}+2\epsilon]\times
[-\epsilon,\epsilon]$ and we put $K=K_{1}\cup K_{2}$ (see again
Figure~\ref{fig:torus}).

\begin{lem}\label{lem:ex1} Let  $\mathcal{F}=\{S_1, S_2, S_3, S_4 \}$
  and assume that $\delta < \frac{\epsilon ^{2}}{2}$. We have:
  \begin{itemize}
  \item[i.] $d_{0}(L',L)=4\epsilon$, $d_{4}(L',L)\leq 2\delta$,
    $l(L',L)=0$, $l_{2\delta}(L',L)=4$.
  \item[ii.] For any weakly-exact Lagrangian $N\subset T^2$ so that
    $N\cap K=\emptyset$, we have
    \begin{equation}\label{eq:intpoints}
      \#(N\cap L')\geq rk\ (HF(N,L))+ \sum_{i=1}^{4} rk\ (HF(N,S_{i}))~.~
    \end{equation}
  \item[iii.] If $N'$ is a weakly exact Lagrangian $N'\subset T^2$,
    then either $N'\cap L\neq\emptyset$ or, for any Hamiltonian
    diffeomorphism $\phi$ so that $\phi(L)=L'$ we have
    $\phi(N')\cap K\neq\emptyset$.
  \end{itemize}
\end{lem}  
Floer homology is considered here with coefficients in $\Z/2$.  Notice
that $HF(N,L')\cong HF(N,L)$ so that the inequality
(\ref{eq:intpoints}) indicates an ``excess'' of intersection
points. An example of a Lagrangian $N$ as at point ii is simply
$N=[-1,1]\times\{-2\epsilon\}$.

\begin{proof} 
 
  By inspecting again Figure \ref{fig:torus} and possibly extending
  the representation of the torus by adding vertically two additional
  fundamental domains to the square $[-1,1]\times [-1,1]$) one can see
  that there is a Hamiltonian isotopy $\phi: T^2 \to T^2$ so that
  $L'=\phi(L)$ (this is because the upper and lower ``bends'' in the
  picture encompass equal areas). The expression in
  (\ref{eq:surgeries}) show that there is a cobordism
  $$V: L'\to(S_{3}, S_{2}, L, S_{1}, S_{4})$$ given as 
  the trace of the respective surgeries (as given in
  \cite{Bi-Co:cob1}) and because the $S_{i}$'s are disjoint and all
  the handles are of area $\delta$ we have
  $\mathcal{S}(V)\leq 2\delta$. The reason for the factor $2$ is that
  the handles associated to the surgeries on the ``left'' and those on
  the ``right'' can not be assumed to have a superposing projection;
  the constant is $2$ and not $4$ because the two handles on the left
  (and similarly for the two handles on the right) can be assumed to
  have overlapping projections.  It is a simple exercise to show that
  $\delta(L';L)= 8\epsilon$.
   
  From the first part of Theorem \ref{thm:fission} we deduce
  $d_{0}(L',L)\geq 4\epsilon$.  It is also easy to see that one can
  find a Hamiltonian $H:T^2 \to \R$ with variation equal to
  $4\epsilon$ and so that $\phi_{1}^{H}(L)=L'$.  Therefore,
  $d_{0}(L',L)=4\epsilon~.~$ On the other hand, recall the assumption
  $\delta < \epsilon ^{2}/2$. Therefore we have:
  $$d_{4}(L',L)\leq 2\delta~.~$$
 
  We now estimate cone-length.  Clearly, the absolute number is
  $l(L',L)=0$.  From the existence of the cobordism $V$ we deduce
  $l_{2 \delta}(L',L)\leq 4$. We want to show $l_{2 \delta}(L',L)=
  4$. Assume that $l_{2\delta}(L',L)\leq 3$.  Therefore there exists a
  cobordism $V':L'\to (L_{1}, L_{2}, L_{3}, L_{4})$ where one of the
  $L_{i}$'s equals $L$ and the other three are picked among the
  $S_{i}$'s (or are void) and the shadow of $V'$ is at most $2\delta$.
  Without loss of generality, assume that $S_{1}$ is not among the
  $L_{i}$'s. We now consider the number
  $\delta(L'; L\cup S_{2}\cup S_{3}\cup S_{4})$. By using a disk
  centered along the part of $S_{1}$ contained in $L'$ we see that
  $\delta(L';L\cup S_{1}\cup S_{2})\geq 8\epsilon$. By the first part of
  Theorem \ref{thm:fission} it follows $\mathcal{S}(V')\geq 4\epsilon$
  which contradicts $\delta\leq 2\epsilon^{2}$.

  The two other points of the Lemma also follow from Theorem
  \ref{thm:fission} (they possibly admit also more elementary, direct
  proofs). Point (b) of the Theorem implies that for any weakly-exact
  Lagrangian $N\subset T^2$ so that $N\cap K=\emptyset$, we have
  (\ref{eq:intpoints}).  Indeed, we may find disks around the (unique)
  intersection point of each of the $S_{i}$'s with $L$ that are of
  area $4\epsilon^{2}$, have the real part along $L$ and the imaginary
  part along $S_{i}$, are contained in $K$, and any two of these disks
  have disjoint interiors.  As $N$ avoids $K$, this means
  $\delta^{\Sigma \mathbb{L}}(\mathbb{L};N)\geq 4\epsilon^{2}$ for
  $\mathbb{L}=L\bigcup\cup_{i} S_{i}$.  The last point of the Lemma
  follows in a similar way. Assuming also that $N'\cap L=\emptyset$ we
  also have $\phi(N')\cap L'=\emptyset$. If we also have
  $\phi(N')\cap K=\emptyset$, then $\phi(N')$ satisfies inequality
  (\ref{eq:intpoints}) (with $\phi(N')$ in the place of $N$).  From
  the fact that $N'$ is weakly exact we deduce that the singular
  homology class of $N'$ is the same as that of $L$ and thus
  $HF(N',S_{i})$ does not vanish.  But this leads to a contradiction
  with $\phi(N')\cap L'=\emptyset$.
\end{proof}

It is easy to construct examples similar to the one above in higher
dimensions.  For instance, one can consider
$M= (T^2 \times T^2, \omega\oplus \omega)$ and take
$\bar{L}=L\times L$, $\bar{S}_{i}= S_{i}\times S_{i}$ etc. We will see
some less trivial extensions in the next subsection.

\begin{rem}\label{rem:pseudo} 
  We use the examples below to underline two  deficiencies
  of the pseudo-metric $d^{\mathcal{F}}$.
  \begin{itemize}
  \item[i.] $d^{\mathcal{F}}$ is generally degenerate. For example,
    $d^{\mathcal{F}}_{3}(S_{1}, S_{2})=0$, hence
    $d^{\mathcal{F}}(S_1, S_2)=0$. Indeed, let
    $V: S_{1} \cobto (S_{1}, S_{2}, S_{2})$ be the cobordism
    $V=\gamma_{0}\times S_{1} \coprod \gamma_{1}\times S_{2}$ where
    $\gamma_{0}=\R+i\subset \C$ and $\gamma_{1}$ is a curve in $\C$
    that has two horizontal negative ends, one at height $2$ and the
    other at height $3$ and is disjoint from $\gamma_{0}$. The same
    construction shows that for any family $\mathcal{F}$ with more
    than one element the resulting pseudo-metric is degenerate.

    In the above examples the cobordisms $V$ are disconnected and they
    also have vanishing shadow. However, there are also examples of
    connected cobordisms $W_{\epsilon}$ with constant ends and {\em
      positive} shadow such that
    $\lim_{\epsilon\to 0}\mathcal{S}(W_{\epsilon})=0$. For instance,
    with the notation above, consider a curve
    $\gamma \subset \mathbb{R}^2$ which has a ``$\supset$'' shape with
    its lower end going to $-\infty$ along the horizontal line $y=-1$
    and its upper end going to $-\infty$ along the horizontal line
    $y=1$. Let $\gamma'$ be the $x$-axis, $y=0$. Consider now the
    surgery
    $W_{\epsilon}: = (\gamma \times S_1) \#_{\epsilon} (\gamma' \times
    L )\subset \mathbb{R}^2 \times T^2$. (Note that, in cotrast to the
    construction of e.g. $L'$ above, the surgery here is performed in
    the space $\mathbb{R}^2 T^2$.) Clearly $W_{\epsilon}$ is a
    (connected) weakly exact Lagrangian cobordism
    $W_{\epsilon}: L \cobto (S_1, L, S_1)$ and
    $\lim_{\epsilon \to 0} \mathcal{S}(W_{\epsilon}) = 0$.

  \item[ii.] In general, even if both Lagrangians $L$ and $L'$ belong
    to the triangulated completion of the family $\mathcal{F}$, it can
    be very difficult to know whether $d^{\mathcal{F}}(L,L')$ is
    finite because there might not be any practical way to construct
    cobordisms with ends $L,L'$ and elements of $\mathcal{F}$.
  \end{itemize}
\end{rem}

\subsubsection{Matching cycles in simple Lefschetz fibrations.}

We revisit here the phenomena described above in a different context
and we also see examples of symplectic diffeomorphisms $\phi :M\to M$
so that $l(\phi^{k}(L),L)=k$ (in these examples $\phi$ is a Dehn
twist).

The manifold $M$ is now taken to be the total space of a Lefschetz
fibration $\pi: M \longrightarrow \mathbb{C}$ over $\mathbb{C}$ with
general fiber the cotangent bundle of a sphere $K$ (in particular $M$
is not compact).  We will assume that the Lefschetz fibration - the we
write as $\pi :M\to \C$ - has exactly three singularities $x_{i}$,
$i=1,2,3$ whose projection on $\C$ is arranged as in Figure
\ref{fig:Lef} below.  We also assume that there are two matching
cycles relating the three singularities that we denote by $S$, from
$x_{1}$ to $x_{2}$, and $L$, from $x_{2}$ to $x_{3}$ - as in the same
figure.

Notice that $L$ and $S$ intersect (transversely) in a single
point. Moreover, recall that with the notations in \cite{Bi-Co:cob1,
  Bi-Co:lefcob-pub} we have that $S\#L$ is Hamiltonian isotopic to the
Dehn twist $\tau_{S}(L)$, and, similarly, $L\#S$ is Hamiltonian
isotopic to $\tau_{S}^{-1}(L)$.  An important point to emphasize here
is that the Dehn twist $\tau_{S}(L)$ is only well defined up to
Hamiltonian isotopy. On the other hand, the models for $\tau_{S}(L)$
(and $\tau_{S}^{-1}(L)$) given by surgery, as before, are precisely
determined as soon as the local data of the surgery is fixed (the
surgery handle and the precise Darboux chart around the intersection
point).  We will also need, two other matching cycles $S_{1}$ and
$S_{2}$ with a projection as in Figure \ref{fig:Lef} a.
 \begin{figure}[htbp]
   \begin{center}
   \includegraphics[width=0.80\textwidth,
    height=0.35\textheight]{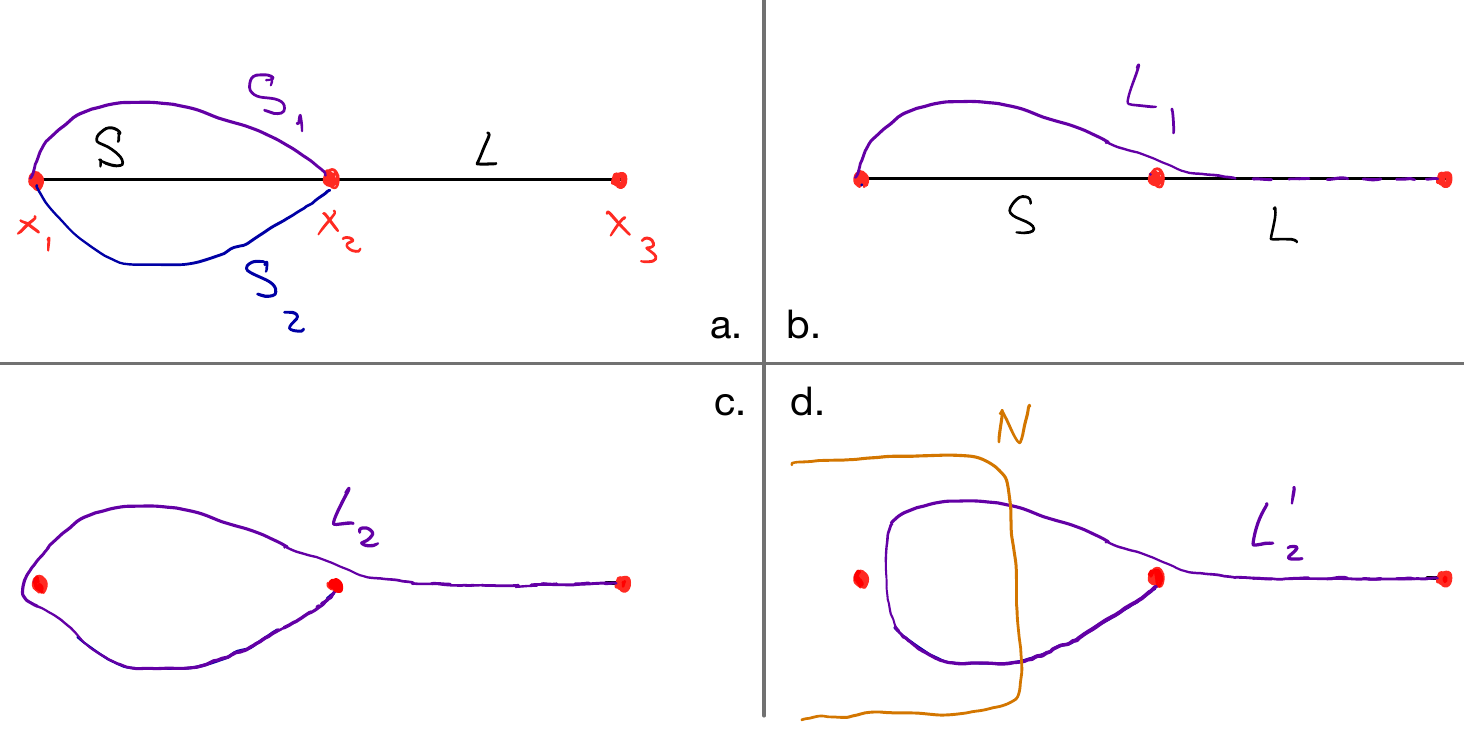}
   \end{center}
   \caption{The matching cycles $S, S_{1},S_{2}$ and $L$ and the
     Lagrangians $L_{1}, L_{2}, L_{2}'$ constructed by surgery (and
     small perturbation) from them.}
   \label{fig:Lef}
\end{figure}

They are both Hamiltonian isotopic to $S$ through Hamiltonian
isotopies.  The two spheres $S_{2}$ and $S_{2}$ intersect transversely
at the points $x_{1}$ and $x_{2}$ and each of them intersects
transversely $L$ at the point $x_{2}$.  We now consider the following
three Lagrangians: $L_{1}$ which is obtained from $S_{1}\# L$ after a
small Hamiltonian isotopy so that its projection is as in Figure
\ref{fig:Lef} b, $L_{2}$ given as a small deformation of
$S_{2}\# L_{1}$ and $L_{2}'$, a small deformation of $L_{1}\#S_{2}$ so
that their projections are as in the same figure, part c and d,
respectively. Notice that $L_{1}$ is a model for $\tau_{S}(L)$ and
that $L_{2}$ and $L_{2}'$ are models for $\tau^{2}_{S}(L)$ and
$L=\tau_{S}^{-1}\tau_{S}(L)$, respectively. In particular, there is a
Hamiltonian isotopy $\phi$ so that $L'_{2}=\phi(L)$.

\

Fix the family $\mathcal{F}=\{S_{1},S_{2}\}$.  The first remark is
that by taking the surgery handles sufficiently small we have
$d_{2}(L'_{2},L)<d_{0}(L'_{2},L)<\infty$. Further, let $K'$ be a
Hamiltonian perturbation of the vanishing sphere $K$ in the general
fiber. Let $N$ be the trail of $K'$ along a curve as in Figure
\ref{fig:Lef} d.  It is not hard to see
$HF(N,L_{2})=HF(N,S_{1})\oplus HF(N,S_{2})=HF(K,K)\oplus HF(K,K)$ (one
can use Seidel's exact triangle associated to a Dehn twist for this
computation). This implies $l(L_{2},L)=2$. On the other hand,
$HF(N,L'_{2})=0$. However, by taking the surgery handles in the
constructions above sufficiently small we see that
$\# (N\cap L'_{2})\geq 2 \ \mathrm{rk\/}(HF(K,K))$, as predicted by
Theorem \ref{thm:fission}. Notice also that if the surgery handle is
not small enough, or, alternatively, $N$ avoids $L_{2}'$ by passing
closer to $x_{1}$, then $N$ is disjoint from $L'_{2}$.

The last remark in this setting is the following.  By taking more
copies of the sphere $S$, (for instance four, as on the left of Figure
\ref{fig:lef2}), we can construct, in a way similar to the above,
models $L_{k}$ for $\tau_{S}^{k}(L)$. In Figure \ref{fig:lef2}, on the
right, we represent in this way $\tau^{4}_{S}(L)$.  As before, it is
easy to compute $HF(N, L_{k})=\oplus_{i=1}^{k}HF(K,K)$. This shows
that $l(L_{k},L)=k$ (this is a reflection of the well-known fact that
$\tau_{S}$ is not a torsion element in $Symp(M)$).

\begin{figure}[htbp]
  \begin{center}
    \includegraphics[width=0.85\textwidth,
    height=0.18\textheight]{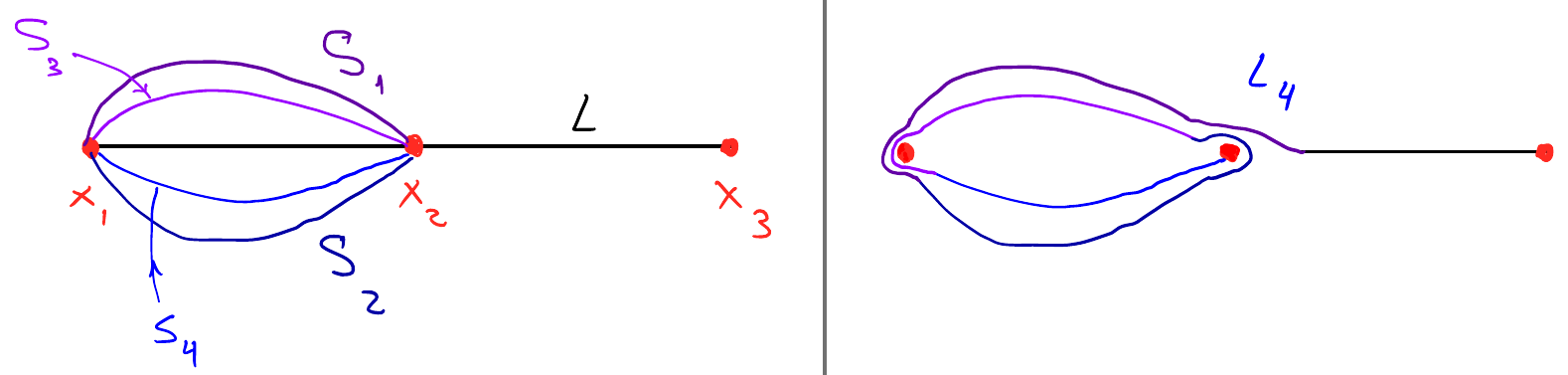}
  \end{center}
  \caption{A model for $\tau^{4}_{S}(L)$.}
  \label{fig:lef2}
\end{figure}

\subsubsection{Trace of surgery.} The numbers $d_{k}$ are hard to compute
as it is difficult in general to identify cobordisms with fixed ends
and with minimal shadow. However, we will see here how to use one of
the intersection results from Theorem~\ref{thm:fission} to show the
``optimality'' of decompositions given by the trace of certain
surgeries at one point.

We focus on just one example. As in \S\ref{subsec:curves_tori} we take
$M=T=S^{1}\times S^{1}$ and we fix $S_{1}$ and $L$ as in that
subsection. We now consider $L''= L\# S_{1}$ and, again as in
\S\ref{subsec:curves_tori}, we assume that the area of the handle used
in the surgery giving $L''$ is equal to $\delta$. We fix
$\mathcal{F}=\{S_1, S_2, S_3, S_4\}$ as in Lemma~\ref{lem:ex1}.
Notice that the shadow of the trace of the surgery
$V:L''=L\# S^{1}\to (L,S_{1})$ is equal to $\delta$.
 
\begin{lem}
  For $\delta$ small enough we  have $d_{1}(L'',L)=\delta$.
\end{lem}
 
In other words, there is no decomposition of $L''$ in terms of the
family $L\cup\mathcal{F}$ through a cobordism with two negative ends
and of shadow smaller than $\delta$.

\begin{proof} Suppose that there is a cobordism
  $V': L'' \to (L_{1}, L_{2})$ so that one of the $L_{i}$'s equals
  $L$, the other equals one of the $S_{i}$'s and
  $\mathcal{S}(V')= \delta' <\delta$.  We first notice that $S_{1}$
  needs to appear among the $L_{i}$'s. Indeed, suppose, for instance
  that $(L_{1}, L_{2})=(L, S_{2})$. In this case, consider a disk
  based on the part of $L''$ that coincides with $S_{1}$ and is
  disjoint from $S_{2}$ as well as from $L$ and whose real part is
  along $L''$. The area of such a disk can be assumed to be as close
  as needed to $2(4\epsilon -\delta)$.  By now applying the first
  part of Theorem \ref{thm:fission} we deduce that
  $\delta > \mathcal{S}(V')\geq 4\epsilon-\delta$ which is a
  contradiction if $\delta$ is small enough. In conclusion, we deduce
  that the two negative ends of $V'$ coincide with $L$ and
  $S_{1}$. Consider now the Lagrangian $N$ as in Figure \ref{fig:ex2}
  and denote by $o$ the intersection of $L$ and $S_{1}$.
  \begin{figure}[htbp]
    \begin{center}
      \includegraphics[width=0.65\textwidth,
      height=0.36\textheight]{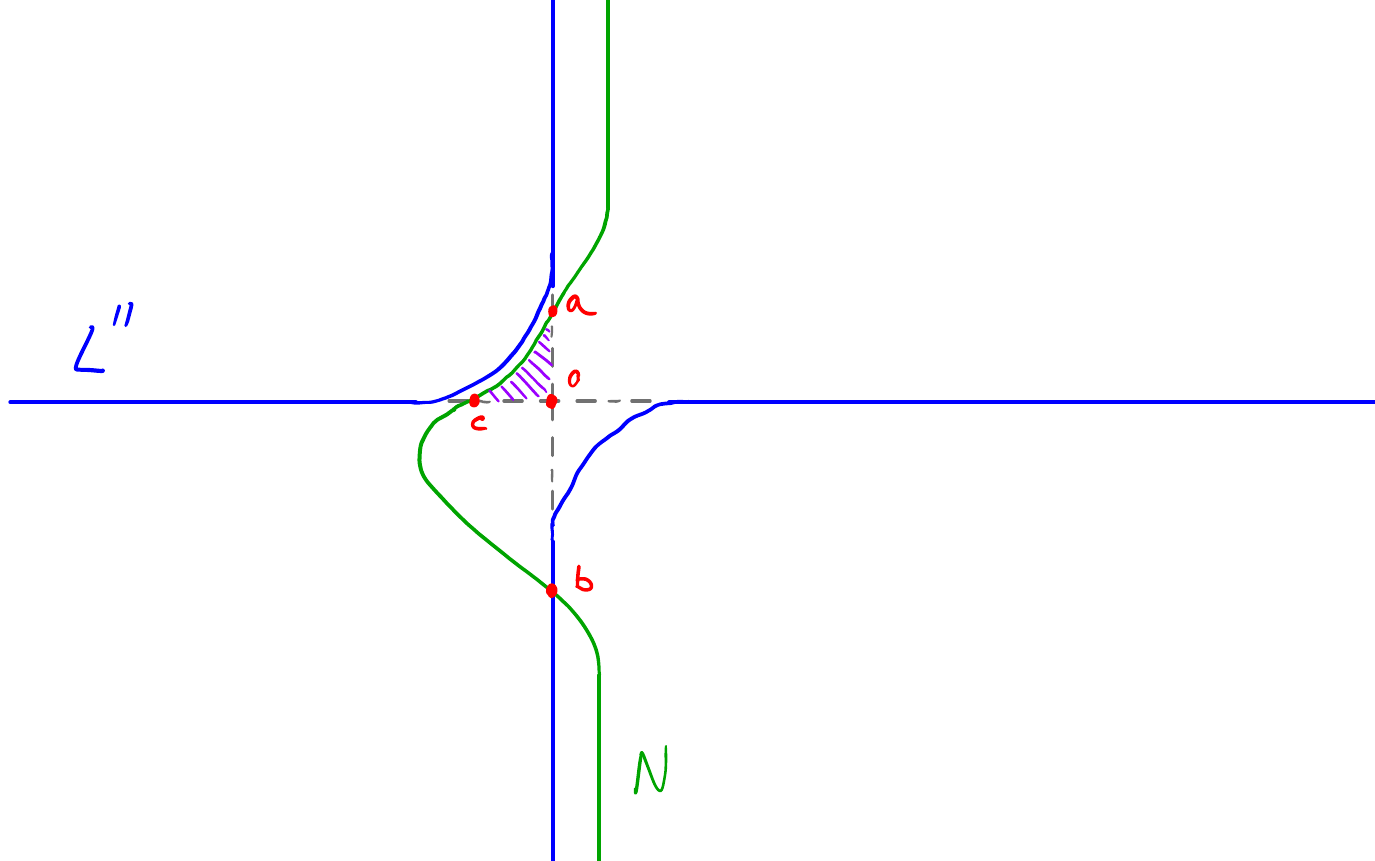}
    \end{center}
    \caption{The triangle $aoc$ is of area $A$ with $\delta> A>\delta'$.}
    \label{fig:ex2}
  \end{figure}

  The properties of $N$ are the following: $N$ is Hamiltonian isotopic
  to $S_{1}$; it intersects $S_{1}$ transversely at precisely two
  points $a$ and $b$ and it intersects $L$ transversely at one point
  $c$ ; $N$ intersects $L''$ transversely at the point $b$; the small
  triangle of vertices $a,c,o$ is of area $A$ with
  $\delta'< A < \delta$.  We use the Lagrangian $N$ as follows. First,
  notice that by assuming $\delta$ small enough and for
  $\mathbb{L}=L\cup S_{1}$, we can find the relevant disks centered at
  $o$ so as to estimate
  $\delta^{\Sigma \mathbb{L}}(\mathbb{L};N)\geq 4A$. By applying the
  point (b) of Theorem \ref{thm:fission} we deduce
  $1=\#(N\cap L'')\geq \dim HF(N, L)+\dim HF(N,S_{1})=3$ which is a
  contradiction and thus proves that $V'$ does not exist.
\end{proof}
 
 \subsection{Algebraic 
 metrics on $\mathcal{L}ag^{\ast}(M)$} \label{subsec:alg-metr-Lag}
 The main purpose of this subsection is to notice that it is possible to define measurements similar to those
 in \S\ref{subsec:shad-metric} but that only exploit the algebraic structures involved and that do not 
 appeal to cobordism.  We emphasize  that, as before, our metrics may take infinite values. The key point is that the proof of the first part of Theorem \ref{thm:fission}  implies not only the non-degeneracy of $\widehat{d}^{\mathcal{F},\mathcal{F}'}$ but also that of its algebraic counterpart. 
When $\mathcal{F}$ generates
 $D\fuk^{\ast}(M)$ some of these algebraic pseudo-metrics are finite by definition, independently of the existence  of cobordisms - see Remark \ref{rem:finite-alg}. While inspired by cobordism constructions, this algebraic approach is independent of them.  We emphasize that we will not 
attempt to develop here this additional algebraic machinery 
 in an extensive way. At the same time, the construction of both the metrics 
 $\widehat{d}^{\mathcal{F},\mathcal{F}'}$ as well as their algebraic counterparts fit a more general, abstract pattern that we will outline. 
 
 \subsubsection{Weighted triangulated categories.}\label{subsubsec:weights1} Let $\mathcal{X}$ be a triangulated category
 and let $\mathcal{X}_{0}$ be a family of objects of $\mathcal{X}$ that generate $\mathcal{X}$ through
 triangular completion.  The purpose of this subsection is to describe a procedure leading to a (pseudo) metric 
 on $\mathcal{X}_{0}$. The pseudo-metrics $d^{\mathcal{F}}$ in \S\ref{subsec:shad-metric} are of this type but, as we shall see further below, other choices are possible.
 
 There is a category denoted by $T^{S}\mathcal{X}$ that was introduced in \cite{Bi-Co:cob1,Bi-Co:lcob-fuk}. This category is monoidal and its objects are finite ordered famillies 
 $(K_{1},\ldots, K_{r})$ with $K_{i}\in \mathcal{O}b(\mathcal{X})$ with the operation given by concatenation. Up to a certain natural equivalence relation, the morphisms in $T^{S}\mathcal{X}$ are direct sums of basic morphisms $\bar{\phi}$ from a family formed of a single object of $\mathcal{X}$ to a general family,  $\bar{\phi}: K\to (K_{1},\ldots, K_{s})$. Such a morphism $\bar{\phi}$ is a triple $(\phi, a, \eta)$ where $a$ is an object in $\mathcal{X}$, $\eta$
 is a cone decomposition of $a$ through iterated  distinguished triangles of the form:
 \begin{equation}
 \label{eq:iterated-cones}
a=\tcn (K_{s}\to \tcn ( K_{s-1}\to \ldots\to \tcn (K_{2}\to K_{1}).. )
 \end{equation}
 and $\phi :K\to a$ is an isomorphism. For such a cone decomposition the family  $(K_{1},\ldots, K_{s})$ is called the linearization of the decomposition.
 In essence, the morphisms in $T^{S}\mathcal{X}$ parametrize all the cone-decompositions of the 
 objects in $\mathcal{X}$. Composition in $T^{s}\mathcal{X}$ comes down to refinement of cone-decompositions. 
 
 Denote by $T^{S}\mathcal{X}_{0}$ the full subcategory of $T^{S}\mathcal{X}$ that has
 objects $(K_{1},\ldots, K_{r})$ with $K_{i}\in \mathcal{X}_{0}$, $1\leq i\leq r$. 
 
 Assume given a weight $w: \mor_{T^{S}\mathcal{X}_{0}}\to [0,\infty]$  so that 
 \begin{equation}\label{eq:weight-tr}
 w(\bar{\phi}\circ \bar{\psi})\leq w(\bar{\phi})+w(\bar{\psi}) \ , \ w(id_{X})=0 \ , \forall \ X
 \end{equation}
  here $id_{X}$ is the identity morphism
 viewed as defined on the family formed by the single object $X$ and with values in the same family.
 We will refer to this $w$ as a weight on $\mathcal{X}$.
 Fix also a family $\mathcal{F}\subset \mathcal{X}_{0}$. 
 
In this setting, we define (compare to (\ref{eq:pseudo-metric})): 
\begin{equation}\label{eq:pseudo-metric2}
s^{\mathcal{F}}(K',K)=\inf \{w(\bar{\phi}) \mid \ \ \bar{\phi} : K'\to (F_{1},\ldots, K,\ldots, F_{r}),\ \ F_{i}\in \mathcal{F}, \forall i \}~.~
\end{equation}

We take $s^{\mathcal{F}}$ to be $=\infty$ in case there is no morphism as in (\ref{eq:pseudo-metric2}).  If $w$ is finite and if $\mathcal{F}$ generates $\mathcal{X}$, then $s^{\mathcal{F}}$ is finite. 
It is clear that $s^{\mathcal{F}}(K',K)$ satisfies the triangle inequality but it is not symmetric in general. We  let
$$\bar{s}^{\mathcal{F}}(K',K)= \frac{s^{\mathcal{F}}(K',K)+s^{\mathcal{F}}(K,K')}{2}$$
thus obtaining a pseudo-metric on the set of objects fo $\mathcal{X}$.

In summary, the pseudo-metrics obtained in this way are based on minimizing the energy needed to split
objects into ``elementary'' pieces belonging to $\mathcal{F}$. We will refer to them as {\em weighted fragmentation} pseudo-metrics.
Obviously, other numbers such as those in  \S\ref{subsec:shad-metric} (basically, weighted notions
of cone-length) can also be defined in the general setting here.

The case of interest in this paper is   $\mathcal{X}=D\fuk^{\ast}(M)$ with  $\mathcal{X}_{0}$ 
consisting of all the Yoneda modules associated to the Lagrangians in $\mathcal{L}ag^{\ast}(M)$. 
In our notation, the category $\fuk^{\ast}(M)$ is defined as described at the beginning of \S\ref{s:floer-theory}, without reference to filtrations. Even in this case, this is not actually a single
$A_{\infty}$ category but rather a family of quasi-equivalent  such categories. Similarly, $D\fuk^{\ast}(M)$ is well defined up to triangulated equivalence.  
With these choices, the pseudo-metric $d^{\mathcal{F}}$ from \S\ref{subsec:shad-metric} 
is a first example of a (class) of weighted fragmentation pseudo-metrics associated to a weight $w$ defined as follows. 

Recall from \cite{Bi-Co:lcob-fuk},\cite{Ch-Co:cob-Seidel} that there is a monoidal cobordism category
$\mathcal{C}ob^{\ast}(M)$ whose objects are families $(L_{1},\ldots, L_{s})$ with $L_{i}\in \mathcal{L}ag^{\ast}(M)$ and with morphisms (formal sums) of cobordisms of the type
$V: L\cobto (L_{1},\ldots, L_{s})$ (modulo an appropriate equivalence relation; the monoidal 
operation is concatenation). There is a monoidal functor, denoted in \cite{Bi-Co:lcob-fuk}
by $\widetilde{\mathcal{F}}$ but that, to avoid confusion in notation, we will
denote here by $\widetilde{\Phi}$ :
\begin{equation}\label{eq:functor-cob-TS}\widetilde{\Phi}:\mathcal{C}ob^{\ast}(M)\to T^{S}(D\fuk^{\ast}(M))~.~
\end{equation}
On objects, this functor associates to  
a Lagrangian $L$ its Yoneda module $\mathcal{L}$ and its properties have 
been used extensively earlier in the paper, starting from \S\ref{sb:wf-icones-cobs}. 

In the setting, $\mathcal{X}=D\fuk^{\ast}(M)$, for a morphism $\bar{\phi}\in \mor_{T^{S}\mathcal{X}_0}$ we define the {\em shadow} weight of $\bar{\phi}$ by: 
\begin{equation}\label{eq:shadow-weight}
w_{\mathcal{S}}(\bar{\phi})=\inf\{\mathcal{S}(V) \mid \ \widetilde{\Phi}(V)=\bar{\phi}\}
\end{equation}
and it is easy to see from the various definitions involved that $d^{\mathcal{F}}$ coincides with the weighted
fragmentation pseudo-metric $\bar{s}^{\mathcal{F}}$ associated to $w_{\mathcal{S}}$. 
Additionally, recall from  Corollary \ref{cor:deblurring}  that, by using an appropriate perturbation $\mathcal{F}'$,
we obtain an actual metric $\widehat{d}^{\mathcal{F},\mathcal{F}'}=\max\{d^{\mathcal{F}} ,d^{\mathcal{F}'}\}$.

\begin{rem}\label{rem:triang-categ} The definition of the category $T^{S}\mathcal{X}$ was inspired by the work on Lagrangian cobordism and might seem artificial in itself. However,  we will remark here that 
(in a slightly modified form) it is the natural categorification of the Grothendieck group $K(\mathcal{X})$.  This group is defined as the quotient of the free abelian group 
 generated by the objects in $\mathcal{X}$ modulo the relations $B=A+C$ whenever $A\to B\to C$ is a distinguished triangle in $\mathcal{X}$. We will work here in a simplified setting and take the identity for the shift functor. As a consequence $K(\mathcal{X})$ is a $\Z/2$ vector space.
 Alternatively, $K(\mathcal{X})$ can also be defined as the free monoid of finite ordered families 
 $(K_{1},\ldots, K_{r})$ where $K_{i}\in \mathcal{O}b(\mathcal{X})$, with the operation being given by
 concatenation of families, modulo the relations $K_{1}+K_{2}+\ldots + K_{s}=0$ 
 whenever there exists a cone decomposition of $0$ with linearization $(K_{1},\ldots, K_{s})$. 
When $\mathcal{X}$ is small, there is a category, $\widehat{T}^{S}\mathcal{X}$, closely associated to $T^{S}\mathcal{X}$, that categorifies $K(\mathcal{X})$ in the usual sense (meaning that it is a monoidal category with the property that the monoid formed by the isomorphism classes of its objects is $K(\mathcal{X})$).  The basic idea is that, in $\widehat{T}^{S}\mathcal{X}$,  families that
are linearizations of acyclic cones are declared isomorphic to $0$.  More formally, $\widehat{T}^{S}\mathcal{X}$ is defined if the category $\mathcal{X}$ is small and is constructed in three steps: first we add to the morphisms in $T^{S}\mathcal{X}$ the morphism $0\to \emptyset$ (and the relevant compositions) thus getting $T^{S}\mathcal{X}^{+}$; secondly, we localize $T^{S}\mathcal{X}^{+}$ at the family of morphisms $$\mathcal{A}=\{\phi\in\mor_{T^{S}\mathcal{X}^{+}} \mid \ \ \phi : 0 \to 
(K_{1},\ldots, K_{s})\}$$ (here $0$ is viewed as a family formed by the single element $0$; this is equivalent to adding inverses to all the morphisms having $0$ as domain and adding relations so that associativity of composition is still  satisfied); finally,  we complete in the monoidal sense by allowing  formal sums for all the new and old
morphisms.  
\end{rem}

 \subsubsection{Energy of retracts of weakly filtered modules.}\label{subsubsec:energy-retr}
Our aim is to define an algebraic weight $w_{\textnormal{alg}}$ that satisfies
 (\ref{eq:weight-tr}) and the associated weighted fragmentation pseudo-metrics (see \S\ref{subsubsec:weights1}). We start in this subsection by introducing a measurement 
 associated to retracts.
 
 Assume that $\mathcal{M}$ is a
weakly filtered module over the weakly filtered $A_{\infty}$ -
category $\mathcal{A}$ with discrepancy $\leq \bm{\epsilon}^{m}$
as in \S \ref{sb:mod} and that $\psi : \mathcal{M}\to \mathcal{M}$ is a weakly
filtered module homomorphism with discrepancy
$\leq \bm{\epsilon}^{h}$ which is null-homotopic.  Following the
terminology in~\S\ref{sec:bdry-depth}, we consider the homotopical
boundary level of $\psi$:

$$B_h(\psi;\bm{\epsilon}^{h}) :=
\beta_{h}(\psi;\bm{\epsilon}^{h})+A(\psi)~.~$$
 Let $f:\mathcal{M}_{0}\to\mathcal{M}_{1}$ be a  morphism of weakly filtered
modules and define:
\begin{equation}\label{eq:retract-weight}
\rho(f)=\inf_{g}\  
(\max\ \{B_h(g\circ f-id;\bm{\epsilon}^{h}), A(g)+A(f), 0\})
\end{equation}
where the infimum is taken over all
weakly filtered module morphisms
$$\ g:\mathcal{M}_{1}\to \mathcal{M}_{0},\ \mathrm{with}\ g\circ f \in \hom^{\epsilon^{h}}(\mathcal{M}_{0},\mathcal{M}_{0})\ ,\ g\circ f\simeq id_{\mathcal{M}_{0}}~.~$$
 In case no such $g$ exists we put 
 $\rho(f)=\infty$. The measurement $\rho$ estimates the 
 minimal energy required to find a left homotopy inverse for $f$.
  \begin{rem} Similar notions  are familiar  in Floer theory, generally to compare two quasi-isomorphic chain complexes, and in that case the infimum above is taken also over all
 morphisms $f$ and one also takes into account a homotopy $f\circ g\simeq id_{\mathcal{M}_{1}}$. For instance, this appears in \cite{Usher-Zhang}.
 \end{rem}
 Two properties of $\rho$ will be useful below.
 \begin{lem}\label{lem:first-ineq} Given $\mathcal{M}_{0}\stackrel{f}{\longrightarrow} \mathcal{M}_{1}$, $\mathcal{M}_{1}\stackrel{f'}{\longrightarrow} \mathcal{M}_{2}$, then:
  \begin{equation}\label{eq:ro-triang}
  \rho(f'\circ f)\leq \rho(f)+\rho(f')~.~
  \end{equation}
  \end{lem}
 \begin{proof}Indeed,  assume $ \mathcal{M}_{1}\stackrel{g}{\longrightarrow}\mathcal{M}_{0}$,
 $\mathcal{M}_{2}\stackrel{g'}{\longrightarrow} \mathcal{M}_{1}$ are weakly filtered module maps
 and $\eta:g\circ f\simeq id_{\mathcal{M}_{0}}$, $\eta:g'\circ f'\simeq id_{\mathcal{M}_{1}}$
 are the respective homotopies.  Assume that $f,g, \eta, f',g',\eta'$ shift filtrations by $\leq s,\ r,\ k,\ s',\ r',\ k'$,
 respectively. These numbers can be taken larger but as close as desired to the respective action levels. 
 Notice that  $f'\circ f$ shifts filtrations by $\leq s+s'$, $g\circ g'$ shifts filtrations by $\leq r+r'$
 and, moreover, the homotopy 
 $$\bar{\eta}=g\circ \eta'\circ f + \eta: g\circ g'\circ f'\circ f\simeq id_{\mathcal{M}_{0}}$$ 
 shifts filtrations by $\leq \max\{r+s+k', k\}$. This implies the claim.  
 \end{proof}
 To state the second property, assume that the weakly filtered module 
$\mathcal{M}_{1}$ can be written as a weakly filtered iterated cone 
$$\mathcal{M}_{1}= \tcn (K_{s}\to  \tcn (K_{s-1}\to\ldots \to \tcn ( \mathcal{N} \to \tcn (K_{i-1}\to \ldots \tcn (K_{2}\to K_{1})\ldots )$$ and that there is another weakly filtered 
module $\mathcal{N}'$ together with weakly filtered maps $u:\mathcal{N}\to \mathcal{N}'$ and $v:\mathcal{N}'\to \mathcal{N}$ and a weakly filtered homotopy $\xi:v\circ u\simeq id_{\mathcal{N}}$. 

\begin{lem}\label{lem:bounds-rho}  There is another
weakly filtered module $\mathcal{M}'_{1}$ that can be written as a filtered  iterated cone of the same 
form as the decomposition for $\mathcal{M}_{1}$ except with $\mathcal{N}'$ replacing $\mathcal{N}$ and
there is an associated map $u':\mathcal{M}_{1} \to\mathcal{M}'_{1}$ so that 
$\rho(u')\leq \max\{A(u)+A(v),A(\xi), 0\}$.
\end{lem}
As a corollary we deduce that given $\mathcal{M}_{1}$, $\mathcal{N}$ as well as $\mathcal{N}'$
and a weakly filtered map $u:\mathcal{N}\to \mathcal{N}'$ with $\rho(u)<\infty$, then, for any $\epsilon>0$, there exists
a weakly filtered module $\mathcal{M}'_{1}$ and a map $u':\mathcal{M}_{1}\to\mathcal{M}'_{1}$
as in the Lemma such that:
\begin{equation}
\label{eq:ineq-rho}
\rho(u')\leq \rho (u)+\epsilon~.~
\end{equation}
\emph{Proof of Lemma \ref{lem:bounds-rho}}. By recurrence, the proof is easily reduced 
to showing the statement for two particular types of decompositions: the first is $\mathcal{M}_{1}=
\tcn (\mathcal{N}\stackrel{\phi}{\longrightarrow} K_{1})$; the second case
 is $\mathcal{M}_{1}=\tcn (K_{2}\stackrel{\phi}{\longrightarrow} \mathcal{N})$.
 We will only treat here the first case the second being entirely similar. Without loss of generality, we may assume that $\phi$ does not shift action filtrations. 
 Assume that the map $v:\mathcal{N}'\to \mathcal{N}$ shifts 
 filtrations by $\leq r$, the map $u$ shifts filtrations by $\leq s$ and $\xi$ shifts filtration by $\leq k$. 
 Following the definitions of weakly filtered cones in \S\ref{sb:wf-mc} we construct $\mathcal{M}'_{1}$
 as follows.
 Let $\bar{v}: S^{-r}\mathcal{N}'\to \mathcal{N}$ be given by the map $v$ after shifting the filtration of its domain up by $r$. 
 Define $\phi'=\phi\circ \bar{v}$, $\phi':S^{-r}\mathcal{N}'\to K_{1}$
 and put $\mathcal{M}'_{1}=\tcn (\phi')$. With the notation in \eqref{eq:fcone}, this cone is defined
 by taking the action shift of $\phi'$ to be $0$.  
 There are module morphisms $v':\mathcal{M}'_{1}\to \mathcal{M}_{1}$
 defined as $v'=(\bar{v}, id_{K_{1}})$ and $u':  \mathcal{M}_{1} \to \mathcal{M}'_{1}$,  $u'=(\bar{u},\phi\circ \xi + id_{K_{1}})$ where $\bar{u}:\mathcal{N}\to S^{-r}\mathcal{N}'$ is the map $u$ with its target with a shifted filtration (these equations have to be interepreted component by component, as in the definition of the structure maps of cones of $A_{\infty}$-modules).
 There is also a homotopy $\xi':\mathcal{M}_{1}\to\mathcal{M}_{1}$, $\xi':v'\circ u'\simeq id$ given by the formula $\xi'=(\xi, 0)$.
 Notice that: $v'$ does not shift filtrations; $u'$ shifts action filtrations by $\leq \max\{r+s, k\}$;   $\xi'$ shifts filtration by $\leq k$.  As we can take $r, s,k$ larger but as close as desired to, respectively,  $A(v), A(u), A(\xi)$ this implies the claim. 
 \qed
 
 \subsubsection{Algebraic weights on $T^{S}D\fuk^{\ast}(M)$.}\label{subsubsec:alg-we} We  now use the measurement 
 $\rho$ introduced in \S\ref{subsubsec:energy-retr} to define an algebraic weight $w$, in the sense of \S\ref{subsubsec:weights1}. We will assume here that $\mathcal{X}=D\fuk^{\ast}(M)$ and that
 $\mathcal{X}_{0}$ consists of the the Yoneda modules associated to the Lagrangians in $\mathcal{L}ag^{\ast}(M)$. We will appeal to the constructions from \S\ref{sb:wf-fukaya}. 
 Recall from  Proposition \ref{p:wf-fukaya} that to a system of coherent perturbation data $p\in E'_{\textnormal{reg}}$ we associate a weakly filtered  $A_{\infty}$ category $\fuk(\mathcal{C};p)$. Recall also from \S\ref{sb:wf-fukaya} that $\mathcal{C}$ is the class of the Lagrangians in use, $\mathcal{C}=\mathcal{L}ag^{\ast}(M)$. We also recall that we denote by $\mathcal{N}$ the family of coherent perturbation data $\mathscr{D}=(K,J)$ with $K\equiv 0$.
Proposition  \ref{p:wf-fukaya} also shows that for $p_{0}\in\mathcal{N}$ the discrepancies 
of the categories $\fuk(\mathcal{C};p)$ tend to $0$ when  $p\to p_{0}$.  We will denote by 
$\fuk(\mathcal{C};p)^{\bigtriangleup}$ the category of all (finite) iterated weakly filtered  cones that one can
construct - as in \S\ref{sb:wf-mc} - out of the objects of $\fuk(\mathcal{C};p)$. There is clearly 
a functor  $\fuk(\mathcal{C};p)^{\bigtriangleup}\to D\fuk^{\ast}(M)$ that forgets filtrations on objects
and associates to each morphism its homology class (again, at the same time forgetting the filtration).
We denote the image of an object $X$ through this functor by $[X]$ and similarly for morphisms. The distinguished triangles in $D\fuk^{\ast}(M)$ are associated through this functor to the cone attachements in
$\fuk(\mathcal{C};p)^{\bigtriangleup}$ and there is a  similar correspondence for the iterated cones.

Let 
$\bar{\phi}: \mathcal{L}\to (\mathcal{L}_{1},\ldots, \mathcal{L}_{k})$, $\bar{\phi}=(\phi, a,\eta)$ be a morphism in $T^{S}\mathcal{X}_{0}$ (see \S\ref{subsubsec:weights1}). 
We define:
\begin{eqnarray}\label{eq:weight-p}
w_{p}(\bar{\phi})=\inf\{\ \rho(\alpha) \mid \alpha\in\mor_{\fuk(\mathcal{C};p)^{\bigtriangleup}},\  \alpha:\mathcal{L}\to \mathcal{M},\ \textnormal{such that}\hspace{0.5in}\\ \nonumber
\mathcal{M}\ \textnormal{admits\ an\ iterated cone decomposition} \ \bar{\eta}\ \textnormal{such that}\   [\alpha]=\phi,\ [\mathcal{M}]=a,\ [\bar{\eta}]=\eta\ \}
\end{eqnarray}
In summary, $w_{p}(\bar{\phi})$ infimizes $\rho$ among all the filtered models of the morphism
$\bar{\phi}$ inside $\fuk(\mathcal{C};p)$. By combining Lemma \ref{lem:bounds-rho} and (\ref{eq:ro-triang}) it is not difficult to deduce that $w_{p}$ satisfies (\ref{eq:weight-tr}) and thus this weight
can be used as in \S\ref{subsubsec:weights1} to define a pseudo-metric $\bar{s}^{\mathcal{F}}_{p}$. 
It is useful to define also similar notions for points $p_{0}\in \mathcal{N}$. For this purpose, we set:
$$w_{p_{0}}(\bar{\phi})=\limsup_{p\to p_{0}}(w_{p}(\bar{\phi}))~.~$$
It is easy to see that $w_{p_{0}}$ continues to satisfy (\ref{eq:weight-tr}) and therefore there is 
a corresponding weighted fragmentation pseudo-metric $\bar{s}^{\mathcal{F}}_{p_{0}}$.

In this setting, we deduce from the first part of Theorem \ref{thm:fission} (more precisely, from the proof of this 
theorem):

\begin{cor}\label{cor:alg-weights} Let $\bar{\phi}: \mathcal{L} \to (\mathcal{L}_{1},\ldots, \mathcal{L}_{k})$ be a morphism in $T^{S}D\fuk^{\ast}(M)$.
\begin{itemize}
\item[i.] There exists $p_{0}\in \mathcal{N}$ so that, with the notation in Theorem \ref{thm:fission}, we have
$$w_{p_{0}}(\bar{\phi})\geq \frac{1}{2}\delta(L; S)~.~$$ 
\item[ii.] If there exists a Lagrangian cobordism $V:L\cobto (L_{1},\ldots, L_{k})$ with $\widetilde{\Phi}(V)=\bar{\phi}$ ( where $\widetilde{\Phi}$ is the functor from (\ref{eq:functor-cob-TS})), then 
for any $p\in E'_{\textnormal{reg}}$ we have:
$$\mathcal{S}(V)\geq w_{p}(\bar{\phi})~.~$$
\end{itemize}
\end{cor}
This corollary is a refinement of the first part of Theorem \ref{thm:fission}. Using this  Corollary
we define a family of weighted fragmentation metrics on $\mathcal{L}ag^{\ast}(M)$ that are more algebraic in nature compared to the shadow 
metric $\widehat{d}^{\mathcal{F},\mathcal{F}'}$.

For $\bar{\phi}$ a morphism as in Corollary \ref{cor:alg-weights}, define 
$$w_{\textnormal{alg}}(\bar{\phi})=\sup_{p_{0}\in \mathcal{N}}w_{p_{0}}(\bar{\phi})~.~$$
It is immediate that the weight $w_{\textnormal{alg}}$  still satisfies (\ref{eq:weight-tr}) and, as in  \S\ref{subsubsec:weights1}, there is an associated pseudo-metric, $\bar{s}^{\mathcal{F}}_{\textnormal{alg}}$ on $\mathcal{L}ag^{\ast}(M)$. Point i of \ref{cor:alg-weights}  implies that Corollary \ref{cor:vanishing} remains valid with $\bar{s}^{\mathcal{F}}_{\textnormal{alg}}$ taking the place of $d^{\mathcal{F}}$. Moreover, if $\mathcal{F}$, $\mathcal{F}'$ satisfy the assumption in Corollary \ref{cor:deblurring}, then the formula 
\begin{equation}\label{eq:sum-pseudo}\widehat{s}^{\mathcal{F}, \mathcal{F}'}_{\textnormal{alg}}= \max\{\bar{s}^{\mathcal{F}}_{\textnormal{alg}},
\bar{s}^{\mathcal{F}'}_{\textnormal{alg}}\}
\end{equation}
defines a metric on $\mathcal{L}ag^{\ast}(M)$. 
Point ii of Corollary \ref{cor:alg-weights} shows that the metric $\widehat{s}^{\mathcal{F},\mathcal{F}'}_{\textnormal{alg}}$ is bounded from above by the
 shadow metric $\widehat{d}^{\mathcal{F},\mathcal{F}'}$ from \ref{cor:deblurring}.

\begin{rem} \label{rem:finite-alg} Assume  that $\mathcal{F}$ and $\mathcal{F}'$ satisfy the hypothesis in Corollary \ref{cor:deblurring} and that they both generate $D\fuk^{\ast}(M)$. In this case, the  weights
$w_{p}$ are finite and thus the pseudo-metrics
$\bar{s}^{\mathcal{F}}_{p}$ as well as $\widehat{s}^{\mathcal{F}, \mathcal{F}'}_{p}$ (which is defined by the obvious analogue of (\ref{eq:sum-pseudo})) are also finite. On the other hand,  for a fixed $p$
it is not clear that the pseudo-metric $\widehat{s}^{\mathcal{F}, \mathcal{F}'}_{p}$ is non-degenerate.
By contrast, $\widehat{s}^{\mathcal{F},\mathcal{F}'}_{\textnormal{alg}}$  is non-degenerate but
might be infinite.\end{rem}

\noindent\emph{Proof of Corollary \ref{cor:alg-weights}.} Let $\bar{\phi}=(\phi, a,\eta)$ and consider
a category $\fuk(\mathcal{C},p)$ and a map $\alpha: \mathcal{L}\to \mathcal{M}$ so that $[\alpha]=\phi$,
$[\mathcal{M}]=a$, and so that the cone-decomposition $\eta$ corresponds to the writing of
$\mathcal{M}$ as a weakly filtered iterated cone:
$$\mathcal{M}=\tcn (\mathcal{L}_{k}\to \tcn (\mathcal{L}_{k-1}\ldots \to \tcn (\mathcal{L}_{2}\to \mathcal{L}_{1})\ldots )~.~$$
Let $\beta:\mathcal{M}\to \mathcal{L}$ be another map and assume that $\zeta:\mathcal{L}\to \mathcal{L}$ is a homotopy so that
$\zeta:\beta\circ \alpha \simeq id_{\mathcal{L}}$. Assume that $\alpha$ shifts filtrations by $\leq s$, $\beta$ shifts filtrations by $\leq r$ and $\zeta$ shifts filtrations by $\leq k$. Consider $\mathcal{M}_{1}=Cone(\mathcal{M}\stackrel{\beta}{\longrightarrow}\mathcal{L})$ and the inclusion $i: \mathcal{L}\to \mathcal{M}_{1}$, $i=(0,id_{\mathcal{L}})$.  As described
in \S\ref{sb:wf-mc}, when defininig the cone $\mathcal{M}_{1}$ we use the value $r$ to write
$\mathcal{M}_{1}= S^{-r}\mathcal{M}\oplus \mathcal{L}$. We now notice that
the map $\bar{\zeta}=(\alpha,\zeta): \mathcal{L}\to \mathcal{M}_{1}$ is a homotopy $\bar{\zeta}:i\simeq 0$ and we
 see that $\bar{\zeta}$ shifts filtrations by $\leq \max\{r+s, k\}$. We deduce: 
\begin{equation}\label{eq:b-and-rho}
B_{h}(i)\leq \rho(\alpha)~.~
\end{equation}
Using this remark we now return to the setting of the proof of Theorem \ref{thm:fission}. In particular,
we pick the choice of perturbation data $p$ as in (\ref{eq:CF-L_0L_0}) and, for coherence of notation, we put
$L_{0}=L$.  Instead of the complex $\mathscr{C}_{p,h}$ which has a geometric construction
 we use the complex $\mathcal{M}_{1}(L_{0})$ constructed above.  The inequality   
(\ref{eq:boundary-fund})  is a consequence of (\ref{eq:Bh-C}).  If we replace inequality
 (\ref{eq:Bh-C}) with  (\ref{eq:b-and-rho}) ,  we can still deduce an inequality similar to (\ref{eq:boundary-fund}) but with $\rho(\alpha)$ instead of $\mathcal{S}(W)$. In other words, there is
 \begin{equation}\label{eq:boundary-fund2}
b' \in \mathcal{M}_{1}(L_{0})\ \textnormal{with}\ 
A(b'; \mathcal{M}_{1}(L_{0})) \leq A(e_{L_0};\mathcal{M}_{1}(L_{0})) +
\rho(\alpha)+ \frac{\epsilon}{2}\ .
\end{equation}
 The reason is that we do not need to use in this argument the boundary depth of the chain complex $\mathcal{M}_{1}(L_{0})$ (which in our algebraic setting might not even be acyclic)
 but only the boundary depth of the element $e_{L_{0}}$ wich is controled by the boundary depth of the 
 map $i:\mathcal{L}=\mathcal{L}_{0}\to \mathcal{M}_{1}$ which in turn is controled by $\rho(\alpha)$.  Given that $w_{p}(\bar{\phi})=\inf_{[\alpha]=\phi} \rho(\alpha)$) we may assume that $\rho(\alpha)\leq w_{p}(\bar{\phi})+\epsilon'''$  and by continuing as in the proof of Theorem \ref{thm:fission} we obtain, after making $p\to p_{0}$
 that there is a Floer polygon $v_{0}$ (compare to (\ref{eq:om-v_0})) such that 
 $$\omega(v_{0})\leq w_{p_{0}}(\bar{\phi})+\epsilon/2 +\epsilon'''~.~$$
 The argument ends by the same type of application of the Lelong inequality as in the proof of the Theorem \ref{thm:fission}.
  
 The proof of the second point of the Corollary is again basically contained in  the proof of Theorem \ref{thm:fission}. It uses the isotopy pictured in Figure \ref{f:gamma-gamma'}  but applies the construction there directly to the cobordism $V$ in Figure  \ref{f:cob-v-w} (and not to $W$). As in (\ref{eq:filtered-cob-it-cone-W}) we deduce the existence of a a weakly filtered module 
 \begin{equation} \label{eq:filtered-cob-it-cone-V2}
  \begin{aligned}
    & \mathcal{M}_{V; \gamma, p, h} = \; & \tcn
    (\mathcal{L}_k \xrightarrow{\; \phi_k \;}
    \tcn(\mathcal{L}_{k-1} \xrightarrow{\; \phi_{k-1} \;}
    \tcn( \cdots 
       \tcn(\mathcal{L}_2 \xrightarrow{\; \phi_2 \;}
    \mathcal{L}_1 )) {\cdot}{\cdot}{\cdot}))),
  \end{aligned}
\end{equation}
(where we neglect a small shift that can be made to $\to 0$). There is also
a similar module $\mathcal{M}_{V; \gamma', p, h}$ which is identified with the Yoneda module
of $L$. The isotopy $\gamma'\to \gamma$ of Hofer length $\leq \mathcal{S}(V)+\epsilon/2$ (see above 
(\ref{eq:Bh-C})) induces module homomorphisms (see for instance \cite{FO3:book-vol1} Chapter 5, as least 
for modules over an $A_{\infty}$-algebra, the case of $A_{\infty}$ categories is similar; alternatively, a direct argument based on moving boundary conditions is also possible)  $\alpha : \mathcal{M}_{V; \gamma', p, h}\to 
\mathcal{M}_{V; \gamma, p, h}$, $\beta: \mathcal{M}_{V; \gamma, p, h}\to 
\mathcal{M}_{V; \gamma', p, h}$ as well as homotopies $\eta:\beta\circ \alpha\simeq id$, $\eta':\alpha\circ \beta\simeq id$ that are all shifting actions by not more than $\mathcal{S}(V)+\epsilon/2$. By  the definition
of the functor $\widetilde{\Phi}$ we have that $\widetilde{\Phi}(V)=(\phi, a,\eta)$ and $[\alpha]= \phi$,
$a=[\mathcal{M}_{V; \gamma, p, h}]$ and, as we just indicated, we also have $\rho(\alpha)\leq \mathcal{S}(V)+\epsilon/2$. This means that by definition (\ref{eq:weight-p}), $w_{p}(\bar{\phi})\leq \mathcal{S}(V)+\epsilon/2$.

\qed


\bibliography{bibliography}

%
%
%

\end{document}